\pdfoutput=1
\documentclass[a4paper,11pt]{amsart}
\usepackage[hmarginratio={1:1},vmarginratio={1:1},lmargin=80.0pt,tmargin=90.0pt]{geometry}
\usepackage{lmodern}

\usepackage{calligra}
\usepackage[T1]{fontenc}
\usepackage[ugly]{nicefrac}
\usepackage{mathtools}
\mathtoolsset{mathic}
\usepackage[final]{microtype}
\usepackage{multicol}
\usepackage{tabularx,tabu}
\usepackage{latexsym,exscale,enumitem,amsfonts,amssymb,mathtools}
\usepackage{amsmath,amsthm,amsfonts,amssymb,amscd,textcomp,bbm}
\usepackage{stmaryrd}
\SetSymbolFont{stmry}{bold}{U}{stmry}{m}{n}
\usepackage[normalem]{ulem}
\usepackage{thmtools}
\usepackage{etex}
\usepackage{fancybox}
\usepackage[table]{xcolor}
\usepackage{mathrsfs}
\DeclareMathAlphabet{\mathpzc}{OT1}{pzc}{m}{it}
\DeclareMathAlphabet{\mathscrbf}{OMS}{mdugm}{b}{n}

\usepackage{caption} 
\usepackage[all]{xy}
\SelectTips{cm}{}

\definecolor{myorange}{RGB}{225,127,0}
\definecolor{mygreen}{RGB}{0,225,0}
\definecolor{mypurple}{RGB}{128,0,128}
\definecolor{myred}{RGB}{255,0,0}
\definecolor{myblue}{RGB}{0,0,195}
\definecolor{myyellow}{RGB}{210,210,0}
\definecolor{mycream}{RGB}{200,200,150}

\definecolor{dummy}{RGB}{10,10,10}
\definecolor{mygray}{gray}{0.7}

\definecolor{verdigris}{RGB}{65,165,165}
\definecolor{orchid}{RGB}{143,40,194}
\definecolor{lava}{RGB}{207,16,32}

\definecolor{mydarkblue}{RGB}{10,10,170}

\usepackage{tikz}
\tikzset{anchorbase/.style={baseline={([yshift=-0.5ex]current bounding box.center)}},
  tinynodes/.style={font=\tiny,text height=0.75ex,text depth=0.15ex},
  arc/.style={line width=1.5, color=black},
  arcd/.style={line width=1.5, color=white},
  arcdo/.style={line width=1.5, dotted, color=mygray},
  doline/.style={line width=1.0, color=orchid, densely dotted},
  daline/.style={line width=1.0, color=lava, densely dashed},
}
\usetikzlibrary{cd}
\usetikzlibrary{decorations}
\usetikzlibrary{decorations.markings}
\usetikzlibrary{decorations.pathreplacing}
\usetikzlibrary{decorations.pathmorphing}
\usetikzlibrary{shapes,positioning,matrix,calc}
\usetikzlibrary{shapes.callouts}
\tikzstyle directed=[postaction={decorate,decoration={markings,
    mark=at position #1 with {\arrow[line width=0.35mm]{>}}}}]
\tikzstyle rdirected=[postaction={decorate,decoration={markings,
    mark=at position #1 with {\arrow[line width=0.35mm]{<}}}}]
\tikzstyle ddirected=[postaction={decorate,decoration={markings,
    mark=at position #1 with {\arrow[line width=0.35mm]{<>}}}}]
\tikzstyle{snakeline} = [decorate, decoration={pre length=0.2cm,
                         post length=0.2cm, snake, amplitude=.4mm,
                         segment length=2mm},thick]

\makeatletter
\newcommand{\setword}[2]{%
  \phantomsection
  #1\def\@currentlabel{\unexpanded{#1}}\label{#2}%
}
\makeatother

\newcommand{\neatafrac}[2]{#1/#2}

\makeatletter
\providecommand{\leftsquigarrow}{%
  \mathrel{\mathpalette\reflect@squig\relax}%
}
\newcommand{\reflect@squig}[2]{%
  \reflectbox{$\m@th#1\rightsquigarrow$}%
}
\makeatother

\newcommand{\K}{\mathbb{K}}
\newcommand{\fieldp}[1][p]{\mathbb{F}_{#1}}
\newcommand{\Z}{\mathbb{Z}}
\newcommand{\N}{\mathbb{Z}_{\geq 0}}
\newcommand{\Np}{\mathbb{Z}_{> 0}}
\newcommand{\R}{\mathbb{R}}
\newcommand{\Sp}{\mathbb{S}^{1}}

\newcommand{\sets}[1]{\mathsf{#1}}
\newcommand{\elem}[1]{\mathpzc{#1}}
\newcommand{\algebras}[1]{\mathfrak{#1}}

\newcommand{\anyalg}{\algebras{A}}
\newcommand{\calg}{\algebras{R}}
\newcommand{\ualg}{\algebras{C}}

\newcommand{\core}{\algebras{C}ore}

\newcommand{\cset}{\sets{X}}
\newcommand{\cmset}{\sets{M}}
\newcommand{\ciset}{\sets{E}}
\newcommand{\coset}{\sets{O}}
\newcommand{\cbasis}{\sets{C}}

\newcommand{\ceps}{\boldsymbol{\varepsilon}}
\newcommand{\cepsmap}{\boldsymbol{\underline{\varepsilon}}}

\newcommand{\csetf}{|\sets{X}|<\infty}

\newcommand{\cbas}[2]{\mathtt{C}_{#1}^{#2}}

\newcommand{\invo}{\star}

\newcommand{\order}{<}

\newcommand{\ord}[1]{<_{#1}}
\newcommand{\Ord}[1]{\leq_{#1}}

\newcommand{\csubset}{\sets{I}}
\newcommand{\cideal}{\mathrm{I}}
\newcommand{\ccideal}{\sets{J}}
\newcommand{\cccideal}{\sets{J}^{\oplus}}

\newcommand{\cpair}[1]{\phi_{#1}}
\newcommand{\Cpair}[1]{\phi^{#1}}

\newcommand{\prmod}{P}
\newcommand{\dmod}{\Delta}
\newcommand{\lmod}{L}

\newcommand{\dbas}[2]{\mathtt{M}_{#1}^{#2}}

\newcommand{\diso}[1]{\Theta^{#1}}
\newcommand{\Diso}{\Theta^{-1}}
\newcommand{\piso}[1]{\Gamma^{#1}}

\newcommand{\dnumber}[1]{\mathrm{d}_{#1}}
\newcommand{\dmatrix}{\elem{D}}
\newcommand{\cnumber}[1]{\mathrm{c}_{#1}}
\newcommand{\cmatrix}{\elem{C}}

\newcommand{\img}{\mathrm{im}}

\newcommand{\acts}{\centerdot}

\newcommand{\isum}{\,\raisebox{.225cm}{\rotatebox{270}{$\uplus$}}\,}

\newcommand{\friends}[1][]{(\dagger_{#1})}

\newcommand{\radi}[1]{\mathrm{rad}(#1)}
\newcommand{\Radi}{\mathrm{Rad}}

\newcommand{\csetz}{\sets{X}_{0}}

\newcommand{\homc}{\mathrm{Hom}_{\calg}}
\newcommand{\homk}{\mathrm{Hom}_{\K}}

\newcommand{\ehomc}{\mathrm{End}_{\calg}}

\newcommand{\extc}{\mathrm{Ext}_{\calg}}

\newcommand{\kdual}{\mathcal{D}}

\newcommand{\typeA}{\mathrm{A}}
\newcommand{\typeAt}{\tilde{\mathrm{A}}}

\newcommand{\marrow}{\tikz[baseline=-2.5,scale=0.25]{\draw[->] (0,0) to (.75,0);}}

\newcommand{\videm}[1]{\mathtt{#1}}
\newcommand{\paths}[2]{\videm{#1}\marrow\videm{#2}}
\newcommand{\pathss}[3]{\videm{#1}\marrow\videm{#2}\marrow\videm{#3}}
\newcommand{\pathsss}[4]{\videm{#1}\marrow\videm{#2}\marrow\videm{#3}\marrow\videm{#4}}
\newcommand{\pathssss}[5]{\videm{#1}\marrow\videm{#2}\marrow\videm{#3}\marrow\videm{#4}\marrow\videm{#5}}

\newcommand{\funnyvee}{\tikz[baseline=-4.5,scale=0.25]{\draw[very thick, verdigris] (0,0) to (.5,-1) to (1,0);}}
\newcommand{\funnywedge}{\tikz[baseline=3.5,scale=0.25]{\draw[very thick, verdigris] (0,0) to (.5,1) to (1,0);}}

\newcommand{\funnyVee}{\tikz[baseline=-4.5,scale=0.25]{\draw[very thick, mygreen] (0,0) to (.5,-1) to (1,0);}}
\newcommand{\funnyWedge}{\tikz[baseline=3.5,scale=0.25]{\draw[very thick, mygreen] (0,0) to (.5,1) to (1,0);}}

\newcommand{\aarc}[1][n]{\algebras{A}rc^{\mathrm{ann}}_{#1}}
\newcommand{\uarc}[1][n]{\algebras{A}rc_{#1}}

\newcommand{\numberarcs}{n}

\newcommand{\down}{{\scriptstyle\vee}}
\newcommand{\up}{{\scriptstyle\wedge}}

\newcommand{\Down}{\funnyvee}
\newcommand{\Up}{\funnywedge}
\newcommand{\Downn}{\funnyVee}
\newcommand{\Upp}{\funnyWedge}

\newcommand{\rotright}{\rho}

\newcommand{\cups}{\tikz[baseline=-5.9,scale=0.4]{\draw[thin, white] (0,-.5) to [out=90, in=180] (.25,0) to [out=0, in=90] (.5,-.5);\draw[thin] (0,0) to [out=270, in=180] (.25,-.5) to [out=0, in=270] (.5,0);}}
\newcommand{\caps}{\tikz[baseline=-5.9,scale=0.4]{\draw[thin, white] (0,0) to [out=270, in=180] (.25,-.5) to [out=0, in=270] (.5,0);\draw[thin] (0,-.5) to [out=90, in=180] (.25,0) to [out=0, in=90] (.5,-.5);}}

\newcommand{\circles}{\mathcal{C}}

\newcommand{\icup}{i\hspace*{.01cm}\cups}
\newcommand{\icap}{i\hspace*{.01cm}\caps}
\newcommand{\ecup}{e\hspace*{.01cm}\cups}
\newcommand{\ecap}{e\hspace*{.01cm}\caps}

\newcommand{\lcup}{l\hspace*{.01cm}\cups}
\newcommand{\lcap}{l\hspace*{.01cm}\caps}
\newcommand{\ucup}{u\hspace*{.01cm}\cups}
\newcommand{\ucap}{u\hspace*{.01cm}\caps}

\newcommand{\setp}[1][p]{\sets{F}_{#1}}

\newcommand{\slt}[1][2]{\mathfrak{sl}_{#1}}
\newcommand{\uslt}[1][0]{\algebras{u}_{#1}(\slt)}

\newcommand{\qpar}{\elem{q}}
\newcommand{\suslt}{\algebras{u}_{\qpar}(\slt)}

\newcommand{\egen}{\elem{E}}
\newcommand{\fgen}{\elem{F}}
\newcommand{\hgen}{\elem{H}}

\newcommand{\hidem}{\elem{1}}
\newcommand{\hidemb}{{\color{black}\elem{1}}}

\theoremstyle{definition}
\newtheorem{theoremm}{Theorem}[section]

\declaretheorem[style=definition,name=Theorem,qed=$\square$,numberlike=theoremm]{theorem}
\declaretheorem[style=definition,name=Theorem,qed=$\blacksquare$,numberlike=theoremm]{theoremqed}
\declaretheorem[style=definition,name=Corollary,qed=$\blacksquare$,numberlike=theoremm]{corollary}
\declaretheorem[style=definition,name=Lemma,qed=$\square$,numberlike=theoremm]{lemma}
\declaretheorem[style=definition,name=Lemma,qed=$\blacksquare$,numberlike=theoremm]{lemmaqed}

\declaretheorem[style=definition,name=Proposition,qed=$\square$,numberlike=theoremm]{proposition}
\declaretheorem[style=definition,name=Proposition,qed=$\blacksquare$,numberlike=theoremm]{propositionqed}

\declaretheorem[style=definition,name=Example,qed=$\blacktriangle$,numberlike=theorem]{example}
\declaretheorem[style=definition,name=Definition,qed=$\blacktriangle$,numberlike=theorem]{definition}
\declaretheorem[style=definition,name=Remark,qed=$\blacktriangle$,numberlike=theorem]{remark}

\declaretheorem[style=definition,name=Further directions,qed=$\blacktriangle$,numberlike=theorem]{furtherdirections}

\allowdisplaybreaks

\numberwithin{equation}{section}
\usepackage[hypertexnames=false]{hyperref}
\usepackage{cleveref,bookmark}
\hypersetup{
    pdftoolbar=true,        
    pdfmenubar=true,        
    pdffitwindow=false,     
    pdfstartview={FitH},    
    pdftitle={Relative cellular algebras},    
    pdfauthor={Michael Ehrig and Daniel Tubbenhauer},     
    pdfsubject={},   
    pdfcreator={Michael Ehrig and Daniel Tubbenhauer},   
    pdfproducer={Michael Ehrig and Daniel Tubbenhauer}, 
    pdfkeywords={}, 
    pdfnewwindow=true,      
    colorlinks=true,       
    linkcolor=mydarkblue,          
    citecolor=teal,        
    filecolor=magenta,      
    urlcolor=orchid,          
    linkbordercolor=lava,
    citebordercolor=teal,
    urlbordercolor=orchid,  
    linktocpage=true
}
\let\fullref\autoref
\def\makeautorefname#1#2{\expandafter\def\csname#1autorefname\endcsname{#2}}
\makeautorefname{equation}{Equation}%
\makeautorefname{footnote}{footnote}%
\makeautorefname{item}{item}%
\makeautorefname{figure}{Figure}%
\makeautorefname{table}{Table}%
\makeautorefname{part}{Part}%
\makeautorefname{appendix}{Appendix}%
\makeautorefname{chapter}{Chapter}%
\makeautorefname{section}{Section}%
\makeautorefname{subsection}{Section}%
\makeautorefname{subsubsection}{Section}%
\makeautorefname{paragraph}{Paragraph}%
\makeautorefname{subparagraph}{Paragraph}%
\makeautorefname{theorem}{Theorem}%
\makeautorefname{theo}{Theorem}%
\makeautorefname{thm}{Theorem}%
\makeautorefname{addendum}{Addendum}%
\makeautorefname{addend}{Addendum}%
\makeautorefname{add}{Addendum}%
\makeautorefname{maintheorem}{Main theorem}%
\makeautorefname{mainthm}{Main theorem}%
\makeautorefname{corollary}{Corollary}%
\makeautorefname{corol}{Corollary}%
\makeautorefname{coro}{Corollary}%
\makeautorefname{cor}{Corollary}%
\makeautorefname{lemma}{Lemma}%
\makeautorefname{lemm}{Lemma}%
\makeautorefname{lem}{Lemma}%
\makeautorefname{sublemma}{Sublemma}%
\makeautorefname{sublem}{Sublemma}%
\makeautorefname{subl}{Sublemma}%
\makeautorefname{proposition}{Proposition}%
\makeautorefname{proposit}{Proposition}%
\makeautorefname{propos}{Proposition}%
\makeautorefname{propo}{Proposition}%
\makeautorefname{prop}{Proposition}%
\makeautorefname{property}{Property}
\makeautorefname{proper}{Property}
\makeautorefname{scholium}{Scholium}%
\makeautorefname{step}{Step}%
\makeautorefname{conjecture}{Conjecture}%
\makeautorefname{conject}{Conjecture}%
\makeautorefname{conj}{Conjecture}%
\makeautorefname{question}{Question}
\makeautorefname{questn}{Question}
\makeautorefname{quest}{Question}
\makeautorefname{ques}{Question}
\makeautorefname{qn}{Question}
\makeautorefname{definition}{Definition}%
\makeautorefname{defin}{Definition}%
\makeautorefname{defi}{Definition}%
\makeautorefname{def}{Definition}%
\makeautorefname{dfn}{Definition}%
\makeautorefname{notation}{Notation}
\makeautorefname{nota}{Notation}
\makeautorefname{notn}{Notation}
\makeautorefname{remark}{Remark}%
\makeautorefname{rema}{Remark}%
\makeautorefname{rem}{Remark}%
\makeautorefname{rmk}{Remark}%
\makeautorefname{rk}{Remark}%
\makeautorefname{remarks}{Remarks}%
\makeautorefname{rems}{Remarks}%
\makeautorefname{rmks}{Remarks}%
\makeautorefname{rks}{Remarks}%
\makeautorefname{example}{Example}%
\makeautorefname{examp}{Example}%
\makeautorefname{exmp}{Example}%
\makeautorefname{exam}{Example}%
\makeautorefname{exa}{Example}%
\makeautorefname{algorithm}{Algorith}%
\makeautorefname{algo}{Algorith}%
\makeautorefname{alg}{Algorith}%
\makeautorefname{axiom}{Axiom}%
\makeautorefname{axi}{Axiom}%
\makeautorefname{ax}{Axiom}%
\makeautorefname{case}{Case}%
\makeautorefname{claim}{Claim}%
\makeautorefname{clm}{Claim}%
\makeautorefname{assumption}{Assumption}%
\makeautorefname{assumpt}{Assumption}%
\makeautorefname{conclusion}{Conclusion}%
\makeautorefname{concl}{Conclusion}%
\makeautorefname{conc}{Conclusion}%
\makeautorefname{condition}{Condition}%
\makeautorefname{condit}{Condition}%
\makeautorefname{cond}{Condition}%
\makeautorefname{construction}{Construction}%
\makeautorefname{construct}{Construction}%
\makeautorefname{const}{Construction}%
\makeautorefname{cons}{Construction}%
\makeautorefname{criterion}{Criterion}%
\makeautorefname{criter}{Criterion}%
\makeautorefname{crit}{Criterion}%
\makeautorefname{exercise}{Exercise}%
\makeautorefname{exer}{Exercise}%
\makeautorefname{exe}{Exercise}%
\makeautorefname{problem}{Problem}%
\makeautorefname{problm}{Problem}%
\makeautorefname{probm}{Problem}%
\makeautorefname{prob}{Problem}%
\makeautorefname{solution}{Solution}%
\makeautorefname{soln}{Solution}%
\makeautorefname{sol}{Solution}%
\makeautorefname{summary}{Summary}%
\makeautorefname{summ}{Summary}%
\makeautorefname{sum}{Summary}%
\makeautorefname{operation}{Operation}%
\makeautorefname{oper}{Operation}%
\makeautorefname{observation}{Observation}%
\makeautorefname{observn}{Observation}%
\makeautorefname{obser}{Observation}%
\makeautorefname{obs}{Observation}%
\makeautorefname{ob}{Observation}%
\makeautorefname{convention}{Convention}%
\makeautorefname{convent}{Convention}%
\makeautorefname{conv}{Convention}%
\makeautorefname{cvn}{Convention}%
\makeautorefname{warning}{Warning}%
\makeautorefname{warn}{Warning}%
\makeautorefname{note}{Note}%
\makeautorefname{furtherdirections}{Further directions}%

\begin{document}
\vbadness=10001
\hbadness=10001
\title[Relative cellular algebras]{Relative cellular algebras}
\author[M. Ehrig and D. Tubbenhauer]{Michael Ehrig and Daniel Tubbenhauer}

\address{M.E.: Beijing Institute of Technology, School of Mathematics and Statistics, Liangxiang Campus of Beijing Institute of Technology, Fangshan District, 100488 Beijing, China}
\email{micha.ehrig@outlook.com}

\address{D.T.: Institut f\"ur Mathematik, Universit\"at Z\"urich, Winterthurerstrasse 190, Campus Irchel, Office Y27J32, CH-8057 Z\"urich, Switzerland, \href{www.dtubbenhauer.com}{www.dtubbenhauer.com}}
\email{daniel.tubbenhauer@math.uzh.ch}

\begin{abstract}
In this paper we generalize 
cellular algebras by allowing 
different partial orderings relative to 
fixed idempotents. For these 
relative cellular algebras we classify and construct simple modules, and 
we obtain other characterizations in analogy to cellular algebras.

We also give several examples
of algebras that are relative cellular,
but not cellular. Most prominently, the restricted enveloping 
algebra and the small quantum group for $\slt$, and an annular version of arc algebras.
\end{abstract}
\maketitle
\tableofcontents
\renewcommand{\theequation}{\thesection-\arabic{equation}}
%
%%%%%%%%%%%%%%%%%%%%%%%%%%%
\section{Introduction}\label{section:intro}
%%%%%%%%%%%%%%%%%%%%%%%%%%%

Arguably the two main problems in the representation theory of, say, algebras
are the classification and the construction of simple modules. However, for most algebras 
both problems -- non-linear in nature -- are out of reach.

In pioneering work \cite{gl1} Graham--Lehrer 
introduced the notion of a \textit{cellular algebra}, i.e. 
an algebra equipped with a so-called \textit{cell datum}. 
For example, of key importance for this paper, the cell datum 
comes with a set $\cset$ and a partial order $<$ on it; the latter  
plays an important role since it yields an ``upper triangular way'' 
to construct certain ``standard, easy'' modules, called \textit{cell modules}.
The usefulness of 
the cell datum comes from the fact that it provides a method 
to systematically reduce hard questions about the 
representation theory of such algebras to problems in 
linear algebra. 
In well-behaved cases these linear algebra problems 
can be solved, giving e.g. a parametrization of 
the isomorphism classes of simple modules 
via a subset of $\cset$, and a construction 
of a representative for each class. Thus, cellular algebras 
provide a method to solve the
classification and the construction problem. Other 
upshots of cellular algebras are that they have certain
reciprocity laws -- allowing to recover the multiplicities 
of simple modules in indecomposable projective modules via the multiplicities of simple modules in cell modules -- 
or that they give various ways to study the blocks of the algebra in question.

After Graham--Lehrer's paper appeared a lot of 
interesting algebras have found to be cellular -- among the more popular 
ones are various diagram algebras and Hecke algebras of finite Coxeter type -- 
and proving cellularity of algebras has turned out to be a very useful tool 
in representation theory. In fact, another motivation for studying cellular
algebras is to understand these various examples of the theory by putting them into an
axiomatic framework, revealing hidden connections.
However, by far not all algebras are cellular since e.g. their Cartan matrix 
has to be positive definite.

In this paper we (strictly) generalize the 
notion of a cellular algebra to what we call 
a \textit{relative cellular algebra}, i.e. 
an algebra equipped with a \textit{relative cell datum}. 
For example, the relative cell datum 
comes with a set $\cset$, but now with several partial 
orders $\ord{\ceps}$ on it, one for each idempotent $\ceps$ 
from a preselected set of idempotents. 
Taking only one idempotent $\ceps=1$, namely 
the unit, and only one partial order $\ord{1}=<$, we recover the 
setting of Graham--Lehrer. 

Surprisingly, most of the theory of cellular algebras 
still works in this relative setup. Thus, relative 
cellular algebras generalize the useful framework of cellular 
algebras to a larger class. For example, relative cellular algebras 
can have a positive semidefinite Cartan matrix.

However, the proofs are 
fairly different from the original ones, carefully incorporating the 
various partial orders.
The purpose of our 
paper is to explain this in detail.

Along the way we give examples of algebras 
that are relative cellular, but not cellular in the sense of 
Graham--Lehrer.

%%%%%%%%%%%%%%%%%%%%%%%%%%%
\subsection*{The papers content in a nutshell}\label{subsection:intro-a}
%%%%%%%%%%%%%%%%%%%%%%%%%%%

Our exposition follows closely \cite{gl1}.

\begin{enumerate}[label=(\roman*)]

\setlength\itemsep{.15cm}

\item In \fullref{section:basic-cell} we introduce our generalization 
of cellularity. The crucial new ingredient 
hereby is (\ref{definition:cell-algebra}.c) asking for a 
set $\ciset$ of idempotents and partial orders 
$\ord{\ceps}$ for each 
$\ceps\in\ciset$. Then 
we define cell modules for relative cellular algebras, 
and discuss a basis free version of 
relative cellularity.
Further, in \fullref{subsection:cell-examples}, 
we give some first non-trivial examples of 
relative cellular algebras that are not cellular. 

\item \fullref{section:basic-cell-props} is the main technical heart of 
the paper where we recover relative versions of some of the 
facts that hold for cellular algebras. 
Most prominently, the construction and classification
of simple modules in \fullref{theorem:simple-set}, and 
some reciprocity laws in \fullref{subsec:BGG}.

\item In the fourth 
section, see \fullref{section:rel-cell-res},
we show that the restricted enveloping algebras of $\slt$  
in positive characteristic are relative cellular algebras. 
We recover the entire 
(well-known, of course) representation theory of these 
algebras from the general theory of relative cellular algebras. 
We note that the case of the small 
quantum groups for $\slt$ at roots of unity works mutatis mutandis, 
giving very similar statements.

\item Finally, in \fullref{section:arc-stuff}
we discuss another, and in some sense the motivating, 
example for relative cellularity:
an annular version of arc algebras. 
We think of this section as being interesting in its own right 
since annular arc algebras have potential
connections to e.g. homological knot theory, 
exotic $t$-structures, Springer fibers and 
modular representation theory.

\end{enumerate}

Moreover, we tried to make the paper reasonably 
self-contained, and we tried to 
keep the exposition as easy as possible.
In fact, throughout the text 
we have included several remarks 
about potential further directions.

\begin{remark}\label{remark:sbased}
Note that any finite-dimensional algebra over an algebraically closed field is 
standardly based (``cellular without involution'')
in the sense of \cite{duru}, cf. \cite[Theorem 6.4.1]{cozh}.
However, the ``naive'' standard defining base which one can produce via an algorithm 
can be fairly useless. We see relative cellular algebras as 
being in between cellular and standardly based algebras, keeping 
some of the nice properties of cellular algebras as e.g.
reciprocity laws, a symmetric $\mathrm{Ext}$-quiver and a more 
useful cell structure as we will see in our examples.
\end{remark}

\subsection*{Conventions}\label{subsec:conventions}

We work over any field $\K$ and 
algebras, maps etc. are assumed to 
be over $\K$, $\K$-linear etc., and 
$\otimes=\otimes_{\K}$. Moreover, 
if not stated otherwise we work with 
finite-dimensional, left modules. 
(Even for potentially infinite-dimensional algebras.)
By an idempotent $\ceps$ we always understand 
a non-zero element in some algebra $\anyalg$ with 
$\ceps^{2}=\ceps$.

We use some colors in this paper, none of which are essential, 
and reading the paper in black-and-white is entirely possible.

\subsection*{Acknowledgements}\label{subsec:acknowledgements}

We thank Gwyn Bellamy, Kevin Coulembier, Andrew Mathas,
Catharina Stroppel and Oded Yacobi for discussions and inspirations, 
and the referees for a careful reading of the manuscript and helpful comments. 
M.E. likes to thank Vinoth Nandakumar for making him aware of the annular arc algebra.
We acknowledge a still nameless toilet paper 
roll for visualizing the concept of a cylinder for us.

A part of this paper was written during the 
Junior Hausdorff Trimester Program ``Symplectic Geometry and 
Representation Theory'' of the Hausdorff Research Institute 
for mathematics (HIM) in Bonn, 
and the hospitality of the HIM during this 
period is gratefully acknowledged.
M.E. was partially supported by the Australian Research Council Grant DP150103431.
\section{Relative cellularity}\label{section:basic-cell}

%%%%%%%%%%%%%%%%%%%%%%%%%%%
\subsection{A generalization of cellularity}\label{subsection:cell-def}
%%%%%%%%%%%%%%%%%%%%%%%%%%%

Following \cite{gl1} we define:

\begin{definition}\label{definition:cell-algebra}
A \textit{relative cellular algebra} is an associative algebra $\calg$ together with a \textit{(relative) cell datum}, i.e.
\begin{gather}\label{eq:rel-cell-datum}
(\cset,\cmset,\cbasis,
{}^{\invo},
\ciset,\coset,\cepsmap)
\end{gather}
such that the following hold.
\smallskip
\begin{enumerate}[label=(\alph*)]

\setlength\itemsep{.15cm}

\item We have a set $\cset$, 
and $\cmset = \{\cmset(\lambda)\mid\lambda\in\cset\}$ is
a collection of finite, non-empty sets such that 
\begin{gather}
\cbasis\colon 
{\textstyle\coprod_{\lambda\in\cset}}\,
\cmset(\lambda)\times\cmset(\lambda)\rightarrow\calg
\end{gather}
is an injective map with image forming a basis of $\calg$.  
For $S,T\in\cmset(\lambda)$ we write 
$\cbasis(S,T)=\cbas{S,T}{\lambda}$ from now on.

\item We have an anti-involution 
${}^{\invo}\colon\calg\to\calg$ such that 
$(\cbas{S,T}{\lambda})^\invo=\cbas{T,S}{\lambda}$.

\item We have a set $\ciset$ of pairwise orthogonal 
idempotents, all fixed by 
${}^{\invo}$, i.e. $\ceps^{\invo}=\ceps$ for all $\ceps\in\ciset$.
Further, $\coset=\{\ord{\ceps}\mid\ceps\in\ciset\}$ 
is a set
of partial orders $\ord{\ceps}$ on $\cset$, and  
$\cepsmap$ is a map 
$\cepsmap\colon\coprod_{\lambda\in\cset}
\cmset(\lambda)\rightarrow\ciset$ sending $S$ 
to $\cepsmap(S)=\ceps_{S}$ such that
\\
\noindent\begin{tabularx}{0.9\textwidth}{XX}
\begin{equation}\hspace{-7.5cm}\label{eq:idem-props-1}
\ceps\calg\ceps\,\cbas{S,T}{\lambda} \in \calg(\Ord{\ceps}\!\lambda),
\phantom{\begin{cases}
a,
\\
b,
\end{cases}}\hspace*{-1cm}
\end{equation} &
\begin{equation}\hspace{-6.5cm}\label{eq:idem-props-2}
\ceps\cbas{S,T}{\lambda}=
\begin{cases}
\cbas{S,T}{\lambda}, &\text{if }\ceps_{S}=\ceps,
\\
0, &\text{if }\ceps_{S}\neq\ceps,
\end{cases}
\end{equation}
\end{tabularx}\\
for all $\lambda\in\cset$, $S,T\in\cmset(\lambda)$ and $\ceps \in \ciset$.
Hereby, for $\ceps\in\ciset$, we let
\begin{gather}\label{eq:notation-new}
\calg(\Ord{\ceps}\!\lambda)=\K\{
\cbas{S,T}{\lambda}\mid\mu\in\cset,\mu\Ord{\ceps}\!\lambda,S,T\in\cmset(\mu)
\},
\end{gather}
a notation which we also use for $\ord{\ceps}$ rather than for $\Ord{\ceps}$, having the evident meaning.

\item For $\lambda\in\cset$, 
$S,T\in\cmset(\lambda)$ and $a\in\calg$ we have
\begin{gather}\label{eq:mult-left}
a \cbas{S,T}{\lambda} 
\in
{\textstyle\sum_{S^{\prime} \in \cmset(\lambda)}}\, 
r_{a}(S^{\prime},S)\,\cbas{S^{\prime},T}{\lambda} 
+ 
\calg(\ord{\ceps_{T}}\!\lambda)
\ceps_{T},
\end{gather}
with scalars $r_{a}(S^{\prime},S)\in\K$ only depending on 
$a,S,S^{\prime}$.
\end{enumerate}
\smallskip

We call the set 
$\{\cbas{S,T}{\lambda}\mid\lambda\in\cset,S,T\in \cmset(\lambda)\}$ 
a \textit{relative cellular basis}.
\end{definition}

The first examples of relative cellular 
algebras are cellular algebras $\ualg$ in the sense 
of \cite[Definition 1.1]{gl1}.
As we will see in (\ref{proposition:cell-relative-cell}.b) below, 
the relative cell datum in this case is
$(\cset,\cmset,\cbasis,{}^{\invo},\{1\},\{\ord{1}\},\cepsmap)$, 
with $\cepsmap$ mapping everything to $1$.

As in the cellular setup, a relative 
cell datum is not unique. 
Nevertheless, we say that an algebra 
$\calg$ \textit{is relative cellular} if 
there exist some relative cell datum. 
(Similarly, if we have already fixed part 
of the relative cell datum as e.g. 
the anti-involution ${}^{\invo}$.)

\begin{remark}\label{remark:fin-dim}
The basic 
properties of relative cellular algebras do not 
require $\csetf$; an extra assumption equivalent to $\calg$ being finite-dimensional, cf. 
(\ref{definition:cell-algebra}.a).
However, numerous results later on, for example 
\fullref{theorem:simple-set}, will make this additional assumption.
\end{remark}

The following is our version of an observation from \cite[Remark 2.4]{GoGr-cellularity-jones-basic}.

\begin{lemma}\label{lemma:involution-char-2}
Let $\mathrm{char}(\K)\neq 2$. 
If $\calg$ has a datum as in \fullref{definition:cell-algebra} 
except that
\begin{gather}
(\cbas{S,T}{\lambda})^\invo=\cbas{T,S}{\lambda}+\calg(\ord{\ceps_{T}}\!\lambda)
\end{gather}
holds instead of (\ref{definition:cell-algebra}.b), then $\calg$ is relative cellular.
\end{lemma}

\begin{proof}
The proof is the same as in \cite[Remark 2.4]{GoGr-cellularity-jones-basic}: 
The condition $(\cbas{S,T}{\lambda})^\invo=\cbas{T,S}{\lambda}+\calg(\ord{\ceps_{T}}\!\lambda)$ 
implies that, for all $\lambda\in\cset$ and 
$S,T\in\cmset(\lambda)$, we can find a unique $f(\lambda,S,T)\in\calg(\ord{\ceps_{T}}\!\lambda)$ such that 
$(\cbas{S,T}{\lambda})^\invo=\cbas{T,S}{\lambda}+f(\lambda,S,T)$. Then the set
$\{\cbas{S,T}{\lambda}+\tfrac{1}{2}f(\lambda,S,T)\mid\lambda\in\cset,S,T\in\cmset(\lambda)\}$ 
can be taken as a relative cellular basis.
\end{proof}

\begin{remark}\label{remark:involution-char-2}
Note that \fullref{lemma:involution-char-2} implies that 
imposing $(\cbas{S,T}{\lambda})^\invo=\cbas{T,S}{\lambda}+\calg(\ord{\ceps_{T}}\!\lambda)$ 
is equivalent to imposing $(\cbas{S,T}{\lambda})^\invo=\cbas{T,S}{\lambda}$ unless $\mathrm{char}(\K)=2$. 
However, in contrast to the case of cellular algebras where $\cepsmap$ is constant, 
$(\cbas{S,T}{\lambda})^\invo=\cbas{T,S}{\lambda}+\calg(\ord{\ceps_{T}}\!\lambda)$ is 
not symmetric (this comes from our choice to 
work with left modules) and some of our arguments in \fullref{section:basic-cell-props} fail 
if we would only require $(\cbas{S,T}{\lambda})^\invo=\cbas{T,S}{\lambda}+\calg(\ord{\ceps_{T}}\!\lambda)$ 
instead of $(\cbas{S,T}{\lambda})^\invo=\cbas{T,S}{\lambda}$.
\end{remark}

\begin{furtherdirections}\label{remark:ground-field}
We could also work more generally over rings 
instead of the field $\K$, as e.g. Graham--Lehrer \cite{gl1}. 
This could be useful to extend the notion of relative cellularity 
to some affine setup as in \cite{kx3}. 
However, most of the results in \fullref{section:basic-cell-props} use 
the fact that we work over a 
field. So, for convenience, we decided not to do so.
\end{furtherdirections}

If not stated otherwise, fix a relative cellular 
algebra $\calg$ in the following. Moreover, let us introduce a notation that will appear throughout 
the paper: for a subset $\csubset\subset\cset$ 
we fix the linear subspace
\begin{gather}\label{eq:ideal-algebra-notation}
\calg(\csubset)
=
\K\{
\cbas{S,T}{\lambda}
\mid
\lambda \in \csubset,S,T\in\cmset(\lambda)
\}
\subset
\calg.
\end{gather}
Often these subspaces will be defined with 
respect to $\ord{\ceps}$, for this 
we abuse notation and, for example, $\calg(\ord{\ceps}\!\lambda)$ 
can be understood as 
$\calg(\{\mu\in\cset\mid\mu\ord{\ceps}\!\lambda\})$ 
and similar for analogous expressions.
Further, by an \textit{ideal $\cideal$ in the 
poset $(\cset,\ord{\ceps})$}, 
\textit{$\ord{\ceps}$-ideal} for short, we understand 
a subset of $\emptyset\neq\cideal\subset\cset$ such that $\cideal$ 
is a directed, lower set in the order-theoretical sense.  
(For example, $\ord{\ceps}\!\lambda=\{\mu\in\cset\mid\mu\ord{\ceps}\!\lambda\}$ is 
an $\ord{\ceps}$-ideal.)

%%%%%%%%%%%%%%%%%%%%%%%%%%%
\subsection{First properties}\label{subsection:cell-lemmas}
%%%%%%%%%%%%%%%%%%%%%%%%%%%

The (very basic) statements below will 
be crucial for the definition of cell modules.

\begin{lemma}\label{lemma:cell-algebra-star-1}
The following properties hold.
\smallskip
\begin{enumerate}[label=(\alph*)]

\setlength\itemsep{.15cm}

\item For 
$\lambda\in\cset$, $S,T\in\cmset(\lambda)$, and 
$\ceps\in\ciset$ we have
\begin{gather}
\cbas{S,T}{\lambda}\,\ceps\calg\ceps\in\calg(\Ord{\ceps}\!\lambda),
\quad\quad 
\cbas{S,T}{\lambda}\ceps =
\begin{cases}
\cbas{S,T}{\lambda}, &\text{if }\ceps_{T}=\ceps,
\\
0, &\text{if }\ceps_{T}\neq\ceps.
\end{cases}
\end{gather}

\item If 
$\ceps\in\ciset$ and $\csubset\subset\cset$, then
$\ceps\calg(\csubset)\subset
\calg(\csubset)\supset\calg(\csubset)\ceps$.

\item For an $\ord{\ceps}$-ideal $\cideal_{\ceps}$
we have that $\calg(\cideal_{\ceps})\ceps$ is 
a left and $\ceps\calg(\cideal_{\ceps})$ 
is a right ideal in $\calg$.

\item For $\lambda\in\cset$, $S,T\in\cmset(\lambda)$, 
and $a\in\calg$ we have
\begin{gather}\label{eq:mult-right}
\cbas{S,T}{\lambda} a 
\in
{\textstyle\sum_{T^{\prime}\in\cmset(\lambda)}}
r_{a^{\invo}}(T^{\prime},T)\,\cbas{S,T^{\prime}}{\lambda} 
+ 
\ceps_{S}\calg(\ord{\ceps_{S}}\!\lambda),
\end{gather}
with the same scalars 
$r_{a^{\invo}}(T^{\prime},T)$ as in (\ref{definition:cell-algebra}.d).\qedhere
\end{enumerate}
\end{lemma}

\begin{proof}
\textit{(\ref{lemma:cell-algebra-star-1}.a).} This follows 
by applying ${}^{\invo}$ to (\ref{definition:cell-algebra}.c).
\medskip

\noindent \textit{(\ref{lemma:cell-algebra-star-1}.b).} 
The first inclusion follows from (\ref{definition:cell-algebra}.c) and the 
second by applying ${}^{\invo}$.
\medskip

\noindent \textit{(\ref{lemma:cell-algebra-star-1}.c).} 
For the left-ideal-statement 
let $\cbas{S,T}{\lambda}\in\calg(\cideal_{\ceps})\ceps$. 
Then -- by (\ref{definition:cell-algebra}.d) -- we have
\begin{gather}
a\cbas{S,T}{\lambda}\ceps
\in
{\textstyle\sum_{S^{\prime}\in\cmset(\lambda)}}\,
r_{a}(S^{\prime},S)\,\cbas{S^{\prime},T}{\lambda}\ceps 
+
\calg(\ord{\ceps_{T}}\!\lambda)\ceps_{T}\ceps.
\end{gather}
But either $\ceps_{T}\ceps=0$ or they 
agree and the last term is inside the linear subspace. 
The right-ideal-statement is again obtained using ${}^{\invo}$.
\medskip

\noindent \textit{(\ref{lemma:cell-algebra-star-1}.d).} 
By applying ${}^{\invo}$ directly to (\ref{definition:cell-algebra}.d).
\end{proof}

Combining (\ref{definition:cell-algebra}.c)
and (\ref{lemma:cell-algebra-star-1}.a) we obtain:

\begin{corollary}\label{lemma:epsilon-structure}
Let $a\in\calg$ such 
that $\ceps a=a= a\ceps$ 
for $\ceps\in\ciset$. Then 
\begin{gather}
a\in\K\{
\cbas{S,T}{\lambda}
\mid
\lambda\in\cset,S,T\in\cmset(\lambda),\ceps_{S}=\ceps_{T}=\ceps
\}.
\end{gather}
The same holds for $a^{\invo}$ as well.
\end{corollary}

Additionally, \fullref{lemma:cell-algebra-star-1} gives 
us a further relation to cellular algebras.

\begin{proposition}\label{proposition:cell-relative-cell}
Let $\calg$ be a relative cellular algebra 
with cell datum $(\cset,\cmset,\cbasis,{}^{\invo},\ciset,\coset,\cepsmap)$, 
and let $\ualg$ be a cellular algebra 
with cell datum 
$(\cset,\cmset,\cbasis,{}^\invo)$ and order 
$\order$ on $\cset$.
\smallskip
\begin{enumerate}[label=(\alph*)]
\setlength\itemsep{.15cm}

\item For all $\ceps\in\ciset$, the algebra $\ceps\calg\ceps$ is a  cellular algebra with cell datum 
$(\cset,\cmset_{\ceps},\cbasis_{\ceps},{}^{\invo})$ 
and the partial order on $\cset$ given by $\ord{\ceps}$,
\begin{gather}
\cmset_{\ceps}(\lambda)
=
\{S\in\cmset(\lambda)
\mid\ceps\cbas{S,T}{\lambda} 
= 
\cbas{S,T}{\lambda}\text{ for }T\in\cmset(\lambda)\},
\end{gather}
and $\cbasis_{\ceps}$ being 
the restriction of $\cbasis$ to 
${\textstyle\coprod_{\lambda\in\cset}}\, 
\cmset_{\ceps}(\lambda)\times\cmset_{\ceps}(\lambda)$.

\item The algebra $\ualg$ is relative 
cellular with relative cell datum 
$(\cset,\cmset,\cbasis,{}^{\invo},\{1\},\{\ord{1}\},\cepsmap)$, 
with $\cepsmap$ mapping everything to $1$.\qedhere
\end{enumerate}
\end{proposition}

\begin{proof}
\textit{(\ref{proposition:cell-relative-cell}.a).} 
That $\cmset_{\ceps}$ and $\cbasis_{\ceps}$ 
give a bijection with a basis of $\ceps\calg\ceps$ 
follows by combining (\ref{definition:cell-algebra}.c) 
and (\ref{lemma:cell-algebra-star-1}.a). 
So we are left with checking the multiplication rule for cellular 
algebras. For $a\in\calg$, $\lambda\in\cset$, 
and $S,T\in\cmset(\lambda)$ with 
$\ceps_{S}=\ceps_{T}=\ceps$, 
we use (\ref{definition:cell-algebra}.c) and get
\begin{gather}
\ceps a \ceps\,\cbas{S,T}{\lambda}\in 
{\textstyle\sum_{S^\prime \in \cmset(\lambda), \ceps_{S^\prime}=\ceps}}\,
r_{\ceps a\ceps}(S^{\prime},S)\,
\cbas{S^\prime,T}{\lambda} 
+ 
\ceps\calg(\ord{\ceps}\lambda)\ceps\subset\ceps\calg\ceps.
\end{gather}
\smallskip

\noindent\textit{(\ref{proposition:cell-relative-cell}.b).} 
By construction, (\ref{definition:cell-algebra}.a) and (\ref{definition:cell-algebra}.b) 
are part of the cell datum $(\cset,\cmset,\cbasis,{}^{\invo})$. 
Next, the set $\ciset$ for (\ref{definition:cell-algebra}.c) can be taken to be 
$\ciset=\{1\}$ (with $1$ being the unit 
of $\calg$) satisfying $1^{\invo} = 1$. The 
partial ordering $\order$ of $\ualg$ is the 
partial ordering $\ord{1}$ for the unit. 
Note hereby that \eqref{eq:idem-props-1} follows 
from (\ref{definition:cell-algebra}.d), while \eqref{eq:idem-props-2} is automatic.
\end{proof}

\begin{remark}\label{remark:diff-from-usual}
For any cellular algebra $\ualg$ and any idempotent $\ceps$ 
fixed by ${}^{\invo}$, $\ceps\ualg\ceps$ is 
cellular, see \cite[Proposition 4.3]{kx1}. However, 
\fullref{proposition:cell-relative-cell} is different since we do not assume 
$\calg$ to be cellular to begin with.
\end{remark}

\begin{remark}\label{remark-example-rel-cell-algebras-1}
As we have seen in the proof of (\ref{proposition:cell-relative-cell}.b), 
the two conditions \eqref{eq:idem-props-1} and \eqref{eq:idem-props-2} 
are ``invisible'' in the non-relative setup. 
However, they are crucial for our purposes e.g. 
\eqref{eq:idem-props-1} is used in \fullref{lemma:cell-algebra-star-2} --
a crucial ingredient for proving \fullref{theorem:simple-set}.
\end{remark}

Note that the map $\cepsmap$ is always 
surjective by (\ref{definition:cell-algebra}.a) and (\ref{definition:cell-algebra}.c). 
Furthermore, only finitely many elements 
of $\ciset$ act non-trivially on a given 
element on $\calg$. Thus, the following is immediate.

\begin{lemmaqed}\label{lemma:decomposition-unit}
If $\csetf$, then $\calg$ is unital 
with unit $\sum_{\ceps\in\ciset}\ceps$. 
Otherwise $\calg$ is locally unital with set 
of local units being all finite sums of elements in $\ciset$.
\end{lemmaqed}

There is a quotient functor from the category 
of $\calg$-modules to modules over $\calg(\ciset)=\bigoplus_{\ceps\in\ciset}\ceps\calg\ceps$. 
By \fullref{lemma:decomposition-unit}, this gives a bijection between the isomorphism 
classes of simples for both algebras. However, some properties of this quotient 
functor depend on the choice of the set $\ciset$, and e.g. $\calg$ is in general not a projective 
$\calg(\ciset)$-module since the projectives 
of both algebras might be fairly different. See also \fullref{remark:homological-orders} below.

%%%%%%%%%%%%%%%%%%%%%%%%%%%
\subsection{Existence of cell modules}\label{subsection:cellmodules}
%%%%%%%%%%%%%%%%%%%%%%%%%%%

We proceed by defining cell modules.

\begin{definition}\label{definition:cell-module}
For $\lambda\in\cset$ and $T\in\cmset(\lambda)$ 
let $\dmod(\lambda;T) = 
\K\{\dbas{S,T}{\lambda}\mid S \in\cmset(\lambda) 
\}$. We define an action $\acts$ of $\calg$ on 
$\dmod(\lambda;T)$ by setting
\begin{gather}
a\acts \dbas{S,T}{\lambda}= 
{\textstyle\sum_{S^{\prime} \in\cmset(\lambda)}}\, 
r_{a}(S^{\prime},S)\,\dbas{S^{\prime},T}{\lambda},
\end{gather}
with $r_{a}(S^{\prime},S)$ being defined by \eqref{eq:mult-left}.
\end{definition}

\begin{lemma}\label{lemma:cell-defined}
The action from \fullref{definition:cell-module} 
defines the structure of an $\calg$-mo\-dule 
on $\dmod(\lambda;T)$. Further, there is an isomorphism 
of $\calg$-modules
$\dmod(\lambda;T)\cong\dmod(\lambda;T^{\prime})$ 
for any $T,T^{\prime}\in\cmset(\lambda)$.
\end{lemma}

\begin{proof}
The coefficient $r_{a}(S^{\prime},S)$ is -- by definition --
additive with respect to $a$, 
and one has $r_{1}(S^{\prime},S)=\delta_{S,S^{\prime}}$. 
Moreover, one also has
\begin{gather}
\begin{aligned}
a^{\prime}(a\cbas{S,T}{\lambda}) 
& \;\in\; a^\prime 
{\textstyle\sum_{S^{\prime}\in\cmset(\lambda)}}\, 
r_{a}(S^{\prime},S)\,\cbas{S^{\prime},T}{\lambda}
+ 
a^{\prime}\calg(\ord{\ceps_{T}}\!\lambda)\ceps_{T}
\\
& \;\subset\;{\textstyle\sum_{S^{\prime},S^{\prime\prime}\in\cmset(\lambda)}}\,
r_{a^{\prime}}(S^{\prime\prime},S^{\prime})r_{a}(S^{\prime},S)\,
\cbas{S^{\prime\prime},T}{\lambda} +\calg(\ord{\ceps_{T}}\!\lambda)\ceps_{T},
\end{aligned}
\end{gather}
where the inclusion is due to 
\eqref{eq:mult-left} and (\ref{lemma:cell-algebra-star-1}.c), and
\begin{gather}
(a^{\prime}a)\cbas{S,T}{\lambda}
\in 
{\textstyle\sum_{S^{\prime\prime}\in\cmset(\lambda)}}\,
r_{a^{\prime}a}(S^{\prime\prime},S)\,\cbas{S^{\prime\prime},T}{\lambda} 
+ 
\calg(\ord{\ceps_{T}}\!\lambda)\ceps_{T}.
\end{gather}
Thus, we have
\begin{gather}
r_{a^{\prime}a}(S^{\prime\prime},S) 
={\textstyle\sum_{S^{\prime}\in\cmset(\lambda)}}\,
r_{a^{\prime}}(S^{\prime\prime},S^{\prime}) r_{a}(S^{\prime},S) 
\;\text{ for }a,a^{\prime}\in\calg. 
\end{gather}
This in turn implies 
$a^{\prime}\acts(a\acts\dbas{S,T}{\lambda}) 
= (a^{\prime}a)\acts\dbas{S,T}{\lambda}$. Hence, we get 
a well-defined $\calg$-module structure on 
$\dmod(\lambda,T)$. Since $r_{a}(S^{\prime},S)$ is independent of the 
second index, the assignment 
$\dbas{S,T}{\lambda}\mapsto\dbas{S,T^{\prime}}{\lambda}$ 
gives an $\calg$-module isomorphism.
\end{proof}

Due to \fullref{lemma:cell-defined} we omit 
the $T$ in the definition and notation of $\dmod(\lambda;T)$. 
We call $\dmod(\lambda)$ a \textit{cell module},
and we denote the basis elements of $\dmod(\lambda)$ 
by $\dbas{S}{\lambda}$ only. Furthermore -- having \fullref{lemma:cell-defined} -- 
we can define right $\calg$-modules:

\begin{definition}\label{definition:cell-dual}
We define the right $\calg$-module $\dmod(\lambda)^{\invo}$ on 
the same vector space as $\dmod(\lambda)$ by 
setting $\dbas{S}{\lambda}\acts a = a^{\invo}\acts\dbas{S}{\lambda}$.
\end{definition}

We get -- by construction -- the following identification:.

\begin{lemmaqed}\label{lemma:cell-dual}
The linear extension of the assignment
\begin{gather}
\diso{\lambda}\colon
\dmod(\lambda)\otimes\dmod(\lambda)^{\invo}
\rightarrow\calg(\{\lambda\}),
\;
\diso{\lambda}(\dbas{S}{\lambda},\dbas{T}{\lambda}) 
=\cbas{S,T}{\lambda},
\end{gather}
is an isomorphism of vector spaces.
\end{lemmaqed}

%%%%%%%%%%%%%%%%%%%%%%%%%%%
\subsection{A basis free definition of relative cellularity}\label{subsection:cell-def-no-basis}
%%%%%%%%%%%%%%%%%%%%%%%%%%%

In this section we let 
$\anyalg$ be an algebra with a 
fixed anti-involution ${}^{\invo}$ and 
a set $\ciset$ of pairwise orthogonal idempotents, all fixed by 
${}^{\invo}$. Furthermore, denote by $\mathbb{K}[\ciset]$ the 
semigroup algebra generated by the elements of $\ciset$. 
Following \cite[Definition 3.2]{kx1} we define:

\begin{definition}\label{definition:cell-ideal}
Let $\ccideal\subset\anyalg$ denote a linear 
subspace, and let $\dmod$ denote a 
finite-dimensional, left $\anyalg$-module. 
Assume that the following hold:
\smallskip
\begin{enumerate}[label=(\alph*)]

\setlength\itemsep{.15cm}

\item The linear subspace $\ccideal$ is 
fixed under ${}^{\invo}$, i.e. $\ccideal^{\invo}=\ccideal$.

\item The linear subspace $\ccideal$ is a $\mathbb{K}[\ciset]$-bimodule.

\item There 
is a $\mathbb{K}[\ciset]$-bimodule isomorphism 
$\Diso\colon\ccideal\xrightarrow{\cong}
\dmod\otimes\dmod^{\invo}$ 
and a diagram
\begin{gather}
\begin{tikzcd}[ampersand replacement=\&,row sep=scriptsize,column sep=scriptsize,arrows={shorten >=-.5ex,shorten <=-.5ex},labels={inner sep=.1ex}]
\ccideal
\arrow{rr}{\Diso}
\arrow[swap]{d}{{}^{\invo}}
\arrow[phantom,xshift=.15cm,yshift=-.1cm,swap]{drr}{\text{$\circlearrowleft$}}
\&
\phantom{.}
\&
\dmod\otimes\dmod^{\invo}\phantom{.}
\arrow[xshift=-.55ex]{d}{\,x\otimes y\mapsto y\otimes x}
\\
\ccideal
\arrow[swap]{rr}{(\Diso)^{\invo}}
\&
\phantom{.}
\&
\dmod\otimes\dmod^{\invo},
\\
\end{tikzcd}
\end{gather}
where $\dmod^{\invo}$ is the right $\anyalg$-module 
on the same 
vector space as $\dmod$ and
right action of $\anyalg$ defined 
via $x\acts a = a^{\invo}\acts x$.
\end{enumerate}
\smallskip

Then we call $\ccideal$ a \textit{cell space}.
\end{definition}

\begin{proposition}\label{proposition:cell-basis-free}
A finite-dimensional algebra $\anyalg$ is relative 
cellular with respect to ${}^{\invo}$ and $\ciset$ if and only if:
\smallskip
\begin{enumerate}[label=(\alph*)]

\setlength\itemsep{.15cm}

\item The elements of $\ciset$ give a decomposition of the unit of $\anyalg$.

\item There is some index set $\cset$ with $\csetf$ and a 
vector space decomposition of 
$\anyalg$ into cell spaces, i.e.
$\anyalg={\textstyle\bigoplus_{\lambda\in\cset}}\,\ccideal_{\lambda}$.

\item For each $\ceps\in\ciset$ there is 
an enumeration $\cset=\{\lambda_{1},\lambda_{2},\cdots,\lambda_{m}\}$ such that
\begin{gather}\label{eq:cell-chain}
\begin{gathered}
0\subset\cccideal_{\lambda_{1}}\ceps
\subset\cccideal_{\lambda_{2}}\ceps\subset\cdots\subset
\cccideal_{\lambda_{i}}\ceps
\subset\cdots\subset
\cccideal_{\lambda_{m}}\ceps
\subset\anyalg\ceps,
\end{gathered}
\end{gather}
is a chain of $\anyalg$-submodules 
$\cccideal_{\lambda_{i}}\ceps=
\bigoplus_{j=1}^{i}\ccideal_{\lambda_{i}}\ceps$.

\item The submodule $\cccideal_{\lambda_{i}}\ceps$ 
as in \eqref{eq:cell-chain}
is a right $\ceps^{\prime}\anyalg\ceps^{\prime}$-module for any $\ceps^{\prime} \in \ciset$.\qedhere
\end{enumerate}
\end{proposition}

\begin{proof}
\textit{\fullref{definition:cell-algebra}$\Rightarrow$\fullref{proposition:cell-basis-free}.}
Fix $\ceps\in\ciset$. Since 
$\ord{\ceps}$ is a partial order on $\cset$, we can
inductively construct 
the linear subspaces 
$\ccideal_{\lambda_{i}}\ceps\subset\anyalg\ceps$ by 
starting with
\begin{gather}
\ccideal_{\lambda_{1}}\ceps=
\K\{
\cbas{S,T}{\lambda_{1}}
\mid
S,T\in\cmset(\lambda_{1}),
\ceps_{T}=\ceps
\}
\stackrel{\eqref{eq:idem-props-2}}{=}
\K\{
\cbas{S,T}{\lambda_{1}}\ceps
\mid
S,T\in\cmset(\lambda_{1}),
\ceps_{T}=\ceps
\}
\end{gather}
for some $\ord{\ceps}$-minimal $\lambda_{1}\in\cset$. 
Then we set $\cccideal_{\lambda_{i}}\ceps=
\bigoplus_{j=1}^{i}\ccideal_{\lambda_{i}}\ceps$, 
and the so constructed linear spaces are submodules and 
satisfy the cell chain 
condition \eqref{eq:cell-chain} by 
(\ref{definition:cell-algebra}.d). Moreover, orthogonality 
and the ${}^{\invo}$-version of \eqref{eq:idem-props-1} 
(see (\ref{lemma:cell-algebra-star-1}.a)) 
shows that (\ref{proposition:cell-basis-free}.d) holds as well.

Further, define
$\ccideal_{\lambda}=
\bigoplus_{\ceps\in\ciset}\ccideal_{\lambda}\ceps$. 
These are cell spaces:
By (\ref{definition:cell-algebra}.b) and the fact that 
$\ceps_{S}=\ceps$ for some 
$\ceps\in\ciset$ we get (\ref{definition:cell-ideal}.a), 
while (\ref{definition:cell-ideal}.b) follows from 
\eqref{eq:idem-props-2}.
Next -- by virtue 
of construction -- $\ccideal_{\lambda}=\anyalg(\{\lambda\})$. 
Thus, we can set $\dmod_{\lambda}\cong\dmod(\lambda)$, whose properties 
-- by \fullref{lemma:cell-dual} --
give (\ref{definition:cell-ideal}.c) by 
defining $\Diso(\cbas{S,T}{\lambda})=
(\dbas{S}{\lambda},\dbas{T}{\lambda})$.
Finally -- by (\ref{definition:cell-algebra}.a), 
\fullref{lemma:decomposition-unit} and 
finite-dimensionality -- we get (\ref{proposition:cell-basis-free}.a) and 
(\ref{proposition:cell-basis-free}.b).
\medskip

\noindent \textit{\fullref{proposition:cell-basis-free}$\Rightarrow$\fullref{definition:cell-algebra}.} 
First, let $\cset=\{\lambda\mid\ccideal_{\lambda}\text{ is a cell space}\}$.
For any cell space $\ccideal_{\lambda}$ we first fix 
a basis $\{\dbas{S}{\lambda}\}$ of its associated $\dmod_{\lambda}$. 
Note that -- by finite-dimensionality
-- we can choose this to be a basis consisting of common eigenvectors for $\K[\ciset]$, and we thus can 
demand that this basis satisfies either 
$\ceps\dbas{S}{\lambda}=\dbas{S}{\lambda}$ or 
$\ceps\dbas{S}{\lambda}=0$ for each $\ceps\in\ciset$.
The $\lambda,S,T$ play hereby the role of some indexes, where 
we set $\cmset(\lambda)$ to be the set of all $S,T$'s that appear 
in this enumeration.
Next, use (\ref{definition:cell-ideal}.c) to define 
$\cbasis(S,T)=\cbas{S,T}{\lambda}=\Diso
(\dbas{S}{\lambda}\otimes\dbas{T}{\lambda})$ for $S,T\in\cmset(\lambda)$.
Since we have already fixed ${}^{\invo}$, this defines 
the relative cell datum up to the part about idempotents.
To define the remaining data, first note that $\ciset$ is 
already given. Moreover, the cell chain condition \eqref{eq:cell-chain} 
gives rise to a partial ordering $\ord{\ceps}$ on $\cset$ 
for each $\ceps\in\ciset$.
Next, observe that $\ceps(S)\cbas{S,T}{\lambda}=\cbas{S,T}{\lambda}$
for precisely one $\ceps(S)\in\ciset$ 
due to the choice of the 
basis $\{\dbas{S}{\lambda}\}$, orthogonality and (\ref{proposition:cell-basis-free}.a). 
Thus, we can define $\ceps_{S}=\ceps(S)$, and we get the last part of the
relative cell datum.

It remains to check that we have defined a relative cell datum. 
First, note that all $\cmset(\lambda)$'s are finite because -- by assumption -- 
the $\dmod_{\lambda}$'s are finite-dimensional, 
while $\csetf$ -- also by assumption.
Second -- by (\ref{proposition:cell-basis-free}.b) -- we have an isomorphism 
of vector spaces $\calg\cong\bigoplus_{\lambda\in\cset}\ccideal_{\lambda}$, 
showing that (\ref{definition:cell-algebra}.a) holds. 
That (\ref{definition:cell-algebra}.b) holds on the nose follows from 
the commutative diagram in (\ref{definition:cell-ideal}.c), while 
(\ref{definition:cell-algebra}.d) follows from (\ref{definition:cell-ideal}.b).
Finally, it remains to show \eqref{eq:idem-props-1} and 
\eqref{eq:idem-props-2}, where the latter is clear by 
construction of $\cepsmap$. The remaining part follows then 
by applying ${}^{\invo}$ to (\ref{proposition:cell-basis-free}.d).
\end{proof}

\begin{furtherdirections}\label{remark:cell-ideals}
As explained in \cite{kx1}, the basis free formulation of cellularity 
is connected to ideals in the setting of 
quasi-hereditary algebras. In the relative setup we lose 
the ideal structure (cf. (\ref{proposition:cell-basis-free}.c) 
and (\ref{proposition:cell-basis-free}.d)) and we 
do not know what the relative version 
of the connection to quasi-hereditary algebras is.
\end{furtherdirections}

%%%%%%%%%%%%%%%%%%%%%%%%%%%
\subsection{Examples of relative cellular algebras}\label{subsection:cell-examples}
%%%%%%%%%%%%%%%%%%%%%%%%%%%

\begin{remark}\label{remark:pos-def}
For the following examples recall that the \textit{Cartan matrix} 
$\cmatrix(\anyalg)$ of some finite-dimensional algebra $\anyalg$ is defined 
by counting the multiplicities of the 
simples $\lmod$ in the 
indecomposable projectives $\prmod$.
Now, it follows from \cite[Proposition 3.2]{kx2}
that $\cmatrix(\ualg)$ is symmetric 
and positive definite in case  
$\ualg$ is a cellular algebra.
\end{remark}

\begin{example}\label{example:rel-cell-algebras-2}
Consider the type 
$\typeA_{n}$ graphs with doubled edges 
(where we exclude the case $n=2$ because it 
requires a slightly different 
setup):
\begin{gather}\label{eq:rel-cell-algebras-2}
\begin{gathered}
\typeA_{n}
=
\xy
(0,-2)*{
\begin{tikzcd}[ampersand replacement=\&,row sep=scriptsize,column sep=scriptsize,arrows={shorten >=-.5ex,shorten <=-.5ex},labels={inner sep=.225ex}]
\videm{1}
\arrow[yshift=.5ex,->]{rr}
\&
\phantom{.}
\&
\videm{2}
\arrow[yshift=-.5ex,->]{ll}
\arrow[yshift=.5ex,->]{rr}
\&
\phantom{.}
\&
\videm{3}
\arrow[yshift=-.5ex,->]{ll}
\arrow[yshift=.5ex,-]{r}
\&
\cdots
\arrow[yshift=-.5ex,->]{l}
\arrow[yshift=.5ex,->]{r}
\&
\videm{n}
\arrow[yshift=-.5ex,-]{l}
\\
\end{tikzcd}};
\endxy
,
\quad
(\paths{i}{j})^{\invo}=\paths{j}{i},
\\
\text{Relations: }
\begin{gathered}
\text{All $2$-cycles at the vertex $\videm{i}$ are equal, i.e. }
\pathss{i}{j}{i}=\pathss{i}{k}{i};
\\
\text{Going two steps in one direction is zero, i.e. }
\pathss{i}{j}{k}=0 \text{, for } \videm{i} \neq \videm{k}.
\end{gathered}
\end{gathered}
\end{gather}
We let $\ualg(\typeA_{n})$ be the quotient of the 
path algebra of $\typeA_{n}$ (multiplication $\circ$ being 
composition of paths $\paths{i}{j}\circ\paths{j}{k}=\pathss{i}{j}{k}$) 
with relations as in \eqref{eq:rel-cell-algebras-2}. 
Up to base change one gets:
\begin{gather}
\cmatrix(\ualg(\typeA_{3}))=
\begin{psmallmatrix}
2 & 1 & 0 \\
1 & 2 & 1 \\
0 & 1 & 2 \\
\end{psmallmatrix}
,\;\;
\cmatrix(\ualg(\typeA_{4}))=
\begin{psmallmatrix}
2 & 1 & 0 & 0\\
1 & 2 & 1 & 0\\
0 & 1 & 2 & 1\\
0 & 0 & 1 & 2\\
\end{psmallmatrix}
,\;\;
\cmatrix(\ualg(\typeA_{5}))=
\begin{psmallmatrix}
2 & 1 & 0 & 0 & 0\\
1 & 2 & 1 & 0 & 0\\
0 & 1 & 2 & 1 & 0\\
0 & 0 & 1 & 2 & 1\\
0 & 0 & 0 & 1 & 2\\
\end{psmallmatrix}
,\;\;
\text{etc.},
\end{gather}
all of which are positive definite.
The algebra $\ualg(\typeA_{n})$ is known as the 
\textit{type $\typeA_{n}$ zigzag algebra}, cf. \cite[Section 3]{hk1}. 
Let us discuss the case $n=2$ with respect to 
cellularity in detail, the general case 
works mutatis mutandis. 

First, the 
$\ualg(\typeA_{3})$-action on itself 
is given by pre-composition of paths, and the algebra
can be equipped with the anti-involution ${}^{\invo}$ 
indicated in \eqref{eq:rel-cell-algebras-2} that fixes the vertex 
idempotents $e_{\videm{1}},e_{\videm{2}},e_{\videm{3}}$. Clearly,
$\ualg(\typeA_{3})$ has one-dimensional simple modules 
$\lmod(\videm{i})$ for $i\in\{1,2,3\}$ where $e_{\videm{j}}$ acts by $\delta_{ij}$.

The algebra $\ualg(\typeA_{3})$ is a relative cellular 
algebra with respect to ${}^{\invo}$. 
As a relative cell datum we can take
\begin{gather}
\begin{gathered}
\cset
=\{\videm{0}\ord{1}\videm{1}\ord{1}\videm{2}\ord{1}\videm{3}\},
\\
\cmset(\videm{0})=\{\paths{1}{2}\},
\;
\cmset(\videm{1})=\{e_{\videm{1}},\paths{2}{1}\},
\;
\cmset(\videm{2})=\{e_{\videm{2}},\paths{3}{2}\},
\;
\cmset(\videm{3})=\{e_{\videm{3}}\},
\;
\cbas{S,T}{\videm{i}}=S\circ T^{\invo},
\\
\ciset=\{1\},
\quad
\cepsmap(\pathss{1}{2}{1})=
\cepsmap(e_{\videm{1}})=
\cepsmap(\paths{2}{1})=
\cepsmap(e_{\videm{2}})=
\cepsmap(\paths{3}{2})=
\cepsmap(e_{\videm{3}})=
1.
\end{gathered}
\end{gather}
Note that $\ciset=\{1\}$ is the same choice as in 
(\ref{proposition:cell-relative-cell}.b), and 
$\ualg(\typeA_{3})$ is actually cellular. 
Now, the cellular 
basis and cell modules are given as follows, where 
we write $\videm{i}$ on top of the columns containing 
$\dmod(\lambda;T)$'s with $\ceps_{T}=e_{\videm{i}}$ 
(in the notation from \fullref{definition:cell-module}):
\begin{gather}
\begin{tikzpicture}[baseline=(current bounding box.center)]
  \matrix (m) [matrix of math nodes, nodes in empty cells, row 3/.style={nodes={minimum height=10mm}}, row sep={.1cm}, column sep={.1cm}, text height=1.5ex, text depth=0.25ex, ampersand replacement=\&] {
 \& \videm{1} \& \& \& \& \videm{2} \& \& \& \& \videm{3} \&
\\
e_{\videm{1}} \& \& \paths{2}{1} \& \& e_{\videm{2}} \& \& \paths{3}{2} \& \& e_{\videm{3}} \& \& \phantom{\paths{1}{3}}
\\
\phantom{\paths{3}{1}} \& \& \pathss{1}{2}{1}
\&
\&
\paths{1}{2} \& \&  \renewcommand*{\arraystretch}{.6}
\begin{matrix}\pathss{2}{1}{2}\\=\\\pathss{2}{3}{2}\end{matrix}
\&
\&
\paths{2}{3} \& \&  \renewcommand*{\arraystretch}{.6}
\begin{matrix}\pathss{3}{1}{3}\\=\\\pathss{3}{2}{3}\end{matrix}
\\
    };
    \draw[thin,dotted] ($(m-1-1.south)+(-.5,0)$) to ($(m-1-1.south)+(9.0,0)$);
	\draw[thin,dotted] ($(m-1-1.east)+(2.45,.25)$) to ($(m-1-1.east)+(2.45,-2.05)$);
	\draw[thin,dotted] ($(m-1-1.east)+(5.7,.25)$) to ($(m-1-1.east)+(5.7,-2.05)$);
	\draw[thin, mygreen] (-1.975,.5) rectangle (-4.5,0);
	\draw[thin, black] (-1.975,-.1) rectangle (-4.5,-1.1);
	\draw[thin, myred] (1.3,.5) rectangle (-1.3,0);
	\draw[thin, mygreen] (1.3,-.1) rectangle (-1.3,-1.1);
	\draw[thin, myblue] (1.925,.5) rectangle (4.55,0);
	\draw[thin, myred] (1.925,-.1) rectangle (4.55,-1.1);
	\draw[thin, white, fill=white] (-3.9,.05) rectangle (-3.0,-.45);
	\draw[thin, white, fill=white] (-.7,.05) rectangle (.2,-.45);
	\draw[thin, white, fill=white] (3.5,.05) rectangle (2.6,-.45);
	\node[mygreen] at (-3.4,.2) {$\dmod(\videm{1})$};
	\node[black] at (-3.4,-.6) {$\dmod(\videm{0})$};
	\node[myred] at (-.15,.2) {$\dmod(\videm{2})$};
	\node[mygreen] at (-.15,-.6) {$\dmod(\videm{1})$};
	\node[myblue] at (3.1,.2) {$\dmod(\videm{3})$};
	\node[myred] at (3.1,-.6) {$\dmod(\videm{2})$};
	\draw[thick,->] (-3.3,0) to node [left] {{\tiny$\ord{1}$}} (-3.3,-.25) to (-3.3,-.4);
	\draw[thick,->] (-.1,0) to node [left] {{\tiny$\ord{1}$}} (-.1,-.25) to (-.1,-.4);
	\draw[thick,->] (3.15,0) to node [left] {{\tiny$\ord{1}$}} (3.15,-.25) to (3.15,-.4);
\end{tikzpicture}
\end{gather}
The left action going in the indicated direction 
(or it stays within the $\dmod$'s)
as one easily checks. Note the directedness: 
$\dmod(\videm{0})\xleftarrow{\text{{\tiny$\ord{1}$}}}\dmod(\videm{1})
\xleftarrow{\text{{\tiny$\ord{1}$}}}\dmod(\videm{2})
\xleftarrow{\text{{\tiny$\ord{1}$}}}\dmod(\videm{3})$, 
making the cell modules well-defined since they are obtained by modding out 
terms that are $\ord{1}$-smaller.

Further, the indecomposable projectives are
\begin{gather}
\begin{tikzpicture}[baseline=(current bounding box.center)]
	\node[mygreen] at (.5,1.5) {$\prmod(\videm{1})=\ualg(\typeA_{3})e_{\videm{1}}$};
	\draw[thin, mygreen, dotted] (-.5,1.25) rectangle (1.5,-.85);
	\node[mygreen] at (0,.5) {$\lmod(\videm{1})$};
	\node[myred] at (0,0) {$\lmod(\videm{2})$};
	\draw[thin, mygreen] (-.4,.75) rectangle (.4,-.25);
	\node[mygreen] at (0,1) {$\dmod(\videm{1})$};
	\node[mygreen] at (1,-.5) {$\lmod(\videm{1})$};
	\draw[thin, black] (.6,-.25) rectangle (1.4,-.75);
	\node[black] at (1,0) {$\dmod(\videm{0})$};
\end{tikzpicture}
,\quad
\begin{tikzpicture}[baseline=(current bounding box.center)]
	\node[myred] at (.5,1.5) {$\prmod(\videm{2})=\ualg(\typeA_{3})e_{\videm{2}}$};
	\draw[thin, myred, dotted] (-.5,1.25) rectangle (1.5,-.85);
	\node[myred] at (0,.5) {$\lmod(\videm{2})$};
	\node[myblue] at (0,0) {$\lmod(\videm{3})$};
	\draw[thin, myred] (-.4,.75) rectangle (.4,-.25);
	\node[myred] at (0,1) {$\dmod(\videm{2})$};
	\node[mygreen] at (1,0) {$\lmod(\videm{1})$};
	\node[myred] at (1,-.5) {$\lmod(\videm{2})$};
	\draw[thin, mygreen] (.6,.25) rectangle (1.4,-.75);
	\node[mygreen] at (1,.5) {$\dmod(\videm{1})$};
\end{tikzpicture}
,\quad
\begin{tikzpicture}[baseline=(current bounding box.center)]
	\node[myblue] at (.5,1.5) {$\prmod(\videm{3})=\ualg(\typeA_{3})e_{\videm{3}}$};
	\draw[thin, myblue, dotted] (-.5,1.25) rectangle (1.5,-.85);
	\node[myblue] at (0,.5) {$\lmod(\videm{3})$};
	\draw[thin, myblue] (-.4,.75) rectangle (.4,.25);
	\node[myblue] at (0,1) {$\dmod(\videm{3})$};
	\node[myred] at (1,0) {$\lmod(\videm{2})$};
	\node[myblue] at (1,-.5) {$\lmod(\videm{3})$};
	\draw[thin, myred] (.6,.25) rectangle (1.4,-.75);
	\node[myred] at (1,.5) {$\dmod(\videm{2})$};
\end{tikzpicture}
\end{gather}
which have the indicated $\dmod$-filtrations. 
We will see in \fullref{proposition:cell-filtration} that this is a general feature, 
with partial order in the filtration being relative. See also \makeautorefname{example}{Examples}
\fullref{example:rel-cell-algebras-3} and 
\ref{example:rel-cell-algebras-4} below.
\end{example}

\makeautorefname{example}{Example}

Morally speaking, in the relative setup we can separate 
parts that are cellular by using the idempotents 
in $\ciset$. Here two prototypical examples:

\begin{example}\label{example:rel-cell-algebras-3}
(We use a notation similar as in \fullref{example:rel-cell-algebras-2}.)
Consider the following family of quivers, i.e. the cycles 
on $n$ vertices with 
double edges:
\begin{gather}\label{eq:rel-cell-algebras-3}
\begin{gathered}
\typeAt_{n}=
\xy
(0,-2)*{
\begin{tikzcd}[ampersand replacement=\&,row sep=scriptsize,column sep=scriptsize,arrows={shorten >=-.5ex,shorten <=-.5ex},labels={inner sep=.225ex}]
\videm{1}
\arrow[yshift=.5ex,->]{rr}
\arrow[xshift=-.5ex,<-]{dd}
\arrow[xshift=.5ex,->]{dd}
\&
\phantom{.}
\&
\videm{2}
\arrow[yshift=-.5ex,->]{ll}
\arrow[xshift=.5ex,->]{dd}
\arrow[xshift=-.5ex,<-]{dd}
\\
\phantom{.}
\&
\phantom{.}
\&
\phantom{.}
\\
\videm{n}
\arrow[yshift=.5ex,-]{r}
\arrow[yshift=-.5ex,<-]{r}
\&
\cdots
\&
\videm{3}
\arrow[yshift=.5ex,<-]{l}
\arrow[yshift=-.5ex,-]{l}
\\
\end{tikzcd}};
\endxy
,
\quad
(\paths{i}{j})^{\invo}=\paths{j}{i},
\\
\text{Relations:} \,
\begin{gathered}
\text{All $2$-cycles at the vertex $\videm{i}$ are equal, i.e. }
\pathss{i}{j}{i}=\pathss{i}{k}{i};
\\
\text{Going two steps in one direction is zero, e.g.}
\pathss{1}{n}{n{-}1}\!=\!0.
\end{gathered}
\end{gathered}
\end{gather}
(The case $n=2$ is special and excluded.)
As in \fullref{example:rel-cell-algebras-2}, 
we let $\calg(\typeAt_{n})$ be the corresponding 
quotient of the path algebra of $\typeAt_{n}$, with relations 
given in \eqref{eq:rel-cell-algebras-3}, and anti-involution 
${}^{\invo}$ 
given by swapping the orientations of the arrows. 
Again, the Cartan matrices are easy to calculate and up to base change:
\begin{gather}
\cmatrix(\calg(\typeAt_{3}))=
\begin{psmallmatrix}
2 & 1 & 1 \\
1 & 2 & 1 \\
1 & 1 & 2 \\
\end{psmallmatrix}
,\;\;
\cmatrix(\calg(\typeAt_{4}))=
\begin{psmallmatrix}
2 & 1 & 0 & 1\\
1 & 2 & 1 & 0\\
0 & 1 & 2 & 1\\
1 & 0 & 1 & 2\\
\end{psmallmatrix}
,\;\;
\cmatrix(\calg(\typeAt_{5}))=
\begin{psmallmatrix}
2 & 1 & 0 & 0 & 1\\
1 & 2 & 1 & 0 & 0\\
0 & 1 & 2 & 1 & 0\\
0 & 0 & 1 & 2 & 1\\
1 & 0 & 0 & 1 & 2\\
\end{psmallmatrix}
,\;\;
\text{etc.}
\end{gather}
The algebra $\calg(\typeAt_{n})$ is known as 
the \textit{type $\typeAt_{n}$ zigzag algebra}, and is for 
example studied in the context of categorical 
actions, see e.g. \cite[Section 3.1]{gtw1} or \cite[Section 2.3]{mt1}.
In contrast to $\ualg(\typeA_{n})$, the algebra $\calg(\typeAt_{n})$ 
is not cellular (at least 
for even $n$ where the Cartan matrix is only positive semidefinite, cf. 
\fullref{remark:pos-def}, although this holds in general, see \cite[Theorem A]{et-zigzag}), but 
it is relative cellular as we discuss now in the case $n=3$, 
the general case again being similar.

In this case we take the following relative cell datum. Let 
$\ceps=e_{\videm{2}}+e_{\videm{3}}$ and let
\begin{gather}
\begin{gathered}
\cset
=\{\videm{2}\ord{e_{\videm{1}}}\videm{3}\ord{e_{\videm{1}}}\videm{1}\}
=\{\videm{1}\ord{\ceps}\videm{2}\ord{\ceps}\videm{3}\},
\\
\cmset(\videm{1})=\{e_{\videm{1}},\paths{2}{1}\},
\quad
\cmset(\videm{2})=\{e_{\videm{2}},\paths{3}{2}\},
\quad
\cmset(\videm{3})=\{e_{\videm{3}},\paths{1}{3}\},
\quad
\cbas{S,T}{\videm{i}}=S\circ T^{\invo},
\\
\ciset=\{e_{\videm{1}},\ceps\},
\,
\cepsmap(e_{\videm{1}})=
\cepsmap(\paths{1}{3})=
e_{\videm{1}},
\cepsmap(e_{\videm{2}})=
\cepsmap(e_{\videm{3}})=
\cepsmap(\paths{2}{1})=
\cepsmap(\paths{3}{2})=\ceps.
\end{gathered}
\end{gather}
Next, the relative cellular basis and the cell modules:
\begin{gather}
\begin{tikzpicture}[baseline=(current bounding box.center)]
  \matrix (m) [matrix of math nodes, nodes in empty cells, row 3/.style={nodes={minimum height=10mm}}, row sep={.1cm}, column sep={.1cm}, text height=1.5ex, text depth=0.25ex, ampersand replacement=\&] {
 \& \videm{1} \& \& \& \& \videm{2} \& \& \& \& \videm{3} \&
\\
e_{\videm{1}} \& \& \paths{2}{1} \& \& e_{\videm{2}} \& \& \paths{3}{2} \& \& e_{\videm{3}} \& \& \paths{1}{3}
\\
\paths{3}{1} \& \&  \renewcommand*{\arraystretch}{.6}
\begin{matrix}\pathss{1}{2}{1}\\=\\\pathss{1}{3}{1}\end{matrix}
\&
\&
\paths{1}{2} \& \&  \renewcommand*{\arraystretch}{.6}
\begin{matrix}\pathss{2}{1}{2}\\=\\\pathss{2}{3}{2}\end{matrix}
\&
\&
\paths{2}{3} \& \&  \renewcommand*{\arraystretch}{.6}
\begin{matrix}\pathss{3}{1}{3}\\=\\\pathss{3}{2}{3}\end{matrix}
\\
    };
    \draw[thin,dotted] ($(m-1-1.south)+(-.5,0)$) to ($(m-1-1.south)+(9.0,0)$);
	\draw[thin,dotted] ($(m-1-1.east)+(2.45,.25)$) to ($(m-1-1.east)+(2.45,-2.05)$);
	\draw[thin,dotted] ($(m-1-1.east)+(5.7,.25)$) to ($(m-1-1.east)+(5.7,-2.05)$);
	\draw[thin, mygreen] (-1.975,.5) rectangle (-4.5,0);
	\draw[thin, myblue] (-1.975,-.1) rectangle (-4.5,-1.1);
	\draw[thin, myred] (1.3,.5) rectangle (-1.3,0);
	\draw[thin, mygreen] (1.3,-.1) rectangle (-1.3,-1.1);
	\draw[thin, myblue] (1.925,.5) rectangle (4.55,0);
	\draw[thin, myred] (1.925,-.1) rectangle (4.55,-1.1);
	\draw[thin, white, fill=white] (-3.9,.05) rectangle (-3.0,-.45);
	\draw[thin, white, fill=white] (-.7,.05) rectangle (.2,-.45);
	\draw[thin, white, fill=white] (3.5,.05) rectangle (2.6,-.45);
	\node[mygreen] at (-3.4,.2) {$\dmod(\videm{1})$};
	\node[myblue] at (-3.4,-.6) {$\dmod(\videm{3})$};
	\node[myred] at (-.15,.2) {$\dmod(\videm{2})$};
	\node[mygreen] at (-.15,-.6) {$\dmod(\videm{1})$};
	\node[myblue] at (3.1,.2) {$\dmod(\videm{3})$};
	\node[myred] at (3.1,-.6) {$\dmod(\videm{2})$};
	\draw[thick,->] (-3.3,0) to node [left] {{\tiny$\ord{e_{\videm{1}}}$}} (-3.3,-.25) to (-3.3,-.4);
	\draw[thick,->] (-.1,0) to node [left] {{\tiny$\ord{\ceps}$}} (-.1,-.25) to (-.1,-.4);
	\draw[thick,->] (3.15,0) to node [left] {{\tiny$\ord{\ceps}$}} (3.15,-.25) to (3.15,-.4);
\end{tikzpicture}
\end{gather}
Hereby we like to stress the difference 
between {\color{mygreen}$\dmod(\videm{1})$} in the left and middle column: 
The one in the left column is $\dmod(\videm{1},e_{\videm{1}})$, 
the other is $\dmod(\videm{1},\paths{2}{1})$, the first of which 
is defined using the partial order $\ord{e_{\videm{1}}}$, the second the 
partial order $\ord{\ceps}$.

The indecomposable projectives themselves are
\begin{gather}
\begin{tikzpicture}[baseline=(current bounding box.center)]
	\node[mygreen] at (.5,1.5) {$\prmod(\videm{1})=\calg(\typeAt_{3})e_{\videm{1}}$};
	\draw[thin, mygreen, dotted] (-.5,1.25) rectangle (1.5,-.85);
	\node[mygreen] at (0,.5) {$\lmod(\videm{1})$};
	\node[myred] at (0,0) {$\lmod(\videm{2})$};
	\draw[thin, mygreen] (-.4,.75) rectangle (.4,-.25);
	\node[mygreen] at (0,1) {$\dmod(\videm{1})$};
	\node[myblue] at (1,0) {$\lmod(\videm{3})$};
	\node[mygreen] at (1,-.5) {$\lmod(\videm{1})$};
	\draw[thin, myblue] (.6,.25) rectangle (1.4,-.75);
	\node[myblue] at (1,.5) {$\dmod(\videm{3})$};
\end{tikzpicture}
,\quad
\begin{tikzpicture}[baseline=(current bounding box.center)]
	\node[myred] at (.5,1.5) {$\prmod(\videm{2})=\calg(\typeAt_{3})e_{\videm{2}}$};
	\draw[thin, myred, dotted] (-.5,1.25) rectangle (1.5,-.85);
	\node[myred] at (0,.5) {$\lmod(\videm{2})$};
	\node[myblue] at (0,0) {$\lmod(\videm{3})$};
	\draw[thin, myred] (-.4,.75) rectangle (.4,-.25);
	\node[myred] at (0,1) {$\dmod(\videm{2})$};
	\node[mygreen] at (1,0) {$\lmod(\videm{1})$};
	\node[myred] at (1,-.5) {$\lmod(\videm{2})$};
	\draw[thin, mygreen] (.6,.25) rectangle (1.4,-.75);
	\node[mygreen] at (1,.5) {$\dmod(\videm{1})$};
\end{tikzpicture}
,\quad
\begin{tikzpicture}[baseline=(current bounding box.center)]
	\node[myblue] at (.5,1.5) {$\prmod(\videm{3})=\calg(\typeAt_{3})e_{\videm{3}}$};
	\draw[thin, myblue, dotted] (-.5,1.25) rectangle (1.5,-.85);
	\node[myblue] at (0,.5) {$\lmod(\videm{3})$};
	\node[mygreen] at (0,0) {$\lmod(\videm{1})$};
	\draw[thin, myblue] (-.4,.75) rectangle (.4,-.25);
	\node[myblue] at (0,1) {$\dmod(\videm{3})$};
	\node[myred] at (1,0) {$\lmod(\videm{2})$};
	\node[myblue] at (1,-.5) {$\lmod(\videm{3})$};
	\draw[thin, myred] (.6,.25) rectangle (1.4,-.75);
	\node[myred] at (1,.5) {$\dmod(\videm{2})$};
\end{tikzpicture}
\end{gather}
which have order depended cyclic patterns.
\end{example}

\begin{remark}\label{remark:homological-orders}
Note that \fullref{example:rel-cell-algebras-3} also shows the 
dependence of the homological characterizations of cell modules on the choice 
of idempotents and their associated partial orders. If one chooses the finer 
set of idempotents $e_{\videm{1}}$, $e_{\videm{2}}$, and $e_{\videm{3}}$ and partial orders
\begin{gather}\label{eq:prime-set}
\cset
=\{\videm{3}\ord{e_\videm{1}}\videm{1}\ord{e_\videm{1}}\videm{2}\}
=\{\videm{1}\ord{e_\videm{2}}\videm{2}\ord{e_\videm{2}}\videm{3}\}
=\{\videm{2}\ord{e_\videm{3}}\videm{3}\ord{e_\videm{3}}\videm{1}\},
\end{gather}
one checks that $\calg(\typeAt_{3})$ is also relative cellular with this choice. 
But, in contrast to the choice in 
\fullref{example:rel-cell-algebras-3}, the cell module  $\dmod(\videm{i},e_{\videm{i}})$ 
is now the maximal quotient of $\prmod(\videm{i})$ with all composition factors $\lmod(\videm{j})$ 
satisfying $\videm{i}\leq_{e_\videm{i}}\videm{j}$. This is reminiscent of the properties of 
standard modules for quasi-hereditary algebras and was for example used in \cite{xi} to give 
homological characterizations of when a cellular algebra is quasi-hereditary. In the relative 
cellular case these homological characterizations depend decisively 
on the choice of idempotents and the partial orders.

We also stress that the $\calg(\typeAt_{3})(\ciset)$-module structure of $\calg(\typeAt_{3})$ 
depends on $\ciset$. This can be seen by comparing 
the cases with $\ciset$ being as in \ref{example:rel-cell-algebras-3} and $\ciset$ being as in \eqref{eq:prime-set}. 
Moreover, in both cases the sets of the isomorphism classes of simples of $\calg(\typeAt_{3})$ and 
$\calg(\typeAt_{3})(\ciset)$ contain three one-dimensional modules, but the indecomposable projectives of 
$\calg(\typeAt_{3})(\ciset)$ depend on the choice of $\ciset$.
\end{remark}

\begin{example}\label{example:rel-cell-algebras-4}
(We use a notation similar as in \fullref{example:rel-cell-algebras-2}.)
As in \fullref{example:rel-cell-algebras-3} we use the 
graphs $\typeAt_{n}$ to define 
a quiver algebra $\calg^{\prime}(\typeAt_{n})$. 
But we impose the relations in \eqref{eq:rel-cell-algebras-4} 
instead of those in \eqref{eq:rel-cell-algebras-3}. 
(We keep the anti-involution ${}^{\invo}$.)
\begin{gather}\label{eq:rel-cell-algebras-4}
\text{Relations: }
\begin{gathered}
\text{All $2$-cycles at the vertex $\videm{i}$ are equal, i.e. }
\pathss{i}{j}{i}=\pathss{i}{k}{i};
\\
\text{Going around the circle is zero, e.g. }
\pathssss{1}{2}{\cdots}{n}{1}=0.
\end{gathered}
\end{gather}
The Cartan matrices are, up to base change, now
\begin{gather}
\cmatrix(\calg^{\prime}(\typeAt_{3}))=
\begin{psmallmatrix}
3 & 3 & 3 \\
3 & 3 & 3 \\
3 & 3 & 3 \\
\end{psmallmatrix}
,\;\;
\cmatrix(\calg^{\prime}(\typeAt_{4}))=
\begin{psmallmatrix}
4 & 4 & 4 & 4\\
4 & 4 & 4 & 4\\
4 & 4 & 4 & 4\\
4 & 4 & 4 & 4\\
\end{psmallmatrix}
,\;\;
\cmatrix(\calg^{\prime}(\typeAt_{5}))=
\begin{psmallmatrix}
5 & 5 & 5 & 5 & 5\\
5 & 5 & 5 & 5 & 5\\
5 & 5 & 5 & 5 & 5\\
5 & 5 & 5 & 5 & 5\\
5 & 5 & 5 & 5 & 5\\
\end{psmallmatrix}
,\;\;
\text{etc.},
\end{gather}
which are not positive definite giving us that the
$\calg^{\prime}(\typeAt_{n})$ are, by see \fullref{remark:pos-def}, not cellular algebras. 
However, they are relative cellular, where we
as before discuss the $n=3$ case in detail, 
the general case being similar. We can take
\begin{gather}
\begin{gathered}
\cset
=\{\videm{3}\ord{e_{\videm{1}}}\videm{2}\ord{e_{\videm{1}}}\videm{1}\}
=\{\videm{1}\ord{e_{\videm{2}}}\videm{3}\ord{e_{\videm{2}}}\videm{2}\}
=\{\videm{2}\ord{e_{\videm{3}}}\videm{1}\ord{e_{\videm{3}}}\videm{3}\},
\\
\cmset(\videm{1})=\{e_{\videm{1}},\paths{3}{1},\pathss{2}{3}{1}\},
\quad
\cmset(\videm{2})=\{e_{\videm{2}},\paths{1}{2},\pathss{3}{1}{2}\},
\quad
\cmset(\videm{3})=\{e_{\videm{3}},\paths{2}{3},\pathss{1}{2}{3}\},
\\
\cbas{S,T}{\videm{i}}=S\circ T^{\invo},
\quad
\ciset=\{e_{\videm{1}},e_{\videm{2}},e_{\videm{3}}\},
\quad
\cepsmap(\paths{i}{\cdot})=e_{\videm{i}}.
\end{gathered}
\end{gather}
The relative cellular basis and the cell modules are then
\begin{gather}
\begin{tikzpicture}[baseline=(current bounding box.center)]
  \matrix (m) [matrix of math nodes, nodes in empty cells, row 1/.style={nodes={minimum height=7mm}}, row 3/.style={nodes={minimum height=13mm}}, row sep={.1cm}, column sep={3.0cm,between origins}, text height=1.5ex, text depth=0.25ex, ampersand replacement=\&] {
 \& \videm{1} \&  
\\
e_{\videm{1}} \& \paths{3}{1} \& \pathss{2}{3}{1}
\\
\paths{2}{1} \& \renewcommand*{\arraystretch}{.6}\begin{matrix}
\pathss{1}{2}{1}\\
=\\
\pathss{1}{3}{1}\end{matrix} \& \renewcommand*{\arraystretch}{.6}\begin{matrix}
\pathsss{3}{1}{2}{1}\\
=\\
\pathsss{3}{1}{3}{1}\end{matrix}
\\
\pathss{3}{2}{1} \& \renewcommand*{\arraystretch}{.6}\begin{matrix}\pathsss{2}{3}{2}{1}
\\ = \\ \pathsss{2}{1}{2}{1}\end{matrix} \& \renewcommand*{\arraystretch}{.6}\begin{matrix}\pathssss{1}{2}{3}{2}{1}
\\ = \\ \pathssss{1}{2}{1}{2}{1}\end{matrix}
\\
    };
	\draw[thin,dotted] ($(m-1-1.south)+(-.75,.2)$) to ($(m-1-1.south)+(11.25,.2)$);
	\draw[thin,dotted] ($(m-1-1.east)+(7.25,.25)$) to ($(m-1-1.east)+(7.25,-3.2)$);
	\draw[thin,dotted] ($(m-1-1.east)+(9.25,.25)$) to ($(m-1-1.east)+(9.25,-3.2)$);
	\draw[thin, mygreen] (3.8,1.1) rectangle (-3.8,.15);
	\draw[thin, myred] (3.8,.1) rectangle (-3.8,-.85);
	\draw[thin, myblue] (3.8,-.9) rectangle (-3.8,-1.85);
	\draw[thin, white, fill=white] (-2.3,.6) rectangle (-1.3,-1.4);
	\node[mygreen] at (-1.7,.6) {$\dmod(\videm{1})$};
	\node[myred] at (-1.7,-.4) {$\dmod(\videm{2})$};
	\node[myblue] at (-1.7,-1.4) {$\dmod(\videm{3})$};
	\draw[thick,->] (-1.7,.35) to node [left] {{\tiny$\ord{e_{\videm{1}}}$}} (-1.7,0) to (-1.7,-.1);
	\draw[thick,->] (-1.7,-.65) to node [left] {{\tiny$\ord{e_{\videm{1}}}$}} (-1.7,-1) to (-1.7,-1.15);
	\node at (5.2,1.35) {$\videm{2}$};
	\node[myred] at (5.2,.6) {$\dmod(\videm{2})$};
	\node[myblue] at (5.2,-.4) {$\dmod(\videm{3})$};
	\node[mygreen] at (5.2,-1.4) {$\dmod(\videm{1})$};
	\draw[thick,->] (5.2,.35) to node [left] {{\tiny$\ord{e_{\videm{2}}}$}} (5.2,0) to (5.2,-.1);
	\draw[thick,->] (5.2,-.65) to node [left] {{\tiny$\ord{e_{\videm{2}}}$}} (5.2,-1) to (5.2,-1.15);
	\node at (7.2,1.35) {$\videm{3}$};
	\node[myblue] at (7.2,.6) {$\dmod(\videm{3})$};
	\node[mygreen] at (7.2,-.4) {$\dmod(\videm{1})$};
	\node[myred] at (7.2,-1.4) {$\dmod(\videm{2})$};
	\draw[thick,->] (7.2,.35) to node [left] {{\tiny$\ord{e_{\videm{3}}}$}} (7.2,0) to (7.2,-.1);
	\draw[thick,->] (7.2,-.65) to node [left] {{\tiny$\ord{e_{\videm{3}}}$}} (7.2,-1) to (7.2,-1.15);
\end{tikzpicture}
\end{gather}
with the cell modules in the second and third columns 
being analog.

The indecomposable projectives themselves are
\begin{gather}
\begin{tikzpicture}[baseline=(current bounding box.center)]
	\node[mygreen] at (1,1.5) {$\prmod(\videm{1})=\calg^{\prime}(\typeAt_{3})e_{\videm{1}}$};
	\draw[thin, mygreen, dotted] (-.5,1.25) rectangle (2.5,-1.85);
	\node[mygreen] at (0,.5) {$\lmod(\videm{1})$};
	\node[myblue] at (0,0) {$\lmod(\videm{3})$};
	\node[myred] at (0,-.5) {$\lmod(\videm{2})$};
	\draw[thin, mygreen] (-.4,.75) rectangle (.4,-.75);
	\node[mygreen] at (0,1) {$\dmod(\videm{1})$};
	\node[myred] at (1,0) {$\lmod(\videm{2})$};
	\node[mygreen] at (1,-.5) {$\lmod(\videm{1})$};
	\node[myblue] at (1,-1) {$\lmod(\videm{3})$};
	\draw[thin, myred] (.6,.25) rectangle (1.4,-1.25);
	\node[myred] at (1,.5) {$\dmod(\videm{2})$};
	\node[myblue] at (2,-.5) {$\lmod(\videm{3})$};
	\node[myred] at (2,-1) {$\lmod(\videm{2})$};
	\node[mygreen] at (2,-1.5) {$\lmod(\videm{1})$};
	\draw[thin, myblue] (1.6,-.25) rectangle (2.4,-1.75);
	\node[myblue] at (2,0) {$\dmod(\videm{3})$};
\end{tikzpicture}
,\quad
\begin{tikzpicture}[baseline=(current bounding box.center)]
	\node[myred] at (1,1.5) {$\prmod(\videm{2})=\calg^{\prime}(\typeAt_{3})e_{\videm{2}}$};
	\draw[thin, myred, dotted] (-.5,1.25) rectangle (2.5,-1.85);
	\node[myred] at (0,.5) {$\lmod(\videm{2})$};
	\node[mygreen] at (0,0) {$\lmod(\videm{1})$};
	\node[myblue] at (0,-.5) {$\lmod(\videm{3})$};
	\draw[thin, myred] (-.4,.75) rectangle (.4,-.75);
	\node[myred] at (0,1) {$\dmod(\videm{2})$};
	\node[myblue] at (1,0) {$\lmod(\videm{3})$};
	\node[myred] at (1,-.5) {$\lmod(\videm{2})$};
	\node[mygreen] at (1,-1) {$\lmod(\videm{1})$};
	\draw[thin, myblue] (.6,.25) rectangle (1.4,-1.25);
	\node[myblue] at (1,.5) {$\dmod(\videm{3})$};
	\node[mygreen] at (2,-.5) {$\lmod(\videm{1})$};
	\node[myblue] at (2,-1) {$\lmod(\videm{3})$};
	\node[myred] at (2,-1.5) {$\lmod(\videm{2})$};
	\draw[thin, mygreen] (1.6,-.25) rectangle (2.4,-1.75);
	\node[mygreen] at (2,0) {$\dmod(\videm{1})$};
\end{tikzpicture}
,\quad
\begin{tikzpicture}[baseline=(current bounding box.center)]
	\node[myblue] at (1,1.5) {$\prmod(\videm{3})=\calg^{\prime}(\typeAt_{3})e_{\videm{3}}$};
	\draw[thin, myblue, dotted] (-.5,1.25) rectangle (2.5,-1.85);
	\node[myblue] at (0,.5) {$\lmod(\videm{3})$};
	\node[myred] at (0,0) {$\lmod(\videm{2})$};
	\node[mygreen] at (0,-.5) {$\lmod(\videm{1})$};
	\draw[thin, myblue] (-.4,.75) rectangle (.4,-.75);
	\node[myblue] at (0,1) {$\dmod(\videm{3})$};
	\node[mygreen] at (1,0) {$\lmod(\videm{1})$};
	\node[myblue] at (1,-.5) {$\lmod(\videm{3})$};
	\node[myred] at (1,-1) {$\lmod(\videm{2})$};
	\draw[thin, mygreen] (.6,.25) rectangle (1.4,-1.25);
	\node[mygreen] at (1,.5) {$\dmod(\videm{1})$};
	\node[myred] at (2,-.5) {$\lmod(\videm{2})$};
	\node[mygreen] at (2,-1) {$\lmod(\videm{1})$};
	\node[myblue] at (2,-1.5) {$\lmod(\videm{3})$};
	\draw[thin, myred] (1.6,-.25) rectangle (2.4,-1.75);
	\node[myred] at (2,0) {$\dmod(\videm{2})$};
\end{tikzpicture}
\end{gather}
which again have (quite heavy) cyclic patterns.
\end{example}

\begin{remark}\label{left-to-you}
In the above three examples we leave it to the reader to check 
that (\ref{definition:cell-algebra}.a) to (\ref{definition:cell-algebra}.d) hold. 
(For \fullref{example:rel-cell-algebras-2}: 
(\ref{definition:cell-algebra}.d) is the most crucial 
thing to be checked, with (\ref{definition:cell-algebra}.c) then being automatic. 
See also the proof of (\ref{proposition:cell-relative-cell}.b) 
and \fullref{remark-example-rel-cell-algebras-1}. 
For \fullref{example:rel-cell-algebras-3}: 
In this case \eqref{eq:idem-props-1} 
needs to be checked. It follows since 
e.g. $e_{\videm{1}}\calg(\typeAt_{3})e_{\videm{1}}$ equals the linear 
span of all $2$-cycles at the vertex $\videm{1}$ that 
are either $e_{\videm{1}}$ or act on everything except 
$e_{\videm{1}}$ as zero.
For \fullref{example:rel-cell-algebras-4}: 
Again, \eqref{eq:idem-props-1} 
is non-trivial. However, it can be checked 
by keeping in mind that 
$e_{\videm{i}}\calg^{\prime}(\typeAt_{3})e_{\videm{i}}$ equals the linear 
span of all $2$-cycles at the vertex $\videm{i}$.)
\end{remark}

\begin{example}\label{example:rel-cell-algebras-res-enveloping-algebras}
Let $\K$ be a field of positive characteristic $p>0$.
In \fullref{section:rel-cell-res} we show that the restricted enveloping 
algebra $\uslt$ is relative cellular, but not cellular. 
(Except in case $p=2$ where $\uslt$ is actually 
already cellular, see \fullref{remark:p-is-two}.)

Similarly,
let $\K$ be any field and fix $\qpar\in\K$ to be a root of unity, $\qpar\neq\pm 1$.
The case of the so-called small quantum group $\suslt$ at $\qpar$
associated to $\slt$ (see e.g. \cite{lu1})
works mutatis mutandis as for $\uslt$, i.e. $\suslt$ is relative cellular, 
but not cellular as long as $\qpar\neq\pm\sqrt{-1}$.
\end{example}

\begin{example}\label{example:rel-cell-algebras-5}
Another example is an annular version of 
arc algebras $\aarc$ that we discuss in detail in 
\fullref{section:arc-stuff}. Note that 
$\aarc$ is again not a cellular algebra, but only 
a relative cellular algebra, cf. \fullref{proposition:not-cell}.
\end{example}

\begin{furtherdirections}\label{remark:tilting}
The most famous examples of cellular algebras are 
coming from centralizer algebras as e.g. Hecke, Temperley--Lieb 
or Brauer algebras. These arise from fairly general 
constructions via the theory of tilting modules, see e.g. 
\cite{ast1} or \cite[Appendix A]{bt2}. We do not know what
the relative version of this is.
\end{furtherdirections}
\section{Simple and projective modules}\label{section:basic-cell-props}

\makeautorefname{theorem}{Theorems}

In the present section we discuss the representation theory 
of relative cellular algebras, following \cite[Sections 2 and 3]{gl1}. 
We stress hereby that some of the statements, 
e.g. \fullref{theorem:simple-set} 
and \ref{theorem:bgg}, hold verbatim as for cellular algebras. 
However, our proofs here are, and have to be, quite different.

\makeautorefname{theorem}{Theorem}

We continue to use the notation from 
\fullref{section:basic-cell}. In particular, 
$\calg$ denotes a relative cellular algebra with 
relative cell datum as in \eqref{eq:rel-cell-datum}.

\subsection{Simple quotients of cell modules}\label{subsection:simples}

First, we define a bilinear 
form on cell modules to get a better handle on their structure.

\begin{lemma}\label{lemma:pairing-first}
Let $\lambda\in\cset$ and $a\in\calg$. 
Then, for $S,T,U,V \in \cmset(\lambda)$, we have
\begin{gather}
\cbas{U,S}{\lambda}\,a\,\cbas{T,V}{\lambda} 
\in 
\cpair{a}(S,T)\cbas{U,V}{\lambda} 
+ 
\left(
\ceps_{U}\calg(\ord{\ceps_{U}}\!\lambda)
\cap
\calg(\ord{\ceps_{V}}\!\lambda)\ceps_{V}
\right),
\end{gather}
where $\cpair{a}(S,T)
=r_{\cbas{U,S}{\lambda}a}(U,T)
=r_{a^{\invo}\cbas{V,T}{\lambda}}(V,S)\in\K$.
\end{lemma}

\begin{proof}
We apply \eqref{eq:mult-left} respectively 
\eqref{eq:mult-right} and compare coefficients. 
The statement then follows immediately.
\end{proof}

Thus, we can define $\cpair{a}(S,T)$ as 
in \fullref{lemma:pairing-first} and this definition
is independent of $U,V\in\cmset(\lambda)$. Of special
importance is the case where $a=\ceps_\lambda$ is 
a local unit for the set $\{\cbas{S,T}{\lambda} \mid S,T \in \cmset(\lambda)\}$, where 
we observe that $\cpair{\ceps_\lambda}(S,T)$ is the same for any such local unit.

\begin{definition}\label{definition:pairing}
For $\lambda\in\cset$ we define a bilinear 
form $\Cpair{\lambda}\colon\dmod(\lambda)\times\dmod(\lambda)\rightarrow\K$ 
by setting $\Cpair{\lambda}(\dbas{S}{\lambda},\dbas{T}{\lambda}) 
=\cpair{\ceps_\lambda}(S,T)$ for $S,T\in\cmset(\lambda)$, and extending bilinearly.
\end{definition}

For (\ref{proposition:form-properties}.c) of the following lemma recall 
$\diso{\lambda}$ as defined in \fullref{lemma:cell-dual}. 
Its proof is mutatis mutandis as in \cite[Proposition 2.4]{gl1} and omitted.

\begin{lemmaqed}\label{proposition:form-properties}
For $\lambda\in\cset$ we have the following.
\smallskip
\begin{enumerate}[label=(\alph*)]

\setlength\itemsep{.15cm}

\item The bilinear form $\Cpair{\lambda}$ is symmetric.

\item For $a\in\calg$ and $x,y\in\dmod(\lambda)$ we 
have $\Cpair{\lambda}(a\acts x,y) 
=\Cpair{\lambda}(x,a^{\invo}\acts y)$.

\item For $u,x,y\in\dmod(\lambda)$ 
we have $\diso{\lambda}(u \otimes x)\acts y 
=\Cpair{\lambda}(x,y)u$.\qedhere

\end{enumerate} 
\end{lemmaqed}

The main use of $\Cpair{\lambda}$ is 
\fullref{corollary:radical} below: Elements of $\dmod(\lambda)$ 
not contained in the radical of $\Cpair{\lambda}$ 
are cyclic generators for $\dmod(\lambda)$. Hereby, 
as usual, the \textit{radical of $\Cpair{\lambda}$} is 
linear subspace of $\dmod(\lambda)$ given by
$\radi{\lambda} = 
\{x\in\dmod(\lambda)\mid\Cpair{\lambda}(x,y)=0 
\text{ for all } y \in \dmod(\lambda)\}$.

\begin{lemma}\label{lemma:cell-cyclic}
Let $\lambda\in\cset$ and 
$z\in\dmod(\lambda)$. Then
\begin{gather}
\calg(\{\lambda\})\acts z 
= 
\img(\Cpair{\lambda}({}_{-},z)) 
\dmod(\lambda)\subset\calg\acts z.
\end{gather}
In particular, if 
$\img(\Cpair{\lambda}({}_{-},z))=\K$, then 
we have $\dmod(\lambda)=\calg(\{\lambda\})\acts z=\calg\acts z$.
\end{lemma}

\begin{proof}
Let $y\in\dmod(\lambda)$ and 
$S,T\in\cmset(\lambda)$. 
By (\ref{proposition:form-properties}.c) we have
\begin{gather}
\cbas{S,T}{\lambda}\acts z 
= 
\diso{\lambda}
(\dbas{S}{\lambda}\otimes\dbas{T}{\lambda})\acts z 
= 
\Cpair{\lambda}(\dbas{T}{\lambda},z)\dbas{S}{\lambda}
\in
\img(\Cpair{\lambda}({}_{-},z)) 
\dmod(\lambda),
\end{gather}
and conversely
\begin{gather}
\Cpair{\lambda}(y,z)\dbas{S}{\lambda} 
= 
\diso{\lambda}(\dbas{S}{\lambda} \otimes y)\acts z 
\in\calg(\{\lambda\})\acts z.
\end{gather}
Hence, we have equality. The special case is then clear.
\end{proof}

Since we work over a field we get as a direct consequence:

\begin{corollary}\label{corollary:radical}
We have
$z\in\dmod(\lambda)\setminus\radi{\lambda}$ if and only if 
$\calg(\{\lambda\})\acts z =\dmod(\lambda)$.
\end{corollary}

Next,  
$\radi{\lambda}$ allows us to deduce that cell modules 
have either a trivial or a simple head.

\begin{proposition}\label{proposition:to-define-Ls}
Let $\lambda\in\cset$.
\smallskip
\begin{enumerate}[label=(\alph*)]

\setlength\itemsep{.15cm}

\item The radical 
$\radi{\lambda}$ is a submodule 
of $\dmod(\lambda)$.

\item If $\Cpair{\lambda}$ is non-zero, then 
$\neatafrac{\dmod(\lambda)}{\radi{\lambda}}$ is simple.

\item If $\Cpair{\lambda}$ is non-zero, then 
$\neatafrac{\dmod(\lambda)}{\radi{\lambda}}$ 
is the head of $\dmod(\lambda)$.\qedhere
\end{enumerate}
\end{proposition}

\begin{proof}
\textit{(\ref{proposition:to-define-Ls}.a).} 
This follows immediately from (\ref{proposition:form-properties}.b).
\smallskip

\noindent\textit{(\ref{proposition:to-define-Ls}.b).} 
By \fullref{corollary:radical}, any 
$z\in\dmod(\lambda)\setminus\radi{\lambda}$ generates $\dmod(\lambda)$.
Thus, the claim follows.
\smallskip

\noindent\textit{(\ref{proposition:to-define-Ls}.c).} 
Again by \fullref{corollary:radical}, any 
$z\in\dmod(\lambda)\setminus\radi{\lambda}$ generates $\dmod(\lambda)$. 
Hence, any proper 
submodule of 
$\dmod(\lambda)$ is 
contained in $\radi{\lambda}$. Thus, $\radi{\lambda}$ 
is the unique maximal submodule of $\dmod(\lambda)$ and 
so equal to the (representation theoretical) radical $\Radi(\dmod(\lambda))$. 
(Recall that $\Radi(\dmod(\lambda))$ is 
intersection of all proper, maximal submodules of $\dmod(\lambda)$.)
\end{proof}

We write $\csetz=\{\lambda\in\cset\mid\Cpair{\lambda}\text{ is non-zero}\}$.
Having \fullref{proposition:to-define-Ls}
we can define:

\begin{definition}\label{definition:the-Ls}
For $\lambda\in\csetz$, 
we set $\lmod(\lambda)=\neatafrac{\dmod(\lambda)}{\radi{\lambda}}$.
\end{definition}

\subsection{Morphisms between cell modules}\label{subsection:cell-morphisms}

In contrast to the setup of cellular algebras, 
the existence of morphisms between cell modules 
is a less useful tool as we will see.

\begin{lemma}\label{lemma:submodule-lemma}
Let $\lambda\in\csetz$, 
$\mu\in\cset$, and 
$f \in \homc(\dmod(\lambda),\neatafrac{\dmod(\mu)}{N})$ non-zero for 
some submodule $N\subset\dmod(\mu)$. Then there exists 
$S\in\cmset(\lambda)$ such that 
$\mu\Ord{\ceps_{S}}\!\lambda$.
\end{lemma}

\begin{proof}
Since $\Cpair{\lambda}$ is non-zero there exists -- 
by \fullref{corollary:radical} -- a generator $z \in \dmod(\lambda)$ 
such that $\calg(\{\lambda\})\acts z=\dmod(\lambda)$. 
Then there exists $a\in\calg(\{\lambda\})$ such 
that $f(a\acts z)=a\acts f(z)\neq 0$, i.e. there exist 
$U,U^{\prime}\in\cmset(\mu)$ such that $r_{a}(U,U^{\prime}) \neq 0$.

This implies that 
there exist $S,T\in\cmset(\lambda)$ such 
that for all $V \in \cmset(\mu)$ the expansion 
of $\cbas{S,T}{\lambda}\cbas{U,V}{\mu}$, using 
(\ref{lemma:cell-algebra-star-1}.d), contains a 
non-zero summand in $\calg(\{\mu\})$. Thus, 
$\mu\Ord{\ceps_{S}}\!\lambda$.
\end{proof}

As can be seen in \fullref{lemma:submodule-lemma}, 
it is possible to have morphism in both 
``directions'', and obtain 
$\lambda\Ord{\ceps}\!\mu\Ord{\ceps^{\prime}}\!\lambda$. 
But we might still have $\lambda\neq\mu$ in case 
$\ceps\neq\ceps^{\prime}$. This is in contrast to 
the framework of cellular algebras.

Let us give an alternative formulation of \fullref{lemma:submodule-lemma}.

\begin{lemma}\label{lemma:submodule-lemma-alt2}
Let $\lambda,\mu\in\cset$ and $S,T\in\cmset(\lambda)$ such 
that $\cbas{S,T}{\lambda}\acts\dmod(\mu)\neq 0$ for some basis element $\cbas{S,T}{\lambda}$. 
Then $\mu\Ord{\ceps_{S}}\!\lambda$.
\end{lemma}

\begin{proof}
By assumption there exists 
$U,V\in\cmset(\mu)$ such that 
the expansion 
of $\cbas{S,T}{\lambda}\cbas{U,V}{\mu}$, using 
(\ref{lemma:cell-algebra-star-1}.d), contains a 
non-zero summand in $\calg(\{\mu\})$. Thus, $\mu\Ord{\ceps_{S}}\lambda$.
\end{proof}

Despite the fact that hom-spaces 
between cell modules are not as useful as 
in the case of cellular algebras, 
the following is surprisingly still true.

\begin{propositionqed}\label{proposition:Delta-maps}
If $\lambda\in\csetz$, then $\ehomc(\dmod(\lambda))=\K$.
\end{propositionqed}

\begin{proof}
We prove the following claim, which immediately 
implies the proposition.
\medskip

\noindent\setword{\ref{proposition:Delta-maps}.Claim}{Delta-maps}.
Let $\lambda\in\csetz$ and let $N\subset\dmod(\lambda)$ be some submodule. 
Then any element $f\in\homc(\dmod(\lambda),\neatafrac{\dmod(\lambda)}{N})$ 
is of the form $f(x)=rx+N$ for some $r\in\K$.
\medskip

\noindent Proof of \ref{Delta-maps}.
By assumption we can 
choose $y,y^{\prime}\in \dmod(\lambda)$ such that 
$\Cpair{\lambda}(y,y^{\prime})=1$. 
(Recall that we work over a field.) Fix $u$ 
such that $f(y^{\prime})=u+N$ 
and set $r=\Cpair{\lambda}(y,u)$. Then
$f(x)=f(\Cpair{\lambda}(y,y^{\prime})x)
=\diso{\lambda}(x\otimes y)\acts f(y^{\prime})=
\diso{\lambda}(x\otimes y)\acts u+N$. 
Hence, we get 
$f(x)= \Cpair{\lambda}(y,u)x +N = rx+N$.
\end{proof}

\subsection{Projective modules}\label{subsec:projectives}

We have already seen in \fullref{subsection:cell-morphisms} 
that some statements from cellular algebras 
are quite different in the relative setup. Even more, from now on 
the relative setup needs some very careful treatment of 
the involved partial orders, all of which is trivial for cellular algebras. 

We start with some 
statements about idempotents. 
In the following we 
call an idempotent 
$e\in\calg$ an 
\textit{idempotent summand of $\ceps\in\ciset$} 
if $\ceps e=e=e \ceps$. 
In this case we write $e\isum\ceps$.

\begin{remark}\label{remark:decomposition-unit}
By \fullref{lemma:decomposition-unit}, at least in case 
$\csetf$, we can restrict 
our attention to $e\isum\ceps$: Since we 
get a(n orthogonal) decomposition of the unit, we can find 
$\ceps\in\ciset$ for all 
indecomposable projectives $\prmod$ of $\calg$ 
such that $\prmod\cong\calg e$ for 
primitive $e\isum\ceps$. Thus, 
up to isomorphism, it suffices to study the projectives 
of the form $\calg e$ for $e\isum\ceps$.
\end{remark}

\begin{lemma}\label{lemma:cell-algebra-star-2}
Let $e\isum\ceps$ and 
$\cideal_{\ceps}$ an 
$\ord{\ceps}$-ideal. Then the following hold.
\smallskip
\begin{enumerate}[label=(\alph*)]

\setlength\itemsep{.15cm}

\item One has $e\calg(\{\lambda\})\subset\calg(\Ord{\ceps}\lambda)\supset\calg(\{\lambda\})e$.

\item One has $e\calg(\cideal_{\ceps})\subset\calg(\cideal_{\ceps})
\supset\calg(\cideal_{\ceps})e$.

\item One has $\calg(\cideal_{\ceps})e 
=\calg(\cideal_{\ceps})\cap\calg e$, 
and $e\calg(\cideal_{\ceps}) = 
\calg(\cideal_{\ceps})\cap e\calg$.

\item One has $e\in\K\{
\cbas{S,T}{\lambda}
\mid
\lambda\in\cset,S,T\in\cmset(\lambda),\ceps_{S}=\ceps_{T}=\ceps
\}$.\qedhere
\end{enumerate}
\end{lemma}

\begin{proof}
\textit{(\ref{lemma:cell-algebra-star-2}.a).} 
By \eqref{eq:idem-props-1} and 
(\ref{lemma:cell-algebra-star-1}.a), since $\ceps e=e=e \ceps$ implies that
$e\in\ceps\calg\ceps$.
\medskip

\noindent\textit{(\ref{lemma:cell-algebra-star-2}.b).} This 
follows from (\ref{lemma:cell-algebra-star-2}.a) 
since $\cideal_{\ceps}$ is an $\ord{\ceps}$-ideal.
\medskip

\noindent \textit{(\ref{lemma:cell-algebra-star-2}.c).} 
We only prove the first statement, the second 
is obtained by applying ${}^{\invo}$. By definition we get
$\calg(\cideal_{\ceps}) e \subset \calg e$, 
and by (\ref{lemma:cell-algebra-star-2}.a) we get
$\calg(\cideal_{\ceps}) e \subset \calg(\cideal_{\ceps})$. Hence, 
the left-hand side is contained in the right-hand side. 
Let $ae \in \calg(\cideal_{\ceps}) \cap \calg e$. 
We expand and -- by assumption -- obtain 
$ae = \sum_{\mu \in \cideal_{\ceps}, S,T \in \cmset(\mu)} 
r_{\mu,S,T} \cbas{S,T}{\mu}$ for some scalars $r_{\mu,S,T}\in\K$. Thus,
\begin{gather}
ae = (ae)e = 
{\textstyle\sum_{\mu \in \cideal_{\ceps}, S,T \in \cmset(\mu)}}\,
r_{\mu,S,T} \cbas{S,T}{\mu} e 
\in\calg(\cideal_{\ceps})e.
\end{gather}
It follows that the right-hand side is also contained 
in the left-hand side.
\medskip

\noindent\textit{(\ref{lemma:cell-algebra-star-2}.d).} 
This follows immediately from 
\fullref{lemma:epsilon-structure} by assumption on $e$.
\end{proof}

\begin{definition}\label{definition:idempotent-order}
For $e\isum\ceps$ 
we define a partial order $\ord{e}$ on $\cset$ as 
being $\ord{\ceps}$.
\end{definition}

We write $\ord{e}=\ord{\ceps}$ etc. in the following.

If the partial order with respect to which an ideal in 
$\cset$ is defined agrees with the partial order 
$\ord{e}$ for some $e\isum\ceps$, then 
we can define submodules inside the corresponding 
projective module $\prmod_{e}=\calg e$ to obtain suitable filtrations.

\begin{lemma}\label{lemma:projective-submodule-lemma}
Let $e\isum\ceps$ and 
$\cideal_{\ceps}$ a $\ord{\ceps}$-ideal. 
Then $\calg(\cideal_{\ceps})e$ is a submodule.

In case $\csetf$, there exists a 
filtration $\prmod_{e}=\prmod_{0} \supset \prmod_{1} 
\supset\cdots 
\supset\prmod_{r} =\{0\}$ such 
that $\neatafrac{\prmod_{i}}{\prmod_{i+1}}=\prmod_{e}(\{\lambda_{i}\})$ 
for some $\lambda_{i}\in\cset$.
\end{lemma}

Hereby, similarly to \eqref{eq:ideal-algebra-notation},
we let 
$\prmod_{e}(\{\lambda\})=
\neatafrac{\calg(\Ord{e}\lambda)e}{\calg(\ord{e}\lambda)e}$.

\begin{proof}
For $\cbas{S,T}{\lambda} \in \calg(\cideal_{\ceps})$ we have
\begin{gather}
a \cbas{S,T}{\lambda} e 
= 
{\textstyle\sum_{S^{\prime} \in \cmset(\lambda)}}\,
r_a(S^{\prime},S) \cbas{S^{\prime},T}{\lambda} e + \friends
\end{gather}
with $\friends\in\calg(\ord{\ceps_{T}}\lambda)\ceps_{T}e$ 
by (\ref{definition:cell-algebra}.d). Then either 
$\ceps_{T}e=0$ in case $\ceps_{T}\neq\ceps$, 
and the extra terms just vanish, or $\ord{e}=\ord{\ceps_{T}}$ 
and $\ceps_{T} e=e$. Hence, $\friends\in\calg(\cideal_{\ceps})$.

Finally, 
choose a maximal chain of $\ord{e}$-ideals -- whose existence is guaranteed 
by $\csetf$ --  
and the statement about the filtration follows immediately.
\end{proof}

Analogously to \fullref{lemma:cell-dual}, we let 
$\piso{\lambda}\colon\dmod(\lambda)
\otimes
(\dmod(\lambda)^{\invo}\acts e)
\rightarrow\prmod_{e}(\{\lambda\})$ defined via
$\piso{\lambda}(\dbas{S}{\lambda} 
\otimes(\dbas{T}{\lambda}\acts e)) 
=\cbas{S,T}{\lambda}e$. 
(Below we write 
$\dbas{S}{\lambda} 
\otimes\dbas{T}{\lambda}\acts e$ etc. for short.)
Note that the first step of the proof of 
\fullref{proposition:lambda-piece-iso} shows that $\piso{\lambda}$ is well-defined.

\begin{proposition}\label{proposition:lambda-piece-iso}
Let $\lambda\in\cset$ and $e\isum\ceps$. 
Then $\piso{\lambda}$ 
is an $\calg$-module isomorphism.
If additionally $\lambda\in\csetz$, then
$\homc(\prmod_{e}(\{\lambda\}),\dmod(\lambda)) 
\cong
\homk(\dmod(\lambda)^{\invo}\acts e,\K)$.
\end{proposition}

\begin{proof}
\textit{Well-definedness of $\piso{\lambda}$.} 
Define $\overline{\piso{\lambda}}(\dbas{S}{\lambda},\dbas{T}{\lambda}\acts e) 
= \cbas{S,T}{\lambda}e$ and extend bilinearly to obtain 
$\overline{\piso{\lambda}}\colon\dmod(\lambda) 
\times 
\dmod(\lambda)^{\invo}\acts e 
\rightarrow\prmod_{e}(\{\lambda\})$. 
If $\overline{\piso{\lambda}}$ is well-defined, 
then it is by definition bilinear. 
So let $\sum_{T\in\cmset(\lambda)} 
r_{T}(\dbas{S}{\lambda},\dbas{T}{\lambda}\acts e)=0$ 
for some scalar $r_{T}\in\K$ and 
some element $[\dbas{S}{\lambda},\dbas{T}{\lambda}\acts e]\in\dmod(\lambda)\times 
\dmod(\lambda)^{\invo}\acts e$. Then
\begin{gather}
{\textstyle\sum_{T \in \cmset(\lambda)}}\,
r_{T}[\dbas{S}{\lambda},\dbas{T}{\lambda}\acts e]
=
{\textstyle\sum_{T,T^{\prime}\in \cmset(\lambda)}}\,
r_{T}r_{e^{\invo}}(T^{\prime},T)[\dbas{S}{\lambda},\dbas{T^{\prime}}{\lambda}].
\end{gather}
Hence, $\sum_{T \in \cmset(\lambda)} 
r_{T}r_{e^{\invo}}(T^{\prime},T)=0$ for all $T^{\prime}\in \cmset(\lambda)$, 
and we have
\begin{gather} 
\begin{gathered}
\overline{\piso{\lambda}}
\left(
{\textstyle\sum_{T \in \cmset(\lambda)}}\,
r_{T}[\dbas{S}{\lambda},\dbas{T}{\lambda}\acts e]
\right)
=
{\textstyle\sum_{T \in \cmset(\lambda)}}\,
r_{T}\cbas{S,T}{\lambda}e
\\
=
{\textstyle\sum_{T,T^{\prime} \in \cmset}}\,
r_{T}r_{e^{\invo}}(T^{\prime},T)
\cbas{S,T^{\prime}}{\lambda} + \friends.
\end{gathered}
\end{gather}
Hereby $\friends\in\calg(\ord{\ceps_{S}}\lambda)$ by 
\eqref{eq:mult-right} and $\friends\in\calg(\Ord{e}\lambda)$ 
by (\ref{lemma:cell-algebra-star-2}.b), 
together giving $\friends\in\calg(\ord{e}\lambda)$. 
Since we also 
have that $\friends=\friends e$, it follows 
that $\friends\in\calg e$. By 
(\ref{lemma:cell-algebra-star-2}.c) we then get that 
$\friends\in\calg(\ord{e}\lambda)e$ and so it 
vanishes in $\prmod_{e}({\lambda})$. Thus, $\overline{\piso{\lambda}}$ 
is well-defined and consequently $\piso{\lambda}$ as well.
\medskip

\noindent\textit{Surjectivity of $\piso{\lambda}$.}
This is immediate by noting that 
-- due to (\ref{lemma:cell-algebra-star-2}.c) -- $\prmod_{e}(\{\lambda\})$ is 
generated by elements of the form $\cbas{S,T}{\lambda}e$ for 
$S,T\in\cmset(\lambda)$ and these are in the image 
of $\piso{\lambda}$.
\medskip

\noindent\textit{Injectivity of $\piso{\lambda}$.}
Let $\sum_{S,T \in \cmset}r_{S,T}\dbas{S}{\lambda} 
\otimes \dbas{T}{\lambda}\acts e$ be in the kernel of $\piso{\lambda}$ 
for some scalars $r_{S,T}\in\K$, i.e. 
$\sum_{S,T\in\cmset(\lambda)}r_{S,T}\cbas{S,T}{\lambda}e 
\in\calg(\ord{e}\lambda)e$. By 
(\ref{lemma:cell-algebra-star-2}.c) we have $\calg(\ord{e}\lambda)e = 
\calg(\ord{e}\lambda)\cap\calg e$ and so 
expanding with \eqref{eq:mult-right} we obtain
\begin{gather}
{\textstyle\sum_{S,T \in \cmset(\lambda)}}\,
r_{S,T} \cbas{S,T}{\lambda}e 
= 
{\textstyle\sum_{S,T,T^{\prime} \in \cmset(\lambda)}}\,
r_{S,T} r_{e^{\invo}}(T^{\prime},T) \cbas{S,T^{\prime}}{\lambda} + \friends,
\end{gather}
with $\friends\in\calg(\ord{e}\lambda)$ by 
\eqref{eq:mult-right} and (\ref{lemma:cell-algebra-star-2}.b). Thus, 
$\sum_{S,T} r_{S,T} r_{e^{\invo}}(T^{\prime},T)=0$ for 
all $T^{\prime}\in\cmset(\lambda)$, due to (\ref{lemma:cell-algebra-star-2}.c). 
This in turn implies that
\begin{gather}
{\textstyle\sum_{S,T \in \cmset(\lambda)}}\,
r_{S,T} \dbas{S}{\lambda}
\otimes
\dbas{T}{\lambda}\acts e 
= 
{\textstyle\sum_{S,T,T^{\prime} \in \cmset(\lambda)}}\,
r_{S,T} r_{e^{\invo}}(T^{\prime},T)
\dbas{S}{\lambda}
\otimes \dbas{T^\prime}{\lambda} = 0.
\end{gather}
Hence, $\piso{\lambda}$ is injective. 
\medskip

\noindent\textit{$\piso{\lambda}$ is a $\calg$-module map.}
For $\piso{\lambda}$ to be a $\calg$-module map we observe that
\begin{gather}
a \cbas{S,T}{\lambda} e 
=
{\textstyle\sum_{S^{\prime} \in \cmset(\lambda)}}\,
r_{a}(S^{\prime},S) \cbas{S^{\prime},T}{\lambda}e + \friends e,
\end{gather}
where $\friends e\in 
\calg(\ord{\ceps_{S}}\lambda)\ceps_{S}e 
\subset \calg(\ord{e} \lambda)e$, 
which is zero in $\prmod_{e}(\{\lambda\})$. Thus, $\piso{\lambda}$ is 
a $\calg$-module map.
\medskip

Finally, for the isomorphism, let $\lambda\in\csetz$. By the above
\begin{gather} 
\begin{aligned}
\homc(\prmod_{e}(\{\lambda\}),\dmod(\lambda)) 
&
\cong
\homc(\dmod(\lambda)\otimes\dmod(\lambda)^{\invo}\acts e,\dmod(\lambda)) 
\\
&\cong 
\homk(\dmod(\lambda)^{\invo}\acts e,
\ehomc(\dmod(\lambda)))
\\
&\cong 
\homk(\dmod(\lambda)^{\invo}\acts e,\K),
\end{aligned}
\end{gather}
where the second isomorphism is the tensor-hom adjunction, and the 
last isomorphism follows from \fullref{proposition:Delta-maps}.
\end{proof}

In addition to statements 
about $\prmod_{e}(\{\lambda\})$, we will also need some 
knowledge about slightly more general quotients 
of $\calg(\cideal_{\ceps})e$.

\begin{lemma}\label{lemma:projective-quotients-sum}
Let $e\isum\ceps$ and $\cideal_{\ceps}$ an 
$\ord{\ceps}$-ideal. 
Assume that $\cideal_{\ceps}$ contains $\ord{\ceps}$-maximal 
elements $\lambda_{1},\cdots,\lambda_{r}$ and let 
$\cideal^{\prime}_{\ceps} = \cideal_{\ceps} 
\setminus \{\lambda_{1},\cdots,\lambda_{r}\}$. Then
\begin{gather}
\neatafrac{\calg(\cideal_{\ceps})e}{\calg(\cideal^{\prime}_{\ceps})e}
\cong 
\prmod_{e}(\{\lambda_{1}\})\oplus\cdots\oplus\prmod_{e}(\{\lambda_{r}\}),
\end{gather}
which is an isomorphism of $\calg$-modules.
\end{lemma}

\begin{proof}
Let $\cideal_{\ceps}^{\Ord{e}\lambda_{k}}
=
\{ \mu \in \cideal_{\ceps} \mid \mu \Ord{e}\lambda_{k} \}$ 
for $k=1,\cdots,r$, 
and define $\cideal_{\ceps}^{\ord{e}\lambda_{k}}$ analogously. 
By assumption, we have 
$\calg(\cideal_{\ceps}) 
=
\sum_{k=1}^r \calg(\cideal_{\ceps}^{\Ord{e}\lambda_{k}})$ 
and 
$\calg(\cideal^{\prime}_{\ceps}) 
=
\sum_{k=1}^{r} \calg(\cideal_{\ceps}^{\ord{e}\lambda_{k}})$. 
Additionally, we clearly 
have $\calg(\cideal^{\prime}_{\ceps})
\cap
\calg(\cideal_{\ceps}^{\Ord{e}\lambda_{k}})
=
\calg(\cideal_{\ceps}^{\ord{e}\lambda_{k}})$. 
Thus -- using (\ref{lemma:cell-algebra-star-2}.c) -- we 
obtain $\calg(\cideal^{\prime}_{\ceps})e \cap
\calg(\cideal_{\ceps}^{\Ord{e}\lambda_{k}})e =
\calg(\cideal_{\ceps}^{\ord{e}\lambda_{k}})e$. 
Hence, the image of 
$\calg(\cideal_{\ceps}^{\Ord{e}\lambda_{k}})e$ 
in $\neatafrac{\calg(\cideal_{\ceps})e}{\calg(\cideal^{\prime}_{\ceps})e}$ 
is isomorphic to 
$\prmod_{e}(\{\lambda_{k}\})=
\neatafrac{\calg(\cideal_{\ceps}^{\Ord{e}\lambda_{k}})e}{\calg(\cideal_{\ceps}^{\ord{e}\lambda_{k}})e}$.

In addition, for $1\leq k,l\leq r$ and $k\neq l$, 
\begin{gather}
\calg(\cideal_{\ceps}^{\Ord{e}\lambda_{k}})e 
\cap
\calg(\cideal_{\ceps}^{\Ord{e}\lambda_{l}})e 
=
\calg(\cideal_{\ceps}^{\Ord{e}\lambda_{k}})
\cap
\calg(\cideal_{\ceps}^{\Ord{e}\lambda_{l}})
\cap\calg e
=
\calg(\cideal_{\ceps}^{\ord{e}\lambda_{k}})e
\cap
\calg(\cideal_{\ceps}^{\ord{e}\lambda_{l}})e.
\end{gather}
Thus, the images of 
$\calg(\cideal_{\ceps}^{\Ord{e}\lambda_{k}})e$ 
and $\calg(\cideal_{\ceps}^{\Ord{e}\lambda_{l}})e$ 
in $\neatafrac{\calg(\cideal_{\ceps})e}{\calg(\cideal^{\prime}_{\ceps})e}$ 
have trivial intersection. Together this gives the statement.
\end{proof}

\subsection{Classification of simples}\label{subsec:classification}

Altogether we are now ready to 
prove the main statement of this section.

\begin{theorem}\label{theorem:simple-set}
Let $\csetf$.
The set $\{[\lmod(\lambda)]\mid\lambda\in\csetz\}$ 
gives a complete, non-redundant set of isomorphism 
classes of simple $\calg$-modules.
\end{theorem}

\begin{proof}
There are three statements to be proven: That the 
$\lmod(\lambda)$'s are simple, that all simples appear, and that 
$\lmod(\lambda)\cong\lmod(\mu)$ if and only if $\lambda=\mu$.
\medskip

\noindent\textit{Simplicity.}
By \fullref{proposition:to-define-Ls},
the $\lmod(\lambda)$ are simple $\calg$-modules.
\medskip

\noindent\textit{Completeness.}
Let $e\isum\ceps$, with $e$ being primitive. Then the head of its associated 
indecomposable projective
$\prmod_{e}$ is simple, and we can obtain every 
simple module by considering the heads of the 
indecomposable projectives of $\calg$.

Let $\cideal_{P}$ denote the $\ord{e}$-ideal 
in $\cset$ generated by $\{\lambda\in\cset\mid\prmod_{e}(\{\lambda\})\neq 0\}$. 
Thus, $\prmod_{e}=\calg(\cideal_{P})e$ and -- by 
(\ref{lemma:cell-algebra-star-2}.c) and by applying 
${}^{\invo}$ -- one has 
$e,e^{\invo}\in\calg(\cideal_{P})$.

Let $\lambda_{\max}\in\cideal_{P}$ be $\ord{e}$-maximal. 
Then -- by construction 
-- $\prmod_{e}(\{\lambda_{\max}\})\neq 0$.
\medskip

\noindent\setword{\ref{theorem:simple-set}.Claim.a}{mainthm-1}.
The form $\Cpair{\lambda_{\max}}$ is non-zero, i.e. $\lambda_{\max}\in\csetz$.
\medskip

\noindent Proof of \ref{mainthm-1}.
Assume $\Cpair{\lambda_{\max}}$ to be zero. 
By \fullref{lemma:cell-cyclic} we know that 
\begin{gather}
\cbas{U,V}{\lambda_{\max}}\acts\dbas{T}{\lambda_{\max}} 
=
\Cpair{\lambda_{\max}}(\dbas{V}{\lambda_{\max}},
\dbas{T}{\lambda_{\max}})\dbas{U}{\lambda_{\max}} 
=
0,
\end{gather}
for all $T,U,V\in\cmset(\lambda_{\max})$.

Expanding 
$e^{\invo} 
=\sum_{\mu\in\cideal_{P},S,T\in\cmset(\mu)}
r(\mu,S,T)
\cbas{S,T}{\mu}$ with $r(\mu,S,T)\in\K$, we see
\begin{gather}\label{eq:theorem-class-sum}
\begin{gathered}
e^{\invo}\cbas{V,U}{\lambda_{\max}} 
=
{\textstyle\sum_{\mu\in\cideal_{P}\setminus\lambda_{\max},S,T \in \cmset(\mu)}}\,
r(\mu,S,T) \cbas{S,T}{\mu}
\cbas{V,U}{\lambda_{\max}}
\\
= 
{\textstyle\sum_{\mu \in \cideal_{P} \setminus \lambda_{\max},S,T \in \cmset(\mu)}}\,
{\textstyle\sum_{T^{\prime}\in\cmset(\mu)}}\,
r(\mu,S,T) r_{\cbas{U,V}{\lambda_{\max}}}(T^{\prime},T)\cbas{S,T^{\prime}}{\mu} 
+
\friends[\mu],
\end{gathered} 
\end{gather}
where $\friends[\mu]\in\ceps_{S} 
\calg(\ord{\ceps_{S}}\mu)$ by \eqref{eq:mult-right}.
Hence, $e^{\invo}\friends[\mu] \in e^{\invo} 
\ceps_{S}\calg(\ord{\ceps_{S}} \mu)$.
Recalling that $e\isum\ceps$, this  
is either zero if $\ceps_{S}\neq \ceps$,
or $e^{\invo}\ceps_{S}
\calg(\ord{\ceps_{S}}\mu) = e^{\invo}
\calg(\ord{e}\mu)\subset\calg(\ord{e}\mu)$, with the final 
inclusion due to (\ref{lemma:cell-algebra-star-2}.a).

Multiplying the sum in \eqref{eq:theorem-class-sum} 
with $e^{\invo}$ we obtain 
an element inside $e^{\invo}\calg(\cideal_{P}\setminus\lambda_{\max})$ 
that is contained in $\calg(\cideal_{P}\setminus\lambda_{\max})$ 
by (\ref{lemma:cell-algebra-star-2}.a). 
Thus, $e^{\invo}\cbas{V,U}{\lambda_{\max}}$ contains no 
summand in $\calg(\{\lambda_{\max}\})$ and we get 
$e^{\invo}\acts\dbas{V}{\lambda_{\max}} = 0$ for all 
$V\in\cmset(\lambda_{\max})$, implying $\dmod(\lambda_{\max})^{\invo}\acts e = 0$.
Since $\prmod_{e}(\{\lambda_{\max}\})\cong\dmod(\lambda_{\max})
\otimes\dmod(\lambda_{\max})^{\invo}\acts e$ 
by \fullref{proposition:lambda-piece-iso}, we thus 
obtain $\prmod_{e}(\{\lambda_{\max}\})=0$. This is a 
contradiction to the choice of 
$\lambda_{\max}$ being a $\ord{e}$-maximal element. 
Thus, $\Cpair{\lambda_{\max}}$ is non-zero.
\medskip

\noindent\setword{\ref{theorem:simple-set}.Claim.b}{mainthm-2}.
$\dmod(\lambda_{\max})$ is a quotient of $\prmod_{e}(\{\lambda_{\max}\})$.
\medskip

\noindent Proof of \ref{mainthm-2}.
First, \ref{mainthm-1} and \fullref{proposition:lambda-piece-iso}
imply that 
\begin{gather}
\homc(\prmod_{e}(\{\lambda_{\max}\}),\dmod(\lambda_{\max})) 
\cong
\homk(\dmod(\lambda_{\max})^{\invo}\acts e,\K)
\neq 0.
\end{gather}
Using this identification, choose a 
linear form $f$ on $\dmod(\lambda_{\max})^{\invo}\acts e$ 
and elements $xe\in\dmod(\lambda_{\max})^{\invo}\acts e$ such that 
$f(xe)=1$ (recalling that 
we work over a field). Let now $z\in\dmod(\lambda_{\max})$ be 
a generator (note that existence of $z$ follows from \fullref{lemma:cell-cyclic}). 
Then, using again 
$\prmod_{e}(\{\lambda_{\max}\})\cong\dmod(\lambda_{\max})
\otimes\dmod(\lambda_{\max})^{\invo}\acts e$, 
we obtain that $f$ corresponds to the map sending 
$z\otimes xe$ to $f(xe)z=z$. Hence, $\dmod(\lambda_{\max})$ 
is a quotient of $\prmod_{e}(\{\lambda_{\max}\})$.
\medskip

By \ref{mainthm-2} and \fullref{proposition:to-define-Ls}, 
we get that $\lmod(\lambda_{\max})$ is a quotient of $\prmod_{e}(\{\lambda_{\max}\})$.
With the choice of $\lambda_{\max}$ being $\ord{e}$-maximal we have that 
$\prmod_{e}(\{\lambda_{\max}\})$ is a 
quotient of $\prmod_{e}$ itself, and thus the head of $\prmod_{e}$ 
contains $\lmod(\lambda_{\max})$. Since $\prmod_{e}$ is indecomposable, 
it has a simple head. Thus, it has to be $\lmod(\lambda_{\max})$.
So the completeness will follow after we have established 
\ref{mainthm-3}:
\medskip

\noindent\setword{\ref{theorem:simple-set}.Claim.c}{mainthm-3}.
There are no primitive idempotents $e$ with $aea^{-1}\!\!\not\!\!\isum\ceps$ 
for all $\ceps\in\ciset$ and all units $a\in\calg$.
\medskip

\noindent Proof of \ref{mainthm-3}. This 
follows from \fullref{lemma:decomposition-unit}, see also 
\fullref{remark:decomposition-unit}.
\medskip

\noindent\textit{Non-redundancy.}
We continue to use the notation from above.
\medskip

\noindent\setword{\ref{theorem:simple-set}.Claim.d}{mainthm-4}.
The ideal $\cideal_{P}$ has a unique $\ord{e}$-maximal element.
\medskip

\noindent Proof of \ref{mainthm-4}.
Assume that $\cideal_{P}$ has $\ord{e}$-maximal elements 
$\lambda_{0},\ldots,\lambda_{r}$. Then for each 
of these we know that $\prmod_{e}(\{\lambda_{k}\})\neq 0$ 
and $\Cpair{\lambda_{k}}$ is non-zero, i.e. $\dmod(\lambda_{k})$ 
has a simple quotient. (This is \ref{mainthm-1}.)
Then -- by \fullref{lemma:projective-quotients-sum} --
we have that
\begin{gather}
\neatafrac{\calg(\cideal_{P})e}
{\calg(\cideal_{P}\setminus \{\lambda_{0},\cdots,\lambda_{r}\})e}
\cong 
\prmod_{e}(\{\lambda_{0}\}) 
\oplus\cdots\oplus 
\prmod_{e}(\{\lambda_{r}\}).
\end{gather}
This in turn implies that 
$\prmod_{e}$ has $\lmod(\lambda_{0})\oplus\cdots\oplus\lmod(\lambda_{r})$ 
as a quotient, which is a contradiction to $\prmod_{e}$ being 
indecomposable. Hence, the ideal $\cideal_{P}$ 
has a unique maximal element that we denote by $\lambda_{\max}$.
\medskip

Now, \ref{mainthm-5} will establish non-redundancy, which will finish the proof.
\medskip

\noindent\setword{\ref{theorem:simple-set}.Claim.e}{mainthm-5}.
$\lmod(\lambda)\cong\lmod(\mu)$ implies $\lambda=\mu$ for $\lambda,\mu\in\csetz$.
\medskip

\noindent Proof of \ref{mainthm-5}.
Without loss of generality, assume that 
$\lambda$ is a $\ord{e}$-maximal element in an 
ideal $\cideal_{P}$ for some 
indecomposable projective $\prmod_{e}$ 
corresponding to $e\isum\ceps$ primitive
for some $\ceps\in\ciset$.
(This is sufficient 
since we already proved above that 
simples obtained for these 
elements of $\cset$ give a complete set of 
isomorphism classes of simples.)

We first observe that we have a quotient map 
\begin{gather}
\pi_{\lambda}\colon
\prmod_{e}\twoheadrightarrow
\lmod(\lambda)\overset{\cong}{\longrightarrow}\lmod(\mu),
\end{gather}
with $z_{\lambda}=\pi_{\lambda}(e)$ 
being a generator of $\lmod(\mu)$. 
Thus, one has $e\acts z_{\lambda}=z_{\lambda}$. 
Note now that $e\in\calg(\Ord{e}\lambda)$
since $\lambda$ is unique $\ord{e}$-maximal by \ref{mainthm-4}.
Thus, (\ref{lemma:cell-algebra-star-2}.d) implies 
that there exists $\eta\Ord{e}\lambda$, $S,T\in\cmset(\eta)$ 
with $\ceps_{S}=\ceps_{T}=\ceps$ 
and $U,V\in\cmset(\mu)$ such that the product
\begin{gather}
\cbas{S,T}{\eta}\cbas{U,V}{\mu}
\in
{\textstyle\sum_{T^{\prime}\in\cmset(\eta)}}\,
r_{\cbas{V,U}{\mu}}(T^{\prime},T) \cbas{S,T^{\prime}}{\eta} 
+
\ceps_{S}^{\invo}\calg(\ord{\ceps_{S}}\eta),
\end{gather}
expanded using \eqref{eq:mult-right}, 
contains a summand in $\calg(\{\mu\})$. 
Hence, with $\ceps_{S}=\ceps$ 
(giving $\ord{e}=\ord{\ceps_{S}}$) it 
follows that $\mu\Ord{e}\eta\Ord{e}\lambda$.

On the other hand -- by \fullref{lemma:cell-cyclic} -- we have 
$\dmod(\mu)=\calg(\{\mu\})\acts z$ for some 
generator $z\in\dmod(\mu)$ giving another quotient map 
\begin{gather}
\psi_{\lambda}\colon
\dmod(\mu)\twoheadrightarrow
\lmod(\mu)\overset{\cong}{\longrightarrow}\lmod(\lambda).
\end{gather}
Fix now $z_{\lambda}$ as above and choose 
$y\in\dmod(\mu)$ with 
$\psi_{\lambda}(y)=z_{\lambda}$. 
Then there exists $a\in\calg(\{\mu\})$ 
such that $y = a\acts z$, but 
\begin{gather}
\psi_{\lambda}((\ceps a)\acts z) 
= \ceps\acts\psi_{\lambda}(a\acts z) 
= \ceps\acts\psi_{\lambda}(y) 
= \ceps\acts z_{\lambda} 
= e\acts z_{\lambda}=z_{\lambda},
\end{gather}
so we can assume that $\ceps a = a$ 
and $a\acts \psi_{\lambda}(z)\neq 0$. 
So there exist $S,T\in\cmset(\lambda)$ 
and $U,V\in\cmset(\mu)$ with $\ceps_{U}=\ceps$ such that
\begin{gather}
\cbas{U,V}{\mu}\cbas{S,T}{\lambda}
\in
{\textstyle\sum_{T^{\prime}\in\cmset(\eta)}}\,
r_{\cbas{T,S}{\lambda}}(V^\prime,V)\cbas{U,V^{\prime}}{\mu}
+
\ceps_{U}^{\invo}\calg(\ord{\ceps_{U}}\mu),
\end{gather}
expanded using \eqref{eq:mult-right}, contains a 
summand in $\calg(\{\lambda\})$. Thus, with $\ceps_{U}=\ceps$ 
(giving $\ord{e}=\ord{\ceps_{U}}$)
it follows that $\lambda\Ord{e}\mu$.  

Hence, altogether we have $\lambda=\mu$.
\end{proof}

Note that the primitive idempotent $e$ such 
that $\calg e$ has $\lmod(\lambda)$ as its head 
is not unique. But if we demand 
the choice of an idempotent summand of some 
$\ceps_{\lambda}\in\ciset$, then $\ceps_{\lambda}$ is 
unique. In particular, the associated partial order
$\ord{\ceps_{\lambda}}$ is independent 
of the choice of $e$.
Thus -- having \fullref{theorem:simple-set} -- we can define: 

\begin{definition}\label{definition:decomp-matrix}
Let $\csetf$ and $\lambda\in\csetz$.
We denote by $\prmod(\lambda)$ the indecomposable 
projective module corresponding to $\lmod(\lambda)$. 
\end{definition}

The partial 
order associated to $\prmod(\lambda)$ is denoted by $\ord{\lambda}$.

\begin{proposition}\label{proposition:cell-filtration}
Let $\lambda\in\csetz$. Then $\prmod(\lambda)$ has a 
filtration by cell modules $\dmod(\mu)$ such 
that $\mu\Ord{\lambda}\lambda$.
\end{proposition}

\begin{proof}
By the proof of 
\fullref{theorem:simple-set} we know 
that $\prmod(\lambda)=\calg(\Ord{\lambda}\lambda)e$ 
for some $e\isum\ceps_{\lambda}$ primitive. 
The statement follows by
\fullref{lemma:projective-submodule-lemma} and the 
description of the subquotients as direct sums of 
cell modules from \fullref{proposition:lambda-piece-iso}.
\end{proof}

The examples in \fullref{subsection:cell-examples} illustrate \fullref{proposition:cell-filtration}.

\subsection{Reciprocity laws}\label{subsec:BGG}

Throughout the rest of the section assume $\csetf$. 
Let $\lambda\in\csetz$ and $\mu\in\cset$.

We denote by $\dnumber{\mu,\lambda}=[\dmod(\mu):\lmod(\lambda)]$ 
the \textit{Jordan--H\"older multiplicity of $\lmod(\lambda)$ 
in $\dmod(\mu)$} and by 
$\dmatrix=\dmatrix(\calg)=(\dnumber{\mu,\lambda})_{\mu\in\cset,\lambda\in\csetz}$ 
the \textit{decomposition matrix of $\calg$}. (We warn 
the reader that $\dmatrix$ is, in contrast to the Cartan matrix, not necessarily a square matrix.)

In contrast to \cite[Proposition 3.6]{gl1}, the matrix $\dmatrix$ is not upper 
triangular, cf. \fullref{example:CD-matrix}. But we have the 
following relative version.

\begin{proposition}\label{proposition:decomp-matrix}
Let $\lambda\in\csetz$ and $\mu\in\cset$. 
Then $\dnumber{\mu,\lambda}=0$ 
unless $\mu\Ord{\lambda}\lambda$.
Furthermore, we have $\dnumber{\lambda,\lambda}=1$.
\end{proposition}

\begin{proof}
Assume that $\dnumber{\mu,\lambda}\neq 0$. 
Then there exists a non-zero map 
$f\colon\dmod(\lambda)\rightarrow\neatafrac{\dmod(\mu)}{N}$ for 
some submodule $N\subset\dmod(\mu)$. 
Corresponding to $\lmod(\lambda)$ there exists 
some $\ceps\in\ciset$ and $e\isum\ceps$
such that $e$ acts non-trivial on 
$\lmod(\lambda)$. Hence, $e$ acts also non-trivial on $\dmod(\lambda)$, and 
furthermore $e\in\calg(\Ord{\lambda}\lambda)$. 
Since $f$ is an $\calg$-module map, $e$ also 
acts non-trivial on $\neatafrac{\dmod(\mu)}{N}$, and hence also 
non-trivial on $\dmod(\mu)$. Thus, there exists 
$\eta\Ord{\lambda}\lambda$, $S,T\in\cmset(\eta)$ with 
$\ceps_{S}=\ceps$ such that 
$\cbas{S,T}{\eta}\acts\dmod(\mu) \neq 0$. Thus -- by 
\fullref{lemma:submodule-lemma-alt2} -- we have 
that $\mu\Ord{\lambda}\eta\Ord{\lambda}\lambda$.

Assume now that 
$\lambda=\mu\in\csetz$. 
Let $f\colon\dmod(\lambda)\rightarrow\neatafrac{\dmod(\lambda)}{N}$ 
for some submodule $N$ be a non-zero map. 
Then we know -- by \ref{Delta-maps} -- that the map is 
a non-zero $\K$-multiple of the identity of $\dmod(\lambda)$ composed 
with the natural quotient map. Thus, $f$ is always surjective 
and only in case of 
$N=\radi{\lambda}$ is the image simple. This gives 
$\dnumber{\lambda,\lambda} = 1$.
\end{proof}

\begin{lemma}\label{lemma:choices-idempotents}
Let $\lambda\in\csetz$ and $e\isum\ceps_{\lambda}$ primitive. 
Then $\prmod(\lambda)\cong\calg e$ 
if and only if 
$\cideal_{\lambda}
=
\{\mu\in\cset\mid\mu\Ord{\lambda}\lambda\}$ 
is the smallest $\ord{\lambda}$-ideal such 
that $e\in\calg(\cideal_{\lambda})$.
\end{lemma}

\begin{proof}
\textit{$\Rightarrow$.}
Assuming that $\prmod(\lambda)\cong\calg e$, 
we know that $\cideal_{\lambda}$ is an 
$\ord{\lambda}$-ideal such that $e\in\calg(\cideal_{\lambda})$, 
see the proof of \fullref{theorem:simple-set}. 
Assume now that $\cideal$ is another 
$\ord{\lambda}$-ideal such that 
$e \in \calg(\cideal)$. 
If $\lambda\in\cideal$ we are done, 
since $\cideal_{\lambda}\subset\cideal$. 
So assume $\lambda\notin\cideal$ and 
denote by $\left\langle\cideal\cup\lambda\right\rangle$ 
the $\ord{\lambda}$-ideal generated by $\cideal$ and $\lambda$. Then
$\prmod(\{\lambda\}) 
= 
\neatafrac{\calg(\left\langle\cideal\cup\lambda\right\rangle)e}
{\calg(\left\langle\cideal\cup\lambda\right\rangle\setminus \lambda)e} 
= 0$ , since $\prmod(\lambda)=\calg(\cideal)e$. This is a 
contradiction to $\lmod(\lambda)$ being the 
quotient of $\prmod(\lambda)$. Thus, 
$\cideal_{\lambda}$ is the smallest $\ord{\lambda}$-ideal 
with the desired property.
\medskip

\textit{$\Leftarrow$.}
For $\cideal_{\lambda}$ being the smallest $\ord{\lambda}$-ideal 
with $e\in\calg(\cideal_{\lambda})$, let $\mu\in\csetz$ such 
that $\calg e=\prmod(\mu)$. Then $\cideal_{\mu}$ is the smallest 
$\ord{\lambda}$-ideal containing $e$, and thus -- by assumption -- 
equal to $\cideal_{\lambda}$. Hence -- by \fullref{theorem:simple-set} -- 
$\prmod(\mu)$ has simple quotient $\lmod(\lambda)$, giving $\mu=\lambda$.
\end{proof}

Since for a primitive idempotent summand 
of $\ceps_{\lambda}$, the minimal 
$\ord{\lambda}$-ideal $\cideal$ such 
that $e\in\calg(\cideal)$ is equal the 
minimal $\ord{\lambda}$-ideal such that $e^{\invo}\in\calg(\cideal)$, 
the following is immediate.

\begin{corollary}\label{corollary:projective-e-estar}
Let $\lambda\in\csetz$. 
If $\prmod(\lambda)\cong\calg e$ for $e\isum\ceps_{\lambda}$, 
then $\prmod(\lambda)\cong\calg e^{\invo}$.
\end{corollary}

For $\lambda,\mu\in\csetz$ we denote 
by $\cnumber{\lambda,\mu} = [\prmod(\lambda):\lmod(\mu)]$ the 
\textit{Jordan--H\"older multiplicity of $\lmod(\mu)$ in $\prmod(\lambda)$}, 
and by $\cmatrix=\cmatrix(\calg)=(\cnumber{\lambda,\mu})_{\lambda,\mu\in\csetz}$ the 
\textit{Cartan matrix of $\calg$}. (By \fullref{theorem:simple-set} this coincides 
with the definition we used in \fullref{subsection:cell-examples}.)

\begin{theorem}\label{theorem:bgg}
Let $\lambda\in\csetz$, $\mu\in\cset$ and 
$e\isum\ceps_{\lambda}$ primitive
such that $\prmod(\lambda)=\calg e$.
\smallskip
\begin{enumerate}[label=(\alph*)]

\setlength\itemsep{.15cm}

\item The multiplicity $\dnumber{\mu,\lambda}$ 
is equal to $\mathrm{dim}(\dmod(\mu)^\invo.e)$.

\item If $\mu\in\csetz$, then
\begin{gather}
[\prmod(\lambda):\lmod(\mu)] = 
{\textstyle\sum_{\nu\in\cset,\nu\Ord{\lambda}\lambda,\nu \Ord{\mu} \mu}}\, 
[\dmod(\nu):\lmod(\lambda)][\dmod(\nu):\lmod(\mu)].
\end{gather}
(Or $\cmatrix=\dmatrix^{\mathrm{T}}\dmatrix$, written as matrices.)\qedhere
\end{enumerate}
\end{theorem}

\begin{proof}
\textit{(\ref{theorem:bgg}.a).} 
This is straightforward, since
\begin{gather}
\begin{aligned}
\dnumber{\mu,\lambda} 
&= \mathrm{dim}(\homc(\prmod(\lambda),\dmod(\mu))) 
=\mathrm{dim}(\homc(\calg e^{\invo},\dmod(\mu)))
\\
&= 
\mathrm{dim}(e^{\invo}\acts\dmod(\mu)) 
= \mathrm{dim}(\dmod(\mu)^{\invo}\acts e),
\end{aligned}
\end{gather}
with the second equality due to \fullref{corollary:projective-e-estar}.
\medskip

\noindent\textit{(\ref{theorem:bgg}.b).} 
Choose a maximal $\ord{\lambda}$-ideal chain 
inside $\cideal_{\lambda}$. Then we know for each 
subquotient $\prmod(\{\nu\})\cong\dmod(\nu)
\otimes\dmod(\nu)^{\invo}\acts e$ 
as left $\calg$-modules. Thus,
\begin{gather}
\begin{aligned}
\cnumber{\lambda,\mu} 
= 
{\textstyle\sum_{\nu\in\cset,\nu\Ord{\lambda}\lambda}}\,
\mathrm{dim}(\dmod(\nu)^{\invo}\acts e)\dnumber{\nu,\mu} 
= 
{\textstyle\sum_{\nu\in\cset,\nu\Ord{\lambda}\lambda}}\,
\dnumber{\nu,\lambda}\dnumber{\nu,\mu},
\end{aligned}
\end{gather}
where -- by \fullref{proposition:decomp-matrix} -- 
any summand is zero unless $\nu\Ord{\mu}\mu$ 
as well.
\end{proof}

\begin{example}\label{example:CD-matrix}
Coming back to the examples from 
\fullref{subsection:cell-examples}, we have 
for $n=3$
\begin{gather}
\begin{gathered}
\cmatrix(\ualg(\typeA_{3}))=
\begin{psmallmatrix}
2 & 1 & 0 \\
1 & 2 & 1 \\
0 & 1 & 2 \\
\end{psmallmatrix}
=
\begin{psmallmatrix}
1 & 1 & 0 & 0\\
0 & 1 & 1 & 0\\
0 & 0 & 1 & 1\\
\end{psmallmatrix}
\begin{psmallmatrix}
1 & 0 & 0 \\
1 & 1 & 0 \\
0 & 1 & 1 \\
0 & 0 & 1 \\
\end{psmallmatrix}
=
\dmatrix(\ualg(\typeA_{3}))^{\mathrm{T}}
\dmatrix(\ualg(\typeA_{3})),
\\
\cmatrix(\calg(\typeAt_{3}))=
\begin{psmallmatrix}
2 & 1 & 1 \\
1 & 2 & 1 \\
1 & 1 & 2 \\
\end{psmallmatrix}
=
\begin{psmallmatrix}
1 & 1 & 0 \\
0 & 1 & 1 \\
1 & 0 & 1 \\
\end{psmallmatrix}
\begin{psmallmatrix}
1 & 0 & 1 \\
1 & 1 & 0 \\
0 & 1 & 1 \\
\end{psmallmatrix}
=
\dmatrix(\calg(\typeAt_{3}))^{\mathrm{T}}
\dmatrix(\calg(\typeAt_{3})),
\\
\cmatrix(\calg^{\prime}(\typeAt_{3}))=
\begin{psmallmatrix}
3 & 3 & 3 \\
3 & 3 & 3 \\
3 & 3 & 3 \\
\end{psmallmatrix}
=
\begin{psmallmatrix}
1 & 1 & 1 \\
1 & 1 & 1 \\
1 & 1 & 1 \\
\end{psmallmatrix}
\begin{psmallmatrix}
1 & 1 & 1 \\
1 & 1 & 1 \\
1 & 1 & 1 \\
\end{psmallmatrix}
=
\dmatrix(\calg^{\prime}(\typeAt_{3}))^{\mathrm{T}}
\dmatrix(\calg^{\prime}(\typeAt_{3})),
\end{gathered}
\end{gather}
(up to base change)
and analogously for general $n$. Note that 
the decomposition matrices have an 
upper triangular shape for $\ualg(\typeA_{n})$ that 
is a cellular algebra.
\end{example}

As a direct corollary of (\ref{theorem:bgg}.b) and the 
singular value decomposition, 
we get a very easy to check, but weak, necessary criterion 
for an algebra to be relative cellular.

\begin{corollary}\label{corollary:semi-definite}
If $\calg$ is relative cellular, 
then $\cmatrix$ is positive semidefinite.
\end{corollary}

As already discussed in detail in \fullref{subsection:cell-examples}, 
this is in contrast 
to the case of cellular algebras where $\cmatrix$ is 
positive definite, cf. \fullref{remark:pos-def}.

\subsection{Further consequences}\label{subsec:some-further-stuff}

For the next proposition, 
we denote by $\kdual$ the (${}^{\invo}$-twisted) duality on 
$\calg$-modules defined by 
$\kdual(M) = \homk(M^{\invo},\K)$. Note that $\dmod(\lambda)$
is in general not isomorphic to $\kdual(\dmod(\lambda))$ 
as an $\calg$-module. But we have the following.

\begin{proposition}\label{proposition:ext}
Let $\lambda,\mu\in\csetz$. 
Then $\kdual\lmod(\lambda)\cong\lmod(\lambda)$ 
as $\calg$-modules. 
Further, there are isomorphisms
$\extc^{i}(\lmod(\lambda),\lmod(\mu))\cong\extc^{i}(\lmod(\mu),\lmod(\lambda))$ 
for all $i\in\Z_{\geq 0}$.
\end{proposition}

\begin{proof}
Let $e\isum\ceps_{\lambda}$ 
primitive such that 
$\prmod(\lambda)=\calg e 
\cong \calg e^{\invo}$. 
We claim that $\prmod(\lambda)$ is a 
projective cover of the simple 
$\kdual\lmod(\lambda)$. 
For $ae^{\invo}\in \calg e^{\invo}$ we 
define $\theta_{ae^{\invo}}$ by 
$\theta_{ae^{\invo}} (x) = x\acts (ae^{\invo})$ for 
$x\in\lmod(\lambda)^{\invo}$. Here $x\acts (ae^{\invo}) = ea^{\invo}\acts x$ 
is an element in $e\lmod(\lambda)$ that can be 
canonically identified with the endomorphism 
ring of $\lmod(\lambda)$ that -- by 
\fullref{proposition:Delta-maps} -- is $\K$. 
Thus, $\theta_{a e^{\invo}}$ defines a linear 
form on $\lmod(\lambda)^{\invo}$. Clearly, the 
map $a e^{\invo} \mapsto \theta_{a e^{\invo}}$ is 
not the zero map, hence it is surjective 
and so $\prmod(\lambda)$ is the 
projective cover of $\kdual\lmod(\lambda)$.

Using $\extc^{i}(L(\lambda),L(\mu))
\cong
\mathrm{Ext}_{\text{mod}{-}\calg}^{i}(L(\lambda)^{\invo},L(\mu)^{\invo})$, the latter 
being in right $\calg$-modules, 
we obtain the statement about 
$\mathrm{Ext}$-groups since vector space 
duality gives a contravariant equivalence 
between left and right modules for a finite-dimensional algebras.
\end{proof}

\begin{remark}\label{remark:sym-quiver}
As a corollary of \fullref{proposition:ext}, 
the $\mathrm{Ext}$-quiver of a relative cellular 
algebra has a symmetric form. This is a well-known fact 
for cellular algebras.
\end{remark}

Finally, the semisimplicity criterion 
for a relative cellular algebra is as in \cite[Theorem 3.8]{gl1}, 
and the proof -- by using the results from 
\fullref{subsec:BGG} --
is identical (and omitted).

\begin{propositionqed}\label{proposition:semi-simplicity}
Let $\calg$ be a relative cellular algebra. Then the following are equivalent.
\smallskip
\begin{enumerate}[label=(\alph*)]

\setlength\itemsep{.15cm}

\item The algebra $\calg$ is semisimple.

\item The cell modules $\dmod(\lambda)$ for $\lambda\in\csetz$ are simple. 

\item The subspace $\radi{\lambda}=0$ for 
all $\lambda\in\cset$.\qedhere
\end{enumerate}
\end{propositionqed}

\begin{example}\label{example-semisimple}
None of the algebras from 
\fullref{subsection:cell-examples}, nor 
$\aarc$ for $\numberarcs\in\Np$ (for the latter see 
\fullref{section:arc-stuff}) are semisimple. 
There are various ways to see this, but using 
\fullref{proposition:semi-simplicity} this follows since the 
simples are all of dimension one, while the cell modules 
are not.
\end{example}
\section{An extended example I: The restricted enveloping algebra of \texorpdfstring{$\slt$}{sltwo}}\label{section:rel-cell-res}

Throughout this section let $\K$ be any field with $\mathrm{char}(\K)=p>0$.

\subsection{The algebra}\label{subsection:the-algebra-sl2}
We let 
$\fieldp$ be the prime field of $\K$, and
we also use the set 
$\setp=\{0,1,\dots,p-2,p-1\}\subset\N$
underlying $\fieldp$. (Using the identification $\fieldp=\setp$, we will 
sometimes read modulo $p$.)

\begin{definition}\label{definition:res-slt}
The \textit{restricted enveloping algebra of $\slt$}, denoted by $\uslt$, 
is the associative, unital algebra generated by 
$\egen,\fgen,\hgen$ subject to
\begin{gather}\label{eq:slt-rels-usual}
\hgen\egen-\egen\hgen=2\egen,
\quad
\hgen\fgen-\fgen\hgen=-2\fgen,
\quad
\egen\fgen-\fgen\egen=\hgen,
\end{gather} 
\begin{gather}\label{eq:slt-rels-modp}
\egen^{p}=\fgen^{p}=\hgen^{p}-\hgen=0.
\end{gather}
Said otherwise, $\uslt$ is the 
usual enveloping algebra of $\slt$ modulo \eqref{eq:slt-rels-modp}.
\end{definition}

Note that the prime $p$ enters the definition of $\uslt$ in two ways: 
via the ground field, but also via \eqref{eq:slt-rels-modp}.

\begin{remark}\label{remark:chi-is-zero}
Our main source for the basics about $\uslt$ are \cite{fp1} and  
\cite{ja1}. (E.g., \fullref{definition:res-slt} 
is taken from therein.) Note that $\uslt[\chi]$ can be defined 
for a choice of $\chi\in\slt^{\ast}$. But, as we will see below, cf. \fullref{remark:fancy-chi}, 
it is crucial for us that $\chi=0$.
\end{remark}

Recall the following \textit{PBW theorem}, cf. \cite[Section 1]{fp1} or \cite[Section A.3]{ja1}:

\begin{theoremqed}\label{theorem:PBW-slt}
The set
\begin{gather}\label{eq:PBW-slt}
\{
\fgen^{x}\hgen^{y}\egen^{z}
\mid
x,y,z\in\setp
\}
\end{gather}
is a basis of $\uslt$.
\end{theoremqed}

Our relative cellular basis for $\uslt$ will 
be an idempotent version of \eqref{eq:PBW-slt}. 
For this we need the following weight idempotents. Let $\lambda\in\setp$ 
and define
\begin{gather}
\hidem_{\lambda}=
-{\textstyle\prod_{\mu\in\fieldp,\mu\neq\lambda}}\,
(\hgen-\mu).
\end{gather}

\begin{lemma}\label{lemma:prim-idem-sl2}
The set $\{\hidem_{\lambda}\mid\lambda\in\setp\}$ is a complete set of pairwise orthogonal idempotents.
\end{lemma}

We stress that the $\hidem_{\lambda}$'s are not primitive idempotents of $\uslt$, but 
rather the primitive idempotents of the 
semisimple subalgebra spanned by the $\hgen$'s.

\begin{proof}
Observe that $\hidem_{\lambda}$ is a 
degree $p-1$ polynomial in $\hgen$ and therefore 
determined by its values in $\fieldp$. Now,
substituting $\hgen$ with any element of $\fieldp$, we see 
-- by Wilson's theorem -- that $\hidem_{\lambda}$ 
is an idempotent. Similarly, 
orthogonality follows from Fermat's little theorem. 
Finally -- by construction -- 
$\sum_{\lambda\in\cset}\hidem_{\lambda}$ evaluates for any substitution 
$\hgen\mapsto\mu\in\fieldp$ to $1$.
\end{proof}

The following tedious calculations, which we will 
use throughout, are omitted.

\begin{lemmaqed}\label{lemma:some-calcs-slt}
Let $\lambda\in\setp$ and $S,T\in\setp$.
\smallskip
\begin{enumerate}[label=(\alph*)]

\setlength\itemsep{.15cm}

\item For $k\in\setp$ we have
\begin{gather}
\begin{gathered}
\hgen^{k}\egen^{T}=\egen^{T}(\hgen+2T)^{k},
\quad
\hgen^{k}\fgen^{S}=\fgen^{S}(\hgen-2S)^{k},
\\
\hidem_{\lambda}\egen=\egen\hidem_{\lambda{-}2},
\quad
\egen\hidem_{\lambda}=\hidem_{\lambda{+}2}\egen,
\quad
\hidem_{\lambda}\fgen=\fgen\hidem_{\lambda{+}2},
\quad
\fgen\hidem_{\lambda}=\hidem_{\lambda{-}2}\fgen,
\quad
\hgen\hidem_{\lambda}=\lambda\hidem_{\lambda}=\hidem_{\lambda}\hgen.
\end{gathered}
\end{gather}

\item We have
\begin{gather}
\begin{gathered}
\egen^{T}\fgen^{S}\hidem_{\lambda}=
{\textstyle\sum_{j=0}^{\min(S,T)}\tfrac{S!T!}{(S-j)!(T-j)!}
\binom{T{-}S{+}\lambda}{j}}
\fgen^{S{-}j}\egen^{T{-}j}\hidem_{\lambda},
\\
\fgen^{S}\egen^{T}\hidem_{\lambda}=
{\textstyle\sum_{j=0}^{\min(S,T)}\tfrac{S!T!}{(S-j)!(T-j)!}
\binom{S{-}T{-}\lambda}{j}}
\egen^{T{-}j}\fgen^{S{-}j}\hidem_{\lambda},
\end{gathered}
\end{gather}
with usual factorials and binomials taken modulo $p$.\qedhere
\end{enumerate}
\end{lemmaqed}

\begin{remark}\label{remark:p-is-two}
For $p=2$ it is -- by \fullref{lemma:some-calcs-slt} -- 
not hard to see that $\uslt$ is isomorphic to 
a direct sum of
$\neatafrac{\K[X,Y]}{(X^{2},Y^{2})}$ 
and a semisimple algebra. 
Thus, $\uslt$ is already cellular, 
and we from now on assume that $p>2$.
\end{remark}

\subsection{The cell datum}\label{subsection:the-cell-datum-sl2}

Next, we want to define the relative cell datum for $\uslt$. To this end, we 
let $\cset=\setp$ and $\cmset(\lambda)=\setp$ for all $\lambda\in\cset$. 
Moreover -- by \fullref{lemma:prim-idem-sl2} -- we can let 
$\ciset=\{\hidem_{\lambda}\mid\lambda\in\cset\}$ be our idempotent set.

Further, we let $\cbas{S,T}{\lambda}=\fgen^{S}\hidem_{\lambda}\egen^{T}$, 
and set $(\fgen^{S}\hidem_{\lambda}\egen^{T})^{\invo}=\fgen^{T}\hidem_{\lambda}\egen^{S}$. And finally, 
let the partial orders $\coset=\{\ord{\hidem_{\lambda}}\mid\lambda\in\cset\}$, on $\cset$,  be defined via
\begin{gather}
\lambda + 2(p-1) \ord{\hidem_{\lambda}} \cdots\ord{\hidem_{\lambda}}\lambda+4\ord{\hidem_{\lambda}}
\lambda+2\ord{\hidem_{\lambda}}\lambda,
\end{gather}
and $\ceps_{S}=\hidem_{\lambda{+}2S}$ 
for $S\in\cmset(\lambda)$. Note that these partial orders on $\cset$ 
are well-defined since 
$2$ generates $\fieldp$ since we assume that $p>2$.

To summarize, we have our cell datum
\begin{gather}\label{eq:uslt-datum}
(\cset,\cmset,\cbasis,{}^{\invo},\ciset,\coset,\cepsmap).
\end{gather}

A direct consequence of \fullref{lemma:some-calcs-slt} is:

\begin{lemmaqed}\label{lemma:some-calcs-slt-2}
Let $\cbas{S+1,T}{\lambda}=\cbas{S-1,T}{\lambda}=\cbas{S,T{+}1}{\lambda}=0$ 
in case $S,T\notin\setp$. Then
\begin{gather}
\egen\,\cbas{S,T}{\lambda}
=
S(1-S+\lambda)
\cbas{S{-}1,T}{\lambda}
+
\cbas{S,T{+}1}{\lambda{+}2},
\quad
\fgen\,\cbas{S,T}{\lambda}
=
\cbas{S{+}1,T}{\lambda},
\quad
\hgen\,\cbas{S,T}{\lambda}
=
(\lambda-2S)\cbas{S,T}{\lambda},
\end{gather}
Similar formulas hold 
for the right action of $\uslt$ on 
the $\cbas{S,T}{\lambda}$'s.
\end{lemmaqed}

\subsection{\texorpdfstring{$p=3$}{p is 3} exemplified}\label{subsection:small-examples-sl2}

\begin{example}\label{example:p-is-three}
Let $p=3$. Then $\hidem_{0}=-(\hgen-1)(\hgen-2)$, 
$\hidem_{1}=-(\hgen-0)(\hgen-2)$ and $\hidem_{2}=-(\hgen-0)(\hgen-1)$.
Moreover, the partial orders are
\begin{gather}
\cset
=
\{\fcolorbox{myred}{mycream!5}{$1$}
\ord{\hidem_{0}}
\fcolorbox{myblue}{mycream!5}{$2$}
\ord{\hidem_{0}}
\fcolorbox{mygreen}{mycream!5}{$0$}\}
=
\{\fcolorbox{myblue}{mycream!5}{$2$}
\ord{\hidem_{1}}
\fcolorbox{mygreen}{mycream!10}{$0$}
\ord{\hidem_{1}}
\fcolorbox{myred}{mycream!5}{$1$}\}
=
\{\fcolorbox{mygreen}{mycream!5}{$0$}
\ord{\hidem_{2}}
\fcolorbox{myred}{mycream!5}{$1$}
\ord{\hidem_{2}}
\fcolorbox{myblue}{mycream!5}{$2$}\}.
\end{gather}
Further, $\hidem_{\mu}\uslt\hidem_{\mu}$ 
consists of elements $\fgen^{S}\hidem_{\lambda}\egen^{S}$ 
such that $\lambda=\mu-2S$. Having all this, 
it is easy to see that \eqref{eq:uslt-datum}
defines a cell datum for $\uslt$.

We get projectives and cell modules (here exemplified in case $\lambda=0$):
\begin{gather}\label{eq:baby-verma-sl2}
\begin{tikzpicture}[baseline=(current bounding box.center)]
	\node[mygreen] at (0,1) {$\uslt\,\hidemb_{0}$};
	\draw[thin, mygreen, dotted] (-4.05,1.25) rectangle (3.05,-.85);
	\node[mygreen] at (-2.2,.5) {$\fgen^{2}\hidem_{0}\,\hidemb_{0}$};
	\node[mygreen] at (0,.5) {$\fgen\hidem_{0}\,\hidemb_{0}$};
	\node[mygreen] at (2.2,.5) {$\hidem_{0}\,\hidemb_{0}$};
	\draw[thin, mygreen] (-3.05,.75) rectangle (3.0,.25);
	\node[mygreen] at (-3.6,.5) {$\dmod(0)$};
	\node[myblue] at (-2.2,0) {$\fgen^{2}\hidem_{2}\egen\,\hidemb_{0}$};
	\node[myblue] at (0,0) {$\fgen\hidem_{2}\egen\,\hidemb_{0}$};
	\node[myblue] at (2.2,0) {$\hidem_{2}\egen\,\hidemb_{0}$};
	\draw[thin, myblue] (-3.05,-.225) rectangle (3.0,.225);
	\node[myblue] at (-3.6,0) {$\dmod(2)$};
	\node[myred] at (-2.2,-.5) {$\fgen^{2}\hidem_{1}\egen^{2}\,\hidemb_{0}$};
	\node[myred] at (0,-.5) {$\fgen\hidem_{1}\egen^{2}\,\hidemb_{0}$};
	\node[myred] at (2.2,-.5) {$\hidem_{1}\egen^{2}\,\hidemb_{0}$};
	\draw[thin, myred] (-3.05,-.75) rectangle (3.0,-.25);
	\node[myred] at (-3.6,-.5) {$\dmod(1)$};
\end{tikzpicture}
\quad
\raisebox{.1cm}{
\fcolorbox{myred}{mycream!5}{$
\xy
(0,0)*
{\dmod(1)
\colon
\begin{tikzcd}[ampersand replacement=\&,row sep=small,column sep=small,arrows={shorten >=-.3ex,shorten <=-.3ex},labels={inner sep=.15ex}]
\cbas{2,2}{1}
\arrow[loop above,looseness=4]{}{0}
\arrow[yshift=.4ex,<-]{r}{1}
\arrow[yshift=-.4ex,->,swap]{r}{0}
\&
\cbas{1,2}{1}
\arrow[loop above,looseness=4]{}{2}
\arrow[yshift=.4ex,<-]{r}{1}
\arrow[yshift=-.4ex,->,swap]{r}{1}
\&
\cbas{0,2}{1}
\arrow[loop above,looseness=4]{}{1}
\end{tikzcd}};
(0,-8.5)*{\stackrel{\egen}{\rightarrow}\quad\stackrel{\fgen}{\leftarrow}\quad\stackrel{\hgen}{\curvearrowright}};
\endxy
$}}
\end{gather}
These are either nine or three-dimensional. 
The $\dmod$'s are isomorphic to the so-called 
\textit{baby Verma modules of highest weight 
$\lambda$}. 
For example, the cell module 
$\dmod(1)$ in $\uslt\hidem_{0}$ 
is the left $\uslt$-module as displayed in \eqref{eq:baby-verma-sl2}.

In order to get the simples $\lmod$, we calculate the radical  
and then we use \fullref{theorem:simple-set}.
Note that, the pairing 
$\Cpair{\lambda}(\fgen^{S}\hidem_{\lambda},\fgen^{T}\hidem_{\lambda})$
is zero unless $S=T$. 
For $S=T$ we get:
\begin{gather}
\dmod(0)\colon
\begin{cases} 
1, & \text {if }S=T=0,
\\
0, & \text {if }S=T=1,
\\
0, & \text {if }S=T=2,
\end{cases}
\;
\dmod(1)\colon
\begin{cases} 
1, & \text {if }S=T=0,
\\
1, & \text {if }S=T=1,
\\
0, & \text {if }S=T=2,
\end{cases}
\;
\dmod(2)\colon
\begin{cases} 
1, & \text {if }S=T=0,
\\
2, & \text {if }S=T=1,
\\
1, & \text {if }S=T=2.
\end{cases}
\end{gather}
Hence, using this and \eqref{eq:baby-verma-sl2}
we get in total
\begin{gather}\label{eq:sequence-p-is-three}
\fcolorbox{myred}{mycream!5}{$\lmod(1)$}
\hookrightarrow
\fcolorbox{mygreen}{mycream!5}{$\dmod(0)
\twoheadrightarrow
\lmod(0)$}
\quad
\fcolorbox{mygreen}{mycream!5}{$\lmod(0)$}
\hookrightarrow
\fcolorbox{myred}{mycream!5}{$\dmod(1)
\twoheadrightarrow
\lmod(1)$}
\quad
\fcolorbox{myblue}{mycream!5}{$\dmod(2)
\cong
\lmod(2)$}
\end{gather}
with $\lmod(\lambda)$ of dimension $\lambda$. 
Next, note that we get from \fullref{theorem:bgg} (up to base change)
\begin{gather}\label{eq:cartan-matrix-p=3}
\cmatrix(\uslt)=\begin{psmallmatrix}2 & 2 & 0\\2 & 2 & 0\\0 & 0 & 1\end{psmallmatrix}
=
\begin{psmallmatrix}1 & 1 & 0\\1 & 1 & 0\\0 & 0 & 1\end{psmallmatrix}
\begin{psmallmatrix}1 & 1 & 0\\1 & 1 & 0\\0 & 0 & 1\end{psmallmatrix}
=
\dmatrix(\uslt)^{\mathrm{T}}\dmatrix(\uslt)
\end{gather} 
which -- by \eqref{eq:sequence-p-is-three} -- actually gives 
us the indecomposable projectives $\prmod(\lambda)$
\begin{gather}
\fcolorbox{myred}{mycream!5}{$\dmod(1)$}
\hookrightarrow
\fcolorbox{mygreen}{mycream!5}{$\prmod(0)
\twoheadrightarrow
\dmod(0)$}
\quad
\fcolorbox{mygreen}{mycream!5}{$\dmod(0)$}
\hookrightarrow
\fcolorbox{myred}{mycream!5}{$\prmod(1)
\twoheadrightarrow
\dmod(1)$}
\quad
\fcolorbox{myblue}{mycream!5}{$\dmod(2)
\cong
\prmod(2)$}
\end{gather}
Finally, \eqref{eq:cartan-matrix-p=3} also shows -- by \fullref{remark:pos-def} -- that $\uslt$ 
is not cellular. However -- 
by \fullref{proposition:cell-relative-cell} -- the so-called \textit{core}
\begin{gather}\label{eq:core-def-slt}
\core(\uslt)={\textstyle\bigoplus_{\lambda\in\cset}}\hidem_{\lambda}\uslt\hidem_{\lambda}
=
\hidem_{0}\uslt\hidem_{0}
\oplus
\hidem_{1}\uslt\hidem_{1}
\oplus
\hidem_{2}\uslt\hidem_{2}
\end{gather}
is a cellular algebra. This recovers \cite[Theorem 1.2]{bt2}. 
It also follows from \fullref{proposition:semi-simplicity} that 
$\uslt$ is not semisimple.
\end{example}

\begin{remark}\label{remark:fancy-chi}
We stress that our assumption $\chi=0$ gives \eqref{eq:slt-rels-modp}. This is crucial since e.g. 
\fullref{lemma:some-calcs-slt-2} implies that
\begin{gather}
\egen^{k}\,\cbas{S,T}{\lambda}
\in
{\textstyle\sum_{j=0}^{k}}\,
\K\,
\cbas{S,T}{\lambda{+}2j}.
\end{gather}
Thus, if $\egen^{p}$ would not be zero, then 
$\lambda+2p$ would appear in the above sum and (\ref{definition:cell-algebra}.d) would fail.
\end{remark}

\subsection{Relative cellularity}\label{subsection:rel-cell-sl2}

The following is now the main statement in this section.

\begin{theorem}\label{theorem:main-uslt}
The algebra $\uslt$ is relative 
cellular with cell datum as in \eqref{eq:uslt-datum}.
\end{theorem}

\begin{proof}
\textit{(\ref{definition:cell-algebra}.a).} Up to the statement that the 
$\cbas{S,T}{\lambda}$ form a basis, this is clear. To see the basis 
statement use \fullref{theorem:PBW-slt}.
\medskip

\noindent\textit{(\ref{definition:cell-algebra}.b).} This follows since ${}^{\invo}$ is the 
Chevalley anti-involution.
\medskip

\noindent\textit{(\ref{definition:cell-algebra}.c).} By construction, \makeautorefname{lemma}{Lemmas}
the $\hidem_{\lambda}$'s are fixed 
by ${}^{\invo}$. To see \eqref{eq:idem-props-1} note 
that \fullref{lemma:prim-idem-sl2} 
and \ref{lemma:some-calcs-slt} show that \makeautorefname{lemma}{Lemma}
\begin{gather}
\hidem_{\mu}\uslt\hidem_{\mu}
=
\K\{
\fgen^{S}\hidem_{\nu}\egen^{S}
\mid
\nu=\mu-2S
\}.
\end{gather}
Thus -- by 
\fullref{lemma:some-calcs-slt-2} -- 
all appearing basis elements in $\hidem_{\mu}\uslt\hidem_{\mu}\,\cbas{S,T}{\lambda}$ 
are smaller than $\lambda$ in the order for $\mu$. \makeautorefname{lemma}{Lemmas}
The rest follows from
\fullref{lemma:prim-idem-sl2} and \ref{lemma:some-calcs-slt}. \makeautorefname{lemma}{Lemma}
\medskip

\noindent\textit{(\ref{definition:cell-algebra}.d).} Directly by using 
\fullref{lemma:some-calcs-slt-2} we get
\begin{gather}
\begin{aligned}
\cbas{S,T}{\lambda}\,
\cbas{U,V}{\mu}
&=
\fgen^{S}\hidem_{\lambda}\egen^{T}\,
\fgen^{U}\hidem_{\mu}\egen^{V}
\\
&\in
r(T,U)\,
\fgen^{S}\hidem_{\lambda}\hidem_{\mu}\egen^{V}
+
{\textstyle\sum_{j=0}^{\min(T,U)-1}}\,\K\,
\fgen^{S+T{-}j}\hidem_{\mu{+}2(U{-}j)}\egen^{U{-}j+V}.
\end{aligned}
\end{gather}
Thus, (\ref{definition:cell-algebra}.d) follows since 
$\hidem_{\lambda}\hidem_{\mu}$ equals $\hidem_{\mu}$ or zero
and $\mu+2(U-j)\ord{\hidem_{\mu}}\mu$ for 
$U-j\in\setp$ and $\egen^{U{-}j+V}=0$ for $U-j\geq p$.
\end{proof}

\subsection{Some consequences}\label{subsection:rel-cell-some-results-sl2}

Similarly as in \fullref{example:p-is-three}, we will explain how to recover the 
representation theory of $\uslt$ for general $p>2$. All of 
this is of course known, but the point is that 
we use the general theory of relative cellular algebras to do so.

\begin{proposition}\label{proposition:slt-theory}
From the \fullref{theorem:main-uslt} and the theory 
of relative cellular algebras we obtain the following, 
where $\lambda\in\cset$:
\smallskip
\begin{enumerate}[label=(\alph*)]

\setlength\itemsep{.15cm}

\item The cell modules $\dmod(\lambda)$ 
are of dimension $p$ and isomorphic to baby Verma modules 
of highest weight $\lambda$.

\item The simple quotients $\lmod(\lambda)$
of $\dmod(\lambda)$ 
are of dimension $\lambda$ and we have
\begin{gather}
\fcolorbox{myred}{mycream!5}{$\lmod(p-\lambda-2)$}
\hookrightarrow
\fcolorbox{mygreen}{mycream!5}{$\dmod(\lambda)
\twoheadrightarrow
\lmod(\lambda)$}
\quad
\quad
\fcolorbox{myblue}{mycream!5}{$\dmod(p-1)\cong\lmod(p-1)$}
\end{gather}

\item The indecomposable
projectives $\prmod(\lambda)$ satisfy
\begin{gather}
\fcolorbox{myred}{mycream!5}{$\dmod(p-\lambda-2)$}
\hookrightarrow
\fcolorbox{mygreen}{mycream!5}{$\prmod(\lambda)
\twoheadrightarrow
\dmod(\lambda)$}
\quad
\quad
\fcolorbox{myblue}{mycream!5}{$\prmod(p-1)\cong\dmod(p-1)$}
\end{gather}

\item The algebra $\uslt$ is a non-semisimple, 
non-cellular algebra whose core (defined as in 
\eqref{eq:core-def-slt}) $\core(\uslt)$ is cellular.\qedhere

\end{enumerate}
\end{proposition}

\begin{proof} \makeautorefname{subsection}{Sections}
We use all the lemmas from \fullref{subsection:the-algebra-sl2}
and \ref{subsection:the-cell-datum-sl2}. \makeautorefname{subsection}{Section}
Using these, the general case can be proven verbatim as the $p=3$ case 
in \fullref{example:p-is-three}:
\medskip

\noindent\textit{(\ref{proposition:slt-theory}.a).} Clear by construction.
\medskip

\noindent\textit{(\ref{proposition:slt-theory}.b).} The first claim follows since
\begin{gather}
\Cpair{\lambda}(\fgen^{S}\hidem_{\lambda},\fgen^{S}\hidem_{\lambda})
=
\begin{cases}
{\textstyle
(S!)^{2}
\binom{\lambda}{S}}, & \text {if }S=T,
\\
0, & \text {if }S\neq T.
\end{cases}
\end{gather}
The second claim follows then from (\ref{proposition:slt-theory}.a).
\medskip

\noindent\textit{(\ref{proposition:slt-theory}.c).} By using (\ref{proposition:slt-theory}.b) 
and \fullref{theorem:bgg}.
\medskip

\noindent\textit{(\ref{proposition:slt-theory}.d).} Observe that (\ref{proposition:slt-theory}.b) 
shows -- by \fullref{proposition:semi-simplicity} -- that $\uslt$ is non-semisimple, 
while (\ref{proposition:slt-theory}.c) 
-- by \fullref{remark:pos-def} -- 
shows that $\uslt$ is not cellular. 
The last claim follows from \fullref{theorem:main-uslt} and \fullref{proposition:cell-relative-cell}.
\end{proof}

This 
resembles the known representation theory of $\uslt$  
from the theory of relative cellular algebras.

\begin{remark}\label{remark:small-sltwo}
The case of the small quantum group $\suslt$ for $\qpar$ 
being a complex, primitive $2l^{\mathrm{th}}$ root of unity with 
$l>2$ works -- by carefully keeping track of the quantum numbers -- 
mutatis mutandis as above. Details are omitted.
\end{remark}

\begin{furtherdirections}\label{remark:tilting2}
Having (\ref{proposition:slt-theory}.d), it is tempting to 
ask whether one can extend the setting of 
\cite{bt1} and \cite{bt2}. However, we stress that 
our above basis is too ``naive'' to generalize to higher rank cases 
and certainly is not the relative analog of the basis of $\core(\uslt)$
constructed in \cite[Theorem 4.6]{bt2}.
\end{furtherdirections}
%
%%%%%%%%%%%%%%%%%%%%%%%%%%%
\section{An extended example II: The annular arc algebra}\label{section:arc-stuff}
%%%%%%%%%%%%%%%%%%%%%%%%%%%

Throughout, fix $\numberarcs\in\Np$.
The purpose of this section is to discuss the 
relative cellularity of the annular arc algebra 
$\aarc$ in detail, with \fullref{theorem:main-arc-algebra} 
being the main result. 

The definition of the underlying space and multiplication rule 
for $\aarc$ are due to Anno--Nandakumar \cite[Section 5.3]{an1}, and 
we will recall their\makeautorefname{subsection}{Sections}
definitions in \fullref{subsection:arc-a-word}
to \ref{subsection:arc-mult} in our conventions. Following \cite{aps1}, we show\makeautorefname{subsection}{Section}
well-definedness in \fullref{subsection:arc-algebra}.

%%%%%%%%%%%%%%%%%%%%%%%%%%%
\subsection{The arc algebra in an annulus}\label{subsection:arc-a-word}
%%%%%%%%%%%%%%%%%%%%%%%%%%%

The conventions we use for $\aarc$ are very much 
in the spirit of the type $\typeA$ arc algebra $\uarc$ 
(see e.g. \cite{kh1} or \cite{bs1}), 
but using a TQFT as in \cite{aps1}.  
Consequently, all the definitions below are adaptations of the 
corresponding notions for $\uarc$ to the annulus, 
where we keep the following illustration in mind:  
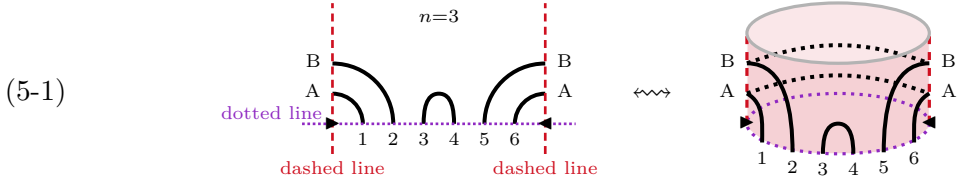
\begin{gather}\label{eq:annulus}
\begin{tikzpicture}[anchorbase, scale=.4, tinynodes]
	\draw[arc] (-2,0) node[below] {1} to [out=90, in=0] (-3,1) node[left] {A};
	\draw[arc] (3,0) node[below] {6} to [out=90, in=180] (4,1) node[right] {A};
	\draw[arc] (-1,0) node[below] {2} to [out=90, in=0] (-3,2) node[left] {B};
	\draw[arc] (2,0) node[below] {5} to [out=90, in=180] (4,2) node[right] {B};
	\draw[arc] (0,0) node[below] {3} to [out=90, in=180] (.5,1) to [out=0, in=90] (1,0) node[below] {4};
	\draw[daline] (-3,-1) node[below] {\text{dashed line}} to (-3,4);
	\draw[daline] (4,-1) node[below] {\text{dashed line}} to (4,4);
	\draw[doline] (-4,0) to (5,0);
	\node[orchid] at (-5,.3) {\text{dotted line}};
	\node at (-3,0) {$\blacktriangleright$};
	\node at (4,0) {$\blacktriangleleft$};
	\node at (.5,3.5) {$\numberarcs{=}3$};
\end{tikzpicture}
\;\;
\leftrightsquigarrow
\quad
\begin{tikzpicture}[anchorbase, scale=.4, tinynodes]
	\draw[white, fill=lava, opacity=.1] (-2,.875) to (-2,3.875) to [out=315, in=180] (1,3.05) to [out=0, in=225] (4,3.875) to (4,.875) to [out=225, in=0] (1,0) to [out=180, in=315] (-2,.875);
	\draw[white, fill=lava, opacity=.1] (-2,.875) to (-2,4.1) to [out=49, in=180] (1,5) to [out=0, in=131] (4,4.1) to (4,.875) to [out=225, in=0] (1,0) to [out=180, in=315] (-2,.875);
	\draw[very thick, orchid, dotted] (4,1) arc[x radius=3, y radius=1, start angle=0, end angle=180];
	\draw[very thick, orchid, dotted] (4,1) arc[x radius=3, y radius=1, start angle=0, end angle=-180];
	\draw[very thick, lava, dashed] (-2,1) to (-2,4);
	\draw[very thick, lava, dashed] (4,1) to (4,4);
	\draw[arc] (-1.5,.375) node[below] {1} to [out=90, in=330] (-2,2) node[left] {A};
	\draw[arc, dotted] (-2,2) to [out=20, in=160] (4,2);
	\draw[arc] (3.5,.375) node[below] {6} to [out=90, in=210] (4,2) node[right] {A};
	\draw[arc] (-.5,.05) node[below] {2} to [out=90, in=0] (-2,3) node[left] {B};
	\draw[arc, dotted] (-2,3) to [out=20, in=160] (4,3);
	\draw[arc] (2.5,.05) node[below] {5} to [out=90, in=180] (4,3) node[right] {B};
	\draw[arc] (.5,0) node[below] {3} to [out=90, in=180] (1,1) to [out=0, in=90] (1.5,0) node[below] {4};
	\draw[very thick, mygray] (1,4) ellipse (3cm and 1cm);
	\node at (-2,1) {$\blacktriangleright$};
	\node at (4,1) {$\blacktriangleleft$};
\end{tikzpicture}
\end{gather}
(In \eqref{eq:annulus}, note that the annulus is topologically a cylinder, 
a perspective that we use silently throughout.) 
Readers familiar with $\uarc$
can immediately check \ref{sec-arcs-merge}
and \ref{sec-arcs-split} in addition 
to \eqref{eq:annulus} before reading the definitions.

%%%%%%%%%%%%%%%%%%%%%%%%%%%
\subsection{Combinatorics of annular arc diagrams}\label{subsection:arc-combinatorics}
%%%%%%%%%%%%%%%%%%%%%%%%%%%

We start by defining the necessary combinatorial data. 
Hereby we closely 
follow the exposition in the non-annular case from 
\cite[Section 2]{bs1} or \cite[Section 3]{est1}. 

\begin{definition}\label{definition-weight}
A (balanced) \textit{weight} (of rank $\numberarcs$) is a 
tuple $\lambda=(\lambda_{i})\in\{\down,\up\}^{2\numberarcs}$ with 
$\numberarcs$ symbols $\down$ and $\numberarcs$ symbols $\up$. 
The set of weights is denoted by $\cset$. 
\end{definition}

Simplifying notation, an example of a weight 
of rank $2$ is $\lambda=\down\,\up\,\up\,\down$.

Let $\Sp$ denote the $1$-sphere.
The \textit{dotted line} is topologically 
$\Sp\times\{0\}$ smoothly embedded in $\R^2\times\{0\}$ 
together with a choice of 
an orientation (this orientation will always be anticlockwise in illustrations),
two distinct points $\blacktriangleright,\blacktriangleleft$
and $2\numberarcs$ discrete points, called \textit{vertices}, 
in the segment $[\blacktriangleright,\blacktriangleleft]$
between $\blacktriangleright$ and $\blacktriangleleft$. 
We number the vertices in order from $1$ to $2\numberarcs$, reading 
along the chosen orientation. 
We view the dotted line as being the bottom (or top) boundary 
of $\Sp\times[0,1]$ (or $\Sp\times[0,-1]$) 
smoothly embedded in $\R^3$, with orientation compatible 
with the one of the dotted line.
Similarly, the \textit{dashed lines} 
are $\{\blacktriangleright\}\times[0,\pm 1]$ and 
$\{\blacktriangleleft\}\times[0,\pm 1]$, see again 
in $\Sp\times[0,\pm 1]$.
Note that each $\lambda=(\lambda_{i})\in\cset$ gives a labeling of 
the vertices of the dotted line by putting 
$\lambda_{i}$ at the $i^{\text{th}}$ vertex.

\begin{definition}\label{definition:arcs}
A(n annular) \textit{cup diagram} $S$ (of rank $\numberarcs$) is a collection 
$\{\gamma_{1},\dots,\gamma_{\numberarcs}\}$ of smooth embeddings 
of $[0,1]$ into $\Sp\times[0,-1]$, called \textit{arcs}, 
such that:
\smallskip
\begin{enumerate}[label=(\alph*)]

\setlength\itemsep{.15cm}

\item The arcs are pairwise non-intersecting 
and have only one critical point.

\item There is a $1{:}1$ correspondence 
between the vertices of the dotted line and 
the boundary points of arcs, identifying the two sets.

\item The arcs cut the dashed lines transversely
and each dashed line at most once. 

\end{enumerate}
\smallskip
Similarly, a(n annular) \textit{cap diagram} $T^{\invo}$ is defined 
inside $\Sp\times[0,1]$.

Observing that (\ref{definition:arcs}.b)
and (\ref{definition:arcs}.c) imply that each arc either 
stays within the region 
$[\blacktriangleright,\blacktriangleleft]\times[0,\pm 1]$ 
or goes around the cylinder once, we can say 
that an arc is of \textit{staying type} or \textit{wrapping type}. 
Similarly, if all arcs of a cup (or cap) diagram are 
of staying type, then we say that the cup (or cap) diagram is 
of \textit{staying type}. 
\end{definition}

Combinatorially speaking, 
we consider arcs to be equal if their endpoints 
connect the same vertices on the dotted line and 
they are of the same type, and the corresponding equivalence 
classes are still called cup and cap diagrams. 
We work with these throughout, and illustrate them 
as exemplified in \eqref{eq:arc-short}. 
We call the corresponding arcs \textit{cups} 
and \textit{caps}, and we usually denote them by $\cups$ respectively by $\caps$.

We note that cup (or cap) diagrams of staying type 
are those appearing for $\uarc$, while all others 
are new in the annular setting.

\begin{definition}\label{definition:orientation-arc}
An \textit{orientated cup diagram} $S\lambda$ 
is a pair of a cup diagram and a weight $\lambda$ such that 
the weight induces an orientation on the arcs of $S$ (seen topologically). 
An \textit{orientated cap diagram} $\lambda T^{\invo}$ 
is defined verbatim.

For $\lambda\in\cset$ we denote by $\cmset(\lambda)$ the set of all 
oriented cup diagrams of the form $S\lambda$.
\end{definition}

Note that we can swap the cylinders 
$\Sp\times[0,-1]\rightleftarrows\Sp\times[0,1]$
by reflecting along the $(x,y,0)$-plane in $\R^3$.
This induces an involution ${}^{\invo}$ 
turning a cup $S$ into a cap diagram $S^{\invo}$, 
and vice versa.
Clearly, $(S^{\invo})^{\invo}=S$, and -- by convention -- 
$(S\lambda)^{\invo}=\lambda S^{\invo}$ and 
$(\lambda S^{\invo})^{\invo}=S\lambda$.

\begin{definition}\label{definition:circle}
A(n annular) \textit{circle diagram} 
$ST^{\invo}$ (of rank $\numberarcs$) is obtained 
from a cup diagram $S$ and a cap diagram 
$T^{\invo}$ (both of rank $\numberarcs$) 
by stacking $T^{\invo}$ on top of $S$, 
inducing a corresponding diagram in $\Sp\times[-1,1]$.

An \textit{oriented circle diagram} is built from 
an oriented cup $S\lambda$ and cap diagram $\lambda T^{\invo}$ 
for the same weight $\lambda$. We denote such 
diagrams by $\cbas{S,T}{\lambda}$, and we say that the 
circle diagram $ST^{\invo}$ is \textit{associated to $\cbas{S,T}{\lambda}$}.
\end{definition}

Similar as cup and cap diagrams are built from arcs, 
circle diagrams are collections 
of (up to $\numberarcs$) \textit{circles} $\circles$, 
with ``circle'' understood in the evident way.

All the above is summarized in \eqref{eq:arc-short} below.

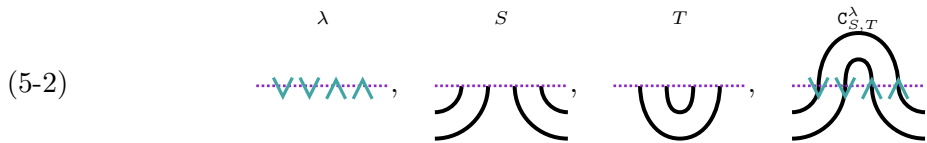
\begin{gather}\label{eq:arc-short}
\begin{tikzpicture}[anchorbase, scale=.35, tinynodes]
	\draw[arcd] (-2,0) to [out=270, in=180] (-.5,-2) to [out=0, in=270] (1,0);
	\draw[arcd] (-1,0) to [out=270, in=180] (-.5,-1) to [out=0, in=270] (0,0);
	\draw[arcd] (-2,0) to [out=90, in=180] (-.5,2) to [out=0, in=90] (1,0);
	\draw[arcd] (-1,0) to [out=90, in=180] (-.5,1) to [out=0, in=90] (0,0);
	\draw[doline] (-3,0) to (2,0);
	\node at (-2,0) {$\Down$};
	\node at (-1,0) {$\Down$};
	\node at (0,0) {$\Up$};
	\node at (1,0) {$\Up$};
	\node at (-.5,2.5) {$\lambda$};
	\node at (-.5,-2.5) {$\phantom{\lambda}$};
\end{tikzpicture}
,\quad
\begin{tikzpicture}[anchorbase, scale=.35, tinynodes]
	\draw[arc] (-2,0) to [out=270, in=0] (-3,-1);
	\draw[arc] (1,0) to [out=270, in=180] (2,-1);
	\draw[arc] (-1,0) to [out=270, in=0] (-3,-2);
	\draw[arc] (0,0) to [out=270, in=180] (2,-2);
	\draw[arcd] (-2,0) to [out=90, in=180] (-.5,2) to [out=0, in=90] (1,0);
	\draw[arcd] (-1,0) to [out=90, in=180] (-.5,1) to [out=0, in=90] (0,0);
	\draw[doline] (-3,0) to (2,0);
	\node at (-.5,2.5) {$S$};
	\node at (-.5,-2.5) {$\phantom{S}$};
\end{tikzpicture}
,\quad
\begin{tikzpicture}[anchorbase, scale=.35, tinynodes]
	\draw[arc] (-2,0) to [out=270, in=180] (-.5,-2) to [out=0, in=270] (1,0);
	\draw[arc] (-1,0) to [out=270, in=180] (-.5,-1) to [out=0, in=270] (0,0);
	\draw[arcd] (-2,0) to [out=90, in=180] (-.5,2) to [out=0, in=90] (1,0);
	\draw[arcd] (-1,0) to [out=90, in=180] (-.5,1) to [out=0, in=90] (0,0);
	\draw[doline] (-3,0) to (2,0);
	\node at (-.5,2.5) {$T$};
	\node at (-.5,-2.5) {$\phantom{T}$};
\end{tikzpicture}
,\quad
\begin{tikzpicture}[anchorbase, scale=.35, tinynodes]
	\draw[arc] (-2,0) to [out=270, in=0] (-3,-1);
	\draw[arc] (1,0) to [out=270, in=180] (2,-1);
	\draw[arc] (-1,0) to [out=270, in=0] (-3,-2);
	\draw[arc] (0,0) to [out=270, in=180] (2,-2);
	\draw[arc] (-2,0) to [out=90, in=180] (-.5,2) to [out=0, in=90] (1,0);
	\draw[arc] (-1,0) to [out=90, in=180] (-.5,1) to [out=0, in=90] (0,0);
	\draw[doline] (-3,0) to (2,0);
	\node at (-2,0) {$\Down$};
	\node at (-1,0) {$\Down$};
	\node at (0,0) {$\Up$};
	\node at (1,0) {$\Up$};
	\node at (-.5,2.5) {$\cbas{S,T}{\lambda}$};
	\node at (-.5,-2.5) {$\phantom{\cbas{S,T}{\lambda}}$};
\end{tikzpicture}
\end{gather}

\begin{definition}\label{definition:essential}
A circle $\circles$ in a circle diagram $ST^{\invo}$ is 
called \textit{essential} if it induces a 
non-trivial element in $\pi_{1}(\Sp\times[-1,1])$, 
and \textit{usual} otherwise.

For an oriented circle diagram $S\lambda T^{\invo}$, 
any circle $\circles$ gets an induced orientation.
Thus, we can say a usual circle is \textit{anticlockwise} or 
\textit{clockwise} (oriented), while essential circles 
are \textit{leftwards} or \textit{rightwards} (oriented).
\end{definition}

The picture illustrating \fullref{definition:essential} is:
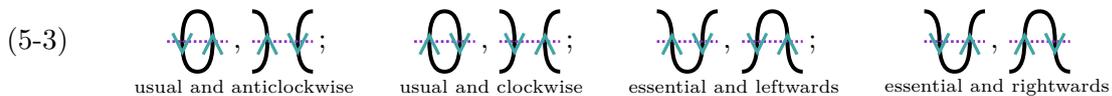
\begin{gather}\label{eq:arc-orientation}
\xy
(0,0)*{\begin{tikzpicture}[anchorbase, scale=.4, tinynodes]
	\draw[arc] (0,0) to [out=90, in=180] (.5,1) to [out=0, in=90] (1,0);
	\draw[arc] (0,0) to [out=270, in=180] (.5,-1) to [out=0, in=270] (1,0);
	\draw[doline] (-.5,0) to (1.5,0);
	\node at (0,0) {$\Down$};
	\node at (1,0) {$\Up$};
\end{tikzpicture}
,
\begin{tikzpicture}[anchorbase, scale=.4, tinynodes]
	\draw[arc] (0,0) to [out=90, in=0] (-.5,1);
	\draw[arc] (1,0) to [out=90, in=180] (1.5,1);
	\draw[arc] (0,0) to [out=270, in=0] (-.5,-1);
	\draw[arc] (1,0) to [out=270, in=180] (1.5,-1);
	\draw[doline] (-.5,0) to (1.5,0);
	\node at (0,0) {$\Up$};
	\node at (1,0) {$\Down$};
\end{tikzpicture};};
(0,-6)*{\text{{\tiny usual and anticlockwise}}};
\endxy
\quad\;\;
\xy
(0,0)*{
\begin{tikzpicture}[anchorbase, scale=.4, tinynodes]
	\draw[arc] (0,0) to [out=90, in=180] (.5,1) to [out=0, in=90] (1,0);
	\draw[arc] (0,0) to [out=270, in=180] (.5,-1) to [out=0, in=270] (1,0);
	\draw[doline] (-.5,0) to (1.5,0);
	\node at (0,0) {$\Up$};
	\node at (1,0) {$\Down$};
\end{tikzpicture}
,
\begin{tikzpicture}[anchorbase, scale=.4, tinynodes]
	\draw[arc] (0,0) to [out=90, in=0] (-.5,1);
	\draw[arc] (1,0) to [out=90, in=180] (1.5,1);
	\draw[arc] (0,0) to [out=270, in=0] (-.5,-1);
	\draw[arc] (1,0) to [out=270, in=180] (1.5,-1);
	\draw[doline] (-.5,0) to (1.5,0);
	\node at (0,0) {$\Down$};
	\node at (1,0) {$\Up$};
\end{tikzpicture};
};
(0,-6)*{\text{{\tiny usual and clockwise}}};
\endxy
\quad\;\;
\xy
(0,0)*{
\begin{tikzpicture}[anchorbase, scale=.4, tinynodes]
	\draw[arc] (0,0) to [out=90, in=0] (-.5,1);
	\draw[arc] (1,0) to [out=90, in=180] (1.5,1);
	\draw[arc] (0,0) to [out=270, in=180] (.5,-1) to [out=0, in=270] (1,0);
	\draw[doline] (-.5,0) to (1.5,0);
	\node at (0,0) {$\Up$};
	\node at (1,0) {$\Down$};
\end{tikzpicture}
,
\begin{tikzpicture}[anchorbase, scale=.4, tinynodes]
	\draw[arc] (0,0) to [out=90, in=180] (.5,1) to [out=0, in=90] (1,0);
	\draw[arc] (0,0) to [out=270, in=0] (-.5,-1);
	\draw[arc] (1,0) to [out=270, in=180] (1.5,-1);
	\draw[doline] (-.5,0) to (1.5,0);
	\node at (0,0) {$\Down$};
	\node at (1,0) {$\Up$};
\end{tikzpicture};
};
(0,-6)*{\text{{\tiny essential and leftwards}}};
\endxy
\quad\;\phantom{.}
\xy
(0,0)*{
\begin{tikzpicture}[anchorbase, scale=.4, tinynodes]
	\draw[arc] (0,0) to [out=90, in=0] (-.5,1);
	\draw[arc] (1,0) to [out=90, in=180] (1.5,1);
	\draw[arc] (0,0) to [out=270, in=180] (.5,-1) to [out=0, in=270] (1,0);
	\draw[doline] (-.5,0) to (1.5,0);
	\node at (0,0) {$\Down$};
	\node at (1,0) {$\Up$};
\end{tikzpicture}
,
\begin{tikzpicture}[anchorbase, scale=.4, tinynodes]
	\draw[arc] (0,0) to [out=90, in=180] (.5,1) to [out=0, in=90] (1,0);
	\draw[arc] (0,0) to [out=270, in=0] (-.5,-1);
	\draw[arc] (1,0) to [out=270, in=180] (1.5,-1);
	\draw[doline] (-.5,0) to (1.5,0);
	\node at (0,0) {$\Up$};
	\node at (1,0) {$\Down$};
\end{tikzpicture}
};
(0,-6)*{\text{{\tiny essential and rightwards}}};
\endxy
\end{gather}
(As in \eqref{eq:arc-orientation}, we say e.g. 
\textit{usual and clockwise} for short.)

%%%%%%%%%%%%%%%%%%%%%%%%%%%
\subsection{The multiplication}\label{subsection:arc-mult}
%%%%%%%%%%%%%%%%%%%%%%%%%%%

We first define the vector space for the annular arc algebra, and 
explain the multiplication afterwards.

\begin{definition}\label{definition:arc-k-vect}
As a vector space, the \textit{annular arc algebra $\aarc$} (of rank $\numberarcs$) is
\begin{gather}
\aarc
=
\K\{
\cbas{S,T}{\lambda}\mid \lambda\in\cset,\, S,T\in\cmset(\lambda)
\},
\end{gather}
i.e. the free vector space on basis given by all 
oriented circle diagrams (of rank $\numberarcs$). 
\end{definition}

Before we define the multiplication by a \textit{surgery procedure},
here a prototypical example, each step called 
a(n oriented) \textit{stacked diagram}:

\begin{gather}\label{eq:arc-surgery}
\begin{tikzpicture}[anchorbase, scale=.4, tinynodes]
	\draw[ultra thick, mygray] (-.5,2) to (-.5,3);
	\draw[arc] (-2,0) to [out=270, in=0] (-3,-1);
	\draw[arc] (1,0) to [out=270, in=180] (2,-1);
	\draw[arc] (-1,0) to [out=270, in=180] (-.5,-1) to [out=0, in=270] (0,0);
	\draw[arc] (-2,0) to [out=90, in=180] (-.5,2) to [out=0, in=90] (1,0);
	\draw[arc] (-1,0) to [out=90, in=180] (-.5,1) to [out=0, in=90] (0,0);
	\draw[arc] (-2,5) to [out=270, in=180] (-.5,3) to [out=0, in=270] (1,5);
	\draw[arc] (-1,5) to [out=270, in=180] (-.5,4) to [out=0, in=270] (0,5);
	\draw[arc] (-2,5) to [out=90, in=180] (-.5,7) to [out=0, in=90] (1,5);
	\draw[arc] (-1,5) to [out=90, in=180] (-.5,6) to [out=0, in=90] (0,5);
	\draw[doline] (-3,0) to (2,0);
	\node at (-2,0) {$\Down$};
	\node at (-1,0) {$\Down$};
	\node at (0,0) {$\Up$};
	\node at (1,0) {$\Up$};
	\draw[doline] (-3,5) to (2,5);
	\node at (-2,5) {$\Down$};
	\node at (-1,5) {$\Up$};
	\node at (0,5) {$\Down$};
	\node at (1,5) {$\Up$};
	\node at (-2.5,1) {$T^{\invo}$};
	\node at (-2.5,4) {$U$};
	\node at (-2.5,-2.5) {$S$};
	\node at (-2.5,7.5) {$V^{\invo}$};
	\node at (-1.75,2.5) {middle};
	\node at (1.6,.3) {$\lambda$};
	\node at (1.6,4.7) {$\mu$};
	\node at (-.5,7.5) {$\cbas{U,V}{\mu}$};
	\node at (-.5,-2.5) {$\cbas{S,T}{\lambda}$};
\end{tikzpicture}
\xrightarrow[\text{gery}]{\text{sur-}}
\begin{tikzpicture}[anchorbase, scale=.4, tinynodes]
	\draw[ultra thick, mygray] (-.5,1) to (-.5,4);
	\draw[arc] (-2,0) to [out=270, in=0] (-3,-1);
	\draw[arc] (1,0) to [out=270, in=180] (2,-1);
	\draw[arc] (-1,0) to [out=270, in=180] (-.5,-1) to [out=0, in=270] (0,0);
	\draw[arc] (-2,0) to (-2,5);
	\draw[arc] (-1,0) to [out=90, in=180] (-.5,1) to [out=0, in=90] (0,0);
	\draw[arc] (1,0) to (1,5);
	\draw[arc] (-1,5) to [out=270, in=180] (-.5,4) to [out=0, in=270] (0,5);
	\draw[arc] (-2,5) to [out=90, in=180] (-.5,7) to [out=0, in=90] (1,5);
	\draw[arc] (-1,5) to [out=90, in=180] (-.5,6) to [out=0, in=90] (0,5);
	\draw[doline] (-3,0) to (2,0);
	\node at (-2,0) {$\Down$};
	\node at (-1,0) {$\Down$};
	\node at (0,0) {$\Up$};
	\node at (1,0) {$\Up$};
	\draw[doline] (-3,5) to (2,5);
	\node at (-2,5) {$\Down$};
	\node at (-1,5) {$\Up$};
	\node at (0,5) {$\Down$};
	\node at (1,5) {$\Up$};
	\node at (-2.5,-2.5) {$S$};
	\node at (-2.5,7.5) {$V^{\invo}$};
	\node at (-.5,7.5) {$\phantom{\cbas{T,V}{\mu}}$};
	\node at (-.5,-2.5) {$\phantom{\cbas{S,T}{\lambda}}$};
\end{tikzpicture}
\xrightarrow[\text{gery}]{\text{sur-}}
\begin{tikzpicture}[anchorbase, scale=.4, tinynodes]
	\draw[arc] (-2,0) to [out=270, in=0] (-3,-1);
	\draw[arc] (1,0) to [out=270, in=180] (2,-1);
	\draw[arc] (-1,0) to [out=270, in=180] (-.5,-1) to [out=0, in=270] (0,0);
	\draw[arc] (-2,0) to (-2,5);
	\draw[arc] (-1,0) to (-1,5);
	\draw[arc] (1,0) to (1,5);
	\draw[arc] (0,0) to (0,5);
	\draw[arc] (-2,5) to [out=90, in=180] (-.5,7) to [out=0, in=90] (1,5);
	\draw[arc] (-1,5) to [out=90, in=180] (-.5,6) to [out=0, in=90] (0,5);
	\draw[doline] (-3,0) to (2,0);
	\node at (-2,0) {$\Down$};
	\node at (-1,0) {$\Upp$};
	\node at (0,0) {$\Downn$};
	\node at (1,0) {$\Up$};
	\draw[doline] (-3,5) to (2,5);
	\node at (-2,5) {$\Down$};
	\node at (-1,5) {$\Up$};
	\node at (0,5) {$\Down$};
	\node at (1,5) {$\Up$};
	\node at (-2.5,-2.5) {$S$};
	\node at (-2.5,7.5) {$V^{\invo}$};
	\node at (1.6,.3) {$\mu$};
	\node at (1.6,4.7) {$\mu$};
	\node at (-.5,7.5) {$\cbas{S,V}{\mu}$};
	\node at (-.5,-2.5) {$\phantom{\cbas{S,T}{\lambda}}$};
\end{tikzpicture}
=
\begin{tikzpicture}[anchorbase, scale=.4, tinynodes]
	\draw[arc] (-2,0) to [out=270, in=0] (-3,-1);
	\draw[arc] (1,0) to [out=270, in=180] (2,-1);
	\draw[arc] (-1,0) to [out=270, in=180] (-.5,-1) to [out=0, in=270] (0,0);
	\draw[arc] (-2,0) to [out=90, in=180] (-.5,2) to [out=0, in=90] (1,0);
	\draw[arc] (-1,0) to [out=90, in=180] (-.5,1) to [out=0, in=90] (0,0);
	\draw[doline] (-3,0) to (2,0);
	\node at (-2,0) {$\Down$};
	\node at (-1,0) {$\Up$};
	\node at (0,0) {$\Down$};
	\node at (1,0) {$\Up$};
	\node at (-2.5,-2.5) {$S$};
	\node at (-2.5,2.5) {$V^{\invo}$};
	\node at (1.6,.3) {$\mu$};
	\node at (-.5,2.5) {$\cbas{S,V}{\mu}$};
	\node at (-.5,-2.5) {$\phantom{\cbas{S,T}{\lambda}}$};
\end{tikzpicture}
\end{gather}
(In our notation, left multiplication is given by concatenation from the bottom.)

To define the multiplication $\mathrm{Mult}\colon\aarc\otimes\aarc\to\aarc$ 
it suffices to explain it on two basis elements $\cbas{S,T}{\lambda}$ 
and $\cbas{U,V}{\mu}$, and extend linearly. The multiplication of such basis 
elements
is defined as follows.
\medskip
\begin{enumerate}[label=(\alph*)]

\setlength\itemsep{.15cm}

\item We let $\cbas{S,T}{\lambda}\cbas{U,V}{\mu}=0$ unless $T=U$. 
Otherwise, put the circle diagram 
associated to $\cbas{U,V}{\mu}$ on top of 
the one associated to $\cbas{S,T}{\lambda}$,
producing a stacked diagram having $T^{\invo}U$ in the middle, 
cf. \eqref{eq:arc-surgery}.

\item For the stacked diagram perform inductively a surgery procedure 
by picking any (note the choice involved)
$\cups\,\text{-}\,\caps$ pair available, meaning 
that the $\cups$ and the $\caps$ can be connected without crossing 
any other arc, and replace it locally via:
\begin{gather}\label{eq:arc-surgery-rule}
\begin{tikzpicture}[anchorbase, scale=.35, tinynodes]
	\draw[ultra thick, mygray] (-1.5,2) to (-1.5,2.5) to node[left] {choose} (-1.5,2.51) to (-1.5,3);
	\draw[arc] (0,5) to [out=270, in=0] (-1.5,3) to (-2.5,3);
	\draw[arc] (3,5) to [out=270, in=180] (4.5,3) to (5.5,3);
	\draw[arc] (0,0) to [out=90, in=0] (-1.5,2) to (-2.5,2);
	\draw[arc] (3,0) to [out=90, in=180] (4.5,2) to (5.5,2);
	\draw[arc] (-2,5) to [out=270, in=180] (-1.5,4) to [out=0, in=270] (-1,5);
	\draw[arc] (1,5) to [out=270, in=180] (1.5,4) to [out=0, in=270] (2,5);
	\draw[arc] (-2,0) to [out=90, in=180] (-1.5,1) to [out=0, in=90] (-1,0);
	\draw[arc] (1,0) to [out=90, in=180] (1.5,1) to [out=0, in=90] (2,0);
	\draw[arc] (4,5) to [out=270, in=180] (4.5,4) to [out=0, in=270] (5,5);
	\draw[arc] (4,0) to [out=90, in=180] (4.5,1) to [out=0, in=90] (5,0);
	\draw[doline] (-2.5,0) to (5.5,0);
	\draw[doline] (-2.5,5) to (5.5,5);
\end{tikzpicture}
\xrightarrow[\text{gery}]{\text{sur-}}
\begin{tikzpicture}[anchorbase, scale=.35, tinynodes]
	\draw[arc] (0,5) to (0,0);
	\draw[arc] (3,5) to (3,0);
	\draw[arc] (-2,5) to [out=270, in=180] (-1.5,4) to [out=0, in=270] (-1,5);
	\draw[arc] (1,5) to [out=270, in=180] (1.5,4) to [out=0, in=270] (2,5);
	\draw[arc] (-2,0) to [out=90, in=180] (-1.5,1) to [out=0, in=90] (-1,0);
	\draw[arc] (1,0) to [out=90, in=180] (1.5,1) to [out=0, in=90] (2,0);
	\draw[arc] (4,5) to [out=270, in=180] (4.5,4) to [out=0, in=270] (5,5);
	\draw[arc] (4,0) to [out=90, in=180] (4.5,1) to [out=0, in=90] (5,0);
	\draw[doline] (-2.5,0) to (5.5,0);
	\draw[doline] (-2.5,5) to (5.5,5);
\end{tikzpicture}
\end{gather}

\item In each step of (\ref{section:arc-stuff}.b) we replace the resulting 
stacked diagrams by a sum of (oriented) stacked diagrams 
as explained below.

\item Finally, collapse the resulting 
stacked diagrams to circle diagrams 
as illustrated on the right in \eqref{eq:arc-surgery}.

\end{enumerate}
\medskip

Observing that each step of (\ref{section:arc-stuff}.b) 
either merges two circles into one, or splits one circle into two, 
we define how to reorient diagrams as follows. 
In all cases, we say ``orient the result'' meaning to put the corresponding 
orientation locally on the stacked diagram after applying 
(\ref{section:arc-stuff}.b), leaving 
all non-involved parts with the same orientation.
\medskip

\noindent \textit{\setword{\ref{section:arc-stuff}.M}{sec-arcs-merge}.} 
Assume that two circles
are merge into one.
\smallskip
\begin{enumerate}[label=(\alph*)]

\setlength\itemsep{.15cm}

\item If one of the circles 
is usual and anticlockwise, then orient the 
result with the orientation induced by the other 
circle.

\item If one the circles is usual and clockwise 
and the other is not usual and anticlockwise, then 
the result is zero.

\item If one the circles is essential and leftwards 
and the other is essential and rightwards, then orient the 
result clockwise.

\item Otherwise, the result is zero.

\end{enumerate}
\begin{gather}
\fcolorbox{myorange}{mycream!10}{$
\xy
(0,0)*{
\begin{tikzpicture}[anchorbase, scale=.35, tinynodes]
	\draw[arc] (0,3) to [out=90, in=0] (-.5,4);
	\draw[arc] (1,3) to [out=90, in=180] (1.5,4);
	\draw[arc] (0,0) to [out=270, in=180] (.5,-1) to [out=0, in=270] (1,0);
	\draw[arc] (0,0) to [out=90, in=180] (.5,1) to [out=0, in=90] (1,0);
	\draw[arc] (0,3) to [out=270, in=180] (.5,2) to [out=0, in=270] (1,3);
	\draw[doline] (-.5,0) to (1.5,0);
	\draw[doline] (-.5,3) to (1.5,3);
	\node at (0,0) {$\Down$};
	\node at (1,0) {$\Up$};
	\node at (1,3) {$\Down$};
	\node at (0,3) {$\Up$};
\end{tikzpicture}
\mapsto
\begin{tikzpicture}[anchorbase, scale=.35, tinynodes]
	\draw[arc] (0,0) to [out=90, in=0] (-.5,1);
	\draw[arc] (1,0) to [out=90, in=180] (1.5,1);
	\draw[arc] (0,0) to [out=270, in=180] (.5,-1) to [out=0, in=270] (1,0);
	\draw[doline] (-.5,0) to (1.5,0);
	\node at (1,0) {$\Down$};
	\node at (0,0) {$\Up$};
\end{tikzpicture}};
(0,-12)*{\text{{\tiny Example for (\ref{section:arc-stuff}.M.a)}}};
\endxy
$}
\;
\fcolorbox{myorange}{mycream!10}{$
\xy
(0,0)*{
\begin{tikzpicture}[anchorbase, scale=.35, tinynodes]
	\draw[arc] (0,3) to [out=90, in=0] (-.5,4);
	\draw[arc] (1,3) to [out=90, in=180] (1.5,4);
	\draw[arc] (0,0) to [out=270, in=180] (.5,-1) to [out=0, in=270] (1,0);
	\draw[arc] (0,0) to [out=90, in=180] (.5,1) to [out=0, in=90] (1,0);
	\draw[arc] (0,3) to [out=270, in=180] (.5,2) to [out=0, in=270] (1,3);
	\draw[doline] (-.5,0) to (1.5,0);
	\draw[doline] (-.5,3) to (1.5,3);
	\node at (1,0) {$\Down$};
	\node at (0,0) {$\Up$};
	\node at (1,3) {$\Down$};
	\node at (0,3) {$\Up$};
\end{tikzpicture}
\mapsto
0};
(0,-12)*{\text{{\tiny Example for (\ref{section:arc-stuff}.M.b)}}};
\endxy
$}
\;
\fcolorbox{myorange}{mycream!10}{$
\xy
(0,0)*{
\begin{tikzpicture}[anchorbase, scale=.35, tinynodes]
	\draw[arc] (0,3) to [out=90, in=0] (-.5,4);
	\draw[arc] (1,3) to [out=90, in=180] (1.5,4);
	\draw[arc] (0,0) to [out=270, in=0] (-.5,-1);
	\draw[arc] (1,0) to [out=270, in=180] (1.5,-1);
	\draw[arc] (0,0) to [out=90, in=180] (.5,1) to [out=0, in=90] (1,0);
	\draw[arc] (0,3) to [out=270, in=180] (.5,2) to [out=0, in=270] (1,3);
	\draw[doline] (-.5,0) to (1.5,0);
	\draw[doline] (-.5,3) to (1.5,3);
	\node at (0,0) {$\Down$};
	\node at (1,0) {$\Up$};
	\node at (0,3) {$\Down$};
	\node at (1,3) {$\Up$};
\end{tikzpicture}
\mapsto
\begin{tikzpicture}[anchorbase, scale=.35, tinynodes]
	\draw[arc] (0,0) to [out=90, in=0] (-.5,1);
	\draw[arc] (1,0) to [out=90, in=180] (1.5,1);
	\draw[arc] (0,0) to [out=270, in=0] (-.5,-1);
	\draw[arc] (1,0) to [out=270, in=180] (1.5,-1);
	\draw[doline] (-.5,0) to (1.5,0);
	\node at (0,0) {$\Down$};
	\node at (1,0) {$\Up$};
\end{tikzpicture}};
(0,-12)*{\text{{\tiny Example for (\ref{section:arc-stuff}.M.c)}}};
\endxy
$}
\;
\fcolorbox{myorange}{mycream!10}{$
\xy
(0,0)*{
\begin{tikzpicture}[anchorbase, scale=.35, tinynodes]
	\draw[arc] (0,3) to [out=90, in=0] (-.5,4);
	\draw[arc] (1,3) to [out=90, in=180] (1.5,4);
	\draw[arc] (0,0) to [out=270, in=0] (-.5,-1);
	\draw[arc] (1,0) to [out=270, in=180] (1.5,-1);
	\draw[arc] (0,0) to [out=90, in=180] (.5,1) to [out=0, in=90] (1,0);
	\draw[arc] (0,3) to [out=270, in=180] (.5,2) to [out=0, in=270] (1,3);
	\draw[doline] (-.5,0) to (1.5,0);
	\draw[doline] (-.5,3) to (1.5,3);
	\node at (0,0) {$\Down$};
	\node at (1,0) {$\Up$};
	\node at (1,3) {$\Down$};
	\node at (0,3) {$\Up$};
\end{tikzpicture}
\mapsto
0};
(0,-12)*{\text{{\tiny Example for (\ref{section:arc-stuff}.M.d)}}};
\endxy
$}
\end{gather}
\smallskip

\noindent \textit{\setword{\ref{section:arc-stuff}.S}{sec-arcs-split}.} 
Assume that one circle
is split into two.

\begin{enumerate}[label=(\alph*)]

\setlength\itemsep{.15cm}

\item If the circle 
is usual and anticlockwise and splits into two usual circles 
$\circles_{1}$ and $\circles_{2}$, then 
take the sum of two copies of the result. In one summand orient 
$\circles_{1}$ clockwise and $\circles_{2}$ anticlockwise, 
in the other swap the roles.

\item If the circle
is usual and clockwise and splits into two usual circles, 
then orient both circles in the result clockwise.

\item If the circle 
is usual and anticlockwise and splits into two essential circles 
$\circles_{1}$ and $\circles_{2}$, then 
take the sum of two copies of the result. In one summand orient 
$\circles_{1}$ leftwards and $\circles_{2}$ rightwards, 
in the other swap the roles.

\item If the circle
is usual and clockwise and splits into two essential circles, 
then the result is zero.

\item If the circle is essential, 
then orient the resulting usual circle clockwise while keeping the 
orientation of the resulting essential circle.

\end{enumerate}
\begin{gather}
\begin{gathered}
\fcolorbox{myorange}{mycream!10}{$
\xy
(0,0)*{
\begin{tikzpicture}[anchorbase, scale=.35, tinynodes]
	\draw[arc] (-1,3) to [out=90, in=180] (-.5,4) to [out=0, in=90] (0,3);
	\draw[arc] (1,3) to [out=90, in=180] (1.5,4) to [out=0, in=90] (2,3);
	\draw[arc] (-1,0) to [out=270, in=180] (-.5,-1) to [out=0, in=270] (0,0);
	\draw[arc] (1,0) to [out=270, in=180] (1.5,-1) to [out=0, in=270] (2,0);
	\draw[arc] (-1,0) to (-1,3);
	\draw[arc] (2,0) to (2,3);
	\draw[arc] (0,0) to [out=90, in=180] (.5,1) to [out=0, in=90] (1,0);
	\draw[arc] (0,3) to [out=270, in=180] (.5,2) to [out=0, in=270] (1,3);
	\draw[doline] (-1.5,0) to (2.5,0);
	\draw[doline] (-1.5,3) to (2.5,3);
	\node at (-1,0) {$\Down$};
	\node at (0,0) {$\Up$};
	\node at (1,0) {$\Down$};
	\node at (2,0) {$\Up$};
	\node at (-1,3) {$\Down$};
	\node at (0,3) {$\Up$};
	\node at (1,3) {$\Down$};
	\node at (2,3) {$\Up$};
\end{tikzpicture}
\mapsto
\begin{tikzpicture}[anchorbase, scale=.35, tinynodes]
	\draw[arc] (-1,3) to [out=90, in=180] (-.5,4) to [out=0, in=90] (0,3);
	\draw[arc] (1,3) to [out=90, in=180] (1.5,4) to [out=0, in=90] (2,3);
	\draw[arc] (-1,0) to [out=270, in=180] (-.5,-1) to [out=0, in=270] (0,0);
	\draw[arc] (1,0) to [out=270, in=180] (1.5,-1) to [out=0, in=270] (2,0);
	\draw[arc] (-1,0) to (-1,3);
	\draw[arc] (2,0) to (2,3);
	\draw[arc] (0,0) to (0,3);
	\draw[arc] (1,0) to (1,3);
	\draw[doline] (-1.5,0) to (2.5,0);
	\draw[doline] (-1.5,3) to (2.5,3);
	\node at (0,0) {$\Down$};
	\node at (-1,0) {$\Up$};
	\node at (1,0) {$\Down$};
	\node at (2,0) {$\Up$};
	\node at (0,3) {$\Down$};
	\node at (-1,3) {$\Up$};
	\node at (1,3) {$\Down$};
	\node at (2,3) {$\Up$};
\end{tikzpicture}
+
\begin{tikzpicture}[anchorbase, scale=.35, tinynodes]
	\draw[arc] (-1,3) to [out=90, in=180] (-.5,4) to [out=0, in=90] (0,3);
	\draw[arc] (1,3) to [out=90, in=180] (1.5,4) to [out=0, in=90] (2,3);
	\draw[arc] (-1,0) to [out=270, in=180] (-.5,-1) to [out=0, in=270] (0,0);
	\draw[arc] (1,0) to [out=270, in=180] (1.5,-1) to [out=0, in=270] (2,0);
	\draw[arc] (-1,0) to (-1,3);
	\draw[arc] (2,0) to (2,3);
	\draw[arc] (0,0) to (0,3);
	\draw[arc] (1,0) to (1,3);
	\draw[doline] (-1.5,0) to (2.5,0);
	\draw[doline] (-1.5,3) to (2.5,3);
	\node at (-1,0) {$\Down$};
	\node at (0,0) {$\Up$};
	\node at (2,0) {$\Down$};
	\node at (1,0) {$\Up$};
	\node at (-1,3) {$\Down$};
	\node at (0,3) {$\Up$};
	\node at (2,3) {$\Down$};
	\node at (1,3) {$\Up$};
\end{tikzpicture}
};
(0,-12)*{\text{{\tiny Example for (\ref{section:arc-stuff}.S.a)}}};
\endxy
$}
\;
\fcolorbox{myorange}{mycream!10}{$
\xy
(0,0)*{
\begin{tikzpicture}[anchorbase, scale=.35, tinynodes]
	\draw[arc] (-1,3) to [out=90, in=180] (-.5,4) to [out=0, in=90] (0,3);
	\draw[arc] (1,3) to [out=90, in=180] (1.5,4) to [out=0, in=90] (2,3);
	\draw[arc] (-1,0) to [out=270, in=180] (-.5,-1) to [out=0, in=270] (0,0);
	\draw[arc] (1,0) to [out=270, in=180] (1.5,-1) to [out=0, in=270] (2,0);
	\draw[arc] (-1,0) to (-1,3);
	\draw[arc] (2,0) to (2,3);
	\draw[arc] (0,0) to [out=90, in=180] (.5,1) to [out=0, in=90] (1,0);
	\draw[arc] (0,3) to [out=270, in=180] (.5,2) to [out=0, in=270] (1,3);
	\draw[doline] (-1.5,0) to (2.5,0);
	\draw[doline] (-1.5,3) to (2.5,3);
	\node at (0,0) {$\Down$};
	\node at (-1,0) {$\Up$};
	\node at (2,0) {$\Down$};
	\node at (1,0) {$\Up$};
	\node at (0,3) {$\Down$};
	\node at (-1,3) {$\Up$};
	\node at (2,3) {$\Down$};
	\node at (1,3) {$\Up$};
\end{tikzpicture}
\mapsto
\begin{tikzpicture}[anchorbase, scale=.35, tinynodes]
	\draw[arc] (-1,3) to [out=90, in=180] (-.5,4) to [out=0, in=90] (0,3);
	\draw[arc] (1,3) to [out=90, in=180] (1.5,4) to [out=0, in=90] (2,3);
	\draw[arc] (-1,0) to [out=270, in=180] (-.5,-1) to [out=0, in=270] (0,0);
	\draw[arc] (1,0) to [out=270, in=180] (1.5,-1) to [out=0, in=270] (2,0);
	\draw[arc] (-1,0) to (-1,3);
	\draw[arc] (2,0) to (2,3);
	\draw[arc] (0,0) to (0,3);
	\draw[arc] (1,0) to (1,3);
	\draw[doline] (-1.5,0) to (2.5,0);
	\draw[doline] (-1.5,3) to (2.5,3);
	\node at (0,0) {$\Down$};
	\node at (-1,0) {$\Up$};
	\node at (2,0) {$\Down$};
	\node at (1,0) {$\Up$};
	\node at (0,3) {$\Down$};
	\node at (-1,3) {$\Up$};
	\node at (2,3) {$\Down$};
	\node at (1,3) {$\Up$};
\end{tikzpicture}
};
(0,-12)*{\text{{\tiny Example for (\ref{section:arc-stuff}.S.b)}}};
\endxy
$}
\\
\fcolorbox{myorange}{mycream!10}{$
\xy
(0,0)*{
\begin{tikzpicture}[anchorbase, scale=.35, tinynodes]
	\draw[arc] (0,3) to [out=90, in=0] (-1.5,5);
	\draw[arc] (1,3) to [out=90, in=180] (2.5,5);
	\draw[arc] (-1,3) to [out=90, in=0] (-1.5,4);
	\draw[arc] (2,3) to [out=90, in=180] (2.5,4);
	\draw[arc] (0,0) to [out=270, in=180] (.5,-1) to [out=0, in=270] (1,0);
	\draw[arc] (-1,0) to [out=270, in=180] (.5,-2) to [out=0, in=270] (2,0);
	\draw[arc] (-1,0) to (-1,3);
	\draw[arc] (0,0) to (0,3);
	\draw[arc] (1,0) to [out=90, in=180] (1.5,1) to [out=0, in=90] (2,0);
	\draw[arc] (1,3) to [out=270, in=180] (1.5,2) to [out=0, in=270] (2,3);
	\draw[doline] (-1.5,0) to (2.5,0);
	\draw[doline] (-1.5,3) to (2.5,3);
	\node at (-1,0) {$\Down$};
	\node at (0,0) {$\Up$};
	\node at (1,0) {$\Down$};
	\node at (2,0) {$\Up$};
	\node at (-1,3) {$\Down$};
	\node at (0,3) {$\Up$};
	\node at (1,3) {$\Down$};
	\node at (2,3) {$\Up$};
\end{tikzpicture}
\!\!\mapsto\!\!
\begin{tikzpicture}[anchorbase, scale=.35, tinynodes]
	\draw[arc] (0,3) to [out=90, in=0] (-1.5,5);
	\draw[arc] (1,3) to [out=90, in=180] (2.5,5);
	\draw[arc] (-1,3) to [out=90, in=0] (-1.5,4);
	\draw[arc] (2,3) to [out=90, in=180] (2.5,4);
	\draw[arc] (0,0) to [out=270, in=180] (.5,-1) to [out=0, in=270] (1,0);
	\draw[arc] (-1,0) to [out=270, in=180] (.5,-2) to [out=0, in=270] (2,0);
	\draw[arc] (-1,0) to (-1,3);
	\draw[arc] (0,0) to (0,3);
	\draw[arc] (1,0) to (1,3);
	\draw[arc] (2,0) to (2,3);
	\draw[doline] (-1.5,0) to (2.5,0);
	\draw[doline] (-1.5,3) to (2.5,3);
	\node at (0,0) {$\Down$};
	\node at (-1,0) {$\Up$};
	\node at (2,0) {$\Down$};
	\node at (1,0) {$\Up$};
	\node at (0,3) {$\Down$};
	\node at (-1,3) {$\Up$};
	\node at (2,3) {$\Down$};
	\node at (1,3) {$\Up$};
\end{tikzpicture}
\!\!+\!\!
\begin{tikzpicture}[anchorbase, scale=.35, tinynodes]
	\draw[arc] (0,3) to [out=90, in=0] (-1.5,5);
	\draw[arc] (1,3) to [out=90, in=180] (2.5,5);
	\draw[arc] (-1,3) to [out=90, in=0] (-1.5,4);
	\draw[arc] (2,3) to [out=90, in=180] (2.5,4);
	\draw[arc] (0,0) to [out=270, in=180] (.5,-1) to [out=0, in=270] (1,0);
	\draw[arc] (-1,0) to [out=270, in=180] (.5,-2) to [out=0, in=270] (2,0);
	\draw[arc] (-1,0) to (-1,3);
	\draw[arc] (0,0) to (0,3);
	\draw[arc] (1,0) to (1,3);
	\draw[arc] (2,0) to (2,3);
	\draw[doline] (-1.5,0) to (2.5,0);
	\draw[doline] (-1.5,3) to (2.5,3);
	\node at (-1,0) {$\Down$};
	\node at (0,0) {$\Up$};
	\node at (1,0) {$\Down$};
	\node at (2,0) {$\Up$};
	\node at (-1,3) {$\Down$};
	\node at (0,3) {$\Up$};
	\node at (1,3) {$\Down$};
	\node at (2,3) {$\Up$};
\end{tikzpicture}};
(0,-16)*{\text{{\tiny Example for (\ref{section:arc-stuff}.S.c)}}};
\endxy
$}
\;
\fcolorbox{myorange}{mycream!10}{$
\xy
(0,0)*{
\begin{tikzpicture}[anchorbase, scale=.35, tinynodes]
	\draw[arc] (0,3) to [out=90, in=0] (-1.5,5);
	\draw[arc] (1,3) to [out=90, in=180] (2.5,5);
	\draw[arc] (-1,3) to [out=90, in=0] (-1.5,4);
	\draw[arc] (2,3) to [out=90, in=180] (2.5,4);
	\draw[arc] (0,0) to [out=270, in=180] (.5,-1) to [out=0, in=270] (1,0);
	\draw[arc] (-1,0) to [out=270, in=180] (.5,-2) to [out=0, in=270] (2,0);
	\draw[arc] (-1,0) to (-1,3);
	\draw[arc] (0,0) to (0,3);
	\draw[arc] (1,0) to [out=90, in=180] (1.5,1) to [out=0, in=90] (2,0);
	\draw[arc] (1,3) to [out=270, in=180] (1.5,2) to [out=0, in=270] (2,3);
	\draw[doline] (-1.5,0) to (2.5,0);
	\draw[doline] (-1.5,3) to (2.5,3);
	\node at (0,0) {$\Down$};
	\node at (-1,0) {$\Up$};
	\node at (2,0) {$\Down$};
	\node at (1,0) {$\Up$};
	\node at (0,3) {$\Down$};
	\node at (-1,3) {$\Up$};
	\node at (2,3) {$\Down$};
	\node at (1,3) {$\Up$};
\end{tikzpicture}
\!\!\mapsto\! 0};
(0,-16)*{\text{{\tiny Example for (\ref{section:arc-stuff}.S.d)}}};
\endxy
$}
\;
\fcolorbox{myorange}{mycream!10}{$
\xy
(0,0)*{
\begin{tikzpicture}[anchorbase, scale=.35, tinynodes]
	\draw[arcd] (-1,3) to [out=90, in=180] (.5,5) to [out=0, in=90] (2,3);
	\draw[arc] (0,3) to [out=90, in=180] (.5,4) to [out=0, in=90] (1,3);
	\draw[arc] (-1,3) to [out=90, in=0] (-1.5,4);
	\draw[arc] (2,3) to [out=90, in=180] (2.5,4);
	\draw[arc] (0,0) to [out=270, in=180] (.5,-1) to [out=0, in=270] (1,0);
	\draw[arc] (-1,0) to [out=270, in=180] (.5,-2) to [out=0, in=270] (2,0);
	\draw[arc] (-1,0) to (-1,3);
	\draw[arc] (0,0) to (0,3);
	\draw[arc] (1,0) to [out=90, in=180] (1.5,1) to [out=0, in=90] (2,0);
	\draw[arc] (1,3) to [out=270, in=180] (1.5,2) to [out=0, in=270] (2,3);
	\draw[doline] (-1.5,0) to (2.5,0);
	\draw[doline] (-1.5,3) to (2.5,3);
	\node at (0,0) {$\Down$};
	\node at (-1,0) {$\Up$};
	\node at (2,0) {$\Down$};
	\node at (1,0) {$\Up$};
	\node at (0,3) {$\Down$};
	\node at (-1,3) {$\Up$};
	\node at (2,3) {$\Down$};
	\node at (1,3) {$\Up$};
\end{tikzpicture}
\!\!\mapsto\!\!
\begin{tikzpicture}[anchorbase, scale=.35, tinynodes]
	\draw[arcd] (-1,3) to [out=90, in=180] (.5,5) to [out=0, in=90] (2,3);
	\draw[arc] (0,3) to [out=90, in=180] (.5,4) to [out=0, in=90] (1,3);
	\draw[arc] (-1,3) to [out=90, in=0] (-1.5,4);
	\draw[arc] (2,3) to [out=90, in=180] (2.5,4);
	\draw[arc] (0,0) to [out=270, in=180] (.5,-1) to [out=0, in=270] (1,0);
	\draw[arc] (-1,0) to [out=270, in=180] (.5,-2) to [out=0, in=270] (2,0);
	\draw[arc] (-1,0) to (-1,3);
	\draw[arc] (0,0) to (0,3);
	\draw[arc] (1,0) to (1,3);
	\draw[arc] (2,0) to (2,3);
	\draw[doline] (-1.5,0) to (2.5,0);
	\draw[doline] (-1.5,3) to (2.5,3);
	\node at (1,0) {$\Down$};
	\node at (-1,0) {$\Up$};
	\node at (2,0) {$\Down$};
	\node at (0,0) {$\Up$};
	\node at (1,3) {$\Down$};
	\node at (-1,3) {$\Up$};
	\node at (2,3) {$\Down$};
	\node at (0,3) {$\Up$};
\end{tikzpicture}};
(0,-16)*{\text{{\tiny Example for (\ref{section:arc-stuff}.S.e)}}};
\endxy
$}
\end{gathered}
\end{gather}
\medskip

We leave it to the reader to check that \ref{sec-arcs-merge} 
and \ref{sec-arcs-split} are all possible configurations. 

\subsection{Well-definedness via annular TQFTs}\label{subsection:arc-algebra}

We first prove well-definedness of $\aarc$.

\begin{proposition}\label{proposition:asso}
The multiplication is well-defined, i.e. 
independent of all involved choices. 
This turns $\aarc$ into an associative, unital, finite-dimensional 
algebra with 
\begin{gather}\label{eq:the-idem-set}
\ciset=\{
\cbas{S,S}{\lambda}\mid 
S\lambda S^{\invo}\text{ contains only usual and anticlockwise }\circles\text{s}
\}
\end{gather}
being a complete set of pairwise orthogonal idempotents.
\end{proposition}

\begin{proof}
With the well-definedness as an exception, the 
statements are easy to verify. 
We will sketch now why 
the multiplication is well-defined. 
(A detailed treatment in case of the non-annular arc algebras, that can be adapted 
to the annular setup, is 
explained e.g. in \cite{est1}.) The main idea is to identify 
the algebraically defined annular arc algebra with a topological algebra -- 
whose elements are certain surfaces -- obtained via 
a TQFT. For this topological incarnation of $\aarc$ the 
well-definedness boils down to isotopies of surfaces, and the 
main problem is to find the TQFT realizing $\aarc$. However, in our case this 
is easy since we modeled $\aarc$ on such a topological defined algebra using 
the TQFT from \cite{aps1}. 
(For further details about this TQFT 
see e.g. \cite[Section 2]{r1}, \cite[Section 4.2]{glw1} or \cite[Section 2.4]{bpw1}.) 
To be a bit more precise, using this TQFT one can define -- following e.g. \cite{est1} -- 
the topological incarnation of $\aarc$. Then, after choosing a cup basis as in \cite{est1}, one 
checks that on this basis the topological algebra satisfies the multiplication rules 
of $\aarc$.
\end{proof}

\begin{furtherdirections}\label{arc-connections-1}
The TQFT used in the proof of \fullref{proposition:asso} 
originates in the context of versions of
annular link homologies, see e.g. the references above.
It would be interesting to know a connection 
between $\aarc$ and those homologies.
\end{furtherdirections}

\begin{example}\label{example:arc-mult-1}
Here the multiplication for symmetric pictures 
in case $\numberarcs=1$:
\begin{gather}
\begin{gathered}
\fcolorbox{myorange}{mycream!10}{$
\xy
(0,0)*{
\begin{tikzpicture}[anchorbase, scale=.4, tinynodes]
	\draw[arc] (0,3) to [out=90, in=180] (.5,4) to [out=0, in=90] (1,3);
	\draw[arc] (0,0) to [out=270, in=180] (.5,-1) to [out=0, in=270] (1,0);
	\draw[arc] (0,0) to [out=90, in=180] (.5,1) to [out=0, in=90] (1,0);
	\draw[arc] (0,3) to [out=270, in=180] (.5,2) to [out=0, in=270] (1,3);
	\draw[doline] (-.5,0) to (1.5,0);
	\draw[doline] (-.5,3) to (1.5,3);
	\node at (0,0) {$\Down$};
	\node at (1,0) {$\Up$};
	\node at (0,3) {$\Down$};
	\node at (1,3) {$\Up$};
\end{tikzpicture}
\mapsto
\begin{tikzpicture}[anchorbase, scale=.4, tinynodes]
	\draw[arc] (0,0) to [out=90, in=180] (.5,1) to [out=0, in=90] (1,0);
	\draw[arc] (0,0) to [out=270, in=180] (.5,-1) to [out=0, in=270] (1,0);
	\draw[doline] (-.5,0) to (1.5,0);
	\node at (0,0) {$\Down$};
	\node at (1,0) {$\Up$};
\end{tikzpicture}};
(7.25,-7)*{\text{$\begin{smallmatrix}\down\,\up\\\down\,\up\end{smallmatrix}\mapsto\down\,\up$}};
\endxy
$
}
\;
\fcolorbox{myorange}{mycream!10}{$
\xy
(0,0)*{
\begin{tikzpicture}[anchorbase, scale=.4, tinynodes]
	\draw[arc] (0,3) to [out=90, in=180] (.5,4) to [out=0, in=90] (1,3);
	\draw[arc] (0,0) to [out=270, in=180] (.5,-1) to [out=0, in=270] (1,0);
	\draw[arc] (0,0) to [out=90, in=180] (.5,1) to [out=0, in=90] (1,0);
	\draw[arc] (0,3) to [out=270, in=180] (.5,2) to [out=0, in=270] (1,3);
	\draw[doline] (-.5,0) to (1.5,0);
	\draw[doline] (-.5,3) to (1.5,3);
	\node at (0,0) {$\Down$};
	\node at (1,0) {$\Up$};
	\node at (1,3) {$\Down$};
	\node at (0,3) {$\Up$};
\end{tikzpicture}
\mapsto
\begin{tikzpicture}[anchorbase, scale=.4, tinynodes]
	\draw[arc] (0,0) to [out=90, in=180] (.5,1) to [out=0, in=90] (1,0);
	\draw[arc] (0,0) to [out=270, in=180] (.5,-1) to [out=0, in=270] (1,0);
	\draw[doline] (-.5,0) to (1.5,0);
	\node at (1,0) {$\Down$};
	\node at (0,0) {$\Up$};
\end{tikzpicture}};
(7.25,-7)*{\text{$\begin{smallmatrix}\up\,\down\\\down\,\up\end{smallmatrix}\mapsto\up\,\down$}};
\endxy
$}
\;
\fcolorbox{myorange}{mycream!10}{$
\xy
(0,0)*{
\begin{tikzpicture}[anchorbase, scale=.4, tinynodes]
	\draw[arc] (0,3) to [out=90, in=180] (.5,4) to [out=0, in=90] (1,3);
	\draw[arc] (0,0) to [out=270, in=180] (.5,-1) to [out=0, in=270] (1,0);
	\draw[arc] (0,0) to [out=90, in=180] (.5,1) to [out=0, in=90] (1,0);
	\draw[arc] (0,3) to [out=270, in=180] (.5,2) to [out=0, in=270] (1,3);
	\draw[doline] (-.5,0) to (1.5,0);
	\draw[doline] (-.5,3) to (1.5,3);
	\node at (0,3) {$\Down$};
	\node at (1,3) {$\Up$};
	\node at (1,0) {$\Down$};
	\node at (0,0) {$\Up$};
\end{tikzpicture}
\mapsto
\begin{tikzpicture}[anchorbase, scale=.4, tinynodes]
	\draw[arc] (0,0) to [out=90, in=180] (.5,1) to [out=0, in=90] (1,0);
	\draw[arc] (0,0) to [out=270, in=180] (.5,-1) to [out=0, in=270] (1,0);
	\draw[doline] (-.5,0) to (1.5,0);
	\node at (1,0) {$\Down$};
	\node at (0,0) {$\Up$};
\end{tikzpicture}};
(7.25,-7)*{\text{$\begin{smallmatrix}\down\,\up\\\up\,\down\end{smallmatrix}\mapsto\up\,\down$}};
\endxy
$}
\;
\fcolorbox{myorange}{mycream!10}{$
\xy
(0,0)*{
\begin{tikzpicture}[anchorbase, scale=.4, tinynodes]
	\draw[arc] (0,3) to [out=90, in=180] (.5,4) to [out=0, in=90] (1,3);
	\draw[arc] (0,0) to [out=270, in=180] (.5,-1) to [out=0, in=270] (1,0);
	\draw[arc] (0,0) to [out=90, in=180] (.5,1) to [out=0, in=90] (1,0);
	\draw[arc] (0,3) to [out=270, in=180] (.5,2) to [out=0, in=270] (1,3);
	\draw[doline] (-.5,0) to (1.5,0);
	\draw[doline] (-.5,3) to (1.5,3);
	\node at (1,0) {$\Down$};
	\node at (0,0) {$\Up$};
	\node at (1,3) {$\Down$};
	\node at (0,3) {$\Up$};
\end{tikzpicture}
\mapsto
0};
(7.75,-7)*{\text{$\begin{smallmatrix}\up\,\down\\\up\,\down\end{smallmatrix}\mapsto 0$}};
\endxy
$}
\\
\fcolorbox{myorange}{mycream!10}{$
\xy
(0,0)*{
\begin{tikzpicture}[anchorbase, scale=.4, tinynodes]
	\draw[arc] (0,3) to [out=90, in=0] (-.5,4);
	\draw[arc] (1,3) to [out=90, in=180] (1.5,4);
	\draw[arc] (0,0) to [out=90, in=0] (-.5,1);
	\draw[arc] (1,0) to [out=90, in=180] (1.5,1);
	\draw[arc] (0,3) to [out=270, in=0] (-.5,2);
	\draw[arc] (1,3) to [out=270, in=180] (1.5,2);
	\draw[arc] (0,0) to [out=270, in=0] (-.5,-1);
	\draw[arc] (1,0) to [out=270, in=180] (1.5,-1);
	\draw[doline] (-.5,0) to (1.5,0);
	\draw[doline] (-.5,3) to (1.5,3);
	\node at (1,0) {$\Down$};
	\node at (0,0) {$\Up$};
	\node at (1,3) {$\Down$};
	\node at (0,3) {$\Up$};
\end{tikzpicture}
\mapsto
\begin{tikzpicture}[anchorbase, scale=.4, tinynodes]
	\draw[arc] (0,0) to [out=90, in=0] (-.5,1);
	\draw[arc] (1,0) to [out=90, in=180] (1.5,1);
	\draw[arc] (0,0) to [out=270, in=0] (-.5,-1);
	\draw[arc] (1,0) to [out=270, in=180] (1.5,-1);
	\draw[doline] (-.5,0) to (1.5,0);
	\node at (1,0) {$\Down$};
	\node at (0,0) {$\Up$};
\end{tikzpicture}};
(7.25,-7)*{\text{$\begin{smallmatrix}\up\,\down\\\up\,\down\end{smallmatrix}\mapsto\up\,\down$}};
(15.8,-7)*{\phantom{a}};
\endxy
$}
\;
\fcolorbox{myorange}{mycream!10}{$
\xy
(0,0)*{
\begin{tikzpicture}[anchorbase, scale=.4, tinynodes]
	\draw[arc] (0,3) to [out=90, in=0] (-.5,4);
	\draw[arc] (1,3) to [out=90, in=180] (1.5,4);
	\draw[arc] (0,0) to [out=90, in=0] (-.5,1);
	\draw[arc] (1,0) to [out=90, in=180] (1.5,1);
	\draw[arc] (0,3) to [out=270, in=0] (-.5,2);
	\draw[arc] (1,3) to [out=270, in=180] (1.5,2);
	\draw[arc] (0,0) to [out=270, in=0] (-.5,-1);
	\draw[arc] (1,0) to [out=270, in=180] (1.5,-1);
	\draw[doline] (-.5,0) to (1.5,0);
	\draw[doline] (-.5,3) to (1.5,3);
	\node at (1,0) {$\Down$};
	\node at (0,0) {$\Up$};
	\node at (0,3) {$\Down$};
	\node at (1,3) {$\Up$};
\end{tikzpicture}
\mapsto
\begin{tikzpicture}[anchorbase, scale=.4, tinynodes]
	\draw[arc] (0,0) to [out=90, in=0] (-.5,1);
	\draw[arc] (1,0) to [out=90, in=180] (1.5,1);
	\draw[arc] (0,0) to [out=270, in=0] (-.5,-1);
	\draw[arc] (1,0) to [out=270, in=180] (1.5,-1);
	\draw[doline] (-.5,0) to (1.5,0);
	\node at (0,0) {$\Down$};
	\node at (1,0) {$\Up$};
\end{tikzpicture}};
(7.25,-7)*{\text{$\begin{smallmatrix}\down\,\up\\\up\,\down\end{smallmatrix}\mapsto\down\,\up$}};
\endxy
$}
\;
\fcolorbox{myorange}{mycream!10}{$
\xy
(0,0)*{
\begin{tikzpicture}[anchorbase, scale=.4, tinynodes]
	\draw[arc] (0,3) to [out=90, in=0] (-.5,4);
	\draw[arc] (1,3) to [out=90, in=180] (1.5,4);
	\draw[arc] (0,0) to [out=90, in=0] (-.5,1);
	\draw[arc] (1,0) to [out=90, in=180] (1.5,1);
	\draw[arc] (0,3) to [out=270, in=0] (-.5,2);
	\draw[arc] (1,3) to [out=270, in=180] (1.5,2);
	\draw[arc] (0,0) to [out=270, in=0] (-.5,-1);
	\draw[arc] (1,0) to [out=270, in=180] (1.5,-1);
	\draw[doline] (-.5,0) to (1.5,0);
	\draw[doline] (-.5,3) to (1.5,3);
	\node at (1,3) {$\Down$};
	\node at (0,3) {$\Up$};
	\node at (0,0) {$\Down$};
	\node at (1,0) {$\Up$};
\end{tikzpicture}
\mapsto
\begin{tikzpicture}[anchorbase, scale=.4, tinynodes]
	\draw[arc] (0,0) to [out=90, in=0] (-.5,1);
	\draw[arc] (1,0) to [out=90, in=180] (1.5,1);
	\draw[arc] (0,0) to [out=270, in=0] (-.5,-1);
	\draw[arc] (1,0) to [out=270, in=180] (1.5,-1);
	\draw[doline] (-.5,0) to (1.5,0);
	\node at (0,0) {$\Down$};
	\node at (1,0) {$\Up$};
\end{tikzpicture}};
(7.25,-7)*{\text{$\begin{smallmatrix}\up\,\down\\\down\,\up\end{smallmatrix}\mapsto\down\,\up$}};
\endxy
$}
\;
\fcolorbox{myorange}{mycream!10}{$
\xy
(0,0)*{
\begin{tikzpicture}[anchorbase, scale=.4, tinynodes]
	\draw[arc] (0,3) to [out=90, in=0] (-.5,4);
	\draw[arc] (1,3) to [out=90, in=180] (1.5,4);
	\draw[arc] (0,0) to [out=90, in=0] (-.5,1);
	\draw[arc] (1,0) to [out=90, in=180] (1.5,1);
	\draw[arc] (0,3) to [out=270, in=0] (-.5,2);
	\draw[arc] (1,3) to [out=270, in=180] (1.5,2);
	\draw[arc] (0,0) to [out=270, in=0] (-.5,-1);
	\draw[arc] (1,0) to [out=270, in=180] (1.5,-1);
	\draw[doline] (-.5,0) to (1.5,0);
	\draw[doline] (-.5,3) to (1.5,3);
	\node at (0,0) {$\Down$};
	\node at (1,0) {$\Up$};
	\node at (0,3) {$\Down$};
	\node at (1,3) {$\Up$};
\end{tikzpicture}
\mapsto 0};
(7.75,-7)*{\text{$\begin{smallmatrix}\down\,\up\\\down\,\up\end{smallmatrix}\mapsto 0$}};
\endxy
$}
\end{gathered}
\end{gather}
Note the changed roles of the weights.
\end{example}

\begin{example}\label{example:arc-mult-2}
The list of the idempotents 
from \eqref{eq:the-idem-set} in case $\numberarcs=2$ is
\begin{gather}
\begin{gathered}
\phantom{.}
\begin{tikzpicture}[anchorbase, scale=.4, tinynodes]
	\draw[arc] (1,0) to [out=90, in=180] (1.5,1) to [out=0, in=90] (2,0);
	\draw[arc] (0,0) to [out=90, in=180] (1.5,2) to [out=0, in=90] (3,0);
	\draw[arc] (1,0) to [out=270, in=180] (1.5,-1) to [out=0, in=270] (2,0);
	\draw[arc] (0,0) to [out=270, in=180] (1.5,-2) to [out=0, in=270] (3,0);
	\draw[doline] (-.5,0) to (3.5,0);
	\node at (0,0) {$\Down$};
	\node at (1,0) {$\Down$};
	\node at (2,0) {$\Up$};
	\node at (3,0) {$\Up$};
	\node at (1.5,2.5) {$e_{1}$};
	\node at (1.5,-2.5) {$\phantom{e_{1}}$};
\end{tikzpicture}
,\quad
\begin{tikzpicture}[anchorbase, scale=.4, tinynodes]
	\draw[arc] (0,0) to [out=90, in=180] (.5,1) to [out=0, in=90] (1,0);
	\draw[arc] (2,0) to [out=90, in=0] (-.5,2);
	\draw[arc] (3,0) to [out=90, in=210] (3.5,2);
	\draw[arc] (0,0) to [out=270, in=180] (.5,-1) to [out=0, in=270] (1,0);
	\draw[arc] (2,0) to [out=270, in=0] (-.5,-2);
	\draw[arc] (3,0) to [out=270, in=150] (3.5,-2);
	\draw[doline] (-.5,0) to (3.5,0);
	\node at (0,0) {$\Down$};
	\node at (3,0) {$\Down$};
	\node at (2,0) {$\Up$};
	\node at (1,0) {$\Up$};
	\node at (1.5,2.5) {$e_{2}$};
	\node at (1.5,-2.5) {$\phantom{e_{2}}$};
\end{tikzpicture}
,\quad
\begin{tikzpicture}[anchorbase, scale=.4, tinynodes]
	\draw[arc] (2,0) to [out=90, in=180] (2.5,1) to [out=0, in=90] (3,0);
	\draw[arc] (0,0) to [out=90, in=330] (-.5,2);
	\draw[arc] (1,0) to [out=90, in=180] (3.5,2);
	\draw[arc] (2,0) to [out=270, in=180] (2.5,-1) to [out=0, in=270] (3,0);
	\draw[arc] (0,0) to [out=270, in=30] (-.5,-2);
	\draw[arc] (1,0) to [out=270, in=180] (3.5,-2);
	\draw[doline] (-.5,0) to (3.5,0);
	\node at (2,0) {$\Down$};
	\node at (1,0) {$\Down$};
	\node at (0,0) {$\Up$};
	\node at (3,0) {$\Up$};
	\node at (1.5,2.5) {$e_{3}$};
	\node at (1.5,-2.5) {$\phantom{e_{3}}$};
\end{tikzpicture}
,\quad
\begin{tikzpicture}[anchorbase, scale=.4, tinynodes]
	\draw[arc] (1,0) to [out=90, in=0] (-.5,2);
	\draw[arc] (2,0) to [out=90, in=180] (3.5,2);
	\draw[arc] (0,0) to [out=90, in=0] (-.5,1);
	\draw[arc] (3,0) to [out=90, in=180] (3.5,1);
	\draw[arc] (1,0) to [out=270, in=0] (-.5,-2);
	\draw[arc] (2,0) to [out=270, in=180] (3.5,-2);
	\draw[arc] (0,0) to [out=270, in=0] (-.5,-1);
	\draw[arc] (3,0) to [out=270, in=180] (3.5,-1);
	\draw[doline] (-.5,0) to (3.5,0);
	\node at (2,0) {$\Down$};
	\node at (3,0) {$\Down$};
	\node at (0,0) {$\Up$};
	\node at (1,0) {$\Up$};
	\node at (1.5,2.5) {$e_{4}$};
	\node at (1.5,-2.5) {$\phantom{e_{4}}$};
\end{tikzpicture}
,\quad
\begin{tikzpicture}[anchorbase, scale=.4, tinynodes]
	\draw[arc] (0,0) to [out=90, in=180] (.5,1) to [out=0, in=90] (1,0);
	\draw[arc] (2,0) to [out=90, in=180] (2.5,1) to [out=0, in=90] (3,0);
	\draw[arc] (0,0) to [out=270, in=180] (.5,-1) to [out=0, in=270] (1,0);
	\draw[arc] (2,0) to [out=270, in=180] (2.5,-1) to [out=0, in=270] (3,0);
	\draw[doline] (-.5,0) to (3.5,0);
	\node at (0,0) {$\Down$};
	\node at (1,0) {$\Up$};
	\node at (2,0) {$\Down$};
	\node at (3,0) {$\Up$};
	\node at (1.5,2.5) {$e_{5}$};
	\node at (1.5,-2.5) {$\phantom{e_{5}}$};
\end{tikzpicture}
,\quad
\begin{tikzpicture}[anchorbase, scale=.4, tinynodes]
	\draw[arc] (1,0) to [out=90, in=180] (1.5,1) to [out=0, in=90] (2,0);
	\draw[arc] (0,0) to [out=90, in=0] (-.5,1);
	\draw[arc] (3,0) to [out=90, in=180] (3.5,1);
	\draw[arc] (1,0) to [out=270, in=180] (1.5,-1) to [out=0, in=270] (2,0);
	\draw[arc] (0,0) to [out=270, in=0] (-.5,-1);
	\draw[arc] (3,0) to [out=270, in=180] (3.5,-1);
	\draw[doline] (-.5,0) to (3.5,0);
	\node at (1,0) {$\Down$};
	\node at (0,0) {$\Up$};
	\node at (3,0) {$\Down$};
	\node at (2,0) {$\Up$};
	\node at (1.5,2.5) {$e_{6}$};
	\node at (1.5,-2.5) {$\phantom{e_{6}}$};
\end{tikzpicture}
\end{gathered}
\end{gather}
(For later use, cf. \fullref{example:big-arc-example-2}, we denote them by $e_{i}$ 
for $i=1,\dots,6$.)
\end{example}

\begin{furtherdirections}\label{arc-connections-2}
Our conventions here differ slightly from the ones in \cite[Section 5.3]{an1} and it would be interesting to find an explicit isomorphism between the two algebras.
\end{furtherdirections}

%%%%%%%%%%%%%%%%%%%%%%%%%%%
\subsection{Relative cellularity: The cell datum}\label{subsection:arc-cell-basis}
%%%%%%%%%%%%%%%%%%%%%%%%%%%

Let us now give the relative cell datum.

First, as already indicated by our notation in 
\fullref{subsection:arc-combinatorics}, the set $\cset$ is the set of 
weights, while the sets $\cmset(\lambda)$ 
are those cup diagrams $S$ such that $S\lambda$ is oriented. 
The map $\cbasis$ is then given by the defined basis elements 
$\cbas{S,T}{\lambda}$. 
The anti-involution ${}^{\invo}$ is given by reflection.

Furthermore
-- by \fullref{proposition:asso} -- 
we let 
$\ciset$ be as in \eqref{eq:the-idem-set},
and we can associate
to a cup diagram 
$S$ the idempotent $\ceps_{S}=\cbas{S,S}{\lambda}\in\ciset$. This in turn 
defines the map $\cepsmap(S)=\ceps_{S}$.

To define the partial orders $\ord{\ceps_{S}}$ with respect 
to the idempotents in $\ciset$, note that 
there is a rotation map $\rotright\colon\cset\to\cset$ 
given by rotating rightwards. 
This is formally done by renumbering the vertices on the dotted line 
to $2,3,\cdots,2\numberarcs,1$. The same is done for cup diagrams, e.g.:
\begin{gather}\label{eq:ord-example}
\xy
(0,4)*{
\begin{tikzpicture}[anchorbase, scale=.4, tinynodes]
	\draw[doline] (-2.5,0) to (1.5,0);
	\node at (-2,0) {$\Up$};
	\node at (-1,0) {$\Down$};
	\node at (0,0) {$\Up$};
	\node at (1,0) {$\Down$};
\end{tikzpicture}
\xrightarrow{\rotright}
\begin{tikzpicture}[anchorbase, scale=.4, tinynodes]
	\draw[doline] (-2.5,0) to (1.5,0);
	\node at (-2,0) {$\Down$};
	\node at (-1,0) {$\Up$};
	\node at (0,0) {$\Down$};
	\node at (1,0) {$\Up$};
\end{tikzpicture}};
(0,-4)*{
\begin{tikzpicture}[anchorbase, scale=.4, tinynodes]
	\draw[doline] (-2.5,0) to (1.5,0);
	\node at (-2,0) {$\Up$};
	\node at (-1,0) {$\Down$};
	\node at (0,0) {$\Down$};
	\node at (1,0) {$\Up$};
\end{tikzpicture}
\xrightarrow{\rotright}
\begin{tikzpicture}[anchorbase, scale=.4, tinynodes]
	\draw[doline] (-2.5,0) to (1.5,0);
	\node at (-2,0) {$\Up$};
	\node at (-1,0) {$\Up$};
	\node at (0,0) {$\Down$};
	\node at (1,0) {$\Down$};
\end{tikzpicture}
};
\endxy
,
\quad
\begin{tikzpicture}[anchorbase, scale=.4, tinynodes]
	\draw[arc] (0,0) to [out=270, in=0] (-2.5,-2);
	\draw[arc] (1,0) to [out=270, in=150] (1.5,-2);
	\draw[arc] (-2,0) to [out=270, in=180] (-1.5,-1) to [out=0, in=270] (-1,0);
	\draw[doline] (-2.5,0) to (1.5,0);
	\node at (-.5,-2.6) {$S$};
	\node at (-.5,.825) {$\phantom{S}$};
\end{tikzpicture}
\xrightarrow{\rotright}
\begin{tikzpicture}[anchorbase, scale=.4, tinynodes]
	\draw[arc] (-2,0) to [out=270, in=180] (-.5,-2) to [out=0, in=270] (1,0);
	\draw[arc] (-1,0) to [out=270, in=180] (-.5,-1) to [out=0, in=270] (0,0);
	\draw[doline] (-2.5,0) to (1.5,0);
	\node at (-.5,-2.6) {$\rotright(S)$};
	\node at (-.5,.825) {$\phantom{S}$};
\end{tikzpicture}
\end{gather}

We note two lemmas whose 
(very easy) proofs we omit.

\begin{lemmaqed}\label{lemma:rotation-map}
The map $\rotright$ defined on the 
basis as $\rotright(\cbas{S,T}{\lambda})=
\cbas{\rotright(S),\rotright(T)}{\rotright(\lambda)}$ defines an 
algebra automorphism of $\aarc$.
\end{lemmaqed}

\begin{lemmaqed}\label{lemma:partial-orders}
For each cup diagram $S$ there is 
$k\in\N$ such that the cup diagram 
$\rotright^{k}(S)$ is of staying type.
\end{lemmaqed}

The set $\cset$ has a partial order $\prec_{\uarc}$ generated 
by saying that an ordered pair $\down\,\up$ swapped 
to $\up\,\down$ creates a smaller element of $\cset$. 
(This is actually the partial order for $\uarc$, cf. \cite[Section 2]{bs1}.)
Starting from this partial order we will define -- 
by using \fullref{lemma:partial-orders} -- our 
partial orders using the rotation $\rotright$.

\begin{definition}\label{definition:partial-orders}
Let $S$ be a cup diagram and $\lambda,\mu\in\cset$. 
Let $k\in\N$ be minimal such 
that $\rotright^{k}(S)$ is of 
staying type. Then we define $\mu\ord{\ceps_{S}}\lambda$ 
if $\rotright^{k}(\mu)\prec_{\uarc}\rotright^{k}(\lambda)$.
\end{definition}

For example, $\up\,\down\,\down\,\up\ord{\ceps_{S}}\up\,\down\,\up\,\down$, 
but $\up\,\down\,\up\,\down\ord{\ceps_{\rotright(S)}}\up\,\down\,\down\,\up$
for $S$ as in \eqref{eq:ord-example}.

Now -- by \fullref{definition:partial-orders} -- 
we set $\coset=\{\ord{\ceps_{S}}\mid S\text{ is a cup diagram}\}$, and have 
\begin{gather}\label{eq:annular-datum}
(\cset,\cmset,\cbasis,
{}^{\invo},\ciset,\coset,\cepsmap)
\end{gather}
as the candidate for the relative cell datum.

The main ingredient to prove relative cellularity is the following 
that is similar to 
\cite[Theorem 3.1]{bs1}, but more involved to prove. Its proof appears in 
\fullref{subsection:the-proofs} below.

\begin{theorem}\label{theorem:mult-rule-arc}
Let $\lambda,\mu\in\cset$, $S,T\in\cmset(\lambda)$ 
and $U,V\in\cmset(\mu)$. Then
\begin{gather}
\cbas{S,T}{\lambda}\cbas{U,V}{\mu}= 
\begin{cases}
0, 
& \text{ if } T \neq U, 
\\
r(\cbas{S,T}{\lambda},U)\cbas{U,V}{\mu} + \friends, 
& \text{ if } T=U \text{ and } V \in \cmset(\mu),
\\
\friends, 
& \text{ otherwise,}
\end{cases}
\end{gather}
with $r(\cbas{S,T}{\lambda},U)\in\{0,1\}\subset\K$, 
$\friends\in \aarc(\ord{\ceps_{V}}\!\mu)$ 
and $\ceps_{S}\friends=\friends=\friends\ceps_{V}$.
\end{theorem}

This in turn implies the relative cellularity of the annular arc algebra.

\begin{theorem}\label{theorem:main-arc-algebra}
The algebra $\aarc$ is relative 
cellular with cell datum as in \eqref{eq:annular-datum}.
\end{theorem}

\begin{proof}
\textit{(\ref{definition:cell-algebra}.a).}
The sets $\cset$ and $\cmset(\lambda)$ are 
clearly finite, and
the assignment $\cbasis$ gives -- by definition -- an 
injective map with image forming a basis of $\aarc$.
\smallskip

\noindent\textit{(\ref{definition:cell-algebra}.b).}
Clearly, ${}^{\invo}$ is an anti-involution with 
$(\cbas{S,T}{\lambda})^{\invo}=\cbas{T,S}{\lambda}$.
\smallskip

\noindent\textit{(\ref{definition:cell-algebra}.c).}
All statements about the idempotents 
and the mapping $\cepsmap$ are -- by e.g.
\fullref{proposition:asso} -- immediate 
except \eqref{eq:idem-props-1}. 
For \eqref{eq:idem-props-1} 
we note that $\ceps\aarc\ceps\cbas{S,T}{\lambda}$ is 
zero unless $\ceps=\ceps_{S}$. In this case $\ceps\aarc\ceps$ is 
spanned by elements of the form $\cbas{S,S}{\mu}$ for $\mu\in\cset$. 
The multiplication $\cbas{S,S}{\mu}\cbas{S,T}{\lambda}$ 
will be a merge in each step and the only non-trivial operation 
is that some circles in $ST^{\invo}$ are reoriented 
from anticlockwise to clockwise. 
However -- by \fullref{lemma:reorient-larger} below -- 
this will decrease 
the weight with respect to both, 
$\ord{\ceps_{S}}$ and $\ord{\ceps_{T}}$. 
\smallskip

\noindent\textit{(\ref{definition:cell-algebra}.d).}
We note that 
\fullref{theorem:mult-rule-arc} is a stronger version 
of (\ref{definition:cell-algebra}.d).
\end{proof}

%%%%%%%%%%%%%%%%%%%%%%%%%%%
\subsection{Further properties}\label{subsection:general-rank-statements}
%%%%%%%%%%%%%%%%%%%%%%%%%%%

By \fullref{theorem:main-arc-algebra} we can use the 
notions from \fullref{section:basic-cell-props} regarding 
simples, cell and indecomposable projective $\aarc$-modules.

\begin{proposition}\label{proposition:some-multiplicities}
Let $\lambda,\mu\in\cset$ and $S\in\cmset(\lambda)$, $T\in\cmset(\mu)$ 
such that $\ceps_{S}=\cbas{S,S}{\lambda}$ 
and $\ceps_{T}=\cbas{T,T}{\mu}$. Then the following hold.
\smallskip
\begin{enumerate}[label=(\alph*)]

\setlength\itemsep{.15cm}

\item We have $[\dmod(\lambda):\lmod(\mu)] = 1$ if and 
only if $T\lambda$ is oriented, otherwise it is zero.

\item The projective $\prmod(\lambda)$ 
has a filtration by cell modules 
of the form $\dmod(\nu)$ such that $S\nu$ is oriented. 
Further, 
it has a filtration by $2^{\numberarcs}$ cell modules, 
each occurring once.

\item The value 
$[\prmod(\lambda):\lmod(\mu)]$ can be computed by 
counting the number of 
orientations of $ST^{\invo}$, with each 
orientation $\nu$ giving the occurrence 
of $\lmod(\mu)$  in $\dmod(\nu)$ inside the cell 
module filtration given by (\ref{proposition:some-multiplicities}.b).\qedhere
\end{enumerate}

\end{proposition}

\begin{proof}
\textit{(\ref{proposition:some-multiplicities}.a).}
This follows immediately by noting that the basis 
elements of $\dmod(\lambda)$ are compatible with the 
choice of primitive idempotents, with exactly one of 
the idempotents acting as $1$ on a given 
basis element and all others acting by $0$.
\smallskip

\noindent\textit{(\ref{proposition:some-multiplicities}.b).} 
The first statement follows by construction of the cell filtration 
in \fullref{proposition:cell-filtration}.
The second statement follows 
since the number of orientations for $S$ is exactly $\numberarcs$.
\smallskip

\noindent\textit{(\ref{proposition:some-multiplicities}.c).} By combining (\ref{proposition:some-multiplicities}.a) and (\ref{proposition:some-multiplicities}.b).
\end{proof}

\begin{remark}\label{remark:annular-modules}
Note that -- by the proof of (\ref{proposition:some-multiplicities}.a) -- it also 
follows that simple modules always have dimension one. 
On the other hand, a cell module $\dmod(\lambda)$ has dimension 
equal to the number of cup diagrams $S$ such that $S\lambda$ is 
oriented. Thus, $\mathrm{dim}\,(\dmod(\lambda))>1$. 
Furthermore, (\ref{proposition:some-multiplicities}.b) implies that $\prmod(\lambda)$ is also 
always different from $\dmod(\lambda)$.
To summarize: No cell module is simple or projective.
\end{remark}

\begin{proposition}\label{proposition:annular-Frobenius}
The algebra $\aarc$ is a non-semisimple Frobenius algebra 
of infinite global dimension.
\end{proposition}

\begin{proof}
A bilinear form $\sigma\colon\aarc\otimes\aarc\to\K$ is given by 
$\sigma(\cbas{S,T}{\lambda},\cbas{U,V}{\mu}) = 0$ for $S\neq V$, and 
otherwise $\sigma(\cbas{S,T}{\lambda},\cbas{U,V}{\mu})$ 
is set to be the coefficient of $\cbas{S,S}{\nu}$ in 
the product $\cbas{S,T}{\lambda}\cbas{U,V}{\mu}$, 
where $\nu$ is chosen such that all circles in $S\nu S^{\invo}$ are oriented clockwise. 
Associativity and non-degeneracy can be shown using the same TQFT methods as 
in \cite{est1}, using the TQFT as in the proof of \fullref{proposition:asso}, i.e. 
both are immediate for the topological incarnation of $\aarc$ due to the TQFT involved in 
the construction. (Associativity being again an isotopy; non-degeneracy 
follows from the non-degeneracy of the involved TQFT.)

From the dimension observations in \fullref{remark:annular-modules} it follows 
that $\aarc$ is non-semisimple. Further, recall that a Frobenius algebra 
has finite global 
dimension if and only if it is 
semisimple. Thus,  $\aarc$ is of 
infinite global dimension.
\end{proof}

\begin{remark}\label{remark:Frobenius-algebraic}
The Frobenius property in \fullref{proposition:annular-Frobenius} 
can be proven directly using combinatorics. While associativity 
of $\sigma$ follows immediately, the non-degeneracy can be checked by 
carefully looking at products of the form $\cbas{S,T}{\lambda}\cbas{T,S}{\mu}$ 
and noting that the surgeries can be ordered so that merges are performed 
first followed by splits. Thus, for a given weight $\lambda$, the $\mu$ can be 
chosen appropriately so that all circles, after performing the merges, 
are usual and clockwise and then the splits will all create usual and clockwise 
circles, giving the non-degeneracy of $\sigma$.
\end{remark}

\begin{proposition}\label{proposition:not-cell}
The matrix $\cmatrix(\aarc)$ is positive 
semidefinite with determinant zero.
\end{proposition}

\begin{proof}
By \fullref{corollary:semi-definite} it remains to check 
that the Cartan matrix is not of full rank.

The case $\numberarcs=1$ is done 
explicitly in \fullref{example:big-arc-example-1} below.

For the case $\numberarcs>1$, let $S$ be the 
cup diagram having only arcs of staying type with one 
arc connecting vertices $1$ and $2\numberarcs$ and 
arcs connecting $2i$ and $2i+1$ for $1 \leq i \leq (\numberarcs{-}1)$. 
Let $\ceps_{S}= S\lambda S^{\invo}$ 
for $\lambda=\down\,(\down\,\up)(\down\,\up)\cdots(\down\,\up)\up$. 
The multiplicity $[\prmod(\lambda):\lmod(\mu)]$ 
is -- by (\ref{proposition:some-multiplicities}.c) --
obtained by
counting the number of possible orientations 
of the diagram $ST^{\invo}$, where $T$ is 
the unique diagram such that $T\mu T^{\invo}$ is a primitive idempotent. 
Note that this is $2^{m}$ for $m$ being the number of 
circles in $ST^{\invo}$. Next, the number of such 
orientations is the same as the number of orientations 
of $\rotright^{2}(S)T^{\invo}$. 
This holds true since $\rotright^{2}(S)$ 
connects the same vertices as $S$, just with arcs that 
are not of staying type. Thus, 
$[\prmod(\lambda):\lmod(\mu)]=[\prmod(\rotright^{2}(\lambda)) :\lmod(\mu)]$ 
and with the assumption $\numberarcs>1$ we obtain $\rotright^{2}(\lambda)\neq\lambda$. 
In total, the matrix $\cmatrix(\aarc)$ has two equal columns.
\end{proof}

%%%%%%%%%%%%%%%%%%%%%%%%%%%
\subsection{Low rank examples}\label{subsection:arc-cell}
%%%%%%%%%%%%%%%%%%%%%%%%%%%

Let us discuss the cases $\numberarcs=1$ and $\numberarcs=2$ in detail. 
This will be very much as in 
\fullref{subsection:cell-examples}, whose notions 
we recommend to recall.

\begin{example}\label{example:big-arc-example-1}
Let $\numberarcs=1$. Then the relative cell datum 
of $\aarc[1]$ is as follows.
\begin{gather}
\begin{gathered}
\cset
=\{\fcolorbox{myred}{mycream!10}{$\up\,\down$}
\ord{e_{1}}
\fcolorbox{mygreen}{mycream!10}{$\down\,\up$}\}
=
\{\fcolorbox{mygreen}{mycream!10}{$\down\,\up$}
\ord{e_{2}}
\fcolorbox{myred}{mycream!10}{$\up\,\down$}\},
\quad
{}^{\invo}\rightsquigarrow\text{reflect diagrams},
\\
\cmset(\down\,\up)=
\left\{
\begin{tikzpicture}[anchorbase, scale=.4, tinynodes]
	\draw[arc] (0,0) to [out=270, in=180] (.5,-1) to [out=0, in=270] (1,0);
	\draw[doline] (-.5,0) to (1.5,0);
	\node at (0,0) {$\Down$};
	\node at (1,0) {$\Up$};
\end{tikzpicture}
,
\begin{tikzpicture}[anchorbase, scale=.4, tinynodes]
	\draw[arc] (0,0) to [out=270, in=0] (-.5,-1);
	\draw[arc] (1,0) to [out=270, in=180] (1.5,-1);
	\draw[doline] (-.5,0) to (1.5,0);
	\node at (0,0) {$\Down$};
	\node at (1,0) {$\Up$};
\end{tikzpicture}
\right\},
\quad
\cmset(\up\,\down)=
\left\{
\begin{tikzpicture}[anchorbase, scale=.4, tinynodes]
	\draw[arc] (0,0) to [out=270, in=0] (-.5,-1);
	\draw[arc] (1,0) to [out=270, in=180] (1.5,-1);
	\draw[doline] (-.5,0) to (1.5,0);
	\node at (1,0) {$\Down$};
	\node at (0,0) {$\Up$};
\end{tikzpicture}
,
\begin{tikzpicture}[anchorbase, scale=.4, tinynodes]
	\draw[arc] (0,0) to [out=270, in=180] (.5,-1) to [out=0, in=270] (1,0);
	\draw[doline] (-.5,0) to (1.5,0);
	\node at (1,0) {$\Down$};
	\node at (0,0) {$\Up$};
\end{tikzpicture}
\right\},
\quad
\cbas{S,T}{\lambda}\rightsquigarrow\text{cf. \eqref{eq:arc-short}},
\\
\ciset=
\left\{
e_{1}=
\begin{tikzpicture}[anchorbase, scale=.4, tinynodes]
	\draw[arc] (0,0) to [out=90, in=180] (.5,1) to [out=0, in=90] (1,0);
	\draw[arc] (0,0) to [out=270, in=180] (.5,-1) to [out=0, in=270] (1,0);
	\draw[doline] (-.5,0) to (1.5,0);
	\node at (0,0) {$\Down$};
	\node at (1,0) {$\Up$};
\end{tikzpicture}
,
e_{2}=
\begin{tikzpicture}[anchorbase, scale=.4, tinynodes]
	\draw[arc] (0,0) to [out=90, in=0] (-.5,1);
	\draw[arc] (1,0) to [out=90, in=180] (1.5,1);
	\draw[arc] (0,0) to [out=270, in=0] (-.5,-1);
	\draw[arc] (1,0) to [out=270, in=180] (1.5,-1);
	\draw[doline] (-.5,0) to (1.5,0);
	\node at (1,0) {$\Down$};
	\node at (0,0) {$\Up$};
\end{tikzpicture}
\right\},
\quad
\cepsmap(\cmset(\down\,\up))=
e_{1}
,
\cepsmap(\cmset(\up\,\down))=
e_{2}.
\end{gathered}
\end{gather}
Now, as for the usual arc algebra, the 
indecomposable projectives $\prmod(\down\,\up)=\aarc[1]e_1$ 
and $\prmod(\up\,\down)=\aarc[1]e_2$ are given by fixing 
(in our notation) the top shape. 
In contrast, the cell modules $\dmod(\down\,\up)$ and 
$\dmod(\up\,\down)$ are given by fixing 
the weight, and we get
{
\abovedisplayskip0.2em
\belowdisplayskip0.2em
\begin{gather}
\begin{tikzpicture}[baseline=(current bounding box.center)]
  \matrix (m) [matrix of math nodes, nodes in empty cells, row 1/.style={nodes={minimum height=9mm}}, row 3/.style={nodes={minimum height=12mm}}, row sep={.1cm}, column sep={.1cm}, text height=1.5ex, text depth=0.25ex, ampersand replacement=\&] {
  \& \down\,\up \& \& \& \& \up\,\down \&  
\\
\begin{tikzpicture}[anchorbase, scale=.4, tinynodes]
	\draw[arc] (0,0) to [out=270, in=180] (.5,-1) to [out=0, in=270] (1,0);
	\draw[arc] (0,0) to [out=90, in=180] (.5,1) to [out=0, in=90] (1,0);
	\draw[doline] (-.5,0) to (1.5,0);
	\node at (0,0) {$\Down$};
	\node at (1,0) {$\Up$};
\end{tikzpicture}
\&
\&
\begin{tikzpicture}[anchorbase, scale=.4, tinynodes]
	\draw[arc] (0,0) to [out=270, in=0] (-.5,-1);
	\draw[arc] (1,0) to [out=270, in=180] (1.5,-1);
	\draw[arc] (0,0) to [out=90, in=180] (.5,1) to [out=0, in=90] (1,0);
	\draw[doline] (-.5,0) to (1.5,0);
	\node at (0,0) {$\Down$};
	\node at (1,0) {$\Up$};
\end{tikzpicture}
\&
\&
\begin{tikzpicture}[anchorbase, scale=.4, tinynodes]
	\draw[arc] (0,0) to [out=270, in=0] (-.5,-1);
	\draw[arc] (1,0) to [out=270, in=180] (1.5,-1);
	\draw[arc] (0,0) to [out=90, in=0] (-.5,1);
	\draw[arc] (1,0) to [out=90, in=180] (1.5,1);
	\draw[doline] (-.5,0) to (1.5,0);
	\node at (1,0) {$\Down$};
	\node at (0,0) {$\Up$};
\end{tikzpicture}
\&
\&
\begin{tikzpicture}[anchorbase, scale=.4, tinynodes]
	\draw[arc] (0,0) to [out=270, in=180] (.5,-1) to [out=0, in=270] (1,0);
	\draw[arc] (0,0) to [out=90, in=0] (-.5,1);
	\draw[arc] (1,0) to [out=90, in=180] (1.5,1);
	\draw[doline] (-.5,0) to (1.5,0);
	\node at (1,0) {$\Down$};
	\node at (0,0) {$\Up$};
\end{tikzpicture}
\\
\begin{tikzpicture}[anchorbase, scale=.4, tinynodes]
	\draw[arc] (0,0) to [out=270, in=0] (-.5,-1);
	\draw[arc] (1,0) to [out=270, in=180] (1.5,-1);
	\draw[arc] (0,0) to [out=90, in=180] (.5,1) to [out=0, in=90] (1,0);
	\draw[doline] (-.5,0) to (1.5,0);
	\node at (1,0) {$\Down$};
	\node at (0,0) {$\Up$};
\end{tikzpicture}
\&
\&
\begin{tikzpicture}[anchorbase, scale=.4, tinynodes]
	\draw[arc] (0,0) to [out=270, in=180] (.5,-1) to [out=0, in=270] (1,0);
	\draw[arc] (0,0) to [out=90, in=180] (.5,1) to [out=0, in=90] (1,0);
	\draw[doline] (-.5,0) to (1.5,0);
	\node at (1,0) {$\Down$};
	\node at (0,0) {$\Up$};
\end{tikzpicture}
\&
\&
\begin{tikzpicture}[anchorbase, scale=.4, tinynodes]
	\draw[arc] (0,0) to [out=270, in=180] (.5,-1) to [out=0, in=270] (1,0);
	\draw[arc] (0,0) to [out=90, in=0] (-.5,1);
	\draw[arc] (1,0) to [out=90, in=180] (1.5,1);
	\draw[doline] (-.5,0) to (1.5,0);
	\node at (0,0) {$\Down$};
	\node at (1,0) {$\Up$};
\end{tikzpicture}
\&
\&
\begin{tikzpicture}[anchorbase, scale=.4, tinynodes]
	\draw[arc] (0,0) to [out=270, in=0] (-.5,-1);
	\draw[arc] (1,0) to [out=270, in=180] (1.5,-1);
	\draw[arc] (0,0) to [out=90, in=0] (-.5,1);
	\draw[arc] (1,0) to [out=90, in=180] (1.5,1);
	\draw[doline] (-.5,0) to (1.5,0);
	\node at (0,0) {$\Down$};
	\node at (1,0) {$\Up$};
\end{tikzpicture}
\\
};
	\draw[thin,dotted] ($(m-1-1.south)+(-.5,.2)$) to ($(m-1-1.south)+(6.5,.2)$);
	\draw[thin,dotted] ($(m-1-1.east)+(2.8,.25)$) to ($(m-1-1.east)+(2.8,-2.35)$);
	\draw[thin, mygreen] (-.25,.65) rectangle (-3.5,-.35);
	\draw[thin, myred] (-.25,-.4) rectangle (-3.5,-1.4);
	\draw[thin, myred] (.25,.65) rectangle (3.5,-.35);
	\draw[thin, mygreen] (.25,-.4) rectangle (3.5,-1.4);
	\draw[thin, white, fill=white] (-2.4,-.1) rectangle (-1.5,-.6);
	\draw[thin, white, fill=white] (2.2,-.1) rectangle (1.3,-.6);
	\node[mygreen] at (-1.85,.1) {$\dmod(\down\,\up)$};
	\node[myred] at (-1.85,-.9) {$\dmod(\up\,\down)$};
	\node[myred] at (1.85,.1) {$\dmod(\up\,\down)$};
	\node[mygreen] at (1.85,-.9) {$\dmod(\down\,\up)$};
	\draw[thick,->] (-1.7,-.1) to node [left] {{\tiny$\ord{e_{1}}$}} (-1.7,-.65) to (-1.7,-.7);
	\draw[thick,->] (2.0,-.1) to node [left] {{\tiny$\ord{e_{2}}$}} (2.0,-.65) to (2.0,-.7);
\end{tikzpicture}
\;
\begin{tikzpicture}[baseline=(current bounding box.center)]
	\node[mygreen] at (.625,1.5) {$\prmod(\down\,\up)$};
	\draw[thin, mygreen, dotted] (-.6,1.25) rectangle (1.9,-.85);
	\node[mygreen] at (0,.5) {$\lmod(\down\,\up)$};
	\node[myred] at (0,0) {$\lmod(\up\,\down)$};
	\draw[thin, mygreen] (-.5125,.75) rectangle (.5125,-.25);
	\node[mygreen] at (0,1) {$\dmod(\down\,\up)$};
	\node[mygreen] at (1.25,-.5) {$\lmod(\down\,\up)$};
	\node[myred] at (1.25,0) {$\lmod(\up\,\down)$};
	\draw[thin, myred] (.725,.25) rectangle (1.775,-.75);
	\node[myred] at (1.25,.5) {$\dmod(\up\,\down)$};
	\node at (.625,1.7) {\phantom{a}};
\end{tikzpicture}
\;
\begin{tikzpicture}[baseline=(current bounding box.center)]
	\node[myred] at (.625,1.5) {$\prmod(\up\,\down)$};
	\draw[thin, myred, dotted] (-.6,1.25) rectangle (1.9,-.85);
	\node[myred] at (0,.5) {$\lmod(\up\,\down)$};
	\node[mygreen] at (0,0) {$\lmod(\down\,\up)$};
	\draw[thin, myred] (-.5125,.75) rectangle (.5125,-.25);
	\node[myred] at (0,1) {$\dmod(\up\,\down)$};
	\node[myred] at (1.25,-.5) {$\lmod(\up\,\down)$};
	\node[mygreen] at (1.25,0) {$\lmod(\down\,\up)$};
	\draw[thin, mygreen] (.725,.25) rectangle (1.775,-.75);
	\node[mygreen] at (1.25,.5) {$\dmod(\down\,\up)$};
	\node at (.625,1.7) {\phantom{a}};
\end{tikzpicture}
\end{gather}}
Note that looking at the bottom picture 
determines the action of the primitive idempotents 
$e_1$ and $e_2$, and thus the simple module as illustrated.

Finally, the above gives us the Cartan matrix
$\cmatrix(\aarc[1])=\begin{psmallmatrix}2 & 2\\ 2 & 2\end{psmallmatrix}$,
showing that $\aarc[1]$ is not cellular.
\end{example}

\begin{example}\label{example:big-arc-example-2}
Let $\numberarcs=2$. We are now going to explain how the 
relative cellular datum of $\aarc[2]$ looks like.
The relative cellular datum will be very much in the spirit as in 
\fullref{example:big-arc-example-1}, with 
partial orderings relative to the idempotents in 
\fullref{example:arc-mult-2}. 
But since the 
algebra $\aarc[2]$ is of dimension $108$, we will 
only highlight some features by focussing on 
$e_{2}$ and $e_{5}$.

First, we have
$
\cset=\{
\fcolorbox{mypurple}{mycream!10}{$\down\,\down\,\up\,\up$},
\fcolorbox{myblue}{mycream!10}{$\down\,\up\,\down\,\up$},
\fcolorbox{myorange}{mycream!10}{$\up\,\down\,\down\,\up$},
\fcolorbox{mygreen}{mycream!10}{$\down\,\up\,\up\,\down$},
\fcolorbox{myred}{mycream!10}{$\up\,\down\,\up\,\down$},
\fcolorbox{myyellow}{mycream!10}{$\up\,\up\,\down\,\down$}
\}
$
as the set of weights.
As explained in \fullref{subsection:arc-cell-basis}, 
the partial orderings for the idempotents 
are obtained from the 
usual one (i.e. \eqref{eq:the-usual-order} 
in this case) by `rotation of the cylinder', e.g.
{
\abovedisplayskip0.2em
\belowdisplayskip0.2em
\begin{align}\label{eq:ordering-k2}
&
\begin{tikzcd}[ampersand replacement=\&,row sep=tiny,column sep=normal,arrows={shorten >=-.2ex,shorten <=-.2ex},labels={inner sep=.15ex}]
\phantom{.}
\&
\phantom{.}
\&
\fcolorbox{mypurple}{mycream!10}{$\down\,\down\,\up\,\up$}
\arrow[->]{rd}[description]{\ord{e_{2}}}
\&
\phantom{.}
\&
\phantom{.}
\\
\fcolorbox{mygreen}{mycream!10}{$\down\,\up\,\up\,\down$}
\arrow[->]{r}[description]{\ord{e_{2}}}
\&
\fcolorbox{myred}{mycream!10}{$\up\,\down\,\up\,\down$}
\arrow[->]{ru}[description]{\ord{e_{2}}}
\arrow[->]{rd}[description]{\ord{e_{2}}}
\&
\phantom{.}
\&
\fcolorbox{myblue}{mycream!10}{$\down\,\up\,\down\,\up$}
\arrow[->]{r}[description]{\ord{e_{2}}}
\&
\fcolorbox{myorange}{mycream!10}{$\up\,\down\,\down\,\up$}
\\
\phantom{.}
\&
\phantom{.}
\&
\fcolorbox{myyellow}{mycream!10}{$\up\,\up\,\down\,\down$}
\arrow[->]{ru}[description]{\ord{e_{2}}}
\&
\phantom{.}
\&
\phantom{.}
\\
\end{tikzcd}
\\
\label{eq:the-usual-order}
&
\begin{tikzcd}[ampersand replacement=\&,row sep=tiny,column sep=normal,arrows={shorten >=-.2ex,shorten <=-.2ex},labels={inner sep=.15ex}]
\phantom{.}
\&
\phantom{.}
\&
\fcolorbox{myorange}{mycream!10}{$\up\,\down\,\down\,\up$}
\arrow[->]{rd}[description]{\ord{e_{5}}}
\&
\phantom{.}
\&
\phantom{.}
\\
\fcolorbox{mypurple}{mycream!10}{$\down\,\down\,\up\,\up$}
\arrow[->]{r}[description]{\ord{e_{5}}}
\&
\fcolorbox{myblue}{mycream!10}{$\down\,\up\,\down\,\up$}
\arrow[->]{ru}[description]{\ord{e_{5}}}
\arrow[->]{rd}[description]{\ord{e_{5}}}
\&
\phantom{.}
\&
\fcolorbox{myred}{mycream!10}{$\up\,\down\,\up\,\down$}
\arrow[->]{r}[description]{\ord{e_{5}}}
\&
\fcolorbox{myyellow}{mycream!10}{$\up\,\up\,\down\,\down$}
\\
\phantom{.}
\&
\phantom{.}
\&
\fcolorbox{mygreen}{mycream!10}{$\down\,\up\,\up\,\down$}
\arrow[->]{ru}[description]{\ord{e_{5}}}
\&
\phantom{.}
\&
\phantom{.}
\\
\end{tikzcd}
\end{align}
}
The other partial orderings are similar, but rotated.

Now, the relative cell datum is
{
\abovedisplayskip0.2em
\belowdisplayskip0.2em
\begin{gather}
\begin{gathered}
\cset
\rightsquigarrow\text{cf. \eqref{eq:ordering-k2}},
\quad
{}^{\invo}
\rightsquigarrow\text{reflect diagrams},
\quad
\cmset
\rightsquigarrow\text{cf. \fullref{example:big-arc-example-2}},
\\
\cbas{S,T}{\lambda}
\rightsquigarrow\text{cf. \eqref{eq:arc-short}},
\quad
\ciset
\rightsquigarrow\text{cf. \fullref{example:arc-mult-2}}
\quad
\cepsmap
\rightsquigarrow\text{cf. \fullref{example:big-arc-example-2}}.
\end{gathered}
\end{gather}
}
Having these, the cell modules in $\prmod(\down\up\up\down)$ are
\begin{gather}
\scalebox{.9}{$
\begin{gathered}
\begin{tikzpicture}[baseline=(current bounding box.center)]
  \matrix (m) [matrix of math nodes, nodes in empty cells, row 1/.style={nodes={minimum height=18mm}}, row 3/.style={nodes={minimum height=32mm}}, row 5/.style={nodes={minimum height=32mm}}, row sep={.1cm}, column sep={.1cm}, text height=1.5ex, text depth=0.25ex, ampersand replacement=\&] {
\down\,\up\,\up\,\down
\\
\begin{tikzpicture}[anchorbase, scale=.35, tinynodes]
	\draw[arc] (0,0) to [out=270, in=180] (.5,-1) to [out=0, in=270] (1,0);
	\draw[arc] (2,0) to [out=270, in=0] (-.5,-2);
	\draw[arc] (3,0) to [out=270, in=150] (3.5,-2);
	\draw[arc] (0,0) to [out=90, in=180] (.5,1) to [out=0, in=90] (1,0);
	\draw[arc] (2,0) to [out=90, in=0] (-.5,2);
	\draw[arc] (3,0) to [out=90, in=210] (3.5,2);
	\draw[doline] (-.5,0) to (3.5,0);
	\node at (0,0) {$\Down$};
	\node at (3,0) {$\Down$};
	\node at (2,0) {$\Up$};
	\node at (1,0) {$\Up$};
\end{tikzpicture}
\,
\begin{tikzpicture}[anchorbase, scale=.35, tinynodes]
	\draw[arcd] (0,0) to [out=270, in=180] (.5,-2) to [out=0, in=270] (1,0);
	\draw[arc] (0,0) to [out=270, in=180] (.5,-1) to [out=0, in=270] (1,0);
	\draw[arc] (2,0) to [out=270, in=180] (2.5,-1) to [out=0, in=270] (3,0);
	\draw[arc] (0,0) to [out=90, in=180] (.5,1) to [out=0, in=90] (1,0);
	\draw[arc] (2,0) to [out=90, in=0] (-.5,2);
	\draw[arc] (3,0) to [out=90, in=210] (3.5,2);
	\draw[doline] (-.5,0) to (3.5,0);
	\node at (0,0) {$\Down$};
	\node at (3,0) {$\Down$};
	\node at (2,0) {$\Up$};
	\node at (1,0) {$\Up$};
\end{tikzpicture}
\,
\begin{tikzpicture}[anchorbase, scale=.35, tinynodes]
	\draw[arc] (2,0) to [out=270, in=180] (2.5,-1) to [out=0, in=270] (3,0);
	\draw[arc] (0,0) to [out=270, in=30] (-.5,-2);
	\draw[arc] (1,0) to [out=270, in=180] (3.5,-2);
	\draw[arc] (0,0) to [out=90, in=180] (.5,1) to [out=0, in=90] (1,0);
	\draw[arc] (2,0) to [out=90, in=0] (-.5,2);
	\draw[arc] (3,0) to [out=90, in=210] (3.5,2);
	\draw[doline] (-.5,0) to (3.5,0);
	\node at (0,0) {$\Down$};
	\node at (3,0) {$\Down$};
	\node at (2,0) {$\Up$};
	\node at (1,0) {$\Up$};
\end{tikzpicture}
\\
\begin{tikzpicture}[anchorbase, scale=.35, tinynodes]
	\draw[arcd] (0,0) to [out=270, in=180] (.5,-2) to [out=0, in=270] (1,0);
	\draw[arc] (1,0) to [out=270, in=180] (1.5,-1) to [out=0, in=270] (2,0);
	\draw[arc] (0,0) to [out=270, in=0] (-.5,-1);
	\draw[arc] (3,0) to [out=270, in=180] (3.5,-1);
	\draw[arc] (0,0) to [out=90, in=180] (.5,1) to [out=0, in=90] (1,0);
	\draw[arc] (2,0) to [out=90, in=0] (-.5,2);
	\draw[arc] (3,0) to [out=90, in=210] (3.5,2);
	\draw[doline] (-.5,0) to (3.5,0);
	\node at (1,0) {$\Down$};
	\node at (3,0) {$\Down$};
	\node at (2,0) {$\Up$};
	\node at (0,0) {$\Up$};
\end{tikzpicture}
\,
\begin{tikzpicture}[anchorbase, scale=.35, tinynodes]
	\draw[arc] (0,0) to [out=270, in=180] (.5,-1) to [out=0, in=270] (1,0);
	\draw[arc] (2,0) to [out=270, in=0] (-.5,-2);
	\draw[arc] (3,0) to [out=270, in=150] (3.5,-2);
	\draw[arc] (0,0) to [out=90, in=180] (.5,1) to [out=0, in=90] (1,0);
	\draw[arc] (2,0) to [out=90, in=0] (-.5,2);
	\draw[arc] (3,0) to [out=90, in=210] (3.5,2);
	\draw[doline] (-.5,0) to (3.5,0);
	\node at (1,0) {$\Down$};
	\node at (3,0) {$\Down$};
	\node at (2,0) {$\Up$};
	\node at (0,0) {$\Up$};
\end{tikzpicture}
\,
\begin{tikzpicture}[anchorbase, scale=.35, tinynodes]
	\draw[arc] (0,0) to [out=270, in=180] (1.5,-2) to [out=0, in=270] (3,0);
	\draw[arc] (1,0) to [out=270, in=180] (1.5,-1) to [out=0, in=270] (2,0);
	\draw[arc] (0,0) to [out=90, in=180] (.5,1) to [out=0, in=90] (1,0);
	\draw[arc] (2,0) to [out=90, in=0] (-.5,2);
	\draw[arc] (3,0) to [out=90, in=210] (3.5,2);
	\draw[doline] (-.5,0) to (3.5,0);
	\node at (1,0) {$\Down$};
	\node at (3,0) {$\Down$};
	\node at (2,0) {$\Up$};
	\node at (0,0) {$\Up$};
\end{tikzpicture}
\,
\begin{tikzpicture}[anchorbase, scale=.35, tinynodes]
	\draw[arc] (0,0) to [out=270, in=0] (-.5,-1);
	\draw[arc] (1,0) to [out=270, in=0] (-.5,-2);
	\draw[arc] (2,0) to [out=270, in=180] (3.5,-2);
	\draw[arc] (3,0) to [out=270, in=180] (3.5,-1);
	\draw[arc] (0,0) to [out=90, in=180] (.5,1) to [out=0, in=90] (1,0);
	\draw[arc] (2,0) to [out=90, in=0] (-.5,2);
	\draw[arc] (3,0) to [out=90, in=210] (3.5,2);
	\draw[doline] (-.5,0) to (3.5,0);
	\node at (1,0) {$\Down$};
	\node at (3,0) {$\Down$};
	\node at (2,0) {$\Up$};
	\node at (0,0) {$\Up$};
\end{tikzpicture}
\,
\begin{tikzpicture}[anchorbase, scale=.35, tinynodes]
	\draw[arc] (2,0) to [out=270, in=180] (2.5,-1) to [out=0, in=270] (3,0);
	\draw[arc] (0,0) to [out=270, in=30] (-.5,-2);
	\draw[arc] (1,0) to [out=270, in=180] (3.5,-2);
	\draw[arc] (0,0) to [out=90, in=180] (.5,1) to [out=0, in=90] (1,0);
	\draw[arc] (2,0) to [out=90, in=0] (-.5,2);
	\draw[arc] (3,0) to [out=90, in=210] (3.5,2);
	\draw[doline] (-.5,0) to (3.5,0);
	\node at (1,0) {$\Down$};
	\node at (3,0) {$\Down$};
	\node at (2,0) {$\Up$};
	\node at (0,0) {$\Up$};
\end{tikzpicture}
\,
\begin{tikzpicture}[anchorbase, scale=.35, tinynodes]
	\draw[arcd] (0,0) to [out=270, in=180] (.5,-2) to [out=0, in=270] (1,0);
	\draw[arc] (0,0) to [out=270, in=180] (.5,-1) to [out=0, in=270] (1,0);
	\draw[arc] (2,0) to [out=270, in=180] (2.5,-1) to [out=0, in=270] (3,0);
	\draw[arc] (0,0) to [out=90, in=180] (.5,1) to [out=0, in=90] (1,0);
	\draw[arc] (2,0) to [out=90, in=0] (-.5,2);
	\draw[arc] (3,0) to [out=90, in=210] (3.5,2);
	\draw[doline] (-.5,0) to (3.5,0);
	\node at (1,0) {$\Down$};
	\node at (3,0) {$\Down$};
	\node at (2,0) {$\Up$};
	\node at (0,0) {$\Up$};
\end{tikzpicture}
\\
\begin{tikzpicture}[anchorbase, scale=.35, tinynodes]
	\draw[arcd] (0,0) to [out=270, in=180] (.5,-2) to [out=0, in=270] (1,0);
	\draw[arc] (0,0) to [out=270, in=180] (.5,-1) to [out=0, in=270] (1,0);
	\draw[arc] (2,0) to [out=270, in=180] (2.5,-1) to [out=0, in=270] (3,0);
	\draw[arc] (0,0) to [out=90, in=180] (.5,1) to [out=0, in=90] (1,0);
	\draw[arc] (2,0) to [out=90, in=0] (-.5,2);
	\draw[arc] (3,0) to [out=90, in=210] (3.5,2);
	\draw[doline] (-.5,0) to (3.5,0);
	\node at (0,0) {$\Down$};
	\node at (2,0) {$\Down$};
	\node at (3,0) {$\Up$};
	\node at (1,0) {$\Up$};
\end{tikzpicture}
\,
\begin{tikzpicture}[anchorbase, scale=.35, tinynodes]
	\draw[arc] (2,0) to [out=270, in=180] (2.5,-1) to [out=0, in=270] (3,0);
	\draw[arc] (0,0) to [out=270, in=30] (-.5,-2);
	\draw[arc] (1,0) to [out=270, in=180] (3.5,-2);
	\draw[arc] (0,0) to [out=90, in=180] (.5,1) to [out=0, in=90] (1,0);
	\draw[arc] (2,0) to [out=90, in=0] (-.5,2);
	\draw[arc] (3,0) to [out=90, in=210] (3.5,2);
	\draw[doline] (-.5,0) to (3.5,0);
	\node at (0,0) {$\Down$};
	\node at (2,0) {$\Down$};
	\node at (3,0) {$\Up$};
	\node at (1,0) {$\Up$};
\end{tikzpicture}
\,
\begin{tikzpicture}[anchorbase, scale=.35, tinynodes]
	\draw[arc] (0,0) to [out=270, in=0] (-.5,-1);
	\draw[arc] (1,0) to [out=270, in=0] (-.5,-2);
	\draw[arc] (2,0) to [out=270, in=180] (3.5,-2);
	\draw[arc] (3,0) to [out=270, in=180] (3.5,-1);
	\draw[arc] (0,0) to [out=90, in=180] (.5,1) to [out=0, in=90] (1,0);
	\draw[arc] (2,0) to [out=90, in=0] (-.5,2);
	\draw[arc] (3,0) to [out=90, in=210] (3.5,2);
	\draw[doline] (-.5,0) to (3.5,0);
	\node at (0,0) {$\Down$};
	\node at (2,0) {$\Down$};
	\node at (3,0) {$\Up$};
	\node at (1,0) {$\Up$};
\end{tikzpicture}
\,
\begin{tikzpicture}[anchorbase, scale=.35, tinynodes]
	\draw[arc] (0,0) to [out=270, in=180] (1.5,-2) to [out=0, in=270] (3,0);
	\draw[arc] (1,0) to [out=270, in=180] (1.5,-1) to [out=0, in=270] (2,0);
	\draw[arc] (0,0) to [out=90, in=180] (.5,1) to [out=0, in=90] (1,0);
	\draw[arc] (2,0) to [out=90, in=0] (-.5,2);
	\draw[arc] (3,0) to [out=90, in=210] (3.5,2);
	\draw[doline] (-.5,0) to (3.5,0);
	\node at (0,0) {$\Down$};
	\node at (2,0) {$\Down$};
	\node at (3,0) {$\Up$};
	\node at (1,0) {$\Up$};
\end{tikzpicture}
\,
\begin{tikzpicture}[anchorbase, scale=.35, tinynodes]
	\draw[arc] (0,0) to [out=270, in=180] (.5,-1) to [out=0, in=270] (1,0);
	\draw[arc] (2,0) to [out=270, in=0] (-.5,-2);
	\draw[arc] (3,0) to [out=270, in=150] (3.5,-2);
	\draw[arc] (0,0) to [out=90, in=180] (.5,1) to [out=0, in=90] (1,0);
	\draw[arc] (2,0) to [out=90, in=0] (-.5,2);
	\draw[arc] (3,0) to [out=90, in=210] (3.5,2);
	\draw[doline] (-.5,0) to (3.5,0);
	\node at (0,0) {$\Down$};
	\node at (2,0) {$\Down$};
	\node at (3,0) {$\Up$};
	\node at (1,0) {$\Up$};
\end{tikzpicture}
\,
\begin{tikzpicture}[anchorbase, scale=.35, tinynodes]
	\draw[arcd] (0,0) to [out=270, in=180] (.5,-2) to [out=0, in=270] (1,0);
	\draw[arc] (1,0) to [out=270, in=180] (1.5,-1) to [out=0, in=270] (2,0);
	\draw[arc] (0,0) to [out=270, in=0] (-.5,-1);
	\draw[arc] (3,0) to [out=270, in=180] (3.5,-1);
	\draw[arc] (0,0) to [out=90, in=180] (.5,1) to [out=0, in=90] (1,0);
	\draw[arc] (2,0) to [out=90, in=0] (-.5,2);
	\draw[arc] (3,0) to [out=90, in=210] (3.5,2);
	\draw[doline] (-.5,0) to (3.5,0);
	\node at (0,0) {$\Down$};
	\node at (2,0) {$\Down$};
	\node at (3,0) {$\Up$};
	\node at (1,0) {$\Up$};
\end{tikzpicture}
\\
\begin{tikzpicture}[anchorbase, scale=.35, tinynodes]
	\draw[arc] (2,0) to [out=270, in=180] (2.5,-1) to [out=0, in=270] (3,0);
	\draw[arc] (0,0) to [out=270, in=30] (-.5,-2);
	\draw[arc] (1,0) to [out=270, in=180] (3.5,-2);
	\draw[arc] (0,0) to [out=90, in=180] (.5,1) to [out=0, in=90] (1,0);
	\draw[arc] (2,0) to [out=90, in=0] (-.5,2);
	\draw[arc] (3,0) to [out=90, in=210] (3.5,2);
	\draw[doline] (-.5,0) to (3.5,0);
	\node at (1,0) {$\Down$};
	\node at (2,0) {$\Down$};
	\node at (3,0) {$\Up$};
	\node at (0,0) {$\Up$};
\end{tikzpicture}
\,
\begin{tikzpicture}[anchorbase, scale=.35, tinynodes]
	\draw[arcd] (0,0) to [out=270, in=180] (.5,-2) to [out=0, in=270] (1,0);
	\draw[arc] (0,0) to [out=270, in=180] (.5,-1) to [out=0, in=270] (1,0);
	\draw[arc] (2,0) to [out=270, in=180] (2.5,-1) to [out=0, in=270] (3,0);
	\draw[arc] (0,0) to [out=90, in=180] (.5,1) to [out=0, in=90] (1,0);
	\draw[arc] (2,0) to [out=90, in=0] (-.5,2);
	\draw[arc] (3,0) to [out=90, in=210] (3.5,2);
	\draw[doline] (-.5,0) to (3.5,0);
	\node at (1,0) {$\Down$};
	\node at (2,0) {$\Down$};
	\node at (3,0) {$\Up$};
	\node at (0,0) {$\Up$};
\end{tikzpicture}
\,
\begin{tikzpicture}[anchorbase, scale=.35, tinynodes]
	\draw[arc] (0,0) to [out=270, in=180] (.5,-1) to [out=0, in=270] (1,0);
	\draw[arc] (2,0) to [out=270, in=0] (-.5,-2);
	\draw[arc] (3,0) to [out=270, in=150] (3.5,-2);
	\draw[arc] (0,0) to [out=90, in=180] (.5,1) to [out=0, in=90] (1,0);
	\draw[arc] (2,0) to [out=90, in=0] (-.5,2);
	\draw[arc] (3,0) to [out=90, in=210] (3.5,2);
	\draw[doline] (-.5,0) to (3.5,0);
	\node at (1,0) {$\Down$};
	\node at (2,0) {$\Down$};
	\node at (3,0) {$\Up$};
	\node at (0,0) {$\Up$};
\end{tikzpicture}
\\
};
	\draw[thin,dotted] ($(m-1-1.south)+(-5.7,.7)$) to ($(m-1-1.south)+(5.7,.7)$);
	\draw[thin, mygreen] (-5.5,3.6) rectangle (5.5,1.65);
	\draw[thin, myred] (-5.5,1.55) rectangle (5.5,-.3);
	\draw[thin, myblue] (-5.5,-.4) rectangle (5.5,-2.25);
	\draw[thin, myorange] (-5.5,-2.35) rectangle (5.5,-4.3);
\end{tikzpicture}
\end{gathered}$}
\end{gather}
and in $\prmod(\down\up\down\up)$
\begin{gather}
\scalebox{.9}{$
\begin{gathered}
\begin{tikzpicture}[baseline=(current bounding box.center)]
  \matrix (m) [matrix of math nodes, nodes in empty cells, row 1/.style={nodes={minimum height=18mm}}, row 3/.style={nodes={minimum height=32mm}}, row 5/.style={nodes={minimum height=32mm}}, row sep={.1cm}, column sep={.1cm}, text height=1.5ex, text depth=0.25ex, ampersand replacement=\&] {
\down\,\up\,\down\,\up
\\
\begin{tikzpicture}[anchorbase, scale=.35, tinynodes]
	\draw[arcd] (0,0) to [out=270, in=180] (.5,-2) to [out=0, in=270] (1,0);
	\draw[arc] (0,0) to [out=270, in=180] (.5,-1) to [out=0, in=270] (1,0);
	\draw[arc] (2,0) to [out=270, in=180] (2.5,-1) to [out=0, in=270] (3,0);
	\draw[arcd] (0,0) to [out=90, in=180] (.5,2) to [out=0, in=90] (1,0);
	\draw[arc] (0,0) to [out=90, in=180] (.5,1) to [out=0, in=90] (1,0);
	\draw[arc] (2,0) to [out=90, in=180] (2.5,1) to [out=0, in=90] (3,0);
	\draw[doline] (-.5,0) to (3.5,0);
	\node at (0,0) {$\Down$};
	\node at (2,0) {$\Down$};
	\node at (3,0) {$\Up$};
	\node at (1,0) {$\Up$};
\end{tikzpicture}
\,
\begin{tikzpicture}[anchorbase, scale=.35, tinynodes]
	\draw[arc] (2,0) to [out=270, in=180] (2.5,-1) to [out=0, in=270] (3,0);
	\draw[arc] (0,0) to [out=270, in=30] (-.5,-2);
	\draw[arc] (1,0) to [out=270, in=180] (3.5,-2);
	\draw[arcd] (0,0) to [out=90, in=180] (.5,2) to [out=0, in=90] (1,0);
	\draw[arc] (0,0) to [out=90, in=180] (.5,1) to [out=0, in=90] (1,0);
	\draw[arc] (2,0) to [out=90, in=180] (2.5,1) to [out=0, in=90] (3,0);
	\draw[doline] (-.5,0) to (3.5,0);
	\node at (0,0) {$\Down$};
	\node at (2,0) {$\Down$};
	\node at (3,0) {$\Up$};
	\node at (1,0) {$\Up$};
\end{tikzpicture}
\,
\begin{tikzpicture}[anchorbase, scale=.35, tinynodes]
	\draw[arc] (2,0) to [out=270, in=180] (3.5,-2);
	\draw[arc] (3,0) to [out=270, in=180] (3.5,-1);
	\draw[arc] (0,0) to [out=270, in=0] (-.5,-1);
	\draw[arc] (1,0) to [out=270, in=0] (-.5,-2);
	\draw[arcd] (0,0) to [out=90, in=180] (.5,2) to [out=0, in=90] (1,0);
	\draw[arc] (0,0) to [out=90, in=180] (.5,1) to [out=0, in=90] (1,0);
	\draw[arc] (2,0) to [out=90, in=180] (2.5,1) to [out=0, in=90] (3,0);
	\draw[doline] (-.5,0) to (3.5,0);
	\node at (0,0) {$\Down$};
	\node at (2,0) {$\Down$};
	\node at (3,0) {$\Up$};
	\node at (1,0) {$\Up$};
\end{tikzpicture}
\,
\begin{tikzpicture}[anchorbase, scale=.35, tinynodes]
	\draw[arc] (0,0) to [out=270, in=180] (1.5,-2) to [out=0, in=270] (3,0);
	\draw[arc] (1,0) to [out=270, in=180] (1.5,-1) to [out=0, in=270] (2,0);
	\draw[arcd] (0,0) to [out=90, in=180] (.5,2) to [out=0, in=90] (1,0);
	\draw[arc] (0,0) to [out=90, in=180] (.5,1) to [out=0, in=90] (1,0);
	\draw[arc] (2,0) to [out=90, in=180] (2.5,1) to [out=0, in=90] (3,0);
	\draw[doline] (-.5,0) to (3.5,0);
	\node at (0,0) {$\Down$};
	\node at (2,0) {$\Down$};
	\node at (3,0) {$\Up$};
	\node at (1,0) {$\Up$};
\end{tikzpicture}
\,
\begin{tikzpicture}[anchorbase, scale=.35, tinynodes]
	\draw[arc] (0,0) to [out=270, in=180] (.5,-1) to [out=0, in=270] (1,0);
	\draw[arc] (2,0) to [out=270, in=0] (-.5,-2);
	\draw[arc] (3,0) to [out=270, in=150] (3.5,-2);
	\draw[arcd] (0,0) to [out=90, in=180] (.5,2) to [out=0, in=90] (1,0);
	\draw[arc] (0,0) to [out=90, in=180] (.5,1) to [out=0, in=90] (1,0);
	\draw[arc] (2,0) to [out=90, in=180] (2.5,1) to [out=0, in=90] (3,0);
	\draw[doline] (-.5,0) to (3.5,0);
	\node at (0,0) {$\Down$};
	\node at (2,0) {$\Down$};
	\node at (3,0) {$\Up$};
	\node at (1,0) {$\Up$};
\end{tikzpicture}
\,
\begin{tikzpicture}[anchorbase, scale=.35, tinynodes]
	\draw[arcd] (0,0) to [out=270, in=180] (.5,-2) to [out=0, in=270] (1,0);
	\draw[arc] (1,0) to [out=270, in=180] (1.5,-1) to [out=0, in=270] (2,0);
	\draw[arc] (0,0) to [out=270, in=0] (-.5,-1);
	\draw[arc] (3,0) to [out=270, in=180] (3.5,-1);
	\draw[arcd] (0,0) to [out=90, in=180] (.5,2) to [out=0, in=90] (1,0);
	\draw[arc] (0,0) to [out=90, in=180] (.5,1) to [out=0, in=90] (1,0);
	\draw[arc] (2,0) to [out=90, in=180] (2.5,1) to [out=0, in=90] (3,0);
	\draw[doline] (-.5,0) to (3.5,0);
	\node at (0,0) {$\Down$};
	\node at (2,0) {$\Down$};
	\node at (3,0) {$\Up$};
	\node at (1,0) {$\Up$};
\end{tikzpicture}
\\
\begin{tikzpicture}[anchorbase, scale=.35, tinynodes]
	\draw[arc] (0,0) to [out=270, in=180] (.5,-1) to [out=0, in=270] (1,0);
	\draw[arc] (2,0) to [out=270, in=0] (-.5,-2);
	\draw[arc] (3,0) to [out=270, in=150] (3.5,-2);
	\draw[arcd] (0,0) to [out=90, in=180] (.5,2) to [out=0, in=90] (1,0);
	\draw[arc] (0,0) to [out=90, in=180] (.5,1) to [out=0, in=90] (1,0);
	\draw[arc] (2,0) to [out=90, in=180] (2.5,1) to [out=0, in=90] (3,0);
	\draw[doline] (-.5,0) to (3.5,0);
	\node at (1,0) {$\Down$};
	\node at (2,0) {$\Down$};
	\node at (3,0) {$\Up$};
	\node at (0,0) {$\Up$};
\end{tikzpicture}
\,
\begin{tikzpicture}[anchorbase, scale=.35, tinynodes]
	\draw[arc] (0,0) to [out=270, in=180] (.5,-1) to [out=0, in=270] (1,0);
	\draw[arc] (2,0) to [out=270, in=180] (2.5,-1) to [out=0, in=270] (3,0);
	\draw[arcd] (0,0) to [out=90, in=180] (.5,2) to [out=0, in=90] (1,0);
	\draw[arc] (0,0) to [out=90, in=180] (.5,1) to [out=0, in=90] (1,0);
	\draw[arc] (2,0) to [out=90, in=180] (2.5,1) to [out=0, in=90] (3,0);
	\draw[doline] (-.5,0) to (3.5,0);
	\node at (1,0) {$\Down$};
	\node at (2,0) {$\Down$};
	\node at (3,0) {$\Up$};
	\node at (0,0) {$\Up$};
\end{tikzpicture}
\,
\begin{tikzpicture}[anchorbase, scale=.35, tinynodes]
	\draw[arc] (2,0) to [out=270, in=180] (2.5,-1) to [out=0, in=270] (3,0);
	\draw[arc] (0,0) to [out=270, in=30] (-.5,-2);
	\draw[arc] (1,0) to [out=270, in=180] (3.5,-2);
	\draw[arcd] (0,0) to [out=90, in=180] (.5,2) to [out=0, in=90] (1,0);
	\draw[arc] (0,0) to [out=90, in=180] (.5,1) to [out=0, in=90] (1,0);
	\draw[arc] (2,0) to [out=90, in=180] (2.5,1) to [out=0, in=90] (3,0);
	\draw[doline] (-.5,0) to (3.5,0);
	\node at (1,0) {$\Down$};
	\node at (2,0) {$\Down$};
	\node at (3,0) {$\Up$};
	\node at (0,0) {$\Up$};
\end{tikzpicture}
\quad\quad\quad
\begin{tikzpicture}[anchorbase, scale=.35, tinynodes]
	\draw[arc] (2,0) to [out=270, in=180] (2.5,-1) to [out=0, in=270] (3,0);
	\draw[arc] (0,0) to [out=270, in=30] (-.5,-2);
	\draw[arc] (1,0) to [out=270, in=180] (3.5,-2);
	\draw[arcd] (0,0) to [out=90, in=180] (.5,2) to [out=0, in=90] (1,0);
	\draw[arc] (0,0) to [out=90, in=180] (.5,1) to [out=0, in=90] (1,0);
	\draw[arc] (2,0) to [out=90, in=180] (2.5,1) to [out=0, in=90] (3,0);
	\draw[doline] (-.5,0) to (3.5,0);
	\node at (0,0) {$\Down$};
	\node at (3,0) {$\Down$};
	\node at (2,0) {$\Up$};
	\node at (1,0) {$\Up$};
\end{tikzpicture}
\,
\begin{tikzpicture}[anchorbase, scale=.35, tinynodes]
	\draw[arc] (0,0) to [out=270, in=180] (.5,-1) to [out=0, in=270] (1,0);
	\draw[arc] (2,0) to [out=270, in=180] (2.5,-1) to [out=0, in=270] (3,0);
	\draw[arcd] (0,0) to [out=90, in=180] (.5,2) to [out=0, in=90] (1,0);
	\draw[arc] (0,0) to [out=90, in=180] (.5,1) to [out=0, in=90] (1,0);
	\draw[arc] (2,0) to [out=90, in=180] (2.5,1) to [out=0, in=90] (3,0);
	\draw[doline] (-.5,0) to (3.5,0);
	\node at (0,0) {$\Down$};
	\node at (3,0) {$\Down$};
	\node at (2,0) {$\Up$};
	\node at (1,0) {$\Up$};
\end{tikzpicture}
\,
\begin{tikzpicture}[anchorbase, scale=.35, tinynodes]
	\draw[arc] (0,0) to [out=270, in=180] (.5,-1) to [out=0, in=270] (1,0);
	\draw[arc] (2,0) to [out=270, in=0] (-.5,-2);
	\draw[arc] (3,0) to [out=270, in=150] (3.5,-2);
	\draw[arcd] (0,0) to [out=90, in=180] (.5,2) to [out=0, in=90] (1,0);
	\draw[arc] (0,0) to [out=90, in=180] (.5,1) to [out=0, in=90] (1,0);
	\draw[arc] (2,0) to [out=90, in=180] (2.5,1) to [out=0, in=90] (3,0);
	\draw[doline] (-.5,0) to (3.5,0);
	\node at (0,0) {$\Down$};
	\node at (3,0) {$\Down$};
	\node at (2,0) {$\Up$};
	\node at (1,0) {$\Up$};
\end{tikzpicture}
\\
\begin{tikzpicture}[anchorbase, scale=.35, tinynodes]
	\draw[arcd] (0,0) to [out=270, in=180] (.5,-2) to [out=0, in=270] (1,0);
	\draw[arc] (1,0) to [out=270, in=180] (1.5,-1) to [out=0, in=270] (2,0);
	\draw[arc] (0,0) to [out=270, in=0] (-.5,-1);
	\draw[arc] (3,0) to [out=270, in=180] (3.5,-1);
	\draw[arcd] (0,0) to [out=90, in=180] (.5,2) to [out=0, in=90] (1,0);
	\draw[arc] (0,0) to [out=90, in=180] (.5,1) to [out=0, in=90] (1,0);
	\draw[arc] (2,0) to [out=90, in=180] (2.5,1) to [out=0, in=90] (3,0);
	\draw[doline] (-.5,0) to (3.5,0);
	\node at (1,0) {$\Down$};
	\node at (3,0) {$\Down$};
	\node at (2,0) {$\Up$};
	\node at (0,0) {$\Up$};
\end{tikzpicture}
\,
\begin{tikzpicture}[anchorbase, scale=.35, tinynodes]
	\draw[arc] (0,0) to [out=270, in=180] (.5,-1) to [out=0, in=270] (1,0);
	\draw[arc] (2,0) to [out=270, in=0] (-.5,-2);
	\draw[arc] (3,0) to [out=270, in=150] (3.5,-2);
	\draw[arcd] (0,0) to [out=90, in=180] (.5,2) to [out=0, in=90] (1,0);
	\draw[arc] (0,0) to [out=90, in=180] (.5,1) to [out=0, in=90] (1,0);
	\draw[arc] (2,0) to [out=90, in=180] (2.5,1) to [out=0, in=90] (3,0);
	\draw[doline] (-.5,0) to (3.5,0);
	\node at (1,0) {$\Down$};
	\node at (3,0) {$\Down$};
	\node at (2,0) {$\Up$};
	\node at (0,0) {$\Up$};
\end{tikzpicture}
\,
\begin{tikzpicture}[anchorbase, scale=.35, tinynodes]
	\draw[arc] (0,0) to [out=270, in=180] (1.5,-2) to [out=0, in=270] (3,0);
	\draw[arc] (1,0) to [out=270, in=180] (1.5,-1) to [out=0, in=270] (2,0);
	\draw[arcd] (0,0) to [out=90, in=180] (.5,2) to [out=0, in=90] (1,0);
	\draw[arc] (0,0) to [out=90, in=180] (.5,1) to [out=0, in=90] (1,0);
	\draw[arc] (2,0) to [out=90, in=180] (2.5,1) to [out=0, in=90] (3,0);
	\draw[doline] (-.5,0) to (3.5,0);
	\node at (1,0) {$\Down$};
	\node at (3,0) {$\Down$};
	\node at (2,0) {$\Up$};
	\node at (0,0) {$\Up$};
\end{tikzpicture}
\,
\begin{tikzpicture}[anchorbase, scale=.35, tinynodes]
	\draw[arc] (2,0) to [out=270, in=180] (3.5,-2);
	\draw[arc] (3,0) to [out=270, in=180] (3.5,-1);
	\draw[arc] (0,0) to [out=270, in=0] (-.5,-1);
	\draw[arc] (1,0) to [out=270, in=0] (-.5,-2);
	\draw[arcd] (0,0) to [out=90, in=180] (.5,2) to [out=0, in=90] (1,0);
	\draw[arc] (0,0) to [out=90, in=180] (.5,1) to [out=0, in=90] (1,0);
	\draw[arc] (2,0) to [out=90, in=180] (2.5,1) to [out=0, in=90] (3,0);
	\draw[doline] (-.5,0) to (3.5,0);
	\node at (1,0) {$\Down$};
	\node at (3,0) {$\Down$};
	\node at (2,0) {$\Up$};
	\node at (0,0) {$\Up$};
\end{tikzpicture}
\,
\begin{tikzpicture}[anchorbase, scale=.35, tinynodes]
	\draw[arc] (2,0) to [out=270, in=180] (2.5,-1) to [out=0, in=270] (3,0);
	\draw[arc] (0,0) to [out=270, in=30] (-.5,-2);
	\draw[arc] (1,0) to [out=270, in=180] (3.5,-2);
	\draw[arcd] (0,0) to [out=90, in=180] (.5,2) to [out=0, in=90] (1,0);
	\draw[arc] (0,0) to [out=90, in=180] (.5,1) to [out=0, in=90] (1,0);
	\draw[arc] (2,0) to [out=90, in=180] (2.5,1) to [out=0, in=90] (3,0);
	\draw[doline] (-.5,0) to (3.5,0);
	\node at (1,0) {$\Down$};
	\node at (3,0) {$\Down$};
	\node at (2,0) {$\Up$};
	\node at (0,0) {$\Up$};
\end{tikzpicture}
\,
\begin{tikzpicture}[anchorbase, scale=.35, tinynodes]
	\draw[arcd] (0,0) to [out=270, in=180] (.5,-2) to [out=0, in=270] (1,0);
	\draw[arc] (0,0) to [out=270, in=180] (.5,-1) to [out=0, in=270] (1,0);
	\draw[arc] (2,0) to [out=270, in=180] (2.5,-1) to [out=0, in=270] (3,0);
	\draw[arcd] (0,0) to [out=90, in=180] (.5,2) to [out=0, in=90] (1,0);
	\draw[arc] (0,0) to [out=90, in=180] (.5,1) to [out=0, in=90] (1,0);
	\draw[arc] (2,0) to [out=90, in=180] (2.5,1) to [out=0, in=90] (3,0);
	\draw[doline] (-.5,0) to (3.5,0);
	\node at (1,0) {$\Down$};
	\node at (3,0) {$\Down$};
	\node at (2,0) {$\Up$};
	\node at (0,0) {$\Up$};
\end{tikzpicture}
\\
};
	\draw[thin,dotted] ($(m-1-1.south)+(-5.7,.7)$) to ($(m-1-1.south)+(5.7,.7)$);
	\draw[thin, myblue] (-5.5,2.0) rectangle (5.5,.05);
	\draw[thin, mygreen] (-6.0,-.05) rectangle (-.5,-1.9);
	\draw[thin, myorange] (.5,-.05) rectangle (6.0,-1.9);
	\draw[thin, myred] (-5.5,-2.0) rectangle (5.5,-3.85);
\end{tikzpicture}
\end{gathered}$}
\end{gather}
These are ordered as in \eqref{eq:ordering-k2} and \eqref{eq:the-usual-order}. Next, 
the indecomposable projectives are
\begin{gather}
\scalebox{.9}{$
\begin{gathered}
\begin{tikzpicture}[baseline=(current bounding box.center)]
	\node[mygreen] at (-.01,0) {$\lmod(\down\up\up\down)$};
	\node[myblue] at (-.01,-.5) {$\lmod(\down\up\down\up)$};
	\node[myred] at (3.5,-.5) {$\lmod(\up\down\up\down)$};
	\node[myblue] at (9.1,-.5) {$\lmod(\down\up\down\up)$};
	\node[myorange] at (-.01,-1) {$\lmod(\up\down\down\up)$};
	\node[mygreen] at (1.4,-1) {$\lmod(\down\up\up\down)$};
	\node[mypurple] at (2.8,-1) {$\lmod(\down\down\up\up)$};
	\node[myyellow] at (4.2,-1) {$\lmod(\up\up\down\down)$};
	\node[myorange] at (5.6,-1) {$\lmod(\up\down\down\up)$};
	\node[myorange] at (7.0,-1) {$\lmod(\up\down\down\up)$};
	\node[myyellow] at (8.4,-1) {$\lmod(\up\up\down\down)$};
	\node[mypurple] at (9.8,-1) {$\lmod(\down\down\up\up)$};
	\node[mygreen] at (11.2,-1) {$\lmod(\down\up\up\down)$};
	\node[myorange] at (12.6,-1) {$\lmod(\up\down\down\up)$};
	\node[myblue] at (3.5,-1.5) {$\lmod(\down\up\down\up)$};
	\node[myred] at (9.1,-1.5) {$\lmod(\up\down\up\down)$};
	\node[myblue] at (12.6,-1.5) {$\lmod(\down\up\down\up)$};
	\node[mygreen] at (12.6,-2) {$\lmod(\down\up\up\down)$};
	\node[mygreen] at (6.3,.55) {$\prmod(\down\up\up\down)$};
	\draw[thin, mygreen, dotted] (-.8,.3) rectangle (13.4,-2.25);
	\node[mygreen] at (1.5,0) {$\dmod(\down\up\up\down)$};
	\draw[thin, mygreen] (-.7,.25) rectangle (.675,-1.25);
	\node[myred] at (3.5,-2) {$\dmod(\down\up\up\down)$};
	\draw[thin, myred] (.7,-.275) rectangle (6.275,-1.725);
	\node[myblue] at (9.1,0) {$\dmod(\down\up\up\down)$};
	\draw[thin, myblue] (6.3,-.275) rectangle (11.875,-1.725);
	\node[myorange] at (11.1,-2) {$\dmod(\up\down\down\up)$};
	\draw[thin, myorange] (11.91,-.775) rectangle (13.325,-2.2);
\end{tikzpicture}
\\
\begin{tikzpicture}[baseline=(current bounding box.center)]
	\node[myred] at (3.5,0) {$\lmod(\up\down\up\down)$};
	\node[mygreen] at (-.01,-.5) {$\lmod(\down\up\up\down)$};
	\node[mygreen] at (1.4,-.5) {$\lmod(\down\up\up\down)$};
	\node[mypurple] at (2.8,-.5) {$\lmod(\down\down\up\up)$};
	\node[myyellow] at (4.2,-.5) {$\lmod(\up\up\down\down)$};
	\node[myorange] at (5.6,-.5) {$\lmod(\up\down\down\up)$};
	\node[myorange] at (12.6,-.5) {$\lmod(\up\down\down\up)$};
	\node[myblue] at (-.01,-1) {$\lmod(\down\up\down\up)$};
	\node[myblue] at (3.5,-1) {$\lmod(\down\up\down\up)$};
	\node[myblue] at (9.1,-1) {$\lmod(\down\up\down\up)$};
	\node[myblue] at (12.6,-1) {$\lmod(\down\up\down\up)$};
	\node[myorange] at (-.01,-1.5) {$\lmod(\up\down\down\up)$};
	\node[myorange] at (7.0,-1.5) {$\lmod(\up\down\down\up)$};
	\node[myyellow] at (8.4,-1.5) {$\lmod(\up\up\down\down)$};
	\node[mypurple] at (9.8,-1.5) {$\lmod(\down\down\up\up)$};
	\node[mygreen] at (11.2,-1.5) {$\lmod(\down\up\up\down)$};
	\node[mygreen] at (12.6,-1.5) {$\lmod(\down\up\up\down)$};
	\node[myred] at (9.1,-2) {$\lmod(\up\down\up\down)$};
	\node[myred] at (6.3,.55) {$\prmod(\down\up\up\down)$};
	\draw[thin, myred, dotted] (-.8,.3) rectangle (13.4,-2.25);
	\node[mygreen] at (.3,-2) {$\dmod(\down\up\up\down)$};
	\draw[thin, mygreen] (-.7,-0.25) rectangle (.675,-1.75);
	\node[myred] at (3.5,-1.5) {$\dmod(\down\up\up\down)$};
	\draw[thin, myred] (.7,.25) rectangle (6.275,-1.2);
	\node[myblue] at (9.1,-.5) {$\dmod(\down\up\up\down)$};
	\draw[thin, myblue] (6.3,-.775) rectangle (11.875,-2.2);
	\node[myorange] at (12.3,0) {$\dmod(\up\down\down\up)$};
	\draw[thin, myorange] (11.91,-.25) rectangle (13.325,-1.75);
\end{tikzpicture}
\end{gathered}$}
\end{gather}
(Note: From $\numberarcs=3$ onwards the $\prmod(\lambda)$'s are not of 
the same size anymore. 
That is, $\aarc[3]$ is of dimension $1664$ with $\prmod(\lambda)$'s 
being of dimension $80$ or $88$.)
By the above we see that the Cartan matrix is (up to similarity)
\begin{gather}
\cmatrix(\aarc[2])
=
\begin{psmallmatrix}
4 & 2 & 2 & 4 & 2 & 4
\\
2 & 4 & 4 & 2 & 4 & 2
\\
2 & 4 & 4 & 2 & 4 & 2
\\
4 & 2 & 2 & 4 & 2 & 4
\\
2 & 4 & 4 & 2 & 4 & 2
\\
4 & 2 & 2 & 4 & 2 & 4
\end{psmallmatrix}
\end{gather}
This again shows that $\aarc[2]$ is not a cellular algebra.
\end{example}

%%%%%%%%%%%%%%%%%%%%%%%%%%%
\subsection{Some concluding comments}\label{subsection:futher-ideas}
%%%%%%%%%%%%%%%%%%%%%%%%%%%

A few potential generalizations 
regarding relative cellularity of $\aarc$ are:

\begin{furtherdirections}\label{remark:no-fancy-stuff}
Everything can be done in the graded setup as well 
with the algebra $\aarc$ having an analogous 
grading as $\uarc$.
In particular, it makes sense to define the notion of a 
\textit{graded, relative cellular algebra}, 
generalizing \cite[Definition 2.1]{hm1}.
\end{furtherdirections}

\begin{furtherdirections}\label{remark:sl3}
$\uarc$ was originally 
defined to construct tangle invariants 
associated to Khovanov homology \cite{kh1}. 
Similarly, so-called \textit{web algebras} appear in the construction 
of tangle invariants associated to 
Khovanov--Rozansky homologies. These web algebras 
are also known to be cellular algebras, see 
\cite[Corollary 5.21]{mpt1}, \cite[Theorem 4.22]{t1} 
and \cite[Theorem 7.7]{m1}. Building on 
\cite{qr1}, it should be possible to defined annular variants, and the 
question whether these are relative cellular arises.
\end{furtherdirections}

\begin{furtherdirections}\label{remark:typeD}
One could also define
annular versions of the \textit{type $\mathrm{D}$ 
arc algebra} as in \cite{es1}, \cite{es2} or \cite{etw1}. 
This algebra is again cellular, see \cite[Corollary 7.3]{es1}, 
and the question about relative cellularity again arises.
\end{furtherdirections}

%%%%%%%%%%%%%%%%%%%%%%%%%%%
\subsection{Relative cellularity: Technicalities}\label{subsection:dull}
%%%%%%%%%%%%%%%%%%%%%%%%%%%

For the proof of \fullref{theorem:mult-rule-arc} we 
need some more control over cups and caps, necessitating a number of definitions 
and lemmas. 

\begin{definition}\label{definition:oriented-cups-caps}
Let $\lambda\in\cset$ and 
$S$ be a cup diagram such that 
$S\lambda$ is oriented. Assume that we have the following local 
situations.
\begin{gather}\label{eq:oriented-cups-caps}
\xy
(0,0)*{
\begin{tikzpicture}[anchorbase, scale=.4, tinynodes]
	\draw[arc] (0,0) to [out=270, in=180] (.5,-1) to [out=0, in=270] (1,0);
	\draw[arcd] (0,0) to [out=90, in=180] (.5,1) to [out=0, in=90] (1,0);
	\draw[doline] (-.5,0) to (1.5,0);
	\node at (0,0) {$\Down$};
	\node at (1,0) {$\Up$};
\end{tikzpicture}
,
\begin{tikzpicture}[anchorbase, scale=.4, tinynodes]
	\draw[arcd] (0,0) to [out=270, in=180] (.5,-1) to [out=0, in=270] (1,0);
	\draw[arc] (0,0) to [out=90, in=180] (.5,1) to [out=0, in=90] (1,0);
	\draw[doline] (-.5,0) to (1.5,0);
	\node at (0,0) {$\Down$};
	\node at (1,0) {$\Up$};
\end{tikzpicture}
,
\begin{tikzpicture}[anchorbase, scale=.4, tinynodes]
	\draw[arc] (0,0) to [out=270, in=0] (-.5,-1);
	\draw[arc] (1,0) to [out=270, in=180] (1.5,-1);
	\draw[arcd] (0,0) to [out=90, in=0] (-.5,1);
	\draw[arcd] (1,0) to [out=90, in=180] (1.5,1);
	\draw[doline] (-.5,0) to (1.5,0);
	\node at (1,0) {$\Down$};
	\node at (0,0) {$\Up$};
\end{tikzpicture}
,
\begin{tikzpicture}[anchorbase, scale=.4, tinynodes]
	\draw[arcd] (0,0) to [out=270, in=0] (-.5,-1);
	\draw[arcd] (1,0) to [out=270, in=180] (1.5,-1);
	\draw[arc] (0,0) to [out=90, in=0] (-.5,1);
	\draw[arc] (1,0) to [out=90, in=180] (1.5,1);
	\draw[doline] (-.5,0) to (1.5,0);
	\node at (1,0) {$\Down$};
	\node at (0,0) {$\Up$};
\end{tikzpicture}
};
(0,-6)*{\text{{\tiny anticlockwise}}};
\endxy
;\quad\quad
\xy
(0,0)*{
\begin{tikzpicture}[anchorbase, scale=.4, tinynodes]
	\draw[arc] (0,0) to [out=270, in=180] (.5,-1) to [out=0, in=270] (1,0);
	\draw[arcd] (0,0) to [out=90, in=180] (.5,1) to [out=0, in=90] (1,0);
	\draw[doline] (-.5,0) to (1.5,0);
	\node at (1,0) {$\Down$};
	\node at (0,0) {$\Up$};
\end{tikzpicture}
,
\begin{tikzpicture}[anchorbase, scale=.4, tinynodes]
	\draw[arcd] (0,0) to [out=270, in=180] (.5,-1) to [out=0, in=270] (1,0);
	\draw[arc] (0,0) to [out=90, in=180] (.5,1) to [out=0, in=90] (1,0);
	\draw[doline] (-.5,0) to (1.5,0);
	\node at (1,0) {$\Down$};
	\node at (0,0) {$\Up$};
\end{tikzpicture}
,
\begin{tikzpicture}[anchorbase, scale=.4, tinynodes]
	\draw[arc] (0,0) to [out=270, in=0] (-.5,-1);
	\draw[arc] (1,0) to [out=270, in=180] (1.5,-1);
	\draw[arcd] (0,0) to [out=90, in=0] (-.5,1);
	\draw[arcd] (1,0) to [out=90, in=180] (1.5,1);
	\draw[doline] (-.5,0) to (1.5,0);
	\node at (0,0) {$\Down$};
	\node at (1,0) {$\Up$};
\end{tikzpicture}
,
\begin{tikzpicture}[anchorbase, scale=.4, tinynodes]
	\draw[arcd] (0,0) to [out=270, in=0] (-.5,-1);
	\draw[arcd] (1,0) to [out=270, in=180] (1.5,-1);
	\draw[arc] (0,0) to [out=90, in=0] (-.5,1);
	\draw[arc] (1,0) to [out=90, in=180] (1.5,1);
	\draw[doline] (-.5,0) to (1.5,0);
	\node at (0,0) {$\Down$};
	\node at (1,0) {$\Up$};
\end{tikzpicture}};
(0,-6)*{\text{{\tiny clockwise}}};
\endxy
\end{gather}
Then we 
call such cups or caps 
\textit{anticlockwise} and \textit{clockwise}, as indicated.
\end{definition}

Comparing to \eqref{eq:arc-orientation}, cups and caps in 
usual circles are always of the corresponding orientation. 
Moreover, cups in essential and rightwards and 
caps in essential and leftwards circles
are clockwise, and vice versa.

\begin{definition} \label{definition:right-left-side}
Let $\circles$ be a circle in a circle 
diagram $ST^{\invo}$. Then $\Sp\times[-1,1]\setminus\circles$ 
has two connected components. For a usual circle the connected 
component containing the boundary of $\Sp\times[-1,1]$ is 
called the \textit{exterior} of $\circles$, the other is called 
the \textit{interior}. For an essential circle the one containing the 
boundary $\Sp\times\{1\}$ is called the \textit{upper} (half), the other 
is called the \textit{lower} (half).
 
Here the picture illustrating these notions:
\begin{gather}\label{eq:sides-of-circles}
\begin{tikzpicture}[anchorbase, scale=.4, tinynodes]
	\draw[arc] (0,0) to [out=270, in=180] (.5,-1) to [out=0, in=270] (1,0);
	\draw[arc] (0,0) to [out=90, in=180] (.5,1) to [out=0, in=90] (1,0);
	\draw[doline] (-.5,0) to (1.5,0);
	\node at (.5,-2) {\text{exterior}};
	\node at (.5,2) {\text{interior}};
	\draw[->] (.5,1.6) to (.5,.4);
\end{tikzpicture}
,
\begin{tikzpicture}[anchorbase, scale=.4, tinynodes]
	\draw[arc] (0,0) to [out=270, in=180] (.5,-1) to [out=0, in=270] (1,0);
	\draw[arc] (0,0) to [out=90, in=330] (-.5,1);
	\draw[arc] (1,0) to [out=90, in=210] (1.5,1);
	\draw[doline] (-.5,0) to (1.5,0);
	\node at (.5,2) {\text{upper}};
	\node at (.5,-2) {\text{lower}};
\end{tikzpicture}
;
\quad\quad
\begin{tikzpicture}[anchorbase, scale=.4, tinynodes]
	\draw[arc] (0,0) to [out=270, in=180] (.5,-1) to [out=0, in=270] (1,0);
	\draw[arc] (0,0) to [out=90, in=180] (.5,1) to [out=0, in=90] (1,0);
	\draw[doline] (-.5,0) to (1.5,0);
	\node at (0,0) {$\Down$};
	\node at (1,0) {$\Up$};
	\node at (.5,2) {\text{left}};
	\node at (.5,-2) {\text{right}};
	\draw[->] (.5,1.6) to (.5,.4);
\end{tikzpicture}
,
\begin{tikzpicture}[anchorbase, scale=.4, tinynodes]
	\draw[arc] (0,0) to [out=270, in=180] (.5,-1) to [out=0, in=270] (1,0);
	\draw[arc] (0,0) to [out=90, in=180] (.5,1) to [out=0, in=90] (1,0);
	\draw[doline] (-.5,0) to (1.5,0);
	\node at (1,0) {$\Down$};
	\node at (0,0) {$\Up$};
	\node at (.5,2) {\text{right}};
	\node at (.5,-2) {\text{left}};
	\draw[->] (.5,1.6) to (.5,.4);
\end{tikzpicture}
;
\quad\quad
\begin{tikzpicture}[anchorbase, scale=.4, tinynodes]
	\draw[arc] (0,0) to [out=270, in=180] (.5,-1) to [out=0, in=270] (1,0);
	\draw[arc] (0,0) to [out=90, in=330] (-.5,1);
	\draw[arc] (1,0) to [out=90, in=210] (1.5,1);
	\draw[doline] (-.5,0) to (1.5,0);
	\node at (1,0) {$\Down$};
	\node at (0,0) {$\Up$};
	\node at (.5,-2) {\text{left}};
	\node at (.5,2) {\text{right}};
\end{tikzpicture}
,
\begin{tikzpicture}[anchorbase, scale=.4, tinynodes]
	\draw[arc] (0,0) to [out=270, in=180] (.5,-1) to [out=0, in=270] (1,0);
	\draw[arc] (0,0) to [out=90, in=330] (-.5,1);
	\draw[arc] (1,0) to [out=90, in=210] (1.5,1);
	\draw[doline] (-.5,0) to (1.5,0);
	\node at (0,0) {$\Down$};
	\node at (1,0) {$\Up$};
	\node at (.5,-2) {\text{right}};
	\node at (.5,2) {\text{left}};
\end{tikzpicture}
\end{gather}
As in \eqref{eq:sides-of-circles},
if furthermore a \textit{small circle} $\circles$ 
(i.e. circles built from one cup and one cap only) is 
endowed with an orientation in $\cbas{S,T}{\lambda}$, 
then we distinguish between a \textit{right} and a \textit{left side of $\circles$} 
by using the orientation.

For more general circles we use repeatedly
\begin{gather}\label{eq:zig-zag}
\begin{tikzpicture}[anchorbase, scale=.4, tinynodes]
	\draw[arc] (0,2) to (0,0) to [out=270, in=0] (-.5,-1) to [out=180, in=270] (-1,0);
	\draw[arc] (-1,0) to [out=90, in=0] (-1.5,1) to [out=180, in=90] (-2,0) to (-2,-2);
	\draw[doline] (.5,0) to (-2.5,0);
	\node at (0,0) {$\Up$};
	\node at (-1,0) {$\Down$};
	\node at (-2,0) {$\Up$};
	\node at (-.5,-1.75) {\text{right}};
	\node at (-1.5,1.75) {\text{left}};
\end{tikzpicture}
\leftsquigarrow
\begin{tikzpicture}[anchorbase, scale=.4, tinynodes]
	\draw[arc] (0,-2) to (0,2);
	\draw[doline] (-.5,0) to (.5,0);
	\node at (0,0) {$\Up$};
	\node at (1,.75) {\text{right}};
	\node at (-1,.75) {\text{left}};
\end{tikzpicture}
\rightsquigarrow
\begin{tikzpicture}[anchorbase, scale=.4, tinynodes]
	\draw[arc] (0,2) to (0,0) to [out=270, in=180] (.5,-1) to [out=0, in=270] (1,0);
	\draw[arc] (1,0) to [out=90, in=180] (1.5,1) to [out=0, in=90] (2,0) to (2,-2);
	\draw[doline] (-.5,0) to (2.5,0);
	\node at (0,0) {$\Up$};
	\node at (1,0) {$\Down$};
	\node at (2,0) {$\Up$};
	\node at (1.5,1.75) {\text{right}};
	\node at (.5,-1.75) {\text{left}};
\end{tikzpicture}
;
\quad\quad
\begin{tikzpicture}[anchorbase, scale=.4, tinynodes]
	\draw[arc] (0,2) to (0,0) to [out=270, in=0] (-.5,-1) to [out=180, in=270] (-1,0);
	\draw[arc] (-1,0) to [out=90, in=0] (-1.5,1) to [out=180, in=90] (-2,0) to (-2,-2);
	\draw[doline] (.5,0) to (-2.5,0);
	\node at (0,0) {$\Down$};
	\node at (-1,0) {$\Up$};
	\node at (-2,0) {$\Down$};
	\node at (-1.5,1.75) {\text{right}};
	\node at (-.5,-1.75) {\text{left}};
\end{tikzpicture}
\leftsquigarrow
\begin{tikzpicture}[anchorbase, scale=.4, tinynodes]
	\draw[arc] (0,-2) to (0,2);
	\draw[doline] (-.5,0) to (.5,0);
	\node at (0,0) {$\Down$};
	\node at (-1,.75) {\text{right}};
	\node at (1,.75) {\text{left}};
\end{tikzpicture}
\rightsquigarrow
\begin{tikzpicture}[anchorbase, scale=.4, tinynodes]
	\draw[arc] (0,2) to (0,0) to [out=270, in=180] (.5,-1) to [out=0, in=270] (1,0);
	\draw[arc] (1,0) to [out=90, in=180] (1.5,1) to [out=0, in=90] (2,0) to (2,-2);
	\draw[doline] (-.5,0) to (2.5,0);
	\node at (0,0) {$\Down$};
	\node at (1,0) {$\Up$};
	\node at (2,0) {$\Down$};
	\node at (.5,-1.75) {\text{right}};
	\node at (1.5,1.75) {\text{left}};
\end{tikzpicture}
\end{gather}
to define the notions \textit{right} and \textit{left side of $\circles$}.
\end{definition}

The following is clear.

\begin{lemmaqed}\label{lemma:righthand-rule}
Let $\circles$ be a circle in an circle diagram $ST^\invo$. 
Then the notions in \fullref{definition:right-left-side} are well-defined and satisfy:
\smallskip
\begin{enumerate}[label=(\alph*)]
\setlength\itemsep{.15cm}

\item If $\circles$ is usual and anticlockwise, then its interior is to the left.
If $\circles$ is usual and clockwise, then its exterior is to the left.

\item If $\circles$ is essential and leftwards, then its lower is to the left.
If $\circles$ is essential and rightwards, then its upper is to the left.\qedhere
\end{enumerate}
\end{lemmaqed}

We also need to distinguish 
certain types of cups and caps.

\begin{definition}\label{definition:convex-concave}
Let $ST^{\invo}$ be a circle diagram and $\circles$ a circle in $ST^{\invo}$.
\smallskip
\begin{enumerate}[label=(\alph*)]
\setlength\itemsep{.15cm}

\item Let $\circles$ be usual. We say that a cup, 
respectively cap, in $\circles$ 
is $\ecup$, respectively 
$\ecap$, if the exterior of $\circles$ is 
directly above the cup, respectively below the cap. 
Otherwise we call it $\icup$, respectively $\icap$.

\item Let $\circles$ be essential. We say that a cup, 
respectively cap, in $\circles$ 
is $\lcup$, respectively 
$\lcap$, if the lower of $\circles$ is 
directly below the cup, respectively below the cap. 
Otherwise we call it $\ucup$, respectively $\ucap$.\qedhere

\end{enumerate}
\end{definition}

Note that \fullref{definition:convex-concave} depends only on the shape, 
and here the picture:
\begin{gather}
\begin{tikzpicture}[anchorbase, scale=.4, tinynodes]
	\draw[arc] (0,0) to [out=270, in=180] (.5,-1) to [out=0, in=270] (1,0);
	\draw[arcdo] (0,0) to [out=90, in=180] (.5,1) to [out=0, in=90] (1,0);
	\draw[doline] (-.5,0) to (1.5,0);
	\node at (.5,-1.5) {\text{ext.}};
	\node[scale=.8] at (.5,1.5) {$\scalebox{1.5}{\raisebox{.015cm}{$e$}}\hspace*{.01cm}\cups$};
\end{tikzpicture}
,
\begin{tikzpicture}[anchorbase, scale=.4, tinynodes]
	\draw[arcdo] (0,0) to [out=270, in=180] (.5,-1) to [out=0, in=270] (1,0);
	\draw[arc] (0,0) to [out=90, in=180] (.5,1) to [out=0, in=90] (1,0);
	\draw[doline] (-.5,0) to (1.5,0);
	\node at (.5,1.5) {\text{ext.}};
	\node[scale=.8] at (.5,-1.5) {$\scalebox{1.5}{\raisebox{.015cm}{$e$}}\hspace*{.01cm}\caps$};
\end{tikzpicture}
;
\quad\quad
\begin{tikzpicture}[anchorbase, scale=.4, tinynodes]
	\draw[arcd] (-1,0) to [out=90, in=180] (.5,2) to [out=0, in=90] (2,0);
	\draw[arc] (0,0) to [out=270, in=180] (.5,-1) to [out=0, in=270] (1,0);
	\draw[arcdo] (-1,0) to [out=90, in=180] (-.5,1) to [out=0, in=90] (0,0);
	\draw[arcdo] (1,0) to [out=90, in=180] (1.5,1) to [out=0, in=90] (2,0);
	\draw[arcdo] (-1,0) to [out=270, in=180] (.5,-2) to [out=0, in=270] (2,0);
	\draw[doline] (-1.5,0) to (2.5,0);
	\node at (.5,-1.5) {\text{int.}};
	\node[scale=.8] at (.5,1.5) {$\scalebox{1.5}{\raisebox{.015cm}{$i$}}\hspace*{.01cm}\cups$};
\end{tikzpicture}
,
\begin{tikzpicture}[anchorbase, scale=.4, tinynodes]
	\draw[arcd] (-1,0) to [out=270, in=180] (.5,-2) to [out=0, in=270] (2,0);
	\draw[arc] (0,0) to [out=90, in=180] (.5,1) to [out=0, in=90] (1,0);
	\draw[arcdo] (-1,0) to [out=270, in=180] (-.5,-1) to [out=0, in=270] (0,0);
	\draw[arcdo] (1,0) to [out=270, in=180] (1.5,-1) to [out=0, in=270] (2,0);
	\draw[arcdo] (-1,0) to [out=90, in=180] (.5,2) to [out=0, in=90] (2,0);
	\draw[doline] (-1.5,0) to (2.5,0);
	\node at (.5,1.5) {\text{int.}};
	\node[scale=.8] at (.5,-1.5) {$\scalebox{1.5}{\raisebox{.015cm}{$i$}}\hspace*{.01cm}\caps$};
\end{tikzpicture}
;
\quad\quad
\begin{tikzpicture}[anchorbase, scale=.4, tinynodes]
	\draw[arc] (0,0) to [out=270, in=180] (.5,-1) to [out=0, in=270] (1,0);
	\draw[arcdo] (0,0) to [out=90, in=0] (-.5,1);
	\draw[arcdo] (1,0) to [out=90, in=180] (1.5,1);
	\draw[doline] (-.5,0) to (1.5,0);
	\node at (.5,-1.5) {\text{low.}};
	\node[scale=.8] at (.5,1.5) {$\scalebox{1.5}{\raisebox{.015cm}{$l$}}\hspace*{.01cm}\cups$};
\end{tikzpicture}
,
\begin{tikzpicture}[anchorbase, scale=.4, tinynodes]
	\draw[arcdo] (0,0) to [out=270, in=0] (-.5,-1);
	\draw[arcdo] (1,0) to [out=270, in=180] (1.5,-1);
	\draw[arc] (0,0) to [out=90, in=180] (.5,1) to [out=0, in=90] (1,0);
	\draw[doline] (-.5,0) to (1.5,0);
	\node at (.5,1.5) {\text{upp.}};
	\node[scale=.8] at (.5,-1.5) {$\scalebox{1.5}{\raisebox{.015cm}{$u$}}\hspace*{.01cm}\caps$};
\end{tikzpicture}
;
\quad\quad
\begin{tikzpicture}[anchorbase, scale=.4, tinynodes]
	\draw[arc] (0,0) to [out=270, in=180] (.5,-1) to [out=0, in=270] (1,0);
	\draw[arcdo] (0,0) to [out=90, in=180] (2.5,2);
	\draw[arcdo] (1,0) to [out=90, in=180] (1.5,1) to [out=0, in=90] (2,0);
	\draw[arcdo] (-1,0) to [out=270, in=180] (.5,-2) to [out=0, in=270] (2,0);
	\draw[arcdo] (-1,0) to [out=90, in=330] (-1.5,2);
	\draw[doline] (-1.5,0) to (2.5,0);
	\node at (.5,-1.5) {\text{upp.}};
	\node[scale=.8] at (-.175,1.5) {$\scalebox{1.5}{\raisebox{.015cm}{$u$}}\hspace*{.01cm}\cups$};
\end{tikzpicture}
,
\begin{tikzpicture}[anchorbase, scale=.4, tinynodes]
	\draw[arc] (0,0) to [out=90, in=180] (.5,1) to [out=0, in=90] (1,0);
	\draw[arcdo] (0,0) to [out=270, in=180] (2.5,-2);
	\draw[arcdo] (1,0) to [out=270, in=180] (1.5,-1) to [out=0, in=270] (2,0);
	\draw[arcdo] (-1,0) to [out=90, in=180] (.5,2) to [out=0, in=90] (2,0);
	\draw[arcdo] (-1,0) to [out=270, in=30] (-1.5,-2);
	\draw[doline] (-1.5,0) to (2.5,0);
	\node at (.55,1.5) {\text{low.}};
	\node[scale=.8] at (-.175,-1.5) {$\scalebox{1.5}{\raisebox{.015cm}{$l$}}\hspace*{.01cm}\caps$};
\end{tikzpicture}
\end{gather}

We write e.g. $\ecup$ instead of $\ecup$ cup for short.

\begin{lemma}\label{lemma:orientation-cups-caps}
Let $\circles$ be a circle in an oriented circle diagram $\cbas{S,T}{\lambda}$.
\begin{enumerate}[label=(\alph*)]
\setlength\itemsep{.15cm}

\item If $\circles$ is usual, then the orientation of 
$\circles$ and any $\ecup$ or $\ecap$ agrees, 
while any $\icup$ or $\icap$ is oriented in the opposite way.

\item If $\circles$ is essential and leftwards, then any
$\lcup$ or $\lcap$ is oriented clockwise, while any 
$\ucup$ or $\ucap$ is oriented anticlockwise.

\item If $\circles$ is essential and rightwards, then any
$\ucup$ or $\ucap$ is oriented clockwise, while any 
$\lcup$ or $\lcap$ is oriented anticlockwise.\qedhere
\end{enumerate}
\end{lemma}

\begin{proof}
All of these are easily proved by induction on the number of cups and caps in the circle. 
Here the induction start:
\begin{gather}
\begin{tikzpicture}[anchorbase, scale=.4, tinynodes]
	\draw[arc] (0,0) to [out=270, in=180] (.5,-1) to [out=0, in=270] (1,0);
	\draw[arc] (0,0) to [out=90, in=180] (.5,1) to [out=0, in=90] (1,0);
	\draw[doline] (-.5,0) to (1.5,0);
	\node at (0,0) {$\Down$};
	\node at (1,0) {$\Up$};
	\node[scale=.8] at (.5,1.4) {$\scalebox{1.5}{\raisebox{.015cm}{$e$}}\hspace*{.01cm}\caps$};
	\node[scale=.8] at (.5,-1.6) {$\scalebox{1.5}{\raisebox{.015cm}{$e$}}\hspace*{.01cm}\cups$};
\end{tikzpicture}
;
\quad\quad
\begin{tikzpicture}[anchorbase, scale=.4, tinynodes]
	\draw[arc] (0,0) to [out=270, in=180] (.5,-1) to [out=0, in=270] (1,0);
	\draw[arc] (0,0) to [out=90, in=180] (.5,1) to [out=0, in=90] (1,0);
	\draw[doline] (-.5,0) to (1.5,0);
	\node at (1,0) {$\Down$};
	\node at (0,0) {$\Up$};
	\node[scale=.8] at (.5,1.4) {$\scalebox{1.5}{\raisebox{.015cm}{$e$}}\hspace*{.01cm}\caps$};
	\node[scale=.8] at (.5,-1.6) {$\scalebox{1.5}{\raisebox{.015cm}{$e$}}\hspace*{.01cm}\cups$};
\end{tikzpicture}
;
\quad\quad
\begin{tikzpicture}[anchorbase, scale=.4, tinynodes]
	\draw[arc] (0,0) to [out=270, in=180] (.5,-1) to [out=0, in=270] (1,0);
	\draw[arc] (0,0) to [out=90, in=0] (-.5,1);
	\draw[arc] (1,0) to [out=90, in=180] (1.5,1);
	\draw[doline] (-.5,0) to (1.5,0);
	\node at (1,0) {$\Down$};
	\node at (0,0) {$\Up$};
	\node[scale=.8] at (.5,1.4) {$\scalebox{1.5}{\raisebox{.015cm}{$u$}}\hspace*{.01cm}\caps$};
	\node[scale=.8] at (.5,-1.6) {$\scalebox{1.5}{\raisebox{.015cm}{$l$}}\hspace*{.01cm}\cups$};
\end{tikzpicture}
,
\begin{tikzpicture}[anchorbase, scale=.4, tinynodes]
	\draw[arc] (0,0) to [out=270, in=0] (-.5,-1);
	\draw[arc] (1,0) to [out=270, in=180] (1.5,-1);
	\draw[arc] (0,0) to [out=90, in=180] (.5,1) to [out=0, in=90] (1,0);
	\draw[doline] (-.5,0) to (1.5,0);
	\node at (0,0) {$\Down$};
	\node at (1,0) {$\Up$};
	\node[scale=.8] at (.5,1.4) {$\scalebox{1.5}{\raisebox{.015cm}{$u$}}\hspace*{.01cm}\caps$};
	\node[scale=.8] at (.5,-1.6) {$\scalebox{1.5}{\raisebox{.015cm}{$l$}}\hspace*{.01cm}\cups$};
\end{tikzpicture}
;
\quad\quad
\begin{tikzpicture}[anchorbase, scale=.4, tinynodes]
	\draw[arc] (0,0) to [out=270, in=180] (.5,-1) to [out=0, in=270] (1,0);
	\draw[arc] (0,0) to [out=90, in=0] (-.5,1);
	\draw[arc] (1,0) to [out=90, in=180] (1.5,1);
	\draw[doline] (-.5,0) to (1.5,0);
	\node at (0,0) {$\Down$};
	\node at (1,0) {$\Up$};
	\node[scale=.8] at (.5,1.4) {$\scalebox{1.5}{\raisebox{.015cm}{$u$}}\hspace*{.01cm}\caps$};
	\node[scale=.8] at (.5,-1.6) {$\scalebox{1.5}{\raisebox{.015cm}{$l$}}\hspace*{.01cm}\cups$};
\end{tikzpicture}
,
\begin{tikzpicture}[anchorbase, scale=.4, tinynodes]
	\draw[arc] (0,0) to [out=270, in=0] (-.5,-1);
	\draw[arc] (1,0) to [out=270, in=180] (1.5,-1);
	\draw[arc] (0,0) to [out=90, in=180] (.5,1) to [out=0, in=90] (1,0);
	\draw[doline] (-.5,0) to (1.5,0);
	\node at (1,0) {$\Down$};
	\node at (0,0) {$\Up$};
	\node[scale=.8] at (.5,1.4) {$\scalebox{1.5}{\raisebox{.015cm}{$u$}}\hspace*{.01cm}\caps$};
	\node[scale=.8] at (.5,-1.6) {$\scalebox{1.5}{\raisebox{.015cm}{$l$}}\hspace*{.01cm}\cups$};
\end{tikzpicture}
\end{gather}
Then one continues using \eqref{eq:zig-zag}.
\end{proof}

For the next two lemmas 
the circles are considered inside an oriented, 
stacked diagram where the surgery is performed. 
Note hereby that we apply (\ref{section:arc-stuff}.b) only, i.e. without 
reorienting the resulting diagram, but rather keeping 
the original orientation. We call this 
\textit{applying the surgery naively}.

\begin{lemma}\label{lemma:split-usual-ess}
Assume an essential circle $\circles$ splits 
into an essential $\circles_{e}$ and an usual 
$\circles_{u}$ circle by naive surgery. 
Then the resulting diagram is oriented, $\circles_{e}$ is 
oriented in the same way as $\circles$ and $\circles_{u}$ is 
oriented opposite to the orientation of the $\cups\,\text{-}\,\caps$ pair 
involved in the naive surgery.
\end{lemma}

\begin{proof}
First -- by (\ref{lemma:orientation-cups-caps}.b) 
and (\ref{lemma:orientation-cups-caps}.c) -- we know 
that the cup and cap involved in the naive surgery have the same 
orientation.
Thus, these are the local possibilities:
\begin{gather}
\begin{tikzpicture}[anchorbase, scale=.4, tinynodes]
	\draw[arcdo] (-1,0) to [out=270, in=180] (-.5,-1) to [out=0, in=270] (0,0);
	\draw[arcdo] (-1,3) to [out=90, in=180] (-.5,4) to [out=0, in=90] (0,3);
	\draw[arcdo] (-1,3) to (-1,0);
	\draw[arc] (0,0) to [out=90, in=180] (.5,1) to [out=0, in=90] (1,0);
	\draw[arc] (0,3) to [out=270, in=180] (.5,2) to [out=0, in=270] (1,3);
	\draw[arc] (1,0) to (1,-1);
	\draw[arc] (1,3) to (1,4);
	\draw[doline] (-1.5,0) to (1.5,0);
	\draw[doline] (-1.5,3) to (1.5,3);
	\node at (1,0) {$\Down$};
	\node at (0,0) {$\Up$};
	\node at (1,3) {$\Down$};
	\node at (0,3) {$\Up$};
	\node at (-.5,1.5) {$\circles$};
\end{tikzpicture}
\mapsto
\begin{tikzpicture}[anchorbase, scale=.4, tinynodes]
	\draw[arcdo] (-1,0) to [out=270, in=180] (-.5,-1) to [out=0, in=270] (0,0);
	\draw[arcdo] (-1,3) to [out=90, in=180] (-.5,4) to [out=0, in=90] (0,3);
	\draw[arcdo] (-1,3) to (-1,0);
	\draw[arc] (0,0) to (0,3);
	\draw[arc] (1,3) to (1,0);
	\draw[arc] (1,0) to (1,-1);
	\draw[arc] (1,3) to (1,4);
	\draw[doline] (-1.5,0) to (1.5,0);
	\draw[doline] (-1.5,3) to (1.5,3);
	\node at (1,0) {$\Down$};
	\node at (0,0) {$\Up$};
	\node at (1,3) {$\Down$};
	\node at (0,3) {$\Up$};
	\node at (.5,1.5) {$\circles_{e}$};
	\node at (-.5,1.5) {$\circles_{u}$};
\end{tikzpicture}
;\;
\begin{tikzpicture}[anchorbase, scale=.4, tinynodes]
	\draw[arcdo] (-1,0) to [out=270, in=180] (-.5,-1) to [out=0, in=270] (0,0);
	\draw[arcdo] (-1,3) to [out=90, in=180] (-.5,4) to [out=0, in=90] (0,3);
	\draw[arcdo] (-1,3) to (-1,0);
	\draw[arc] (0,0) to [out=90, in=180] (.5,1) to [out=0, in=90] (1,0);
	\draw[arc] (0,3) to [out=270, in=180] (.5,2) to [out=0, in=270] (1,3);
	\draw[arc] (1,0) to (1,-1);
	\draw[arc] (1,3) to (1,4);
	\draw[doline] (-1.5,0) to (1.5,0);
	\draw[doline] (-1.5,3) to (1.5,3);
	\node at (0,0) {$\Down$};
	\node at (1,0) {$\Up$};
	\node at (0,3) {$\Down$};
	\node at (1,3) {$\Up$};
	\node at (-.5,1.5) {$\circles$};
\end{tikzpicture}
\mapsto
\begin{tikzpicture}[anchorbase, scale=.4, tinynodes]
	\draw[arcdo] (-1,0) to [out=270, in=180] (-.5,-1) to [out=0, in=270] (0,0);
	\draw[arcdo] (-1,3) to [out=90, in=180] (-.5,4) to [out=0, in=90] (0,3);
	\draw[arcdo] (-1,3) to (-1,0);
	\draw[arc] (0,0) to (0,3);
	\draw[arc] (1,3) to (1,0);
	\draw[arc] (1,0) to (1,-1);
	\draw[arc] (1,3) to (1,4);
	\draw[doline] (-1.5,0) to (1.5,0);
	\draw[doline] (-1.5,3) to (1.5,3);
	\node at (0,0) {$\Down$};
	\node at (1,0) {$\Up$};
	\node at (0,3) {$\Down$};
	\node at (1,3) {$\Up$};
	\node at (.5,1.5) {$\circles_{e}$};
	\node at (-.5,1.5) {$\circles_{u}$};
\end{tikzpicture}
;\;
\begin{tikzpicture}[anchorbase, scale=.4, tinynodes]
	\draw[arcdo] (2,0) to [out=270, in=0] (1.5,-1) to [out=180, in=270] (1,0);
	\draw[arcdo] (2,3) to [out=90, in=0] (1.5,4) to [out=180, in=90] (1,3);
	\draw[arcdo] (2,3) to (2,0);
	\draw[arc] (0,0) to [out=90, in=180] (.5,1) to [out=0, in=90] (1,0);
	\draw[arc] (0,3) to [out=270, in=180] (.5,2) to [out=0, in=270] (1,3);
	\draw[arc] (0,0) to (0,-1);
	\draw[arc] (0,3) to (0,4);
	\draw[doline] (-.5,0) to (2.5,0);
	\draw[doline] (-.5,3) to (2.5,3);
	\node at (0,0) {$\Down$};
	\node at (1,0) {$\Up$};
	\node at (0,3) {$\Down$};
	\node at (1,3) {$\Up$};
	\node at (1.5,1.5) {$\circles$};
\end{tikzpicture}
\mapsto
\begin{tikzpicture}[anchorbase, scale=.4, tinynodes]
	\draw[arcdo] (2,0) to [out=270, in=0] (1.5,-1) to [out=180, in=270] (1,0);
	\draw[arcdo] (2,3) to [out=90, in=0] (1.5,4) to [out=180, in=90] (1,3);
	\draw[arcdo] (2,3) to (2,0);
	\draw[arc] (0,0) to (0,3);
	\draw[arc] (1,3) to (1,0);
	\draw[arc] (0,0) to (0,-1);
	\draw[arc] (0,3) to (0,4);
	\draw[doline] (-.5,0) to (2.5,0);
	\draw[doline] (-.5,3) to (2.5,3);
	\node at (0,0) {$\Down$};
	\node at (1,0) {$\Up$};
	\node at (0,3) {$\Down$};
	\node at (1,3) {$\Up$};
	\node at (.5,1.5) {$\circles_{e}$};
	\node at (1.5,1.5) {$\circles_{u}$};
\end{tikzpicture}
;\;
\begin{tikzpicture}[anchorbase, scale=.4, tinynodes]
	\draw[arcdo] (2,0) to [out=270, in=0] (1.5,-1) to [out=180, in=270] (1,0);
	\draw[arcdo] (2,3) to [out=90, in=0] (1.5,4) to [out=180, in=90] (1,3);
	\draw[arcdo] (2,3) to (2,0);
	\draw[arc] (0,0) to [out=90, in=180] (.5,1) to [out=0, in=90] (1,0);
	\draw[arc] (0,3) to [out=270, in=180] (.5,2) to [out=0, in=270] (1,3);
	\draw[arc] (0,0) to (0,-1);
	\draw[arc] (0,3) to (0,4);
	\draw[doline] (-.5,0) to (2.5,0);
	\draw[doline] (-.5,3) to (2.5,3);
	\node at (1,0) {$\Down$};
	\node at (0,0) {$\Up$};
	\node at (1,3) {$\Down$};
	\node at (0,3) {$\Up$};
	\node at (1.5,1.5) {$\circles$};
\end{tikzpicture}
\mapsto
\begin{tikzpicture}[anchorbase, scale=.4, tinynodes]
	\draw[arcdo] (2,0) to [out=270, in=0] (1.5,-1) to [out=180, in=270] (1,0);
	\draw[arcdo] (2,3) to [out=90, in=0] (1.5,4) to [out=180, in=90] (1,3);
	\draw[arcdo] (2,3) to (2,0);
	\draw[arc] (0,0) to (0,3);
	\draw[arc] (1,3) to (1,0);
	\draw[arc] (0,0) to (0,-1);
	\draw[arc] (0,3) to (0,4);
	\draw[doline] (-.5,0) to (2.5,0);
	\draw[doline] (-.5,3) to (2.5,3);
	\node at (0,0) {$\Up$};
	\node at (1,0) {$\Down$};
	\node at (0,3) {$\Up$};
	\node at (1,3) {$\Down$};
	\node at (.5,1.5) {$\circles_{e}$};
	\node at (1.5,1.5) {$\circles_{u}$};
\end{tikzpicture}
\end{gather}
$\circles_{e}$ is -- by assumption -- essential meaning that 
all other possible situations can be rotated into such positions.
\end{proof}

For two essential circles 
$\circles_{u}$ is \textit{above} $\circles_{l}$ if 
$\circles_{l}$ is contained in the 
lower half of $\circles_{u}$. We also say that 
$\circles_{l}$ is \textit{below} $\circles_{u}$.

\begin{lemma}\label{lemma:split-usual-two-essential}
Let $\circles$ be a usual circle splitting into 
two essential circles, $\circles_{u}$ being 
above $\circles_{l}$, by 
naive surgery. 
Then the result is oriented 
with $\circles_{u}$ being essential and leftwards 
and $\circles_{l}$ essential and rightwards 
in case $\circles$ is anticlockwise, and vice versa, in case 
$\circles$ is clockwise.
\end{lemma}

\begin{proof}
As before by (\ref{lemma:orientation-cups-caps}.a), we know 
that the $\cups\,\text{-}\,\caps$ of the naive surgery have the same 
orientation. 
Thus, there is an induced orientation on the result after naive surgery.

To see the second part of the claim, keeping
\begin{gather}
\begin{tikzpicture}[anchorbase, scale=.4, tinynodes]
	\draw[arcdo] (0,3) to [out=90, in=0] (-1.5,5);
	\draw[arcdo] (1,3) to [out=90, in=180] (2.5,5);
	\draw[arcdo] (-1,3) to [out=90, in=0] (-1.5,4);
	\draw[arcdo] (2,3) to [out=90, in=180] (2.5,4);
	\draw[arcdo] (0,0) to [out=270, in=180] (.5,-1) to [out=0, in=270] (1,0);
	\draw[arcdo] (-1,0) to [out=270, in=180] (.5,-2) to [out=0, in=270] (2,0);
	\draw[arcdo] (-1,0) to (-1,3);
	\draw[arcdo] (0,0) to (0,3);
	\draw[arc] (1,0) to [out=90, in=180] (1.5,1) to [out=0, in=90] (2,0);
	\draw[arc] (1,3) to [out=270, in=180] (1.5,2) to [out=0, in=270] (2,3);
	\draw[doline] (-1.5,0) to (2.5,0);
	\draw[doline] (-1.5,3) to (2.5,3);
	\node at (-.5,1.5) {$\circles$};
\end{tikzpicture}
\mapsto
\begin{tikzpicture}[anchorbase, scale=.4, tinynodes]
	\draw[arcdo] (0,3) to [out=90, in=0] (-1.5,5);
	\draw[arcdo] (1,3) to [out=90, in=180] (2.5,5);
	\draw[arcdo] (-1,3) to [out=90, in=0] (-1.5,4);
	\draw[arcdo] (2,3) to [out=90, in=180] (2.5,4);
	\draw[arcdo] (0,0) to [out=270, in=180] (.5,-1) to [out=0, in=270] (1,0);
	\draw[arcdo] (-1,0) to [out=270, in=180] (.5,-2) to [out=0, in=270] (2,0);
	\draw[arcdo] (-1,0) to (-1,3);
	\draw[arcdo] (0,0) to (0,3);
	\draw[arc] (1,0) to (1,3);
	\draw[arc] (2,0) to (2,3);
	\draw[doline] (-1.5,0) to (2.5,0);
	\draw[doline] (-1.5,3) to (2.5,3);
	\node at (-1.5,1.5) {$\circles_{l}$};
	\node at (.5,1.5) {$\circles_{u}$};
\end{tikzpicture}
\end{gather}
in mind, we use 
(\ref{lemma:righthand-rule}.a) and (\ref{lemma:righthand-rule}.b) 
with the interior of $\circles$ turning into the 
lower of $\circles_{u}$ and the upper of $\circles_{l}$.
\end{proof}

For the following lemma we use the evident notion 
of usual circles to be \textit{nested} inside other usual circles. 
(We also say that one circle is \textit{the outer} having the 
evident meaning.)

\begin{lemma}\label{lemma:reorient-larger}
Let $\cbas{S,T}{\lambda}$ be an oriented circle diagram.
\smallskip
\begin{enumerate}[label=(\alph*)]
\setlength\itemsep{.15cm}

\item Let $\cbas{S,T}{\mu}$ be obtained from 
$\cbas{S,T}{\lambda}$ by reorienting an anticlockwise 
circle $\circles$ clockwise, as well as reorienting an 
arbitrary number of clockwise circles nested inside $\circles$ anticlockwise. 
Then $\mu\ord{\ceps_{S}}\!\lambda$ 
and $\mu\ord{\ceps_{T}}\!\lambda$.

\item Assume that $T$ is of staying type and let 
$\cbas{S,T}{\mu}$ be obtained from $\cbas{S,T}{\lambda}$ 
by reorienting a leftwards circle $\circles$ rightwards, as 
well as reorienting an arbitrary number of rightwards circles 
below $\circles$ leftwards. 
Then $\mu\ord{\ceps_{T}}\!\lambda$.\qedhere
\end{enumerate}
\end{lemma}

\begin{proof}
\textit{(\ref{lemma:reorient-larger}.a).} 
We first use the rotation map $\rotright$ to obtain 
a diagram with $S$ 
of staying type. Then the statement $\mu\ord{\ceps_{S}}\!\lambda$ follows 
by the same arguments as in the usual case and is left to the reader. 
(For a similar proof see \cite[Lemma 7.7]{es1}.) The same can be done to 
obtain a diagram with $T$ of staying type giving $\mu\ord{\ceps_{T}}\!\lambda$.
\medskip

\noindent\textit{(\ref{lemma:reorient-larger}.b).} 
In this case we substitute all cups in $S$ that are not of 
staying type by cups of staying type that connect the same vertices 
to obtain a cup diagram $S^{\prime}$. Then the circle $\circles$ 
determines a circle $\circles^{\prime}$ in $S^{\prime}T^{\invo}$ containing 
the same caps as $\circles$. Observe that $\circles^{\prime}$ is then
anticlockwise. Hence, reorienting this we obtain -- 
by (\ref{lemma:reorient-larger}.a) -- 
a weight 
$\mu\ord{\ceps_{T}}\!\lambda$. 
If there are rightward circles below, they get transformed to 
clockwise circles nested inside $\circles^{\prime}$. So the statement 
also follows by (\ref{lemma:reorient-larger}.a), if some of these are reoriented.
\end{proof}

%%%%%%%%%%%%%%%%%%%%%%%%%%%
\subsection{Relative cellularity: Main proof}\label{subsection:the-proofs}
%%%%%%%%%%%%%%%%%%%%%%%%%%%

We can now proceed and finish with the 
proof of \fullref{theorem:mult-rule-arc} to obtain the main 
part of relative cellularity for $\aarc$.

\begin{proof}[Proof of \fullref{theorem:mult-rule-arc}]
We show a stronger statement. Namely the appropriate 
analog of the claim itself, but for each step within the multiplication process.
In each step the general idea is roughly as follows:
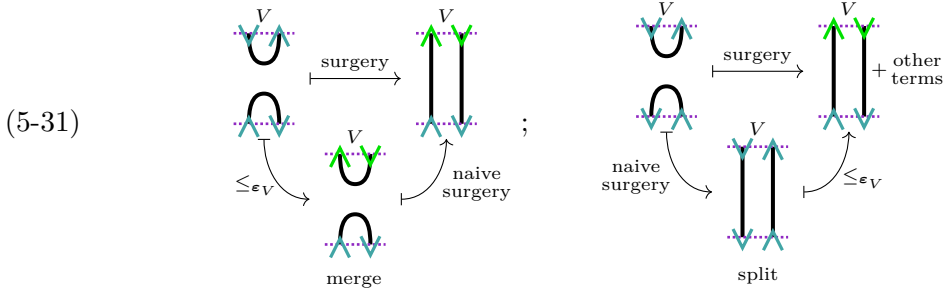
\begin{gather}
\xy
(0,0)*{
\begin{tikzpicture}[anchorbase, scale=.4, tinynodes]
	\draw[arc] (0,3) to [out=270, in=180] (.5,2) to [out=0, in=270] (1,3);
	\draw[arc] (0,0) to [out=90, in=180] (.5,1) to [out=0, in=90] (1,0);
	\draw[doline] (-.5,0) to (1.5,0);
	\draw[doline] (-.5,3) to (1.5,3);
	\node at (0,0) {$\Up$};
	\node at (1,0) {$\Down$};
	\node at (0,3) {$\Down$};
	\node at (1,3) {$\Up$};
	\draw[arc] (3,-1) to [out=270, in=180] (3.5,-2) to [out=0, in=270] (4,-1);
	\draw[arc] (3,-4) to [out=90, in=180] (3.5,-3) to [out=0, in=90] (4,-4);
	\draw[doline] (2.5,-4) to (4.5,-4);
	\draw[doline] (2.5,-1) to (4.5,-1);
	\node at (3,-4) {$\Up$};
	\node at (4,-4) {$\Down$};
	\node at (3,-1) {$\Upp$};
	\node at (4,-1) {$\Downn$};
	\draw[arc] (6,0) to (6,3);
	\draw[arc] (7,0) to (7,3);
	\draw[doline] (5.5,0) to (7.5,0);
	\draw[doline] (5.5,3) to (7.5,3);
	\node at (6,0) {$\Up$};
	\node at (7,0) {$\Down$};
	\node at (6,3) {$\Upp$};
	\node at (7,3) {$\Downn$};
	\node at (.5,3.5) {$V$};
	\node at (3.5,-.5) {$V$};
	\node at (6.5,3.5) {$V$};
	\draw[|->] (2,1.5) to (5,1.5);
	\node at (3.5,2) {$\text{surgery}$};
	\draw[|->] (.5,-.5) to [out=270, in=180] (2,-2.5);
	\draw[|->] (5,-2.5) to [out=0, in=270] (6.5,-.5);
	\node at (7.5,-1.7) {$\text{naive}$};
	\node at (7.5,-2.3) {$\text{surgery}$};
	\node at (.25,-2) {$\Ord{\ceps_{V}}$};
\end{tikzpicture}};
(-2.5,-20)*{\text{{\tiny merge}}};
\endxy
;
\quad\quad
\xy
(0,0)*{
\begin{tikzpicture}[anchorbase, scale=.4, tinynodes]
	\draw[arc] (0,3) to [out=270, in=180] (.5,2) to [out=0, in=270] (1,3);
	\draw[arc] (0,0) to [out=90, in=180] (.5,1) to [out=0, in=90] (1,0);
	\draw[doline] (-.5,0) to (1.5,0);
	\draw[doline] (-.5,3) to (1.5,3);
	\node at (0,0) {$\Down$};
	\node at (1,0) {$\Up$};
	\node at (0,3) {$\Down$};
	\node at (1,3) {$\Up$};
	\draw[arc] (3,-1) to (3,-4);
	\draw[arc] (4,-1) to (4,-4);
	\draw[doline] (2.5,-4) to (4.5,-4);
	\draw[doline] (2.5,-1) to (4.5,-1);
	\node at (3,-4) {$\Down$};
	\node at (4,-4) {$\Up$};
	\node at (3,-1) {$\Down$};
	\node at (4,-1) {$\Up$};
	\draw[arc] (6,0) to (6,3);
	\draw[arc] (7,0) to (7,3);
	\draw[doline] (5.5,0) to (7.5,0);
	\draw[doline] (5.5,3) to (7.5,3);
	\node at (6,0) {$\Up$};
	\node at (7,0) {$\Down$};
	\node at (6,3) {$\Upp$};
	\node at (7,3) {$\Downn$};
	\node at (.5,3.5) {$V$};
	\node at (3.5,-.5) {$V$};
	\node at (6.5,3.5) {$V$};
	\node at (7.5,1.5) {$+$};
	\node at (8.75,1.8) {$\text{other}$};
	\node at (8.75,1.2) {$\text{terms}$};
	\draw[|->] (2,1.5) to (5,1.5);
	\node at (3.5,2) {$\text{surgery}$};
	\draw[|->] (.5,-.5) to [out=270, in=180] (2,-2.5);
	\draw[|->] (5,-2.5) to [out=0, in=270] (6.5,-.5);
	\node at (-.5,-1.7) {$\text{naive}$};
	\node at (-.5,-2.3) {$\text{surgery}$};
	\node at (7,-2) {$\Ord{\ceps_{V}}$};
\end{tikzpicture}};
(-1.9,-20)*{\text{{\tiny split}}};
\endxy
\end{gather}
In words, we reorient before or after the surgery 
such that naive surgery gives the result we want to consider. 
In doing so the reordering will -- by \fullref{lemma:reorient-larger}
-- $\Ord{\ceps_{V}}$-decrease the weight. 
Observe hereby that this reorientation 
process is always possible. But in case of a merge the reorientation 
might happen for circles not touching the upper dotted line. 
(Examples are for instance 
provided by the merge rule (\ref{section:arc-stuff}.a).) Those cases need a bit more care, 
but this will only happen in \ref{arcthm-case3} below.

Let us make this rigorous.
To this end, let $S\lambda T^{\invo} T\mu V^{\invo}$ be a stacked 
diagram. Without loss of generality 
we also assume that the diagram 
is rotated in such a way that $V$ is of staying type.
Further, let $\caps$ denote a cap in $T^{\invo}$ and $\cups$ 
the mirrored cup in $T$ such that one can perform surgery 
with the pair $\cups\,\text{-}\,\caps$. 
In the following, let $\circles$ denote the 
circle containing $\caps$ and $\circles^{\prime}$ the 
circle containing $\cups$. (These need not be distinct in general.)
\medskip

\noindent\setword{\ref{theorem:mult-rule-arc}.Claim.a}{arcthm-1}.
After naive surgery along $\cups\,\text{-}\,\caps$ and reorientation 
one obtains diagrams with an orientation $\mu^{\prime}$ on the upper 
dotted line such that $\mu^{\prime}\ord{\ceps_{V}}\!\mu$. Further, if $\mu$ 
appears, then it appears with coefficient one, independent of $V$.
\medskip

\noindent Proof of \ref{arcthm-1}.
The proof is divided into three parts: 
First we assume that $\caps$ is oriented clockwise, 
then we assume that $\caps$ is anticlockwise and divide 
the cases of $\cups$ being anticlockwise respectively clockwise. 
In all cases we silently use \fullref{lemma:orientation-cups-caps}.
\medskip

\noindent\setword{\ref{theorem:mult-rule-arc}.Case.A}{arcthm-case1}: $\caps$ is clockwise.
We further distinguish depending on the 
properties of the circle $\circles$ that in turn imply further 
properties of $\caps$ and $\cups$.
\smallskip
\begin{enumerate}[label=(\roman*)]

\setlength\itemsep{.15cm}

\item \textit{$\circles$ is usual and anticlockwise.} 
This implies that $\caps$ is $\icap$.
\smallskip
\begin{enumerate}[label=(\arabic*)]

\setlength\itemsep{.15cm}

\item If the surgery is a merge of two circles, 
then $\circles^{\prime}$ must be 
nested inside $\circles$. Hence, $\circles^{\prime}$ is usual as well, 
and $\circles^{\prime}$ and 
$\cups$ have the same orientation. 
In particular, if $\circles^{\prime}$ and $\cups$ are anticlockwise, 
we need to reorient $\circles^{\prime}$ and $\cups$ clockwise 
and then perform the surgery naively. 
The resulting orientation $\mu^{\prime}$ 
on the upper dotted line is 
strictly $\ord{\ceps_{V}}$-smaller than $\mu$. 
If on the other hand 
$\circles^{\prime}$ and $\cups$ are clockwise, we need to 
reorient both $\circles$ and $\circles^{\prime}$ and then perform the surgery naively. 
In this case this 
also produces a $\mu^{\prime}$ strictly $\ord{\ceps_{V}}$-smaller than $\mu$.

\item If the surgery is a split, then $\cups$ is clockwise 
as well. Hence, the surgery will create two usual 
circles both containing arcs in $V$. Note that 
the naive surgery creates two circles that are 
usual and anticlockwise. Thus, for each summand of the 
result one of the two circles needs 
to be reoriented creating
strictly $\ord{\ceps_{V}}$-smaller orientations $\mu^{\prime}$.

\end{enumerate}

\item \textit{$\circles$ is usual and clockwise.} 
In this case $\caps$ is $\ecap$.
\smallskip
\begin{enumerate}[label=(\arabic*)]

\setlength\itemsep{.15cm}

\item If one merges, then the only non-zero result 
occurs when $\circles^{\prime}$ is usual and anticlockwise. 
To obtain the result we need to reorient 
$\circles^{\prime}$, and $\circles$ if it is nested 
inside $\circles^{\prime}$, and then perform naive surgery. 
Since $\circles^{\prime}$ contains arcs in $V$ this will
produce a strictly $\ord{\ceps_{V}}$-smaller orientation $\mu^{\prime}$.

\item If one splits, then $\cups$ is a clockwise $\ecup$. 
The only non-zero result is the split 
into two usual circles, both touching 
the upper dotted line. After performing naive surgery the outer of 
the two created circles is already clockwise, while the nested 
is anticlockwise. Reorienting the nested circle again gives a 
strictly $\ord{\ceps_{V}}$-smaller 
orientation $\mu^{\prime}$.

\end{enumerate}

\item \textit{$\circles$ is essential and leftwards.} 
In this case $\caps$ is $\lcap$.
\smallskip
\begin{enumerate}[label=(\arabic*)]

\setlength\itemsep{.15cm}

\item If the surgery is a merge, the non-zero cases are the ones 
where $\circles^{\prime}$ is 
usual and anticlockwise or essential and rightwards. In the first
case, we have that $\cups$ is anticlockwise as well. In this case $\circles^{\prime}$ 
needs to be reoriented, strictly $\ord{\ceps_{V}}$-decreasing the orientation $\mu^{\prime}$, 
and then naive surgery can be performed. In the second case, 
$\cups$ is clockwise. Performing naive surgery will 
then produce a usual and anticlockwise circle containing arcs in $V$. Thus, reorienting 
the resulting circle gives a $\ord{\ceps_{V}}$-strictly smaller 
orientation $\mu^{\prime}$.

\item If the surgery is a split, then also $\cups$ is 
clockwise. Performing naive surgery will thus produce a 
usual and anticlockwise and an essential and leftwards circle. Since 
both contain arcs in $V$, reorienting the former will again 
yield a strictly $\ord{\ceps_{V}}$-smaller 
orientation $\mu^{\prime}$ by \fullref{lemma:reorient-larger}.

\end{enumerate}

\item \textit{$\circles$ is essential and rightwards.} 
Very similar to the leftwards case and omitted.

\end{enumerate}
\smallskip

\noindent\setword{\ref{theorem:mult-rule-arc}.Case.B}{arcthm-case2}: $\caps$ and $\cups$ are anticlockwise.
In this case the result 
after naive surgery will always be automatically oriented, 
giving a coefficient $1$ for the orientation $\mu$. 
It remains to rule out the case that 
other summands are not $\ord{\ceps_{V}}$-strictly smaller than $\mu$.
\smallskip
\begin{enumerate}[label=(\roman*)]

\setlength\itemsep{.15cm}

\item We first assume that the surgery will be a merge. 
In case that two usual circles are merged, 
the result is already oriented in the correct 
way and no reorientation is necessary. In case 
that two essential circles are merged, note that 
either $\caps$ or $\cups$ is upper, while the other 
is lower. This means that one has an 
essential and leftwards and an essential and rightwards
circle. Since the 
result of the naive surgery is oriented clockwise, 
no reorientation is needed. Further, if the merge 
includes a usual and an essential circle, then 
usual circle is oriented anticlockwise. 
Thus, there is again no need for a reorientation after surgery.

\item Assume now that the surgery is a split. If it is 
a split into two usual circles, then 
original circle was anticlockwise. After 
naive surgery 
we get a usual and anticlockwise outer and a 
usual and clockwise nested circle.
Thus, we obtain this as a summand 
in the result and a summand where both circles are 
reoriented. But since both contain arcs in $V$, this 
creates a strictly $\ord{\ceps_{V}}$-smaller orientation $\mu^{\prime}$ 
on the upper dotted line. In case that the 
split creates a usual and an essential circle, 
then the usual circle is automatically anticlockwise after naive surgery. 
Finally, if the split creates two essential circles, then 
$\circles=\circles^{\prime}$ 
is anticlockwise. Further, the upper 
of the two created circles is essential and leftwards, while the lower 
is essential and rightwards after naive surgery. The second summand in the 
result is obtained by reorienting both circles, but since 
both contain arcs in $V$, we see 
that reorienting both will give a strictly $\ord{\ceps_{V}}$-smaller 
orientation $\mu^{\prime}$.

\end{enumerate}
\medskip

\noindent\setword{\ref{theorem:mult-rule-arc}.Case.C}{arcthm-case3}: $\caps$ is anticlockwise, $\cups$ is clockwise. 
This case is a bit different than the previous cases since
the result will depend on whether the circle 
$\circles$ contains arcs in $V$ or not, and what we show is that 
the result will always be 
independent of $V$.

Before we start, we note that, since the 
orientations of $\caps$ and $\cups$ are different, the surgery 
will always be a merge.
\smallskip
\begin{enumerate}[label=(\roman*)]

\setlength\itemsep{.15cm}

\item First assume that the circle $\circles$ does not contain 
arcs in $V$. If $\circles$ is usual, then a merge with an 
usual or essential circle $\circles^{\prime}$ will be 
performed by reorienting $\circles$ followed by naive surgery. 
Hence, always resulting in the weight $\mu$ in the result. 
In case $\circles$ is essential, the two possibilities for 
$\circles^{\prime}$ are either an essential circle, oriented in the same way as 
$\circles$, or $\circles^{\prime}$ being usual and clockwise. Both cases 
result in zero. Thus, in this case the result is independent of $V$.

\item If on the other hand $\circles$ contains arcs in $V$, then 
$\circles$ being usual will always strictly $\ord{\ceps_{V}}$-decrease the 
weight $\mu^{\prime}$ when $\circles$ is reoriented. While 
the case $\circles$ being essential, would still result in zero 
in all cases. Since in this case $\mu$ never occurs, its 
coefficient is again independent of $V$.

\end{enumerate}

In the last case, the condition whether $\circles$ contains 
arcs in $V$ or not is equivalent to asking whether swapping 
all entries in $\lambda$ contained in the circle $\circles$ 
would give an orientation of $\circles$ or not. If $\circles$ 
does not contain arcs in $V$ then it would just be the opposite 
orientation, while if $\circles$ contains arc in $V$, this would 
not result in an orientation as the orientation on the 
top is unchanged. Doing this for all surgery moves and always 
assuming the case that $\lambda$ appears in every step, thus 
implies that $V\in\cmset(\lambda)$.
\medskip

Taking all above together shows \ref{arcthm-1}. This in turn 
implies the statement.
\end{proof}
%

%%%%%%%%%%%%%%%%%%% end of paper %%%%%%%%%%%%%%%%%%%%%%%%%
\vspace*{-.15cm}

\begin{thebibliography}{GTW17}

\bibitem[AST18]{ast1}
H.H.~Andersen, C.~Stroppel, and D.~Tubbenhauer.
\newblock Cellular structures using {$\textbf{U}_q$}-tilting modules.
\newblock {\em Pacific J. Math.}, 292(1):21--59, 2018.
\newblock URL: \url{http://arxiv.org/abs/1503.00224}, \href
  {http://dx.doi.org/10.2140/pjm.2018.292.21}
  {\path{doi:10.2140/pjm.2018.292.21}}.

\bibitem[AN16]{an1}
R.~Anno and V.~Nandakumar.
\newblock Exotic {$t$}-structures for two-block {S}pringer fibers.
\newblock 2016.
\newblock URL: \url{https://arxiv.org/abs/1602.00768}.

\bibitem[APS04]{aps1}
M.M.~Asaeda, J.H.~Przytycki, and A.S.~Sikora.
\newblock Categorification of the {K}auffman bracket skein module of
  {$I$}-bundles over surfaces.
\newblock {\em Algebr. Geom. Topol.}, 4:1177--1210, 2004.
\newblock URL: \url{https://arxiv.org/abs/math/0409414}, \href
  {http://dx.doi.org/10.2140/agt.2004.4.1177}
  {\path{doi:10.2140/agt.2004.4.1177}}.

\bibitem[BPW19]{bpw1}
A.~Beliakova, K.K.~Putyra, and S.M.~Wehrli.
\newblock Quantum link homology via trace functor {I}.
\newblock {\em Invent. Math.}, 215(2):383--492, 2019.
\newblock URL: \url{https://arxiv.org/abs/1605.03523}, \href
  {http://dx.doi.org/10.1007/s00222-018-0830-0}
  {\path{doi:10.1007/s00222-018-0830-0}}.

\bibitem[BT17]{bt2}
G.~Bellamy and U.~Thiel.
\newblock Cores of graded algebras with triangular decomposition.
\newblock 2017.
\newblock URL: \url{https://arxiv.org/abs/1711.00780}.

\bibitem[BT18]{bt1}
G.~Bellamy and U.~Thiel.
\newblock Highest weight theory for finite-dimensional graded algebras with
  triangular decomposition.
\newblock {\em Adv. Math.}, 330:361--419, 2018.
\newblock URL: \url{https://arxiv.org/abs/1705.08024}, \href
  {http://dx.doi.org/10.1016/j.aim.2018.03.011}
  {\path{doi:10.1016/j.aim.2018.03.011}}.

\bibitem[BS11]{bs1}
J.~Brundan and C.~Stroppel.
\newblock Highest weight categories arising from {K}hovanov's diagram algebra
  {I}: cellularity.
\newblock {\em Mosc. Math. J.}, 11(4):685--722, 821--822, 2011.
\newblock URL: \url{http://arxiv.org/abs/0806.1532}.

\bibitem[CZ19]{cozh}
K.~Coulembier and R.~Zhang.
\newblock Borelic pairs for stratified algebras.
\newblock {\em Adv. Math.}, 345:53--115, 2019.
\newblock URL: \url{https://arxiv.org/abs/1607.01867}, \href
  {http://dx.doi.org/10.1016/j.aim.2019.01.002}
  {\path{doi:10.1016/j.aim.2019.01.002}}.

\bibitem[DR98]{duru}
J.~Du and H.~Rui.
\newblock Based algebras and standard bases for quasi-hereditary algebras.
\newblock {\em Trans. Amer. Math. Soc.}, 350(8):3207--3235, 1998.
\newblock \href {http://dx.doi.org/10.1090/S0002-9947-98-02305-8}
  {\path{doi:10.1090/S0002-9947-98-02305-8}}.

\bibitem[ES16a]{es2}
M.~Ehrig and C.~Stroppel.
\newblock {$2$}-row {S}pringer fibres and {K}hovanov diagram algebras for type
  {D}.
\newblock {\em Canad. J. Math.}, 68(6):1285--1333, 2016.
\newblock URL: \url{https://arxiv.org/abs/1209.4998}, \href
  {http://dx.doi.org/10.4153/CJM-2015-051-4}
  {\path{doi:10.4153/CJM-2015-051-4}}.

\bibitem[ES16b]{es1}
M.~Ehrig and C.~Stroppel.
\newblock Diagrammatic description for the categories of perverse sheaves on
  isotropic {G}rassmannians.
\newblock {\em Selecta Math. (N.S.)}, 22(3):1455--1536, 2016.
\newblock URL: \url{http://arxiv.org/abs/1511.04111}, \href
  {http://dx.doi.org/10.1007/s00029-015-0215-9}
  {\path{doi:10.1007/s00029-015-0215-9}}.

\bibitem[EST17]{est1}
M.~Ehrig, C.~Stroppel, and D.~Tubbenhauer.
\newblock The {B}lanchet--{K}hovanov algebras.
\newblock In {\em Categorification and higher representation theory}, volume
  683 of {\em Contemp. Math.}, pages 183--226. Amer. Math. Soc., Providence,
  RI, 2017.
\newblock URL: \url{http://arxiv.org/abs/1510.04884}, \href
  {http://dx.doi.org/10.1090/conm/683} {\path{doi:10.1090/conm/683}}.

\bibitem[ET18]{et-zigzag}
M.~Ehrig and D.~Tubbenhauer.
\newblock Algebraic properties of zigzag algebras.
\newblock 2018.
\newblock To appear in Comm. Algebra.
\newblock URL: \url{https://arxiv.org/abs/1807.11173}, \href
  {http://dx.doi.org/10.1080/00927872.2019.1632325}
  {\path{doi:10.1080/00927872.2019.1632325}}.

\bibitem[ETW16]{etw1}
M.~Ehrig, D.~Tubbenhauer, and A.~Wilbert.
\newblock Singular {TQFT}s, foams and type {D} arc algebras.
\newblock 2016.
\newblock To appear in Doc. Math.
\newblock URL: \url{https://arxiv.org/abs/1611.07444}.

\bibitem[FP88]{fp1}
E.M.~Friedlander and B.J.~Parshall.
\newblock Modular representation theory of {L}ie algebras.
\newblock {\em Amer. J. Math.}, 110(6):1055--1093, 1988.

\bibitem[GTW17]{gtw1}
A.~Gadbled, A.-L.~Thiel, and E.~Wagner.
\newblock Categorical action of the extended braid group of affine type {$A$}.
\newblock {\em Commun. Contemp. Math.}, 19(3):1650024, 39, 2017.
\newblock URL: \url{https://arxiv.org/abs/1504.07596}, \href
  {http://dx.doi.org/10.1142/S0219199716500243}
  {\path{doi:10.1142/S0219199716500243}}.

\bibitem[GG11]{GoGr-cellularity-jones-basic}
F.M.~Goodman and J.~Graber.
\newblock Cellularity and the {J}ones basic construction.
\newblock {\em Adv. in Appl. Math.}, 46(1-4):312--362, 2011.
\newblock URL: \url{https://arxiv.org/abs/0906.1496}, \href
  {http://dx.doi.org/10.1016/j.aam.2010.10.003}
  {\path{doi:10.1016/j.aam.2010.10.003}}.

\bibitem[GL96]{gl1}
J.J.~Graham and G.~Lehrer.
\newblock Cellular algebras.
\newblock {\em Invent. Math.}, 123(1):1--34, 1996.
\newblock \href {http://dx.doi.org/10.1007/BF01232365}
  {\path{doi:10.1007/BF01232365}}.

\bibitem[GLW18]{glw1}
E.J.~Grigsby, A.M.~Licata, and S.M.~Wehrli.
\newblock Annular {K}hovanov homology and knotted {S}chur--{W}eyl
  representations.
\newblock {\em Compos. Math.}, 154(3):459--502, 2018.
\newblock URL: \url{https://arxiv.org/abs/1505.04386}.

\bibitem[HM10]{hm1}
J.~Hu and A.~Mathas.
\newblock Graded cellular bases for the cyclotomic
  {K}hovanov--{L}auda--{R}ouquier algebras of type ${A}$.
\newblock {\em Adv. Math.}, 225(2):598--642, 2010.
\newblock URL: \url{http://arxiv.org/abs/0907.2985}, \href
  {http://dx.doi.org/10.1016/j.aim.2010.03.002}
  {\path{doi:10.1016/j.aim.2010.03.002}}.

\bibitem[HK01]{hk1}
R.S.~Huerfano and M.~Khovanov.
\newblock A category for the adjoint representation.
\newblock {\em J. Algebra}, 246(2):514--542, 2001.
\newblock URL: \url{https://arxiv.org/abs/math/0002060}, \href
  {http://dx.doi.org/10.1006/jabr.2001.8962}
  {\path{doi:10.1006/jabr.2001.8962}}.

\bibitem[Jan04]{ja1}
J.C.~Jantzen.
\newblock Representations of {L}ie algebras in positive characteristic.
\newblock In {\em Representation theory of algebraic groups and quantum
  groups}, volume~40 of {\em Adv. Stud. Pure Math.}, pages 175--218. Math. Soc.
  Japan, Tokyo, 2004.

\bibitem[Kho02]{kh1}
M.~Khovanov.
\newblock A functor-valued invariant of tangles.
\newblock {\em Algebr. Geom. Topol.}, 2:665--741, 2002.
\newblock URL: \url{http://arxiv.org/abs/math/0103190}, \href
  {http://dx.doi.org/10.2140/agt.2002.2.665}
  {\path{doi:10.2140/agt.2002.2.665}}.

\bibitem[KX12]{kx3}
S.~K{\"o}nig and C.~Xi.
\newblock Affine cellular algebras.
\newblock {\em Adv. Math.}, 229(1):139--182, 2012.
\newblock \href {http://dx.doi.org/10.1016/j.aim.2011.08.010}
  {\path{doi:10.1016/j.aim.2011.08.010}}.

\bibitem[KX98]{kx1}
S.~K{\"o}nig and C.~Xi.
\newblock On the structure of cellular algebras.
\newblock In {\em Algebras and modules, {II} ({G}eiranger, 1996)}, volume~24 of
  {\em CMS Conf. Proc.}, pages 365--386. Amer. Math. Soc., Providence, RI,
  1998.

\bibitem[KX99]{kx2}
S.~K{\"o}nig and C.~Xi.
\newblock When is a cellular algebra quasi-hereditary?
\newblock {\em Math. Ann.}, 315(2):281--293, 1999.
\newblock \href {http://dx.doi.org/10.1007/s002080050368}
  {\path{doi:10.1007/s002080050368}}.

\bibitem[Lus90]{lu1}
G.~Lusztig.
\newblock Finite-dimensional {H}opf algebras arising from quantized universal
  enveloping algebra.
\newblock {\em J. Amer. Math. Soc.}, 3(1):257--296, 1990.

\bibitem[Mac14]{m1}
M.~Mackaay.
\newblock The {$\mathfrak{sl}_n$}-web algebras and dual canonical bases.
\newblock {\em J. Algebra}, 409:54--100, 2014.
\newblock URL: \url{http://arxiv.org/abs/1308.0566}, \href
  {http://dx.doi.org/10.1016/j.jalgebra.2014.02.036}
  {\path{doi:10.1016/j.jalgebra.2014.02.036}}.

\bibitem[MPT14]{mpt1}
M.~Mackaay, W.~Pan, and D.~Tubbenhauer.
\newblock The {$\mathfrak{sl}_3$}-web algebra.
\newblock {\em Math. Z.}, 277(1-2):401--479, 2014.
\newblock URL: \url{http://arxiv.org/abs/1206.2118}, \href
  {http://dx.doi.org/10.1007/s00209-013-1262-6}
  {\path{doi:10.1007/s00209-013-1262-6}}.

\bibitem[MT16]{mt1}
M.~Mackaay and D.~Tubbenhauer.
\newblock Two-color {S}oergel calculus and simple transitive
  {$2$}-representations.
\newblock 2016.
\newblock To appear in Canad. J. Math.
\newblock URL: \url{https://arxiv.org/abs/1609.00962}, \href
  {http://dx.doi.org/10.4153/CJM-2017-061-2}
  {\path{doi:10.4153/CJM-2017-061-2}}.

\bibitem[QR18]{qr1}
H.~Queffelec and D.E.V.~Rose.
\newblock Sutured annular {K}hovanov--{R}ozansky homology.
\newblock {\em Trans. Amer. Math. Soc.}, 370(2):1285--1319, 2018.
\newblock URL: \url{https://arxiv.org/abs/1506.08188}.

\bibitem[Rob13]{r1}
L.P.~Roberts.
\newblock On knot {F}loer homology in double branched covers.
\newblock {\em Geom. Topol.}, 17(1):413--467, 2013.
\newblock URL: \url{https://arxiv.org/abs/0706.0741}, \href
  {http://dx.doi.org/10.2140/gt.2013.17.413}
  {\path{doi:10.2140/gt.2013.17.413}}.

\bibitem[Tub14]{t1}
D.~Tubbenhauer.
\newblock {$\mathfrak{sl}_3$}-web bases, intermediate crystal bases and
  categorification.
\newblock {\em J. Algebraic Combin.}, 40(4):1001--1076, 2014.
\newblock URL: \url{http://arxiv.org/abs/1310.2779}, \href
  {http://dx.doi.org/10.1007/s10801-014-0518-5}
  {\path{doi:10.1007/s10801-014-0518-5}}.

\bibitem[Xi02]{xi}
C.~Xi.
\newblock Standardly stratified algebras and cellular algebras.
\newblock {\em Math. Proc. Cambridge Philos. Soc.}, 133(1):37--53, 2002.
\newblock \href {http://dx.doi.org/10.1017/S0305004102005996}
  {\path{doi:10.1017/S0305004102005996}}.

\end{thebibliography}
\end{document}